\documentclass[11pt]{amsart}
\usepackage{amsmath, amssymb, amsthm, esint,   mathtools, mathrsfs, bbm}
\usepackage{geometry, mathabx}
\newcommand{\bd}{\boldsymbol}
\usepackage{cancel}
\usepackage{xcolor}
\usepackage[ colorlinks=true, citecolor=blue]{hyperref}

\def\R{{\mathbb R}}

\def\bd{\boldsymbol}
\newcommand{\cred}{\color{black}}

\newcommand{\ep}{\varepsilon}
\newcommand{\bu}{{\bd{u}}}
\newcommand{\bq}{{\bd{q}}}
\newcommand{\bg}{{\bf g}}
\newcommand{\bx}{{\bf x}}

\newcommand{\bP}{{\mathbb P}}
\newcommand{\sF}{\mathcal F}
\newcommand{\sO}{{\mathcal{O}}}
\newcommand{\bE}{\mathbb{E}}
\newcommand{\sB}{\mathcal{B}}

\newcommand{\bp}{\boldsymbol{p}}

\theoremstyle{definition}

\newtheorem{theorem}{Theorem}[section]
\newtheorem{corollary}{Corollary}[section]
\newtheorem{definition}{Definition}[section]
\newtheorem{proposition}{Proposition}[section]
\newtheorem{remark}{Remark}[section]
\newtheorem{lemma}{Lemma}[section]

\date{}
\author[J. Kuan and K. Tawri]{Jeffrey Kuan$^1$ and Krutika Tawri$^2$}

\address{\newline	$^1$ Department of Mathematics, University of Maryland, MD, USA.\\
\newline
$^2$  Department of Mathematics, University of California Berkeley, CA, USA.}
	
\email{jeffreyk@umd.edu (Jeffrey Kuan), ktawri@berkeley.edu (Krutika Tawri)}

\begin{document}
\title[Stochastic Compressible FSI]{Existence of weak martingale solutions to a stochastic fluid-structure interaction problem with a compressible viscous fluid}

\begin{abstract}
We study the existence of weak martingale solutions to a stochastic moving boundary problem arising from the interaction between an isentropic compressible fluid and a viscoelastic structure. In the model, we consider a three-dimensional compressible isentropic fluid with adiabatic constant $\gamma > 3/2$ interacting dynamically with an elastic structure on the boundary of the fluid domain described by a plate equation, under the additional influence of stochastic perturbations which randomly force both the compressible fluid and elastic structure equations in time. The problem is nonlinearly coupled in the sense that the a priori unknown (and random) displacement of the elastic structure from its reference configuration determines the a priori unknown (and random) time-dependent fluid domain on which the compressible isentropic Navier-Stokes equations are posed. We use a splitting method, consisting of a fluid and structure subproblem, to construct random approximate solutions to an approximate Galerkin form of the problem with artificial viscosity and artificial pressure. We introduce stopped processes of structure displacements which handle the issues associated with potential fluid domain degeneracy.
In this splitting scheme, we handle mathematical difficulties associated with the a priori unknown and time-dependent fluid domain by using an extension of the fluid equations to a fixed maximal domain and we handle difficulties associated with imposing the no-slip condition in the stochastic setting by using a novel penalty term defined on an external tubular neighborhood of the moving fluid-structure interface. To the best of our knowledge, this is the first existence result for stochastic fluid-structure interaction with compressible fluids. Novelties of this paper further include a first temporal regularity result for the structure displacement and velocity.
\end{abstract}
\maketitle
\section{Introduction}
In this article, we consider a fluid-structure interaction (FSI) problem involving isentropic compressible fluid in a three-dimensional domain, where part of the fluid domain boundary consists of an elastic deformable structure, and where the system is perturbed by stochastic effects. The stochastic forcing is applied to both the momentum equations as a volumetric body force and the structure as an external load to the deformable fluid boundary. The noise coefficients depend on the structure displacement, the structure velocity, the density of the fluid, and its momentum. The isentropic compressible fluid flow is governed by the 3D Navier-Stokes equations whereas Koiter shell-type equations give the elastodynamics of the structure. The adiabatic constant describing the power law for the fluid pressure is assumed to be $>\frac32$. The fluid interacts with the elastic structure at the fluid-structure interface via a two-way coupling, in which the fluid dynamics and elastodynamics of the structure mutually interact with each other. This two-way coupling ensures the continuity of velocities and of contact forces at the randomly moving fluid-structure interface. 
The main result of this manuscript is a  proof of the existence of a local-in-time (until an almost surely positive stopping time) weak martingale solution to this nonlinearly coupled stochastic fluid-structure interaction problem. That is, we prove the existence of solutions that are weak in the analytical sense and in the probabilistic sense, thereby proving the robustness of the underlying deterministic benchmark fluid-structure interaction problem to external stochastic noise.
To the best of our knowledge, {\bf this is the first result in the field of stochastic moving boundary problems involving compressible fluids, specifically when the random and time-dependent displacement of the fluid domain is not known a priori and is itself an unknown in the problem.}


One of the fundamental difficulties in this problem is the random moving domain, which is a priori unknown, on which the fluid equations are posed. 
The fluid-structure interaction is described by two-way coupling conditions, namely, the dynamic coupling condition and the kinematic coupling condition, which results in a problem that is highly nonlinear and presents numerous mathematical challenges. 
In order to handle the dynamic coupling condition between the fluid and the structure, we use a splitting scheme approach in the spirit of \cite{SarkaHeatFSI}, which separates the elastodynamics of the structure and the fluid after discretizing in time via an approximation parameter $\Delta t$. This involves running two subproblems, a fluid subproblem and a structure subproblem, over discretized time intervals of length $\Delta t$, such that the net total contribution of both subproblems on a single time interval, summed over all time intervals, is a discretized approximation of the limiting weak formulation that will converge as the parameter $\Delta t \to 0$.
In the first (structure) subproblem, only the elastodynamics equations are updated while keeping the fluid entities the same as in the previous step. In the second (fluid) subproblem, which is posed on a maximal domain,
the fluid density and velocity are updated, where the fluid subproblem is supplemented with a viscous regularization term for the density in the continuity equation via parameter $\ep$ and an artificial pressure term in the momentum equations via parameter $\delta$. 
These two layers of approximations, inspired by the method of Feireisl \cite{FeireislCompressible}, are intended to be compatible with the Galerkin approximation and to improve the integrability of the pressure.

We address here two additional issues that we come across in the construction of the second subproblem, caused by the stochasticity in the problem. Firstly, we note that the test functions for the fluid and the structure problems are required to satisfy the kinematic coupling condition, which is the essential no-slip boundary condition imposed at the random and time-varying fluid-structure interface. This condition, which is typically embedded in the test space, requires that we define this subproblem in terms of {\it stochastic test functions}.
This is not amenable to the fact that we are constructing martingale solutions.
We overcome this issue by decoupling the structure and the fluid equations along the kinematic coupling condition via a penalty method that essentially treats the structure as a semi-permeable material of thickness of some order of $\delta$.
In particular, at the approximate level, the no-slip condition is not satisfied but it is recovered in the final limit passage of the approximate solutions.
To be precise, we penalize the boundary behavior of the fluid and the structure velocities in a tubular neighborhood of size of some order of $\delta$ outside of the moving interface. This gives us a better control on the fluid density outside of the moving domain (caused by the seepage) that other Brinkman-type penalization terms fail to offer. We emphasize that {\bf the penalty term that we use is new and unique to our current problem}. While this penalization term in the approximate weak formulations allows us to consider a decoupled pair of deterministic test functions, it causes further analytical issues, discussed below, as it allows for seepage of the fluid through the structure.

The second issue is associated with the fact that the fluid domain can degenerate in a random fashion i.e. when the top compliant boundary comes in contact with the bottom rigid boundary. We handle this no-contact condition by introducing appropriate stopped processes that we call ``artificial structure displacements". The construction of these processes provides a deterministic upper bound for the $H^s$-norm, for a fixed $s<2$, of the structure displacement, and ensures that self-interaction of the artificial structure does not occur {\it at any time} by maintaining a minimum distance of $\alpha>0$ between the lateral walls. The discrepancy between the artificial and the original structure displacement is eventually resolved by using a stopping time argument, as in \cite{TC23}, which provides the time-length of existence of the solution. That is, we prove that the two displacements coincide until an almost surely positive stopping time. These deterministic bounds further enable us to work with a maximal domain that contains all the artificial structure displacements.
We are thus able to extend the viscosity coefficients with respect to the artificial structure displacements and pose the second (momentum) subproblem on a {\it fixed maximal domain} containing all associated artificial fluid domains. The method of extending the fluid problem to a maximal domain via extension of the viscosity coefficients has been previously employed in the study of deterministic compressible Navier-Stokes equations on time-dependent domains \cite{FNS2011, FKNNS13, KMNWK18} and deterministic compressible fluid-structure interaction, for example in \cite{SarkaHeatFSI}.


We then derive tightness results for the laws of the approximate solutions, in appropriate phase spaces, by employing compactness arguments. Our noise coefficient depends on the structure displacement, structure velocity, fluid momentum, and fluid density. Passage of the approximation parameters to their limits and the structure of our noise coefficient require almost sure convergence of the structure velocity strongly in $L^2_tL^2_x$. For that purpose, we derive a {\bf novel regularity result for the structure velocity which gives uniform bounds for the fractional time derivative, of some order strictly less than $\frac12$}. This result combined with the Aubin-Lions theorem gives us the tightness of the laws of the structure velocity in $L^2_tL^2_x$ (see Lemma \ref{lem:vtight}). {\bf This is the first temporal regularity result for the structure in this setting of FSI involving a compressible fluid}. Next, due to the nature of compressible Navier-Stokes equations, we obtain tightness of the laws of the fluid entities in the weak topology of their respective phase spaces. 
Thus, upgrading to almost sure convergence on these non-Polish phase spaces requires a variant of the Skorohod Representation theorem which is obtained by a composition of the results in \cite{NTT21} and \cite{J97}, that provides the existence of a sequence of new random variables (on a new probability space) with the same laws as the original variables that converge almost surely in the topology of the phase space, where the new probability space can be taken canonically to be the same probability space $[0, 1]^2$ with the Borel sigma algebra and Lebesgue measure, {\it and where the new Brownian motions on the same probability space are the same (i.e. parameter independent)}. This will be important in order to construct filtrations which are compatible with the process of taking the limit, in the final limit $\delta \to 0$, where we have to construct and work with {\it random} test processes so that the penalty term drops out. More precisely, the approximate random test functions will depend on both the $\delta$-approximate and the limiting structure displacements. Since the stochastic force appears in the momentum equations as a volumetric body force, we must ensure that these test processes are adapted to the filtration that we construct. Hence we require the filtrations for the approximate problems at each $\delta > 0$ approximation level on the new probability space to contain the filtration generated by the limiting solution too. However, constructing a filtration by enlarging the natural filtrations of the approximate solutions by that of the limiting solutions may give us a filtration which is not non-anticipative with respect to the Brownian motions i.e. that their increments in time may not be independent with respect to this enlarged filtration.
Having the same Brownian motion on the new probability space which is independent of $\delta$ fixes this issue and allows us to rigorously pass to the limit as $\delta \to 0$ using random test functions. 

We note here that the usual use of the Bogovski operator to obtain higher integrability of the pressure fails due to the fact that our structure displacement is not Lipschitz continuous in space.  Hence, by constructing an appropriate random test function, we provide better pressure estimates in the interior of the moving domain away from the fluid-structure interface. We then prove that the pressure does not concentrate at the moving interface by using the ideas given in \cite{K09} to deal with rough boundaries and then adapted to the moving domain (deterministic) case in \cite{BreitSchwarzacherNSF}.

Using these almost sure convergence results we are able to pass the approximation parameters to their respective limits in the following order: first we pass the time-discretization step $\Delta t \to 0$ in Section \ref{sec:timedis}, then the Galerkin parameter to $N\to\infty$ in Section \ref{sec:galerkin}, then the viscous regularization parameter to $\ep\to 0$ in Section \ref{sec:viscousreg}, and finally, in Section \ref{sec:delta}, the pressure regularization, the penalty and the extension parameter to $\delta\to 0$.

 
We must also address that even though we have extended the approximate fluid subproblem to a maximal domain, the dynamics, in the limit, really are occurring in only the physical domain. In the {\it deterministic} case, this is accomplished by a vanishing of density result (see  Lemma 3.1 in \cite{SarkaHeatFSI} and Lemma 4.1 in \cite{FKNNS13}), which shows that the vacuum outside of the initial domain is transported outside of the physical domain at all times.
In contrast to the case of deterministic analysis, we do not have that our $\delta$-level approximate densities are zero on the part of the maximal domain that is outside of the physical domain in the stochastic regime, since we cannot handle the random kinematic coupling of test functions until the final limit passage. Thus, we must develop new estimates that show that the total mass of the fluid in the maximal domain that is outside of the physical domain, while not identically zero necessarily, converges to zero, as some order of $\delta$, in the limit as $\delta \to 0$, see Proposition \ref{vacuum2}. These estimates are important in particular, for considering the limit of the advection term, which at the approximate level is posed on the entire maximal domain.

In addition, the final limit passage as $\delta \to 0$ requires careful consideration of moving domains. Because of the way that we have extended the viscosity coefficients for example, we only have uniform estimates of the gradient of the fluid velocity on the moving domain and not on the entire fixed maximal domain. 
Throughout the limit passage as $\delta \to 0$, we require careful consideration of how indicator functions of the moving domain and the exterior tubular neighborhood on which the penalty term is defined, affect the analysis. 
The convergence of stochastic integrals in the limit as $\delta \to 0$ in Lemma \ref{stochint} is in the spirit of convergence results for stochastic compressible fluid dynamics on fixed domains in Chapter 4 of \cite{BFH18}, which involve rigorously identifying how having uniform bounds on approximate solutions allows for weak convergence of compositions of Carath\'{e}odory functions (satisfying appropriate growth conditions) with the approximate solutions, see Theorem \ref{skorohod} for the statement of this result in the context of our current stochastic moving boundary problem. 
This is combined with deterministic convergence lemmas for functions defined on moving domains from \cite{BreitFSI} in order to properly identify the weak limits of \textit{random} nonlinear quantities on \textit{random} moving domains, see Lemma \ref{rhousquared}. To handle the fact that we only have uniform bounds on the fluid velocity on moving domains rather than the entire maximal domain, we also prove a \textit{new extension result} in Theorem \ref{thm:extension}, which states that the extension by zero of an $H^{s}(\mathcal{O})$ function, for a domain $\sO$ which is a hypograph of an $\alpha$-H\"{o}lder continuous function, is in $H^{s\alpha}(\R^{3})$ for all $0 \le s < \alpha/2$. Since domains with boundaries that are $\alpha$-H\"{o}lder continuous but not Lipschitz are found often in fluid-structure interaction problems, we believe that such an extension by zero regularity result is also of independent interest.

One of the final challenges in the last limit passage as $\delta \to 0$ is the fact that we must construct random test functions for the approximate and limiting weak formulations with appropriate measurability and convergence properties and boundary behaviour. 
Namely, the random test pair for the approximate weak formulation at this stage must be adapted, must satisfy the kinematic coupling condition so that the penalty term drops out, and it must converge to the test pair for the limiting weak formulation in appropriately regular spaces as $\delta\to 0$. 
One cannot find these approximate test functions by using usual tricks involving restricting or transforming, via the Arbitrary Lagrangian-Eulerian map, the limiting fluid test function on the approximate fluid domains as these methods inherit the regularity of the structure displacement which, in our case, is only H\"older continuous in space.
Hence, we provide a detailed construction for such a test pair in Section \ref{sec:testfunction}, previously missing from existing literature, that satisfies the requirements mentioned above while possessing appropriate measurability properties (adaptedness) as necessitated by the stochastic integral defined on moving domains.

\section{The nonlinearly coupled stochastic FSI model}

We will consider a fluid-structure interaction model for a compressible fluid in a three-dimensional domain, where part of the boundary is an elastic deformable boundary, and where the system is perturbed by stochastic effects, involving stochastic perturbations in the balance of momentum of the compressible fluid and in the elastodynamics of the structure. 
\begin{figure}[h]
    \centering
    \includegraphics[scale=0.5]{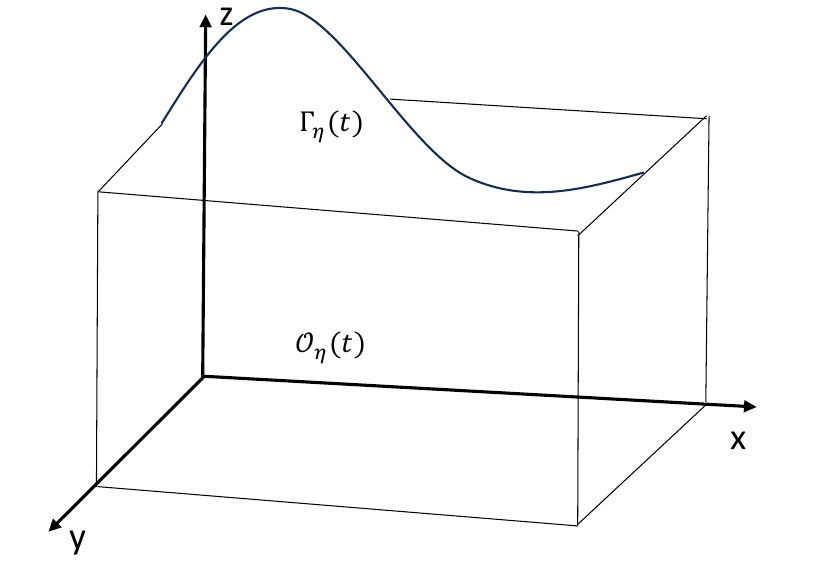}
    \caption{A realization of the fluid domain at some time $t\in[0,T]$}
    \label{domain}
\end{figure}

We will consider the fluid flow in a periodic channel interacting with a complaint structure that sits atop the fluid domain. 
The fluid reference domain, is given by 
$$\sO =  \Gamma\times (0,1),\qquad {\cred \Gamma= \mathbb{T}^2},$$
where $\mathbb{T}^2$ is the 2D torus.
The bottom and top parts of the boundary $\partial\sO$ will be denoted by
{\cred \begin{equation*}
 \Gamma_b = \Gamma \times \{z=0\}, \quad \Gamma_{top} = \Gamma \times \{z=1\}.
\end{equation*}}
We will assume that the elastic structure at the top boundary of $\mathcal{O}$ displaces in only the transverse $z$ direction so that $\eta$ will be the time-dependent (scalar) displacement of the elastic structure from its reference configuration $\Gamma$ in the $z$ direction. Given a particular function $\eta(t, x, y)$ describing the scalar transverse displacement of the elastic structure from $\Gamma$, we can define the time-dependent fluid domain at time $t\in[0,T]$ by (see Fig. \ref{domain}):
\begin{equation}\label{Oeta}
\mathcal{O}_\eta(t) = \{(x, y, z) \in {\mathbb{R}^{3}}: (x, y) \in \Gamma, 0 \le z \le 1 + \eta(t, x, y)\},
\end{equation}
and the top boundary of $\mathcal{O}_\eta(t)$ will be the physical time-dependent configuration of the elastic structure at time $t$, given by
\begin{equation}\label{gammaeta}
\Gamma_{\eta}(t) = \{(x, y, 1 + \eta(t, x, y) : (x, y) \in \Gamma\}.
\end{equation}
We will assume that a compressible, isentropic, viscous fluid occupies the three-dimensional domain $\mathcal{O}_\eta(t)$ at time $t$ and that the dynamics of the fluid and the structure influence each other.

\medskip

\noindent \textbf{Description of the fluid and structure subproblems.} We now describe each of the subproblems separately. 

For the \textbf{fluid subproblem}, we model the fluid velocity $\bu$ and the fluid density $\rho$ by the compressible Navier-Stokes equations for a viscous fluid:
\begin{equation}\label{eqn:fluid}
\left\{\begin{array}{l}
\partial_{t} \rho + \text{div}(\rho u) = 0 \\
\partial_{t}(\rho \bu) + \text{div}(\rho \bu \otimes \bu) - \nabla \cdot \bd{\sigma}(\rho, \bu) = F_{f} \\
\end{array}\right. \quad \text{ on } \mathcal{O}_\eta(t),
\end{equation}
where the Cauchy stress tensor $\bd{\sigma}(\rho, \bu)$ is defined by
\begin{equation*}
\bd{\sigma}(\rho, \bu) = 2\mu \left(\bd{D}(\bu) - \frac{1}{3} \text{div}(\bu) \bd{I}\right) + \left(\pi + \frac{2}{3}\mu\right)  \text{div}(\bu) \bd{I} - p \bd{I},
\end{equation*}
where $\displaystyle \bd{D}(\bu) = \frac{1}{2} (\nabla \bu + (\nabla\bu)^T)$ is the symmetrized gradient of the velocity and $F_{f}$ is an external force, to be specified later, which will eventually represent stochastic perturbations of the fluid momentum. Moreover, we assume that $$\lambda=\pi+\mu>0.$$
The first equation is the \textit{continuity equation}, which represents the conservation of mass, and the second equation represents the balance of momentum. Here, $\mu, \lambda > 0$ are the viscosity coefficients, and $p$ is the fluid pressure, which is given as a function of the fluid density $\rho$ through a constitutive relationship for the problem. In the current setting, we will assume that the viscous compressible fluid is \textit{isentropic} so that the pressure law is given by a power law:
\begin{equation*}
p(\rho) = \rho^{\gamma}, \text{ for some positive constant } \gamma > \frac{3}{2}.
\end{equation*}
{\cred We prescribe periodic boundary condition for the fluid velocity in $x$ and $y$ direction and prescribe the no-slip boundary condition on the bottom boundary i.e. $\bu|_{\Gamma_b}=0$.}

For the \textbf{structure subproblem}, we model the evolution of the structure displacement $\eta$ by the Koiter shell equation:
\begin{equation}\label{eqn:structure}
\partial_{tt} \eta - \alpha \partial_{t}\Delta \eta + \Delta^{2} \eta = F_{s}, \qquad \text{ on } \Gamma,
\end{equation}
where we recall that $\eta$ is the transverse scalar displacement in the $z$ direction of the elastic structure from its reference configuration $\Gamma$. We emphasize that while the fluid equations are posed on the physical time-dependent domain, the elastodynamics of the structure are prescribed in the Eulerian coordinates on the reference domain. Here, the constant $\alpha$ represents a viscoelasticity parameter and we assume that {{$\alpha > 0$}} in the current manuscript to obtain an existence result. 
We supply this elastodynamics equation for the structure displacement $\eta$ with periodic boundary conditions.

\medskip

\noindent \textbf{Description of the coupling conditions.} We next couple these two subproblems together with the following coupling conditions, which are the kinematic coupling condition describing continuity of velocities and the dynamic coupling condition describing the load of the fluid on the elastodynamics of the structure:

\begin{itemize}
\item The \textit{kinematic coupling condition} describes the continuity of velocities via the no-slip condition along the moving (time-dependent) fluid-structure interface so that:
\begin{equation}\label{eqn:kinbc}
\bu = (\partial_{t}\eta) \bd{e}_{z} \qquad \text{ on } \Gamma_\eta(t),
\end{equation}
where $\bd{e}_{z}$ is the unit normal vector in the $z$ direction, and where we recall that the structure is assumed to displace in only the transverse $z$ direction.
\item The \textit{dynamic coupling condition} describes the load of the fluid on the structure via the Cauchy stress tensor. This condition specifies the form of the external force on the structure as
\begin{equation}\label{eqn:dynbc}
F_{s} = -\bd{\sigma}(\rho, \bu) \bd{n}_\eta(t) \cdot \bd{e}_{z} + F_{s}^{stoch},
\end{equation}
where $\bd{n}_\eta(t)$ is the time-dependent {{outward}} normal vector to the moving fluid-structure interface $\Gamma_\eta(t)$, and where $F_{s}^{stoch}$ is an external stochastic forcing acting on the structure that will be specified later.

\end{itemize}

\medskip

\noindent \textbf{Description of the maximal domain and the stochastic forcing.} Next, we specify the nature of the stochastic forcing on the structure and the fluid. Let $W^1$ and $W^2$ be independent cylindrical Wiener processes in a separable Hilbert space $\mathcal{U}_0$. Letting $\{\bd{e}_{k}\}_{k = 1}^{\infty}$ be an orthonormal basis for $\mathcal{U}$, we can consider a Wiener process to formally be the sum of $\beta_{k}(t) \bd{e}_{k}$ over the positive integers $k$, where $\beta_k$ is a standard 1D Brownian motion. We will let the dynamics of the problem depend nonlinearly on the stochastic noise, as a function of the fluid density and the fluid momentum pointwise in the moving domain $\mathcal{O}_\eta(t)$. 

We consider nonlinear noise intensity operators {{$\bd{F}$ and $G$}} such that $\bd{F}(\rho,\rho\bu): \mathcal{U}_{0} \to {H}^{-l}(\R^3)$ for $l>\frac32$ and $G{(\eta,v)}: \mathcal{U}_{0} \to L^{2}(\Gamma)$ are Hilbert-Schmidt. 
\ 
We define, 
\begin{equation*}
F_{f} = \bd{F}(\rho, \rho \bu ) dW^1(t), \qquad F_s^{{\cred stoch}} = G(\eta, v) dW^2(t),
\end{equation*}
For given functions $\rho: \mathcal{O}_{\alpha} \to \mathbb{R}$ and $\bu: \mathcal{O}_{\alpha} \to \mathbb{R}^{3}$, we define
\begin{equation*}
f_{k}(\rho, \rho \bu, \cdot) = \bd{F}(\rho, \rho \bu) \bd{e}_{k}, \qquad g_{k}(\eta, v, \cdot) = G(\eta, v) \bd{e}_{k}.
\end{equation*}
We assume the following noise structure: There exists a sequence of positive constants $\{c_{k}\}_{k = 1}^{\infty}$ with $\displaystyle \sum_{k = 1}^{\infty} c_{k}^{2} < \infty$ such that for all positive integers $k$, the functions {$f_{k}: [0, \infty) \times \R^{3} \times \R^3 \to \R$ and $g_{k}: \R \times \R \times \Gamma \to \R$, which are  $C^{1}$ functions} in all of the inputs, satisfy: 
\begin{align}
&|f_{k}(\rho, \bd{q},x)| \le c_{k}(\rho + |\bd{q}|), \quad |\nabla_{\rho, \bd{q}} f_{k}(\rho, \bd{q}, x)| \le c_{k}.
\label{fassumption} \\
&
|g_{k}(\eta, v, x)| \le c_{k}(|\eta| + |v|), \quad |\nabla_{\eta, v} g_{k}(\eta, v, x)| \le c_{k}.\label{gassumption}
\end{align}

Note in particular that \eqref{fassumption} implies that the intensity of the noise in the momentum equations is identically zero whenever the fluid density is equal to zero which is consistent with deterministic theory. Moreover, this allows us to naturally extend the noise $\bd{F}dW$ outside of the moving domain to the entire space {{$\R^3$}} by assuming that the fluid density is equal to 0 outside of the moving domain. 

\subsection{Literature review}

In this manuscript, we consider a stochastic model of fluid-structure interaction involving the coupled dynamics interaction between a viscous, compressible, isentropic fluid and a viscoelastic Koiter shell, perturbed by stochastic effects in both the fluid and structure elastodynamics. Problems involving compressible fluid dynamics and fluid-structure interaction in both the deterministic and stochastic settings have been of considerable interest in the past mathematical literature.

The mathematical study of existence of weak solutions to the (deterministic) compressible Navier-Stokes equations in $\R^{3}$ goes back to the seminal work of P.-L. Lions in \cite{LionsCompressible, LionsBook}, whose work first establishes the existence of weak solutions to the compressible Navier-Stokes equations where an isentropic compressible fluid is considered with a constitutive pressure law of $p(\rho) = \rho^{\gamma}$ for $\gamma > 9/5$. These results were extended by another important work in \cite{FeireislCompressible}, which uses a multilayer approximation scheme involving a Galerkin approximation along with artificial viscosity and artificial pressure to boost integrability of the pressure to establish existence of weak solutions for adiabatic constants $\gamma > 3/2$, and this multilayer approximation scheme forms the basis for many existence results in the literature for systems involving compressible Navier-Stokes equations, including the one found in the current manuscript. 

Analysis of compressible viscous fluid flows was later extended to the case of compressible fluid flows perturbed by stochastic (random in time) noise. Preliminary work in \cite{Tornatore1, Tornatore2} considered compressible Navier-Stokes equations in one spatial dimension with random perturbations, and work in higher dimensions was achieved first in \cite{Tornatore2DCompressible} in two spatial dimensions and in \cite{FeireislStochastic} for three spatial dimensions with random noise of the form $\rho dW$, though the techniques presented in these two works are ``semi-deterministic" in the sense that the weak formulation can be written without stochastic integrals using an appropriate transformation. 

A subsequent fundamental work in the mathematical analysis of stochastic compressible fluid flows is \cite{BreitHofmanova}, which studies the stochastic compressible Navier-Stokes equations with multiplicative noise. In this work, the existence of \textit{``finite energy weak martingale solutions"} for stochastic isentropic compressible Navier-Stokes equations with adiabatic constant $\gamma > 3/2$ is established following the aforementioned multilevel approximation scheme, but a new application of stochastic analysis techniques. Independently, another fundamental work in the analysis of such stochastic compressible flows is \cite{Smith}, which establishes global existence of weak martingale solutions for $\gamma > 9/5$ for the 3D stochastic isentropic compressible Navier-Stokes, using a similar multilevel approximation, but also using a new application of a \textit{operator splitting technique} (inspired by work in \cite{BerthelinVovelle}), which separates the stochastic and deterministic components of the problem in a time-discretized splitting scheme. 

The study of well-posedness for compressible flows was later extended to the study of coupled deterministic systems involving elastic structures and solids interacting dynamically with compressible viscous fluids, in the context of fluid-structure interaction (FSI). These models of compressible fluid-structure interaction are of two types: (1) immersed elastic bodies in a compressible viscous fluid, and (2) compressible viscous fluids interacting with shells, plates, or more generally, elastic structures of lower spatial dimension than the fluid domain (an example of which could be a compressible viscous fluid flowing through a tube with elastic walls). 

For models of the first type involving immersed solids, work was first done for well-posedness of models involving rigid bodies immersed in surrounding barotropic compressible flows \cite{BoulakiaFSI2Rigid} and heat-conducting compressible flows \cite{HaakFSIRigid}. Elastic bodies immersed in compressible flows were later considered, first in the context of higher order spatial derivative structure regularization of the structure elastodynamics \cite{BoulakiaFSI1} and later without this extra higher order spatial regularity in the elastodynamics equations \cite{BoulakiaFSI3}. These results were improved in terms of initial data regularity in the case of initial structue displacement being equal to zero in \cite{KukavicaCompressible}, and there were also extensions to nonlinearly elastic immersed solids in compressible viscous barotropic flows in the work of \cite{BoulakiaFSI4}.

For models of the second type involving elastic structures of lower spatial dimension than the compressible fluid, one of the first works is \cite{FloriFSI1}, which features a 2D modified compressible Navier-Stokes equation with a linear pressure law interacting with a 1D elastic structure on the boundary of the fluid domain, and this work was later extended to the case of a 3D fluid domain \cite{FloriFSI2} with modified isentropic compressible Navier-Stokes equations. The first analogue of the classical result by Feireisl \cite{FeireislCompressible} for existence of weak solutions to a deterministic compressible FSI model for $\gamma > 12/7$ appears in \cite{BreitFSI}, with the 3D compressible isentropic Navier-Stokes equations coupled to equations for a surrounding Koiter shell with the no-slip condition, where the compressible fluid equations posed on a time-dependent fluid domain that evolves depending on the a priori unknown displacement of the elastic Koiter shell on the boundary of the moving fluid domain, and these results were extended to the case of a similar compressible fluid-structure system with slip boundary conditions in \cite{SarkaSlip}. These results were extended to the existence of weak solutions for the interaction between a heat-conducting compressible viscous fluid and an elastic Koiter shell with the no-slip condition in \cite{BreitSchwarzacherNSF} and the interaction between a heat-conducting compressible viscous fluid and a thermoelastic shell with the no-slip condition in \cite{SarkaHeatFSI}. We also remark that there has also been a recent result on existence of strong solutions to a 2D model of compressible isentropic viscous Navier-Stokes equations interacting with a viscoelastic beam for sufficiently regular initial data with initial displacement equal to zero in \cite{MitraFSI}. 

All of the aforementioned works described have been in the context of deterministic dynamics, and we make hence make some general comments about the literature relating to the study of stochastic fluid-structure systems in general. The study of stochastic coupled fluid-structure dynamics is relatively new, and until the present work, has exclusively involved incompressible viscous fluids, rather than compressible viscous fluids. Work on stochastic incompressible fluid-structure interaction began first with a fully coupled model \cite{KuanCanicFullyCoupled} involving the coupled interaction between the linear Stokes equations and a wave equation with external additive stochastic structure forcing. The problem considered in this paper is linearly coupled in the sense that the movement of the structure displacement is ``linearized" and the fluid equations are posed on a fixed reference domain. These results were then extended to the more challenging case of stochastic incompressible fluid-structure interaction involving the Navier-Stokes equations interacting with an elastic shell, with multiplicative stochastic forcing acting on both the fluid/structure in \cite{TC23}, where the existence of weak martingale solutions is established. The well-posedness of an analogous moving-boundary stochastic model was also established in the different context of noise of transport type perturbing the structure elastodynamics \cite{BreitMoyo}. Later, existence of weak martingale solutions to stochastic coupled moving boundary systems of incompressible Navier-Stokes equations interacting with elastic shells was established for more general models, including models with unrestricted structure displacements \cite{T24} and models with more general Navier-slip boundary conditions \cite{T24slip}.

We emphasize that so far, the past mathematical literature on stochastic fluid-structure dynamics has been restricted the study of incompressible fluids. Hence, the goal of the current manuscript is to extend the study of stochastic fluid-structure interaction to the compressible regime. 

\section{Weak martingale solutions and the main result} 
\begin{definition}[Definition of a weak martingale solution]\label{def:martingale}
    Let the {{deterministic}} initial structure configuration $\eta_0\in H^2(\Gamma)$
	be such that for some $\alpha_0>0$ we have,
	\begin{align}	 \label{etainitial}
\alpha_0<1+\eta_0(w),\quad \forall w\in \Gamma,\quad \text{ and }\quad { \|\eta_0\|_{H^2(\Gamma)}<\frac1{\alpha_0}}.
	\end{align}
 Let  $\rho_0 \in L^{\gamma}(\sO_{\eta_0})$ and $ v_0\in L^2(\Gamma)$ be 
	 given deterministic initial data.
	Assume that when $\rho_0>0$,  {{the initial momentum}} $\bp_0$ satisfies $\frac{|\bp_0|^2}{\rho_0} \in L^1(\sO_{ \eta_0})$. 
	We say that  $(\mathscr{S},\rho,\bu,\eta,\tau^\eta)$ is a  martingale solution to the system \eqref{eqn:fluid}-\eqref{eqn:dynbc} if: 
	\begin{enumerate}
		\item  $\mathscr{S}=(\Omega,\sF,(\sF_t)_{t\geq 0},\bP,W^1,W^2)$ is a stochastic basis, that is, $(\Omega,\sF,(\sF_t)_{t\geq 0},\bP)$ is a filtered probability space satisfying the usual conditions and $W^1$ and $W^2$ are independent $\mathcal{U}$-valued $(\sF_t)_{t\geq 0}-$Wiener processes.
   \item $\eta\in L^2(\Omega;L^\infty(0,T;H^2(\Gamma))\cap W^{1,\infty}(0,T;L^2(\Gamma)) {{\cap H^{1}(0, T; H^{1}(\Gamma))}})$
\item $\bu \in 
L^2(0,T;H^{1}(\sO_\eta(\cdot)))$ and $\rho \in L^\infty(0,T;L^{\gamma}(\sO_\eta(\cdot)))$ such that {{$\rho|\bu|^2\in L^\infty(0,T;L^1(\sO_\eta(\cdot)))$}} $\bP$-almost surely.
		\item $\tau^\eta$ is a $\bP$-almost surely strictly positive $(\sF_t)_{t\geq 0}-$stopping time;
		\item The random distributions
		$ \rho, \bu,\partial_t\eta,$ (in the sense of \cite{BFH18}) and the stochastic process $\eta$, are $(\sF_t)_{t \geq 0}-$progressively measurable.
		\item {{The renormalized continuity equation holds $\bP$-almost surely, for almost every $t \in [0, \tau^\eta)$:}}
		\begin{align}\label{renorm_cont}
			\int_0^t\int_{\sO_\eta(t)} b(\rho)(\partial_t\phi+\bu\cdot\nabla \phi)=\int_0^t\int_{\sO_\eta(t)} (b'(\rho)\rho-b(\rho))(\nabla \cdot\bu)\phi +\int_{\sO_{\eta_0}}\rho_0b(\rho_0)\phi(0),
		\end{align} 
		for any essentially bounded process $\phi $ taking values in $C^\infty_c([0,T)\times \bar\sO_{\eta}(\cdot))$ and $b \in C(\R)$ with $b'(z)=0$ when $z\geq M_b${{, for some constant $M_{b}$ (varying with the specific choice of $b$).}}
		\item	For every $(\sF_t)_{t \geq 0}-$adapted, essentially bounded smooth process $(\bq,\psi)$ such that 
		$\bq|_{\Gamma_\eta}=\psi\bd{e}_z$, and the following equation holds for $\bP$-almost surely, for almost every {\cred $t \in[0,\tau^\eta)$}:
		\begin{equation}\label{weaksol}
			\begin{split}
				&{\int_{\sO_\eta(t)}\rho(t)\bu(t) \cdot \bq(t) d\bx +\int_\Gamma\partial_t\eta(t)\psi(t)dz}= \int_{\sO_{\eta_0}}\bp_0 \cdot \bq(0) d\bx  + \int_\Gamma v_0\psi(0) d\bd{z} \\
				&+\int_0^{t }\int_{\sO_\eta(t)}\rho\bu\cdot \partial_t\bq d\bx dt+\int_0^{t }\int_{\sO_\eta(t)} \rho\bu\otimes\bu:\nabla\bq d\bx dt
				- { \mu}\int_0^{t } \int_{\sO_\eta(t)}\nabla\bu: \nabla\bq d\bx dt\\
				&+\int_0^t\int_{\sO_\eta(t)} \rho^\gamma (\nabla \cdot\bq )d\bx dt- \lambda \int_0^t\int_{\sO_\eta(t)} \text{div}(\bu)\text{div}(\bq)d\bx dt \\
				&+\int_0^{t }\int_\Gamma\partial_t\eta\partial_t\psi d\bd{z}dt - {{ \int_0^{t }\int_\Gamma \nabla\partial_{t}\eta \cdot\nabla \psi d\bd{z} dt}} - \int_{0}^{t} \int_{\Gamma} (\nabla\eta\nabla \psi +\Delta\eta\Delta \psi) d\bd{z}dt\\
				&+\int_0^t\int_{\sO_\eta(t)} F(\rho,\rho\bu)\cdot\bq\,dW^1(t) + \int_0^t\int_\Gamma G(\eta, \partial_t\eta)\psi dW^2(t).
		\end{split}\end{equation}
  \item The kinematic coupling condition holds: $\bu|_{\Gamma_{\eta}}=\partial_t\eta {\bf e}_z$, $\bP$-almost surely {\cred for almost every $t\in[0,\tau^\eta)$.}
	\end{enumerate}
\end{definition}
We are now in position to state the main result of this paper.
\begin{theorem}\label{mainthm}
   Let the initial deterministic structure configuration $\eta_0\in H^2(\Gamma)$
	be such that for some $\alpha_0>0$, we have:
	\begin{align}	 
		{\alpha_0<1+\eta_0(w),\quad \forall w\in \Gamma,\quad \text{ and }\quad { \|\eta_0\|_{H^2(\Gamma)}<\frac1{\alpha_0}}.}
	\end{align}
Consider deterministic initial data, consisting of an initial density $\rho_{0} \in L^{\gamma}(\mathcal{O}_{\eta_{0}})$ with $\rho_{0} \ge 0$, an initial momentum $\bd{p}_{0} \in L^{\frac{2\gamma}{\gamma + 1}}(\sO_{\eta_0})$ such that $\bd{p}_{0}$ vanishes whenever $\rho_{0}$ vanishes and $\frac{|\bp_0|^2}{\rho_0} \in L^1(\sO_{\eta_0})$ (where this quotient is defined to be zero when $\rho_0$ vanishes), 
and an initial structure velocity $ v_0\in L^2(\Gamma)$. Assume that the noise coefficients $\bd{F}$ and $G$ 
  satisfy the growth conditions \eqref{fassumption}-\eqref{gassumption}. Then there exists at least one martingale solution to the system \eqref{eqn:fluid}-\eqref{eqn:dynbc} in the sense of Definition \ref{def:martingale}.
\end{theorem}

\section{The approximation scheme}
Before presenting our splitting scheme in Section \ref{splittingscheme}, we will present the necessary setup. We fix some parameter $\alpha > 0$, which will be the parameter for which the approximate structure displacements we work with at all steps of the existence proof will satisfy $\|\eta\|_{H^{s}(\Gamma)} \le \frac1\alpha$, for some $s \in (3/2, 2)$.  Since $s \in (3/2, 2)$, we have the Sobolev embedding of $H^{s}(\Gamma) \subset C(\Gamma)$. Setting the embedding constant to 1 for simplicity, the aforementioned $H^s$-bound allows us to define a maximal domain $\mathcal{O}_{\alpha}$: 
\begin{equation*}
 \mathcal{O}_{\alpha} = \mathbb{T}^2 \times (0 ,2 + \alpha^{-1}).
\end{equation*}

\subsection{Extension to the maximal domain}\label{extension}

Because we are extending the problem from the physical moving domain to a fixed maximal domain, we will need to extend both the viscosity coefficients and initial data to the maximal domain $\mathcal{O}_{\alpha}$. We first consider the extension of the viscosity coefficients from the time-dependent domain $\sO_{\eta}(t)$ to the maximal domain $\mathcal{O}_{\alpha}$, and for this, it is useful to define smooth bounding functions $a^{\eta}_{\kappa}$ and $b^{\eta}_{\kappa}$ with $a^{\eta}_{\kappa} > \eta > b^{\eta}_{\kappa}$, which define a tubular neighborhood (with width controlled by the parameter $\kappa$) around a given physical structure location determined by the displacement $\eta$. The bounding functions $a^{\eta}_{\kappa}$ will be useful for the extension of viscosity coefficients, and both bounding functions $a^{\eta}_{\kappa}$ and $b^{\eta}_{\kappa}$ will also be useful later for the construction of test functions in Section \ref{sec:testfunction}.

\medskip

\noindent \textbf{Definition of the bounding functions.} We now define the bounding functions $a^{\eta}_{\kappa}$ and $b^{\eta}_{\kappa}$. Given a structure displacement $\eta \in H^{s}(\Gamma) \subset C^{ \frac12}(\Gamma)$ with $\|\eta\|_{H^{s}(\Gamma)} \le \frac1\alpha$ for some $s \in (3/2, 2)$ , there exists $C_{\alpha}$ (depending only on $\alpha$) such that
\begin{equation}\label{Calphadef}
|\eta(w_1) - \eta(w_2)| \le C_{\alpha}|w_1 - w_2|^{1/2}, \quad \text{ for all } w_1, w_2\in \Gamma.
\end{equation}
Given an extension parameter $\kappa>0$, we use this $1/2$-H\"{o}lder continuity constant $C_{\alpha}$ to define
\begin{equation*}
\eta^{\sharp}_{\kappa} = \eta + 2C_{\alpha}\kappa^{1/2}, \quad \eta^{\flat}_{\kappa} = \eta - 2C_{\alpha}\kappa^{1/2},
\end{equation*}
and to define the smooth bounding functions
\begin{equation}\label{abbounding}
a^{ \eta}_{\kappa} = 1 + (\eta^{\sharp}_{\kappa} * \zeta_{\kappa}), \quad \quad
 b^{\eta}_{\kappa} = 1 +(\eta^{\flat}_{\kappa} * \zeta_{\kappa}),
\end{equation}
where $\zeta_{\kappa} = \kappa^{-2} \zeta(\kappa^{-1} \cdot)$ is the usual smooth convolution kernel in $\R^{2}$ with support in a ball of radius $\kappa$.
Since given ${ w_1}, w_2 \in \Gamma$, for all $|w_2 - w_1| \le \kappa$, we have that $\eta^{\sharp}_{\kappa}(w_2) > \eta({ w_1})$ and $\eta^{\flat}_{\kappa}(w_2) < \eta({ w_1})$ by \eqref{Calphadef}. Hence,
\begin{equation}\label{abbound0}
a^{\eta}_{\kappa} > 1 + \eta > b^{ \eta}_{\kappa}.
\end{equation}
 Furthermore, the estimate \eqref{Calphadef} also implies that for $w\in \Gamma$, 
\begin{equation*}
|(\eta^{\sharp}_{\kappa} * \zeta_{\kappa})(w) - \eta^{\sharp}_{\kappa}(w)| \le C_{\alpha}\kappa^{1/2}, \quad |(\eta^{\flat}_{\kappa} * \zeta_{\kappa})(w) - \eta^{\flat}_{\kappa}(w)| \le C_{\alpha}\kappa^{1/2},
\end{equation*}
and hence:
\begin{equation}\label{abound}
C_{\alpha} \kappa^{1/2} \le a^{ \eta}_{\kappa}(w) - (1 + \eta(w)) \le 3C_{\alpha} \kappa^{1/2},
\end{equation}
\begin{equation}\label{bbound}
C_{\alpha} \kappa^{1/2} \le (1 + \eta(w)) - b^{ \eta}_{\kappa}(w) \le 3C_{\alpha} \kappa^{1/2}.
\end{equation}
We end this note by observing the following property of these bounding functions.
\begin{lemma}\label{abconv}
Assume that $\eta_n\to\eta$ in $C(0,T;C(\Gamma))$ as $n\to\infty$. Then, for a fixed but arbitrary  $\kappa \ge 0$, $a^{ \eta_n}_{\kappa} \to a^{ \eta}_{\kappa}$ and $ b^{ \eta_n}_{\kappa} \to b^{ \eta}_{\kappa}$ in $C(0, T; C^k(\Gamma))$ almost surely for any $k\geq 0$.
\end{lemma}
\begin{proof}
Since ${\eta}_{n} \to \eta$ in $C(0, T; C(\Gamma))$,
\begin{equation}\label{convergencesharp}
({ \eta_n})^{\sharp}_{\kappa} \to (\eta)^{\sharp}_{\kappa} \qquad \text{ in $C(0, T; C(\Gamma))$, as } n\to\infty,
\end{equation}
which implies that $\nabla^{k} [(\eta_n)^{\sharp}_{\kappa} * \zeta_{\kappa}] \to \nabla^{k} [(\eta)^{\sharp}_{\kappa} * \zeta_{\kappa}]$ for all positive integers $k$, in $C(0, T; C(\Gamma))$ as $n\to\infty$ by properties of convolution.
The result for $b_\kappa^{\eta_n}$ is analogous.
\end{proof}
\medskip

\noindent \textbf{Extension of viscosity coefficients.} We will use the bounding function $a^{\eta}_{\kappa}$ to extend the viscosity coefficients from the physical fluid domain to the fixed maximal domain $\mathcal{O}_{\alpha}$ via an extension map. To define this extension map, for a parameter $\kappa\ge 0$, we consider a smooth function $\phi_\kappa: \R \to \R$ such that $\phi_\kappa(w) = 1$ when $w \le \frac14$, $\phi_\kappa$ is decreasing on $ [\frac14,\frac34]$, and $\phi_\kappa(w) = \kappa$ for $w \ge \frac34$. Then, given $\eta \in H^{s}(\Gamma)$ for $s \in (3/2, 2)$, we can define an extension operator by defining $\chi^{\eta}_{\kappa}$, for any $\kappa>0$, to be a smooth {{compactly supported}} function on $\Gamma \times [0, \infty)$, given by
\begin{equation}\label{chi}
\chi^{\eta}_{\kappa}(x, y, z) = \phi_\kappa\left(\frac{z - a^{ \eta}_{\kappa}(x, y)}{\kappa^{1/2}}\right). 
\end{equation}
For an appropriately chosen value of $\kappa$, $\chi_\kappa^\eta$ will be used to extend the viscosity parameters $\mu$ and $\lambda$ \eqref{viscosityextension} in the fluid sub-problem \eqref{fluidsubproblem}.
\medskip

\noindent \textbf{Extension and approximation of the initial data.} Recall that we are given deterministic initial data for the fluid density $\rho_{0} \in L^{\gamma}(\mathcal{O}_{\alpha})$ and the fluid momentum $\bp_{0} \in L^{\frac{2\gamma}{\gamma + 1}}(\mathcal{O}_{\alpha})$ defined on the initial moving fluid domain $\mathcal{O}_{\eta_{0}}$, so we must extend the initial data to the maximal domain $\mathcal{O}_{\alpha}$. We must perform this extension of the initial data to the maximal domain $\mathcal{O}_{\alpha}$ carefully, as we will initially need the fluid density to be bounded below by some positive constant in order to apply a comparison principle to obtain existence of first-level approximate solutions, but in a final limit passage (which will later correspond to sending a parameter $\delta \to 0$), we will require the initial density to vanish outside the initial physical fluid domain. So we use two layers of approximation of the initial data, with two parameters $\ep > 0$ and $\delta > 0$, where the following approximations of the initial data are motivated by techniques found in the beginning of Section 4 of \cite{FeireislCompressible}, and (3.6)-(3.8) in \cite{SarkaHeatFSI}.

First, we extend the initially given deterministic initial data $\rho_{0}$ and $\bp_0$ on $\mathcal{O}_{\eta_{0}}$ by zero to get initial data on $\mathcal{O}_{\alpha}$. For each parameter $\delta > 0$, we approximate this initial data by a pair $(\rho_{0, \delta}, \bd{p}_{0, \delta})$, where $\rho_{0, \delta} \in C^{2 + \nu}(\mathcal{O}_{\alpha})$, for some $\nu>0$ 
satisfying 
\begin{align}
&0 \le \rho_{0, \delta} \le \delta^{-1/\beta}, \quad \rho_{0, \delta}|_{\mathcal{O}_{\alpha} \setminus \mathcal{O}_{ \eta_0}} = 0,\quad\lim_{\delta \to 0} |\{\rho_{0, \delta} < \rho_{0}\}| \to 0, \quad \rho_{0, \delta} \to \rho_{0} \text{ in } L^{\gamma}(\mathcal{O}_{\alpha}),\notag\\
&\bd{p}_{0, \delta} = \bd{p}_{0}, \text{ if } \rho_{0, \delta} \ge \rho_{0} \text{ and } \bp_0 = \bd{0} \text{ otherwise}, \quad \int_{\mathcal{O}_{\alpha}} \frac{|\bd{p}_{0, \delta}|^{2}}{\rho_{0, \delta}} \le C \text{ independently of $\delta$.}
\label{p0}
\end{align}
These properties imply that the modified energy $\displaystyle \frac{1}{2} \int_{\mathcal{O}_{\alpha}} \frac{|\bd{p}_{0, \delta}|^{2}}{\rho_{0, \delta}} + \frac{a}{\gamma - 1} \int_{\mathcal{O}_{\alpha}} \rho_{0, \delta}^{\gamma} + \frac{\delta}{\beta - 1} \int_{\mathcal{O}_{\alpha}} \rho_{0, \delta}^{\beta}$ is uniformly bounded independently of $\delta$, and that $\bd{p}_{0, \delta} \to \bd{p}_{0}$ in $L^{\frac{2\gamma}{\gamma + 1}}(\mathcal{O}_{\alpha})$.

For the next layer of approximation involving the parameter $\ep > 0$ for fixed but arbitrary $\delta > 0$, we choose $\rho_{0, \delta, \ep} \in C^{2 + \nu}(\mathcal{O}_{\alpha})$ such that
\begin{equation}
0 < \ep \le \rho_{0, \delta, \ep} \le \delta^{-1/\beta}, \quad \text{ and } \rho_{0, \delta, \ep} \to \rho_{0, \delta} \text{ in $L^{\gamma}(\mathcal{O}_{\alpha})$ as } \ep\to 0, \label{rho0conv}
\end{equation}
and the initial fluid momentum $\bd{p}_{0, \delta, \ep}$, using the argument in the beginning of Section 4 in \cite{FeireislCompressible} and the beginning of Section 6 in \cite{BreitHofmanova}, can be chosen to satisfy
\begin{align*}
&\frac{|\bd{p}_{0, \delta, \ep}|^{2}}{\rho_{0, \delta, \ep}} \text{ is uniformly bounded (independently of $\ep$) in $L^{1}(\mathcal{O}_{\alpha})$},\\
&\bd{p}_{0, \delta, \ep} \to \bd{p}_{0, \delta} \text{ in } L^{1}(\mathcal{O}_{\alpha}), \quad \frac{\bd{p}_{0, \delta, \ep}}{\sqrt{\rho_{0, \delta, \ep}}} \to \frac{\bd{p}_{0, \delta}}{\sqrt{\rho_{0, \delta}}} \text{ in } L^{2}(\mathcal{O}_{\alpha}), \quad \text{ as $\ep \to 0$.}
\end{align*}
We remark that the modified energy $\displaystyle \frac{1}{2} \int_{\mathcal{O}_{\alpha}} \frac{|\bd{p}_{0, \delta, \ep}|^{2}}{\rho_{0, \delta, \ep}} + \frac{a}{\gamma - 1}\int_{\mathcal{O}_{\alpha}} \rho_{0, \delta, \ep}^{\gamma} + \frac{\delta}{\beta - 1} \int_{\mathcal{O}_{\alpha}} \rho_{0, \delta, \ep}^{\beta}$ is similarly bounded uniformly, independently of $\ep$ and $\delta$. 

Given the extended initial data $(\rho_{0, \delta, \ep}, \bd{p}_{0, \delta, \ep})$ on the maximal domain $\mathcal{O}_{\alpha}$, our splitting scheme will include extension of the fluid equations to the larger domain $\mathcal{O}_{\alpha}$, by penalizing the kinematic coupling condition via the parameter $\delta > 0$.  
While this extension seems unnatural at this stage, we will eventually prove that the fluid momentum exists only in the time-dependent moving domain by proving that its density vanishes outside as $\delta\to 0$ (cf. Proposition \ref{vacuum2}) thus retrieving the original formulation on the moving domain.

\subsection{Description of the Galerkin approximation.}

Let {{$\{\bd{\psi}_{i}\}_{i = 1}^{\infty}$}} be an orthonormal basis for $L^{2}(\mathcal{O}_{\alpha})$ and an orthogonal basis for ${ H^l_0(\mathcal{O}_{\alpha})}$ for $l > 5/2$, and let {{$\{\xi_{i}\}_{i = 1}^{\infty}$}} be an orthonormal basis for $L^{2}(\Gamma)$ and an orthogonal basis for $H^2(\Gamma)$. 
Define
\begin{align}\label{Xkf}
X_{n}^{f} &:= \text{span}\{\bd{\psi}_{1}, \bd\psi_{2}, ..., \bd\psi_{n}\} \subset { H^l_0(\mathcal{O}_{\alpha})},\\
X_{n}^{{st}} &:= \text{span}\{{ \varphi_{1}, \varphi_{2}, ..., \varphi_{n}}\} \subset H^2(\Gamma),\label{Xks} 
\end{align}
where $X_{n}^{f}$ is endowed with the $L^{2}(\mathcal{O}_{\alpha})$ inner product and $X_{n}^{st}$ is endowed with the $L^{2}(\Gamma)$ inner product. Define
\begin{equation}\label{galerkinXk}
X_{n} := X_n^{f} \times X_n^{st}
\end{equation}
to be the full fluid-structure Galerkin space with the usual product norm. Let $P^{f}_{n}: L^{2}(\mathcal{O}_{\alpha}) \to X^{f}_{n}$ be the orthogonal projection operator from $L^{2}(\mathcal{O}_{\alpha})$ onto $X^{f}_{n}$ and let $P^{st}_{n}: L^{2}(\Gamma) \to X^{st}_{n}$ be the orthogonal projection operator from $L^{2}(\Gamma)$ onto $X^{st}_{n}$. Let $(X_n^{f})^*$ and $(X_n^{st})^*$ denote the dual spaces of linear functionals on $X^{f}_{n}$ and $X^{st}_{n}$ respectively. We remark that the choice of the function space $H^l_0(\mathcal{O}_{\alpha})$ for $l > 5/2$ {{used for the fluid Galerkin space}} is compatible with the comparison principle for the continuity equation with viscosity stated in \eqref{comparison}, as functions in $H^l_0(\mathcal{O}_{\alpha})$ {{for $l > 5/2$}} have divergences that are in $L^{\infty}(\mathcal{O}_{\alpha})$ by standard Sobolev embedding in three spatial dimensions. 

\subsection{The operator splitting scheme}\label{splittingscheme}

Upon defining the problem on the extended domain $\mathcal{O}_{\alpha}$, we will then use a splitting scheme to construct approximate solutions defined on a fixed time interval $[0, T]$, where $T$ will be a fixed, but arbitrary final time. We will use an operator splitting scheme that divides the entire time interval $[0, T]$ for a parameter $N$ into $N$ subintervals $[t_{j}, t_{j + 1}]$ of length $\Delta t$. For each $j =  0, 1, 2, ..., N-1$, we will run decoupled structure and fluid subproblems to update all of these approximate quantities on $[t_{j}, t_{j + 1}]$.

We will keep track of four approximate solution quantities defined continuously in time on $[0, T]$: the fluid density $\rho_{N}$, the fluid velocity $\bu_{N}$, the structure displacement $\eta_{N}$, and the structure velocity $v_{N}$. In addition, we will have a ``stopped" process $\eta^{*}_{N}$, which is the structure displacement stopped at the first time of leaving desired bounds on the displacement of the structure. 

\medskip

\noindent \textbf{The structure subproblem.} 
We update the structure displacement by solving the following weak formulation for $\eta_{N} \in L^{2}(\Omega; W^{1, \infty}(t_{j}, t_{j + 1}; X_n^{st})))$ for the Galerkin space $X_n^{st}$ defined in \eqref{Xks}:
\begin{equation}\label{structuresubproblem}
\begin{split}
    &\int_{\Gamma} \partial_{t} \eta_{N}((j+1)\Delta t) \psi=\int_{\Gamma} \partial_{t} \eta_{N}(j\Delta t) \psi- \int_{j\Delta t}^{(j+1)\Delta t} \int_{\Gamma} \nabla \eta_{N} \cdot \nabla \psi - \int_{j\Delta t}^{(j+1)\Delta t}\int_{\Gamma} \Delta \eta_{N} \Delta \psi\\
   &- \int_{j\Delta t}^{(j+1)\Delta t} \int_{\Gamma}\nabla \partial_t\eta_N\cdot \nabla \psi -\frac1\delta\int_{j\Delta t}^{(j+1)\Delta t} \int_{T^\delta_{\mathcal{T}_{\Delta t}\eta^*_N}}(\partial_t\eta_N \bd{e}_z-{ \mathcal{T}}_{\Delta t}\bu_N)\cdot\psi \bd{e}_z\\
&+ \int_{j\Delta t}^{(j+1)\Delta t}\int_{\Gamma}G_n(\eta_{N}, \partial_{t}\eta_{N}) \psi dW^2(t) ,
\end{split}
\end{equation}
holds $\bP$-almost surely for every spatial test function $\psi \in  X_n^{st}$. { We let
$\bu_N(t) = { P^f_n}\bu_0$, $\eta_N(t) = P_n^{st}\eta_0$ and $\partial_t\eta_N(t) = P_n^{st}v_0$ for $t \leq 0.$ }
The noise coefficient is 
\begin{equation}\label{Gndef}
    G_n:=P^{st}_{n}G,
\end{equation}
and the time shifts are defined by $\mathcal{T}_{\Delta t} f = f\left(t - {\Delta t}\right)$. {{We remark that this definition \eqref{Gndef} of the projected structure noise operator makes sense, because the structure velocity and displacement will both be in $L^{\infty}(t_{j}, t_{j + 1}; L^{2}(\Gamma))$, and hence $g_{k} \in L^{\infty}(t_j, t_{j+1}; L^{2}(\Gamma))$ by the assumption \eqref{gassumption} on the noise. Thus, it makes sense to apply the projection operator $P^{st}_{n}: L^{2}(\Gamma) \to X^{st}_{n}$ to each $g_{k}(\eta_N, \partial_t \eta_N)$.}} The \textit{penalty term} that decouples the structure equations from the momentum equation
is defined on the following \textbf{(exterior) tubular neighborhood}: for any $\delta>0$,
$T^\delta_{\eta}$ is the tubular neighborhood 
\begin{align}\label{tube}
    T^\delta_{\eta}:=\{(x,y,z) \in   \sO_\alpha\setminus\sO_\eta: 0< (z-1-\eta)<\delta^{(\frac12-\frac1\beta)}\},
\end{align}
 where artificial pressure parameter $\beta>\max\{4,\gamma\}$ also appears in the next subproblem \eqref{fluidsubproblem}.
Finally, the so-called artificial structure displacement is the following stopped process:
\begin{equation}\label{etastar}
\eta^*_N(t) = \eta_{N}(\tau^\eta_N \wedge t),
\end{equation}
where for a fixed $s\in(3/2,2)$, the stopping time $\tau^\eta_N$ is defined as follows:
\begin{align*}
\tau^\eta_N &:=T\wedge \inf\left\{t> 0:\inf_{\Gamma}(1+{\eta_N}(t,\cdot))\le \alpha \text{ or } \|{\eta}_N(t,\cdot)\|_{H^s(\Gamma)} \ge \frac1{\alpha}\right\}.
\end{align*}
\begin{figure}
    \centering
    \includegraphics[scale=0.4]{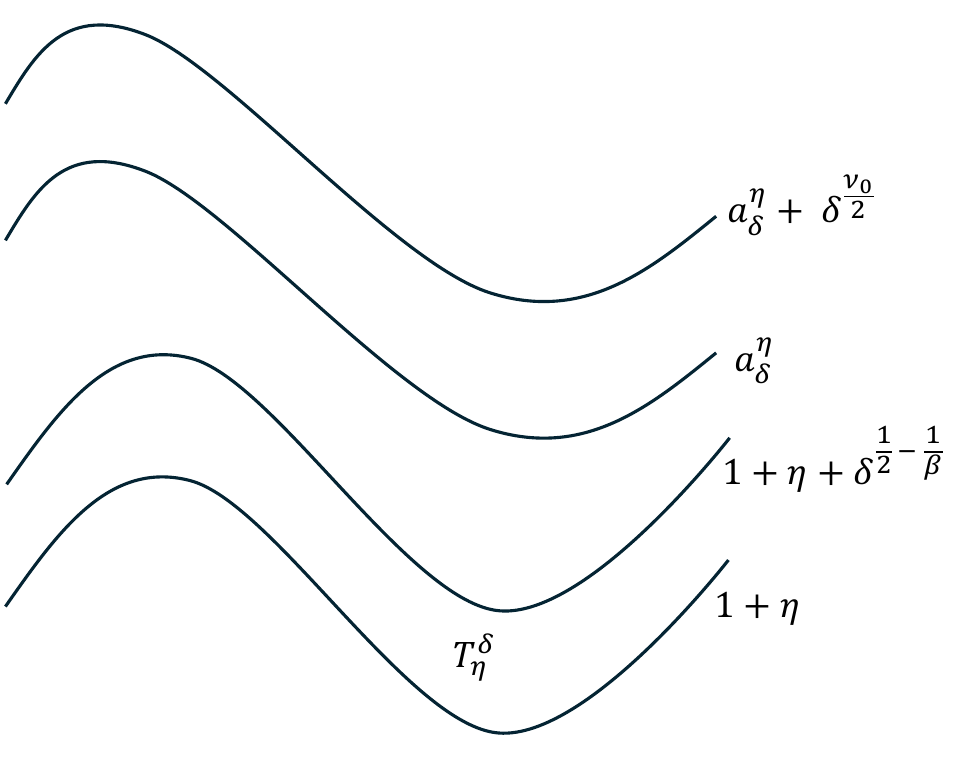}
    \caption{Graphs of functions bounding any structure displacement}
    \label{layers}
\end{figure}
The reason behind using this stopped processes, which will stop the process $\eta_{N}$ at the first time at which the approximate solution leaves the desired bounds, indicated by the parameter $\alpha > 0$, is because we want to avoid self-collision and want to be able to define the fluid subproblem by extending it to a fixed bounded domain. 
Because $\eta_N \in C(0, T; H^{s}(\Gamma))\subset  C(0, T; C(\Gamma))$ for any $s \in (3/2, 2)$, this is indeed a stopping time. 

Then on $[t_{j}, t_{j + 1}]$ we set 
$$v_{N}(t) = \partial_{t}\eta_{N}.$$ 



We use the stopped structure displacement $\eta^*_N$ because this allows us to define the splitting scheme (in particular for the fluid subproblem) on the entire time interval $[0, T]$, as we can use $\eta^*_N$ to define a moving fluid domain that does not exhibit domain degeneracies on the whole time interval $[0, T]$. However, we emphasize the important point that in order to obtain uniform estimates on $[0, T]$ for the structure elastodynamics, we \textit{do not} stop the structure subproblem at the stopping time $\tau^\eta_N$ and instead continue to update and evolve the structure displacement all the way until the final time $T$. 

\medskip

\noindent \textbf{The fluid subproblem.} For the fluid subproblem on the interval $[t_{j}, t_{j+1}]$, we will solve for the fluid density and velocity
\begin{equation*}
\rho_{N} \in L^{2}(\Omega; C(t_{j}, t_{j+1}; { C^{2+\nu}(\mathcal{O}_{\alpha})})), \qquad
\bu_{N} \in L^{2}(\Omega; C(t_{j}, t_{j+1};  X_n^f)),
\end{equation*}
so that the weak formulation of the continuity equation holds for all smooth test functions $\varphi \in C_{c}^{\infty}(\overline{\mathcal{O}}_{\alpha})$:
\begin{align*}
\int_{\mathcal{O}_{\alpha}} \rho_{N}(t) \varphi &= \int_{\mathcal{O}_{\alpha}} \rho_{N}(t_{j}) \varphi
+ \int_{t_{j}}^{t} \int_{\mathcal{O}_{\alpha}} \rho_{N} \bu_{N} \cdot \nabla \varphi - \varepsilon \int_{t_{j}}^{t} \int_{\mathcal{O}_{\alpha}} \nabla \rho_{N} \cdot \nabla \varphi,
\end{align*}
for all $t \in [t_{j}, t_{j + 1}]$ which is the weak formulation for the continuity equation with artificial viscosity, $\varepsilon>0$, posed on the maximal fixed domain $\mathcal{O}_{\alpha}$ with Neumann boundary conditions:
\begin{equation}\label{reg_conteqn}
 \rho_{t} + \text{div}(\rho \bu) = \ep \Delta \rho \ \ \text{ on } \mathcal{O}_{\alpha}, \quad \nabla \rho \cdot \bd{n} = 0 \ \ \text{ on  } \partial \mathcal{O}_{\alpha}.
 \end{equation}
For fixed a fixed artificial pressure parameter $\beta>\max\{4,\gamma\}, \delta>0,\ep>0$, time step $\Delta t = T/N$ and Galerkin parameter $n$, we solve the following weak formulation of the momentum equation holds $\bP$-almost surely for all deterministic and purely spatial test functions $\bq \in  X^f_n$:
\begin{equation}
\begin{split}\label{fluidsubproblem} 
&\int_{\mathcal{O}_{\alpha}} \rho_{N} \bu_{N}((j+1)\Delta t) \cdot \bd{q} = \int_{\mathcal{O}_{\alpha}} \rho_{N} \bu_{N}(j\Delta t) \cdot \bd{q} + \int_{j\Delta t}^{(j+1)\Delta t} \int_{\mathcal{O}_{\alpha}} (\rho_{N} \bu_{N} \otimes \bu_{N}) : \nabla \bd{q} \\
&+ \int_{j\Delta t}^{(j+1)\Delta t}\int_{\mathcal{O}_{\alpha}} \left(a\rho_{N}^{\gamma} + \delta \rho_{N}^{\beta}\right) (\nabla \cdot \bd{q}) - \int_{j\Delta t}^{(j+1)\Delta t} \int_{\mathcal{O}_{\alpha}} \mu_\delta^{\eta^*_{N}}\nabla \bu_{N} : \nabla \bd{q}\\
&- \int_{j\Delta t}^{(j+1)\Delta t}\int_{\mathcal{O}_{\alpha}}  \lambda_\delta^{\eta^{*}_{N}} \text{div}(\bu_{N}) \text{div}(\bd{q})+\ep\int_{j\Delta t}^{(j+1)\Delta t}\int_{\sO_\alpha}\ \rho_N \bu_N \cdot \Delta\bq \\
&- \frac1{\delta} \int_{j\Delta t}^{(j+1)\Delta t}\int_{T^\delta_{\eta^*_N}}(\bu_N-v_N\bd{e}_z) \cdot \bq 
+  \int_{j\Delta t}^{(j+1)\Delta t}\int_{\mathcal{O}_{\alpha}} \mathbbm{1}_{\sO_{\eta_N^*}}\bd{F}_{n}(\rho_{N}, \rho_{N} \bu_{N}) \cdot \bd{q} \,dW^1,
\end{split}
\end{equation}
where {we recall the notation $T^{\delta}_{\eta^{*}_{N}}$ from \eqref{tube} for the tubular neighborhood of width $\delta^{\left(\frac{1}{2} - \frac{1}{\beta}\right)}$ around the moving boundary determined by $\eta^{*}_{N}$, and where} $\bd{F}_n$ defined below in \eqref{Fndef}, appropriately approximates the operator $\bd{F}$ using the Galerkin projection operator. Moreover, recalling definition \eqref{chi}, for some $\nu_0>0$ we let the extended viscosity coefficients 
\begin{equation}\label{viscosityextension}
\mu^{\eta}_{\delta} := \chi^{\eta}_{\delta^{\nu_0}} \mu, \quad \lambda^{\eta}_{\delta} := \chi^{\eta}_{\delta^{\nu_0}} \lambda.
\end{equation}
We will choose $\nu_0 \sim \left(\frac{1}{2} - \frac{1}{\beta}\right)^{2}>0$ which will be justified later (cf. \eqref{nu0} and Proposition \ref{vacuum2}).
To define the noise coefficient, we let,
\begin{equation}\label{multiplier}
(\mathcal{M}[\rho] \bu, \bd{q}) = \int_{\mathcal{O}_{\alpha}} \rho \bu \cdot \bd{q}  \quad \text{ for } \bu,\bd{q} \in X^f_n,\quad { \rho \in L^1(\sO_\alpha)}.
\end{equation}
Note that by identifying $(X_n^f)^*$ with $X^f_{n}$, we can also view $\mathcal{M}[\rho]$ as a linear operator on $X_{n}^f$. 
Now, for each fixed but arbitrary Galerkin parameter $n$, we define $\bd{F}_{n}(\rho, \rho \bu): \mathcal{U}_{0} \to \bd{L}^{1}(\mathcal{O}_{\alpha})$ to be the ``projected" noise operator, defined via the orthonormal basis elements $\{\bd{e}_{k}\}_{k = 1}^{\infty}$ of $\mathcal{U}_0$, by
\begin{equation}\label{Fndef}
\bd{F}_n(\rho, \bd{q}) \bd{e}_{k} = f_{n, k}(\rho, \bd{q}) := \mathcal{M}^{1/2}[\rho] \left(P_{n}^{f}\left(\frac{f_{k}(\rho, \bd{q})}{\rho^{1/2}}\right)\right),
\end{equation}
We note that by the assumption \eqref{fassumption} on the noise coefficients $f_{k}(\rho, \bd{q})$, we have that if $\rho, \rho |\bu|^{2} \in L^{1}(\mathcal{O}_{\alpha})$, as is expected from the a priori estimates, then $\displaystyle \frac{f_{k}(\rho, \bd{q})}{\rho^{1/2}} \in L^{2}(\mathcal{O}_{\alpha})$ so that the orthogonal projection $P_{n}^{f}$ in the above expression makes sense and therefore, {{$\bd{F}_{n}(\rho, \bd{q})\bd{e}_{k} \in (X_n^f)^*$}}. Here, $\mathcal{M}^{1/2}[\rho]: X^f_{n} \to X^f_{n}$ is defined as the square root of $\mathcal{M}[\rho]: X^f_{n} \to X^f_{n}$, see Section 3 in \cite{BreitHofmanova} for details.

\subsection{Solving the structure subproblem}

Since, the noise coefficient $G$ satisfies the conditions \eqref{gassumption}, for any $j< N$ and any $N\in \mathbb{N}$ the existence of a unique solution $\eta^j_{N} \in L^{2}(\Omega; W^{1, \infty}(t_j,t_{j+1}; X^{st}_{n}))$ adapted to the filtration $\{\sF_t\}_{t\geq 0}$ is given by standard methods such as the Picard iteration and we refer the interested reader to \cite{Walsh} (pg. 323) for details.

\subsection{Solving the fluid subproblem}
 Since $\bd{F}$ satisfies the conditions \eqref{fassumption}, there exists a solution $\bu^j_N\in L^2(\Omega,C(t_j,t_{j+1};X^f_n))$ and $\rho^j_N\in L^2(\Omega;C(t_j,t_{j+1};C^{2+\nu}(\bar\sO_\alpha)))$ for any $\nu>0$, adapted to the filtration $\{\sF_t\}_{t\geq 0}$. This existence result is obtained by applying the results from Section 3 in \cite{BreitHofmanova} by identifying the term $\displaystyle \frac1{\delta} \int_{j\Delta t}^{(j+1)\Delta t}\int_{T^\delta_{\eta^*_N}}v_N\bd{e}_z\cdot\bq$ as an external force.
\section{Passage to the time discretization limit $N \to \infty$}\label{sec:timedis}

For each time discretization parameter $N$, we have an approximate solution $(\rho_{N}, \bu_{N}, v_{N}, \eta_{N})$ defined on the initially given probability space $(\Omega, \mathcal{F}, \mathbb{P})$ with Wiener processes $\{W^1\}_{t \ge 0}$ and $\{W^2\}_{t \ge 0}$ with respect to the filtration $\{\mathcal{F}_{t}\}_{t \ge 0}$. Our goal is to pass to the limit as $N \to \infty$, where we omit the explicit dependence of these random approximate solutions on the remaining parameters $k$, $\ep$, and $\delta$, which we will take to be fixed but arbitrary in the limit passage as $N \to \infty$. We have the following semidiscrete formulation for the approximate solutions $(\rho_{N}, \bu_{N}, v_{N}, \eta_{N})$ defined on $[0, T]$, where we emphasize that we are keeping the Galerkin parameter $n$ constant.

\begin{itemize}
\item \textbf{Continuity equation.} For the (approximate) initial data $\rho_{0, \delta,\ep} \in C^{2 }(\overline{\mathcal{O}_{\alpha}})$ we have that $\rho_N \in C(0, T; C^{2+\nu}(\overline{\mathcal{O}_{\alpha}}))$, for any $\nu>0$, and $\bu_{N} \in C(0, T; X_{n}^{f})$ $\mathbb{P}$-almost surely, and they satisfy the following weak formulation $\mathbb{P}$-almost surely for all $\varphi \in C_{c}^{\infty}(\overline{\mathcal{O}}_{\alpha})$ and for all $t \in [0, T]$:
\begin{equation}\label{continuityN}
\int_{\mathcal{O}_{\alpha}} \rho_{N}(t) \varphi = \int_{\mathcal{O}_{\alpha}} \rho_{0, \delta, \ep} \varphi + \int_{0}^{t} \int_{\mathcal{O}_{\alpha}} \rho_{N} \bu_{N} \cdot \nabla \varphi - \epsilon \int_{0}^{t} \int_{\mathcal{O}_{\alpha}}  \nabla \rho_{N} \cdot \nabla \varphi,
\end{equation}
where $\rho_{N}$ satisfies the Neumann boundary condition at the top and the bottom boundaries of the maximal domain $\sO_\alpha$: $\nabla \rho_{N} \cdot \bd{n}|_{\partial\sO_\alpha} = 0$ for all $t \in [0, T]$.\\
\item \textbf{Fluid and structure momentum equations.} For all deterministic test functions $(\bd{q}, \psi) \in X_n$ defined by \eqref{galerkinXk}, the following weak formulation holds almost surely and for every $t\in [0,T]$:
\begin{multline}\label{Nlimitweak}
\int_{\mathcal{O}_{\alpha}} \rho_{N}(t) \bu_{N}(t) \cdot \bd{q} + \int_{\Gamma} v_{N}(t) \psi = \int_{\mathcal{O}_{\alpha}} \bp_{0,\delta,\ep} \cdot \bd{q} + \int_{\Gamma} v_{0}\psi -  \int_{0}^{t} \int_{\Gamma}  \nabla \eta_{N} \cdot \nabla \psi -  \int_{0}^{t} \int_{\Gamma}  \Delta \eta_{N} \Delta \psi \\
+  \int_{0}^{t} \int_{\mathcal{O}_{\alpha}}  (\rho_{N} \bu_{N} \otimes \bu_{N}) : \nabla \bd{q} +  \int_{0}^{t} \int_{\mathcal{O}_{\alpha}} \Big(a \rho_{N}^{\gamma} + \delta \rho_{N}^{\beta}\Big) (\nabla \cdot \bd{q}) \\
-  \int_{0}^{t} \int_{\mathcal{O}_{\alpha}} \mu_\delta^{\eta^*_N}  \nabla \bu_{N} : \nabla \bd{q} -  \int_{0}^{t} \int_{\mathcal{O}_{\alpha}}\lambda_\delta^{\eta^*_{ N}}  \text{div}(\bu_{N}) \text{div}(\bd{q}) +\ep\int_0^t\int_{\sO_\alpha}{{\rho_N\bu_N}} \cdot \Delta\bq\\
- \frac{1}{\delta} \int_{0}^{t} \int_{T_{\mathcal{T}_{\Delta t}\eta^{*}_{N}}^{\delta}}  \left( v_{N} \bd{e}_{z}-\mathcal{T}_{\Delta t}\bu_{N} \right) \cdot \psi \bd{e}_{z} - \frac{1}{\delta} \int_{0}^{t} \int_{T_{\eta^{*}_{N}}^{\delta}}  \left(\bu_{N} - v_{N} \bd{e}_{z}\right) \cdot \bq\\
-  \int_{0}^{t} \int_{\Gamma}  \nabla v_N \cdot \nabla \psi +  \int_{0}^{t} \int_{\mathcal{O}_{\alpha}} \mathbbm{1}_{\sO_{\eta_N^*}}\bd{F}_{n}(\rho_{N}, \rho_{N}\bu_{N}) \cdot \bd{q} dW^1 +  \int_{0}^{t} \int_{\Gamma} G_{n}(\eta_{N}, v_{N}) \psi dW^2,
\end{multline}
where $v_N=\partial_t\eta_N$.
\end{itemize}

To pass to the limit in the approximate solutions as $N \to \infty$, we obtain uniform energy estimates for the solutions that are independent of the parameter $N$.

\medskip

\noindent \textbf{Energy estimate.} We apply the It\^o formula to the functional $\Psi_{st}(v)=\|v\|^2_{L^2(\Gamma)}$ in \eqref{structuresubproblem} and to $\Psi_f(\rho,\bu)=\frac12\langle\mathcal{M}^{-1}[\rho]\bu,\bu\rangle$ in \eqref{fluidsubproblem}. {{Since this computation is standard, we will only describe how the penalty terms are treated and refer the reader to \cite{BFH18, BreitHofmanova} for the rest of the details.}} We will describe how the penalty terms in \eqref{structuresubproblem} and \eqref{fluidsubproblem} are treated. Observe that, since $\partial_t\eta_N = v_N$, we can write 
\begin{align*}
   & \frac1\delta\int_{j\Delta t}^{(j+1)\Delta t} \int_{T^\delta_{\mathcal{T}_{\Delta t}\eta^*_N}}(\partial_t\eta_N\bd{e}_z-\mathcal{T}_{\Delta t}\bu_N)\cdot\partial_t\eta_N \bd{e}_z \\
    & = \frac1{2\delta}\int_{j\Delta t}^{(j+1)\Delta t} 
   \int_{T^\delta_{\mathcal{T}_{\Delta t}\eta^*_N}}\left( |\partial_t\eta_N-\mathcal{T}_{\Delta t}\bu_N|^2+|\partial_t\eta_N|^2-|\mathcal{T}_{\Delta t}\bu_N|^2\right)\\
   &= \frac1{2\delta}\int_{j\Delta t}^{(j+1)\Delta t} 
   \int_{T^\delta_{\mathcal{T}_{\Delta t}\eta^*_N}}\left( |\partial_t\eta_N-\mathcal{T}_{\Delta t}\bu_N|^2-|\mathcal{T}_{\Delta t}\bu_N|^2\right) +\frac1{2\delta^{\frac12+\frac1\beta}}\int_{j\Delta t}^{(j+1)\Delta t}\int_\Gamma |\partial_t\eta_N|^2\\
     &= \frac1{2\delta}\int_{j\Delta t}^{(j+1)\Delta t} 
   \int_{T^\delta_{\mathcal{T}_{\Delta t}\eta^*_N}}|v_N-\mathcal{T}_{\Delta t}\bu_N|^2-\frac1{2\delta}\int_{(j-1)\Delta t}^{j\Delta t} 
   \int_{T^\delta_{\eta^*_N}}|\bu_N|^2+\frac1{2\delta^{\frac12+\frac1\beta}}\int_{j\Delta t}^{(j+1)\Delta t}\int_\Gamma |v_N|^2,
\end{align*}
{{where we used the definition of the tubular neighborhood \eqref{tube} and the fact that $v_N$ is constant along the $z$ coordinate since it is defined only on $\Gamma$}}. Similarly, for the penalty term appearing in \eqref{fluidsubproblem} we obtain,
\begin{align*}
    &\frac1{\delta} \int_{j\Delta t}^{(j+1)\Delta t}\int_{T^\delta_{\eta^*_N}}(\bu_N-v_N\bd{e}_z) \cdot \bu_N = \frac1{2\delta}  \int_{j\Delta t}^{(j+1)\Delta t}\int_{T^\delta_{\eta^*_N}} |\bu_N-v_N\bd{e}_z|^2+|\bu_N|^2-|v_N\bd{e_z}|^2\\
    &= \frac1{2\delta}  \int_{j\Delta t}^{(j+1)\Delta t}\int_{T^\delta_{\eta^*_N}} |\bu_N-v_N\bd{e}_z|^2+ \frac1{2\delta}  \int_{j\Delta t}^{(j+1)\Delta t}\int_{T^\delta_{\eta^*_N}} |\bu_N|^2-\frac1{2\delta^{\frac12+\frac1\beta}} \int_{j\Delta t}^{(j+1)\Delta t}\int_\Gamma |v_N|^2.
\end{align*}
This gives us the following energy estimate for any $t\in [j\Delta t,(j+1)\Delta t]:$
\begin{multline*}
\frac{1}{2} \int_{\mathcal{O}_{\alpha}} \rho_{N}(t) |\bu_{N}(t)|^{2} + \frac{1}{2} \int_{\Gamma} |v_{N}(t)|^{2} + \int_{\mathcal{O}_{\alpha}} \frac{a}{\gamma-1}  \Big(\rho_{N}(t)\Big)^{\gamma} + \int_{\mathcal{O}_{\alpha}} \frac{\delta}{\beta-1} \Big(\rho_{N}(t)\Big)^{\beta} +\int_\Gamma (|\nabla \eta_N|^2 + |\Delta \eta_N|^2)\\ 
+ \int_{0}^{t} \int_{\mathcal{O}_{\alpha}}\mu_\delta^{\eta^*_N}  |\nabla \bu_{N}|^{2}
 + \int_{0}^{t} \int_{\mathcal{O}_{\alpha}}  \lambda_\delta^{\eta^*_N} |\text{div}(\bu_{N})|^{2} +\ep\int_0^t\int_{\sO_\alpha} \rho_N|\nabla\bu_N|^2 +
\int_{0}^{t} \int_{\Gamma}|\nabla v_N|^{2}\\
+ \varepsilon \int_{0}^{t} \int_{\mathcal{O}_{\alpha}}  (a\gamma \rho_{N}^{\gamma - 2} + \delta \beta \rho_{N}^{\beta - 2}) |\nabla \rho_{N}|^{2} + \frac{1}{2\delta} \int_{0}^{t} \int_{T^{\delta}_{\mathcal{T}_{\Delta t}\eta^{*}_{N}}}  |\mathcal{T}_{\Delta t}\bu_{N} - v_{N} \bd{e}_{z}|^{2}+\frac{1}{2\delta} \int_{0}^{t} \int_{T^{\delta}_{\eta^{*}_{N}}}  |\bu_{N} - v_{N} \bd{e}_{z}|^{2} \\
+\frac1{2\delta} \int_{j\Delta t}^t\int_{T^\delta_{\eta^*_N}}|\bu_N|^2+ \int_{0}^{t} \int_{\mathcal{O}_{\alpha}} \mathbbm{1}_{\sO_{\eta_N^*}}\bd{F}_{n}(\rho_{N}, \rho_{N} \bu_{N}  )\cdot \bu_{N} dW^1 +  \int_{0}^{t} \int_{\Gamma}  G_n(\eta_{N}, v_{N}) v_{N} dW^2 \\
+ \sum_{k = 1}^{\infty} \int_{0}^{t} \int_{\mathcal{O}_{\alpha}} \rho_{N}^{-1} \mathbbm{1}_{\sO_{\eta^*_N}}|f_{n, k}( \rho_{N},  \rho_{N} \bu_{N})|^{2} + \sum_{k = 1}^{\infty} \int_{0}^{t} \int_{\Gamma}  |g_{n, k}(\eta_{N}, v_{N})|^{2} 
\\ \leq \frac{1}{2} \int_{\mathcal{O}_{\alpha}} \frac{|\bp_{0,\delta,\ep}|^{2}}{\rho_{0,\delta,\ep} } + \frac{1}{2} \int_{\Gamma} |v_{0}|^{2} + \int_{\mathcal{O}_{\alpha}} a \Big(\rho_{0,\delta,\ep}\Big)^{\gamma} + \int_{\mathcal{O}_{\alpha}} \delta \Big(\rho_{0,\delta,\ep}\Big)^{\beta} {{ + \frac{\Delta t}{2\delta^{\left(\frac{1}{2} + \frac{1}{\beta}\right)}}\int_{T^{\delta}_{\eta^{*}_{N}}} |P_n^f\bd{u}_{0}|^{2}}} , 
\end{multline*}
 where we recall the definition of $f_{n, k}$ from \eqref{Fndef} and where $g_{n, k}(\rho, \bd{q}) := G_{n}(\rho, \bd{q}) \bd{e}_{k}$ for $G_{n}$ as defined in \eqref{Gndef}. 
 Here we have additionally used the facts that, due to the no-slip boundary conditions for $\bu_N$ and the Neumann boundary conditions for $\rho_N$ on $\partial\sO_\alpha$, we have,
\begin{align*}
&	\int_{\mathcal{O}_{\alpha}} (\rho_{N} \bu_{N} \otimes \bu_{N}) : \nabla {\bu_N} = -\int_{\mathcal{O}_{\alpha}} \text{div}((\rho_{N} \bu_{N} \otimes \bu_{N}))  {\bu_N} =-\frac12 \int_{\mathcal{O}_{\alpha}} \text{div}(\rho_N\bu_N)|\bu_N|^2,\\
&	\int_{\sO_{\alpha}}\bu_N\rho_N {\cdot \Delta\bu_N} =-\frac12\int_{\sO_\alpha}\nabla\rho_N\cdot\nabla|\bu_N|^2-\int_{\sO_{\alpha}}\rho_N|\nabla\bu_N|^2.
\end{align*}
Now we raise this equation to a power of $p\geq 1$, then take $\sup_{0\leq t\leq T}$ and expectation, on both sides of this equation. We deal with the stochastic term by applying the Burkholder-Davis-Gundy (BDG) inequality and by using the growth assumptions \eqref{fassumption} on the noise coefficients as follows:
\begin{align*}
	\bE\Bigg(\sup_{0\leq t\leq T} \Bigg|\int_0^t\sum_{k=1}^\infty \int_{\sO_{{\alpha}}}& \mathbbm{1}_{\sO_{\eta^*_N}} f_{n,k}( \rho_{N},  \rho_{N} \bu_{N})\cdot\bu_N dW^1\Bigg|\Bigg)^p \\
& \leq \bE\left(\int_0^T  \sum_{k=1}^\infty\left( \int_{\sO_{{\alpha}}} \mathbbm{1}_{\sO_{\eta^*_N}} f_{n,k}( \rho_{N},  \rho_{N} \bu_{N})\cdot\bu_N\right) ^2 dt \right)^{\frac{p}2} \\
&\leq \bE\left(\int_0^T \sum_{k=1}^\infty\left( \|P^f_n(\rho_N^{-\frac12}f_k( \rho_{N},  \rho_{N} \bu_{N}))\|_{L^2(\sO_\alpha)}\|\rho_N^{\frac12}\bu_n\|_{L^2(\sO_\alpha)}\right) ^2 dt \right)^{\frac{p}2}\\
	&\leq \frac12\bE\left(\sup_{0\leq t\leq T}\int_{\sO_{{\alpha}}}\rho_N|\bu_N|^2\right)^p+ {{C \cdot}} \bE\int_0^T\left(\int_{\sO_\alpha}(\rho_N + \rho_N|\bu_N|^2)dx\right)^pdt,
	\end{align*}
where we used the definition of the operator $\mathcal{M}[\rho]$ in \eqref{multiplier} and the definition of $\displaystyle f_{n, k}(\rho, \bd{q}) := \mathcal{M}^{1/2}[\rho] \left(P^{f}_{n}\left(\frac{f_{k}(\rho, \bd{q})}{\rho^{1/2}}\right)\right)$ in \eqref{Fndef}. We can also estimate the quadratic variation term using the growth condition \eqref{fassumption} and the definition \eqref{Fndef} of $f_{n, k}$:
\begin{align*}
    \bE\sup_{0\leq t\leq T} \left(\sum_{k = 1}^{\infty} \int_{0}^{t} \int_{\mathcal{O}_{{\alpha}}}  \rho_{N}^{-1} \mathbbm{1}_{\mathcal{O}_{\eta^{*}_{N}}} |f_{n, k}(\rho_{N}, \rho_{N}\bu_{N})|^{2}\right)^p & \leq   \sum_{k = 1}^{\infty} \bE\left(\int_{0}^{T} \int_{\mathcal{O}_{{\alpha}}}  \left|\left( \rho_{N}^{-\frac12} f_{k}( \rho_{N},  \rho_{N} \bu_{N})\right) \right|^{2}\right)^p\\
    &\leq \bE \int_0^T\left(\int_{\sO_\alpha}(\rho_N  + \rho_N|\bu_N|^{2})\right)^p dt.
\end{align*}
The other stochastic integral for the structure dynamics is treated identically.

An application of the Gronwall inequality then implies for any $1 \le p < \infty$ that
\begin{multline}\label{moment_timedis}
	\bE\Bigg[\sup_{0\leq t\leq T}\Big( 	\int_{\mathcal{O}_{\alpha}} \rho_{N}(t) |\bu_{N}(t)|^{2} + \int_{\Gamma} |v_{N}(t)|^{2} + \int_{\mathcal{O}_{\alpha}} \frac{a}{\gamma - 1} \Big(\rho_{N}(t)\Big)^{\gamma} + \frac{\delta}{\beta - 1} \Big(\rho_{N}(t)\Big)^{\beta} \\
 +\int_\Gamma (|\nabla \eta_N (t)|^2 + |\Delta \eta_N(t)|^2) +\int_0^t\int_\Gamma |\nabla v_N|^2 
	 +  \int_{0}^{t} \int_{\mathcal{O}_{\alpha}} \mu^{\eta^*_N}_\delta |\nabla \bu_{N}|^{2} + \int_{0}^{t}\int_{\mathcal{O}_{\alpha}} \lambda^{\eta^*_N}_\delta | \text{div}(\bu_{N})|^{2}\\ +\ep\int_0^t\int_{\sO_\alpha}\rho_N|\nabla\bu_N|^2
		+ \varepsilon \int_{0}^{t} \int_{\mathcal{O}_{\alpha}} \left(\frac{4a}{\gamma} |\nabla(\rho^{\gamma/2}_{N})|^{2} + \frac{4\delta}{\beta} |\nabla (\rho^{\beta/2}_{N})|^{2}\right) \\
  + \frac{1}{2\delta} \int_{0}^{t} \int_{T^{\delta}_{\mathcal{T}_{\Delta t}\eta^{*}_{N}}}  |\mathcal{T}_{\Delta t}\bu_{N} - v_{N} \bd{e}_{z}|^{2}+\frac{1}{2\delta} \int_{0}^{t} \int_{T^{\delta}_{\eta^{*}_{N}}}  |\bu_{N} - v_{N} \bd{e}_{z}|^{2} \Big)\Bigg]^p\\
		\le C_p \Bigg( \frac{1}{2} \int_{\mathcal{O}_{\alpha}}  \frac{|\bp_{0,\delta,\ep}|^{2}}{\rho_{0,\delta,\ep}} + \frac{1}{2} \int_{\Gamma} |v_{0}|^{2} + \int_{\mathcal{O}_{\alpha}} a\Big(\rho_{0,\delta,\ep}\Big)^{\gamma} + \delta \Big(\rho_{0,\delta,\ep}\Big)^{\beta}\Bigg)^p,
\end{multline}
for a constant $C_p$ depending on $p$. 
Our next goal is to upgrade the weak convergences that are implied by the uniform in $N$ bounds \eqref{moment_timedis} to almost sure convergence results in appropriate topologies. This will be done by proving that the laws of the approximate solutions are tight in their respective phase spaces.
{ \begin{remark}
    Due to the weak lower semicontinuity of norm, the inequality \eqref{moment_timedis} will hold as we pass $N\to \infty$ in this Section and then as $\ep\to0, \delta\to 0$ in the subsequent sections.
\end{remark}}
\subsection{Tightness result for the limit passage $N \to \infty$}

To pass to the limit, we will use the following path space, which will be the path space for the approximate solutions $(\rho_{N}, \bu_{N}, \eta_{N}, \eta_{N}^{*}, v_{N}, W^1, W^2)$:
\begin{equation}\label{phaseN}
\mathcal{X} = \mathcal{X}_{\rho} \times \mathcal{X}_{\bu} \times \mathcal{X}_{\eta} \times \mathcal{X}_{\eta} \times \mathcal{X}_{v} \times \mathcal{X}_{W},
\end{equation}
where 
\begin{align*}
&\mathcal{X}_{\rho} = C(0, T; L^{\beta}(\sO_\alpha)) \cap L^{\beta}(0, T; W^{1, \beta}(\sO_\alpha)), \qquad \mathcal{X}_{\bu} = C(0, T; L^{2}(\mathcal{O}_{\alpha})),\\
 &\mathcal{X}_{\eta} =  C(0, T; H^s(\Gamma)) \cap (L^{\infty}(0, T; H^2(\Gamma)),{w^*}), \qquad s\in\left(\frac32,2\right),\\
&\mathcal{X}_{v} = C(0, T; L^{2}(\Gamma)), \qquad 
 \mathcal{X}_{W} = C(0, T; \mathcal{U}_0)^2,
\end{align*}
We will denote the law of the approximate solution $(\rho_{N}, \bu_{N}, \eta_{N}, \eta_{N}^{*}, v_{N}, W^1, W^2)$ in the path space $\mathcal{X}$ by $\mu_{N}$, and we will prove in this subsection the following tightness result.

\begin{proposition}
The collection of laws $\{\mu_{N}\}_{N = 1}^{\infty}$ on the path space $\mathcal{X}$ is tight.
\end{proposition}

\begin{proof}
We show tightness for each component separately.

\medskip
\noindent \textbullet{} \textbf{Tightness for structure displacements.} 
First, we show tightness of the laws for the structure displacements $\mu_{\eta}$ in $L^{2}(0, T; L^{2}(\Gamma))$. 
Recall that we have,
\begin{equation}\label{etaNH02}
\mathbb{E}\Big(\|\eta_{N}\|_{W^{1, \infty}(0, T; L^{2}(\Gamma))}^{2} + \|\eta_{N}\|_{L^{\infty}(0, T; H^{2}(\Gamma))}^{2}\Big) \le C
\end{equation}
uniformly in $N$. Combined with the compact embedding 
\begin{equation*}
W^{1, \infty}(0, T; L^{2}(\Gamma)) \cap L^{\infty}(0, T; H^2(\Gamma)) \subset \subset  C(0,T;H^s(\Gamma)),\quad \text{ for any }  s < 2 ,
\end{equation*}
provided by the Aubin-Lions compactness theorem, this shows tightness of $\mu_{\eta}$ in $C(0, T; H^s(\Gamma))$ for our fixed choice of $s\in(\frac32,2)$.

\medskip

\noindent \textbullet{} \textbf{Tightness for structure velocities.} To show the tightness of the laws $\mu_{v}$ in $C(0, T; L^{2}(\Gamma))$, we show tightness of the laws in $C(0, T; X_n^{st})$ for the Galerkin space $X_n^{st}$ defined in \eqref{Xks}. We consider the weak formulation \eqref{Nlimitweak} and note that for all $\psi \in X_n^{st}$ and $t_1 < t_2$:
\begin{multline*}
\int_{\Gamma} (v_{N}(t_2) - v_{N}(t_1)) \psi = \frac{1}{\delta} \int_{t_1}^{t_2} \int_{T^{\delta}_{\mathcal{T}_{\Delta t}\eta^{*}_{N}}} (v_{N} \bd{e}_{z} - \mathcal{T}_{\Delta t}\bu_{N}) \cdot \psi \bd{e}_{z} \\
+  \int_{t_1}^{t_2} \int_{\Gamma}  \nabla \eta_{N} \cdot \nabla \psi +  \int_{t_1}^{t_2} \int_{\Gamma}  \Delta \eta_{N} \Delta \psi + \int_{t_{1}}^{t_{2}} \int_{\Gamma} \nabla v_{N} \cdot \nabla \psi -  \int_{t_1}^{t_2}\int_{\Gamma} G(\eta_{N}, v_{N}) \psi dW^2(t). 
\end{multline*}
If $\|\psi\|_{X_n^{st}} \le 1$, we have the following estimates:
\begin{enumerate}
    \item By using the uniform estimates of $v_{N} \in L^{2}(\Omega; L^{\infty}(0, T; L^{2}(\Gamma)))$ and $\bu_N \in L^{2}(\Omega; L^{2}(0, T; H^{1}(\mathcal{O}_{\alpha})))$,
    \begin{equation*}
        \mathbb{E} \left|\int_{t_1}^{t_2} \int_{T^{\delta}_{\mathcal{T}_{\Delta t}\eta^{*}_{N}}} (v_{N}\bd{e}_{z} - \mathcal{T}_{\Delta t} \bu_{N}) \cdot \psi \bd{e}_{z}\right| \le C|t_1 - t_2|^{1/2},
    \end{equation*}
    for a constant $C$ that is independent of $N$.
    \item Since $\eta_{N} \in L^{2}(\Omega; L^{\infty}(0, T; H^2(\Gamma)))$ by \eqref{etaNH02}, for a constant $C$ that is independent of $N$, we estimate $\displaystyle \mathbb{E} \left|\int_{t_1}^{t_2} \int_{\Gamma} \nabla \eta_{N} \cdot \nabla \psi\right|$ and $\displaystyle \mathbb{E} \left|\int_{t_1}^{t_2} \int_{\Gamma}  \Delta \eta_{N} \Delta \psi\right|$ by 
    \begin{equation*}
         \le \mathbb{E} \int_{t_1}^{t_2} \|\eta_{N}\|_{L^{\infty}(0, T; H^2(\Gamma))} \|\psi\|_{H^2(\Gamma)} \le |t_1 - t_2| \cdot \mathbb{E} \|\eta_{N}\|_{L^{\infty}(0, T; H^2(\Gamma))}^{2} \le C |t_1 - t_2|. 
    \end{equation*}
    Similarly, by the viscoelasticity estimate $v_{N} \in L^{2}(\Omega; L^{2}(0, T; H^{1}(\Gamma)))$, we obtain that 
    \begin{equation*}
    \mathbb{E}\left|\int_{t_{1}}^{t_{2}} \int_{\Gamma} \nabla v_{N} \cdot \nabla \psi\right| \le C|t_{1} - t_{2}|^{1/2}.
    \end{equation*}
    \item Finally, we use the BDG inequality and the fact that $\|\psi\|_{C(\Gamma)} \le C\|\psi\|_{H^2(\Gamma)} \le C\|\psi\|_{X^{st}_n} \le C$ for a constant $C$ that depends only on $n$ (and is independent of $N$) to estimate:
    {{\begin{equation*}
    \mathbb{E} \left|\int_{t_1}^{t_2} \int_{\Gamma} G_n(\eta_{N}, v_{N}) \psi dW^2\right|^{p} \le C\mathbb{E} \left|\int_{t_1}^{t_2} \left(\int_{\Gamma} G_n(\eta_{N}, v_{N})\right)^{2}\right|^{p/2} \le C\mathbb{E} \left|\int_{t_1}^{t_2} \int_{\Gamma} \Big(G_n(\eta_N, v_N)\Big)^{2} \right|^{p/2}. 
    \end{equation*}
    We then use \eqref{gassumption}, \eqref{Gndef}, and the fact that $P_{n}^{st}$ is a bounded operator from $L^{2}(\Gamma)$ to itself, to estimate that this is:
    \begin{align*}
       & \le C\mathbb{E} \left(\int_{t_1}^{t_2} \int_{\Gamma} \sum_{k = 1}^{\infty} |g_{k}(\eta_{N}, v_{N})|^{2}\right)^{p/2} \le C\mathbb{E} \left(\int_{t_1}^{t_2}\int_{\Gamma} (1 + |\eta_{N}|^{2} + |v_{N}|^{2})\right)^{p/2} \\
        &\le C|t_1 - t_2|^{p/2} \mathbb{E}\Big(1 + \|\eta_{N}\|_{L^{\infty}(0, T; H^2(\Gamma))}^{2} + \|v_{N}\|^{2}_{L^{\infty}(0, T; L^{2}(\Gamma))}\Big)^{p/2} \le C|t_1 - t_2|^{p/2}.
    \end{align*}}}
\end{enumerate}
\textit{Conclusion of equicontinuity estimate.} Therefore, by applying the Kolmogorov continuity criterion, we obtain the following equicontinuity estimate on the {}{$C^{1/3}$-seminorm of $v_N$:
\begin{equation*}
        \mathbb{E} [v_{N}]_{C^{1/3}(0, T; (X_n^{st})^{*})} \le C, \qquad \text{ independently of $N$}.
\end{equation*}
Furthermore, since $v_{N}$ are uniformly bounded in $L^{2}(\Omega; L^{\infty}(0, T; X_n^{st}))$ where $X_n^{st}$ is the finite dimensional subspace of $L^{2}(\Gamma)$ defined in \eqref{Xks}, we can use the finite-dimensionality of $X_n^{st}$ as follows. We can identify $(X_n^{st})^{*}$ with $X_n^{st}$ and since $X_n^{st}$ is a finite-dimensional subspace of $L^{2}(\Gamma)$, we can conclude tightness of the laws $\mu_{v}$ in $C(0, T; X_n^{st})$ and hence $C(0, T; L^{2}(\Gamma))$, from the equicontinuity estimate using the Arzela-Ascoli compactness criterion.

\medskip

\noindent \textbullet{} \textbf{Tightness for fluid densities/velocities.} The proof of tightness for the components involving the fluid density and fluid velocity closely follow the proofs given in Section 4.3.2 in \cite{BFH18} and Lemma 4.4 in \cite{Smith} on stochastic compressible isentropic Navier-Stokes equations on a fixed domain, so we just outline the main ideas here, without providing explicit details. 

The first difficulty is that we have a bound on the kinetic energy $\displaystyle \sup_{t \in [0, T]} \int_{\sO_\alpha} \rho_{N}(t) |\bu_{N}(t)|^{2}$ but we want a bound on $\displaystyle \sup_{t \in [0, T]} \int_{\sO_\alpha} |\bu_{N}(t)|^{2}$, which requires an $L^{\infty}$ estimate on the inverse of the density $\rho_{N}^{-1}$ that is uniform in $N$. {{In particular, we want to establish the following uniform convergence estimate (in probability) for the fluid velocity:
\begin{equation}\label{uNbound}
\lim_{M \to \infty} \mathbb{P}\left(\sup_{t \in [0, T]} \int_{\mathcal{O}_{\alpha}} |\bd{u}_{N}(t)|^{2} \ge M\right) = 0, \qquad \text{ uniformly in $N$},
\end{equation}
as this bound will be useful for establishing tightness of the fluid densities/velocities.}} To establish \eqref{uNbound}, it suffices to estimate the following probabilities (uniformly, independently of $N$):
\begin{equation}\label{Nprob}
\mathbb{P}\left(\sup_{t \in [0, T]} \int_{\sO_\alpha} \rho_{N}(t)|\bu_{N}(t)|^{2} \ge M\right), \qquad \mathbb{P}\left(\|\rho_{N}^{-1}\|_{L^{\infty}([0, T] \times \mathcal{O}_{\alpha})} \ge M \right).
\end{equation}
We can estimate the first probability in \eqref{Nprob} by using the uniform boundedness of the kinetic energy, and then to estimate the second probability, we recall that by construction, for fixed $\ep$ and $\delta$, we have that the initial data $\rho_{0, \ep, \delta}$ (which is the initial data for all $N$) satisfies
\begin{equation*}
0 < \ep \le \rho_{0,\delta,\ep} \le \delta^{-1/\beta},
\end{equation*}
see \eqref{rho0conv}. By using the fluid dissipation estimate and Poincar\'{e}'s inequality, $\mathbb{E}\left(\|\bu_{N}\|_{L^{2}(0, T; X_{n}^{f})}^{2}\right) < \infty$. So by equivalence of norms in $X_{n}^{f}$ and the embedding $H^l(\sO_\alpha)$ for $l > 5/2$ into $W^{1, \infty}(\sO_\alpha)$:
\begin{equation*}
\mathbb{E}\left(\|\text{div}(\bu_{N})\|_{L^{2}(0, T; L^{\infty}(\mathcal{O}_{\alpha}))}^2\right) < \infty, \quad \text{uniformly in $N$.} 
\end{equation*}
Hence, by the comparison principle (see e.g. Lemma 2.2 in \cite{FeireislCompressible}),
\begin{equation}\label{comparison}
0<{ \ep} \exp\left(-\int_{0}^{t} \|\text{div} \bu_N(s)\|_{L^{\infty}(\mathcal{O}_{\alpha})} ds\right) \le \rho_N \le { \delta^{-\frac1\beta}} \exp\left(\int_{0}^{t} \|\text{div}\bu_N(s)\|_{L^{\infty}(\mathcal{O}_{\alpha})} ds\right).
\end{equation}
we have a bound on the second probability in \eqref{Nprob}. Thus, we can conclude that the desired uniform convergence \eqref{uNbound} holds as a result of the estimates of both probabilities in \eqref{Nprob}.
We can use this bound along with the following observations to show tightness of the laws of the fluid densities and the fluid velocities. 

\medskip

\noindent \textbf{Tightness of fluid densities.} Using a regularity result (Theorem A.2.2 in \cite{BFH18}, Lemma B.7 in \cite{Smith}) and performing several standard estimates as in Lemma 4.4 in \cite{Smith}, we can obtain the following two estimates: For $\frac{\beta}{\beta + 1} = \theta \left(1 - \frac{2}{3}\right) + (1 - \theta)$ and for some $\beta < q < \infty$ and $1 \le r < \beta$,
\begin{align*}
\|\partial_{t}\rho_{N}\|&_{L^{\beta}((0, T) \times \mathcal{O}_{\alpha})} + \|\rho_{N}\|_{L^{\beta}(0, T; W^{2, \beta}(\mathcal{O}_{\alpha}))} \\
&\le C\left(\|\rho_{0, \delta, \ep}\|_{W^{2, q}(\mathcal{O}_{\alpha})} + \|\bu_{N}\|^{2}_{C(0, T; X^f_{n})} \| \rho_{N}\|^{\theta}_{L^{\infty}(0, T; L^{\beta}(\mathcal{O}_{\alpha}))} \| \rho_{N}^{\beta/2}\|^{\frac{1 - \theta}{2\beta}}_{L^{2}(0, T; H^{1}(\mathcal{O}_{\alpha}))}\right),
\end{align*}
\begin{equation*}
\|\partial_{t} \nabla \rho_{N}\|_{L^{r}((0, T) \times \mathcal{O}_{\alpha})} + \|\nabla \rho_{N}\|_{L^{r}(0, T; W^{2, r}(\mathcal{O}_{\alpha}))} \le C\|\bu_{N}\|_{C(0, T; X^f_{n})} \|\rho_{N}\|_{L^{\beta}(0, T; W^{2, \beta}(\mathcal{O}_{\alpha}))},
\end{equation*}
for a constant $C$ that is independent of $\omega$, $N$, and the remaining approximation parameters. Hence, by using \eqref{uNbound} and the bounds on $\rho_{N}$ in the uniform energy estimate \eqref{moment_timedis}, this gives a uniform estimate on $\|\partial_{t}\rho_{N}\|_{L^{\beta}((0,T) \times \sO_\alpha)}$, $\|\partial_{t} \nabla \rho_{N}\|_{L^{r}((0,T) \times \sO_\alpha)}$, and $\|\rho_{N}\|_{L^{\beta}(0,T;W^{2, \beta}(\sO_\alpha))}$ in probability independently of $N$, which allows us to deduce tightness in the parameter $N$. We obtain tightness of $\rho_{N}$ in $L^{\beta}(0, T; W^{1, \beta}(\mathcal{O}_{\alpha}))$ via Aubin-Lions, using bounds of $\rho_{N}$ in $L^{\beta}(0, T; W^{2, \beta}(\mathcal{O}_{\alpha}))$ and $\partial_{t}\rho_{N}$ in $L^{\beta}((0,T) \times \mathcal{O}_{\alpha})$, and we obtain tightness of $\rho_{N}$ in $C(0, T; L^{\beta}(\mathcal{O}_{\alpha}))$ via Arzela-Ascoli compactness arguments and Sobolev embedding, using bounds of $\rho_{N}$ and $\partial_{t}\rho_{N}$, both in $L^{r}(0, T; W^{1, r}(\mathcal{O}_{\alpha}))$ where we use the regularity of the time derivative to get a bound on $\rho_{N}$ in $C(0, T; W^{1, r}(\mathcal{O}_{\alpha}))$ for an appropriately chosen $1 \le r < \beta$. 
\medskip

\noindent \textbf{Tightness of fluid velocities.} We can use the weak formulation for the momentum equation and $L^{\infty}$ estimates on $\rho_{N}^{-1}$ uniformly in $N$, to get an estimate on the increments of the fluid velocity. Such an increment estimate would allow us to conclude that
\begin{equation*}
\lim_{M \to \infty} \mathbb{P}\left([\bu_{N}]_{C^{1/3}(0,T;X_{n}^{f})} \ge M\right) = 0, \quad \text{ uniformly in $N$},
\end{equation*}
where $[\cdot]_{C^{1/3}(0,T;X_{n}^{f})}$ denotes the $1/3$-H\"{o}lder seminorm for a function taking values in $X^f_{n}$. Combining this with the uniform bound \eqref{uNbound} and the fact that $X^f_{n}$ is finite-dimensional allows us to conclude the tightness of the laws of the approximate fluid velocities via Arzela-Ascoli.
}
\end{proof}

\subsection{Identification of the limit as $N \to \infty$}

Next, we use the Skorohod representation theorem in conjunction with the result of \cite{J97}, in order to obtain a limiting random variable as $N \to \infty$ keeping all other approximation parameters fixed, on a new probability space $(\tilde{\Omega}, \tilde{\mathcal{F}}, \tilde{\mathbb{P}})${\footnote{ Note, here and later, that in the proof of this version of the Skorohod representation theorem $\tilde{ \Omega}=[0,1)\times [0,1)$, $\tilde{\mathcal{F}}$ is the Borel algebra {on $\tilde \Omega$} and $\tilde \bP$ is the Lebesgue measure on $\tilde \Omega$ and thus $(\tilde \Omega, \tilde{\mathcal{F}},\tilde \bP)$ is independent of all the approximate parameters.}}.

Since we are only allowing the time discretization parameter $N$ to vary while keeping all other parameters $n, \ep, \delta$ fixed, we will notate only the dependence of the approximate solutions on $N$, as $\mathcal{U}_{N} := (\rho_{N}, \bu_{N}, \eta_{N}, \eta^{*}_{N}, v_{N}, W^1, W^2)$,  which we consider as random variables taking values in $\mathcal{X}$ defined by \eqref{phaseN}. 

\begin{theorem}\label{skorohod_time}
There exists a probability space $(\tilde{\Omega}, \tilde{\mathcal{F}}, \tilde{\mathbb{P}})$ and random variables
\begin{equation*}
\tilde{\mathcal{U}}_{N} = (\tilde{\rho}_{N}, \tilde{\bu}_{N},   \tilde{\eta}_{N}, \tilde{\eta}^{*}_{N},  \tilde{v}_{N}, \tilde W_N^1, \tilde W_N^2) \text{ for $N = 1, 2, ...$}, \quad 
{\mathcal{U}} = (\rho, \bu,\bu, \eta,  \eta^{*}, v, \tilde W^1, \tilde W^2),
\end{equation*}
defined on this new probability space, such that 
\begin{enumerate}
\item $\tilde{\mathcal{U}}_{N}$ has the same law in $\mathcal{X}$ as $\mathcal{U}_{N}$,
\item $
\tilde{\mathcal{U}}_{N} \to \mathcal{U} \text{ in the topology of $\mathcal{X}$, $\tilde{\mathbb{P}}$-almost surely as $N \to \infty$},$
\item $\tilde \eta^*_N=\tilde \eta_N$ for every $t<\tau^\eta_N$ where, for the fixed $s\in(\frac32,2)$,
		\begin{align*}
			\tau^{\eta}_N &:=T\wedge \inf\left\{t> 0:\inf_{\Gamma}(1+{\tilde\eta_N}(t))\leq \alpha \text{ or } \|\tilde{\eta}_N(t)\|_{H^s(\Gamma)}\geq \frac1{\alpha}\right\}.
		\end{align*}
\item $\partial_{t}\eta = v$, $\tilde{\mathbb{P}}$-almost surely.
\end{enumerate}
\end{theorem}

\begin{proof}
The first two statements essentially follow from the Skorohod representation theorem.  Let $\tilde\sF_t'$ be the $\sigma$-field generated by the random variables ${\bu}(s), v(s), \eta(s),\tilde W^1(s),\tilde W^2(s)$, for all $s \leq t$. Then we define
\begin{align}\label{Ft}
\mathcal{N} &:=\{\mathcal{A}\in \tilde{\mathcal{F}} \ | \ \tilde \bP(\mathcal{A})=0\}\qquad
\tilde{\mathcal{F}}_t :=\bigcap_{s\ge t}\sigma(\tilde{\mathcal{F}}_s' \cup \mathcal{N}).
\end{align}
This gives a complete, right-continuous filtration $(\tilde{\mathcal{F}}_t)_{t \in [0,T]}$, dependent on {{$n,\ep,\delta$}}, on the new probability space $(\tilde \Omega,\tilde{\mathcal{F}},\tilde \bP)$, to which the limiting noise processes and solutions are adapted. The filtration $(\tilde{\mathcal{F}}^N_t)_{t \in [0,T]}$ is constructed similarly.
To prove that the new random variables $\tilde{\mathcal{U}}_N$ satisfy the same weak formulation as the old random variables, we apply the same classical arguments as in \cite{Ben} (see also Theorem 2.9.1 in \cite{BFH18}). 

The third statement follows from Proposition \ref{equalinlaw} in Appendix A.
The fact that $\partial_{t}\eta = v$, $\tilde{\mathbb{P}}$-almost surely follows from the fact that $\partial_{t}\eta_{N} = v_{N}$, $\mathbb{P}$-almost surely by definition of the splitting scheme on the original probability space. So by equivalence of laws, $\partial_{t}\tilde{\eta}_{N} = \tilde{v}_{N}$, $\tilde{\mathbb{P}}$-almost surely, and then we obtain $\partial_{t}\eta = v$, $\tilde{\mathbb{P}}$-almost surely by passing to the limit as $N \to \infty$ and using the convergences given by the Skorohod representation theorem.

We will now prove that $\tilde{\eta}^{*}_{N}$ and $\mathcal{T}_{\Delta t}\tilde{\eta}^{*}_{ N}$ have the same limit $\eta^{*}$ in $L^\infty(0,T;H^2(\Gamma))$ where, we recall, that $\mathcal{T}_{\Delta t} \tilde\eta^{*}_{N}$ denotes the time shift of $\tilde\eta^{*}_{N}$, defined by $\mathcal{T}_{\Delta t} \tilde\eta^{*}_{N} = \eta_{0}$ on $[0, \Delta t]$ and $\mathcal{T}_{\Delta t} \tilde\eta^{*}_{N}(t) = \tilde\eta^{*}_{N}(t - \Delta t)$ for $t \in [\Delta t, T]$. 
To see this, we compute for all $t \in [0, T]$:
\begin{equation*}
\|\tilde\eta^{*}_{N}(t) - \mathcal{T}_{\Delta t} \tilde\eta^{*}_{N}(t)\|_{H^2(\Gamma)} \le \int_{t - {\Delta t}}^{t} \|\partial_{t}\tilde\eta^{*}_{N}(t)\|_{H^2(\Gamma)} \le C\int_{t - {\Delta t}}^{t} \|\tilde v_{N}(t)\|_{L^{2}(\Gamma)} \le C(\Delta t)\|\tilde v_{N}\|_{C(0, T; L^{2}(\Gamma))},
\end{equation*}
where this calculation is justified by the definition of the stopped process $\eta^{*}_{N}$, the fact that $v_{N} = \partial_{t} \eta_{N}$, and equivalence of norms combined with $v_{N} \in C(0, T; X_{n}^{st})$. We note that $C$ depends only on the Galerkin parameter $n$ (which we emphasize is kept fixed in this limit passage $N \to \infty$) and hence is independent of the parameter $N$. So using Sobolev embedding and equivalence of laws, we can transfer this estimate to the new probability space:
\begin{equation*}
\|\tilde{\eta}^{*}_{N}- \mathcal{T}_{\Delta t}\tilde{\eta}^{*}_{ N}\|_{ L^\infty(0, T; H^2(\Gamma))} \le C(\Delta t)\|\tilde{v}_{N}\|_{C(0, T; L^{2}(\Gamma))}.
\end{equation*}
We complete our proof by recalling that we have, on the new probability space, the convergence $\tilde{v}_{N} \to \tilde{v}$ $\tilde{\mathbb{P}}$-almost surely in $C(0, T; L^{2}(\Gamma))$. 
\end{proof}

\medskip

Recall from the definition \eqref{phaseN} of the phase space $\mathcal{X}$ that $(\tilde{\rho}_{N}, \tilde{\bu}_{N}, \tilde{\eta}_{N}, \tilde{\eta}^{*}_{N}, \tilde{v}_{N}, \tilde W_N^1, \tilde W_N^2)$ transferred to the new probability space take values in 
\begin{equation*}
C(0, T; L^{\beta}(\mathcal{O}_{\alpha})) \times C(0, T; L^{2}(\mathcal{O}_{\alpha})) \times C(0, T; C(\Gamma))^2 \times C(0, T; L^{2}(\Gamma)) \times C(0, T; \R)^{2}.
\end{equation*}
Moreover, we argue that 
		\begin{align}\label{etasequal1}
			{\eta}^*(t)={\eta}(t) \quad \text{ for any } t<\tau^{\eta}, \quad \tilde\bP\text{-almost surely.}
		\end{align}	
		where for a given $\alpha$,
	       $$	\tau^{\eta}  := T\wedge\inf\left\{t>0:\inf_{\Gamma}(1+\eta(t))\leq \alpha \text{ or } \|{\eta}(t)\|_{H^s(\Gamma)}\geq \frac1{\alpha} \right\}.$$
		Indeed, observe that for almost any $\omega\in\tilde\Omega$ and $t<\tau^{\eta}$, and for any $\epsilon>0$, there exists an {\color{blue}{$N$}} such that
		\begin{align*}
			\|{\eta}(t)-{\eta}^*(t)\|_{H^s(\Gamma)}&<\|{\eta}(t)-\tilde{\eta}_N(t)\|_{H^s(\Gamma)}+\|{\eta}^*(t)-\tilde{\eta}^*_N(t)\|_{H^s(\Gamma)}
			+\|\tilde{\eta}^*_N(t)-\tilde{\eta}_N(t)\|_{H^s(\Gamma)}
			<\epsilon.
		\end{align*}
		This is true because, the almost sure uniform convergence of the structure displacements implies that for any $\epsilon>0$ there exists an $N_1 \in \mathbb{N}$ such that the first two terms on the right side of the above inequality
		are each bounded by $\frac{\epsilon}2$ for all $N\geq N_1$.
		Moreover, due to the uniform convergence of the structure displacement for a fixed outcome,  $t<\tau^{\eta}$ implies that $t<\tau^{\eta}_N$ for infinitely many $N$'s thence the third term is equal to 0. This concludes the proof of \eqref{etasequal1}.

Finally, we show that the new limiting random variables satisfy the desired weak formulation by passing to the limit in the weak formulation of the continuity equation \eqref{continuityN} and the semidiscrete weak formulation \eqref{Nlimitweak}. We handle the convergence of each of the terms in the weak formulation \eqref{Nlimitweak} for each fixed deterministic test function $\bd{q} \in X_n^f$ and $\psi \in X_n^{st}$ by using the convergence results in Theorem \ref{skorohod_time}. Similar techniques work for passing to the limit in the weak formulation of the approximate continuity equation \eqref{continuityN}, so we focus only on the passage to the limit in the momentum weak formulation \eqref{Nlimitweak}, and note that by the equivalence of laws given by Theorem \ref{skorohod_time}, the weak formulation also holds (with the new approximate random variables on the new probability space) $\tilde{\mathbb{P}}$-almost surely on the new probability space. We only comment on the most involved terms in this limit passage below:

\begin{itemize}
    \item $\displaystyle \int_{0}^{T} \int_{\mathcal{O}_{\alpha}} \mu^{\tilde{\eta}^*_N}_\delta  \nabla \tilde{\bu}_{N} : \nabla \bd{q} \to  \int_{0}^{T} \int_{\mathcal{O}_{\alpha}} \mu_\delta^{\eta^{*}} \nabla \bu : \nabla \bd{q}$, where we use the $\tilde{\mathbb{P}}$-almost sure convergence of $\mu^{\tilde{\eta}^*_N}_\delta$ to $\mu_\delta^{\eta^{*}}$ in $C([0, T] \times \sO_\alpha)$ which follows from $\tilde{\eta}^*_N \to \eta^*$ in $C(0, T; C(\Gamma))$ and the properties of the map $\eta \to \mu_\delta^{\eta}$ (see \eqref{chi} and \eqref{viscosityextension}), and the convergence of $\tilde{\bu}_{N} \to \bu$ in $C(0, T; X^f_{n})$. Similarly, we have $\displaystyle  \int_{0}^{T} \int_{\mathcal{O}_{\alpha}} \lambda^{\tilde{\eta}^*_{N}}_\delta \text{div}(\tilde{\bu}_{N}) \text{div}(\bd{q}) \to \int_{0}^{T} \int_{\mathcal{O}_{\alpha}}  \lambda^{\eta^{*}}_\delta\text{div}(\bu) \text{div}(\bd{q})$. 

    \item Next, we have
    \begin{equation*}
        \frac{1}{\delta} \int_{0}^{T} \int_{T^{\delta}_{\tilde{\eta}^{*}_{N}}} (\tilde{\bu}_{N} - \tilde{v}_{N} \bd{e}_{z}) \cdot \bd{q}-  \frac{1}{\delta} \int_{0}^{T} \int_{T^{\delta}_{\mathcal{T}_{\Delta t}\tilde{\eta}^{*}_{ N}}} (\mathcal{T}_{\Delta t}\tilde{\bu}_{ N} - \tilde{v}_{N} \bd{e}_{z}) \cdot \psi \bd{e}_{z} \to \frac{1}{\delta} \int_{0}^{T} \int_{T^{\delta}_{\eta^*}} (\bu - v \bd{e}_{z}) \cdot (\bd{q} - \psi \bd{e}_{z}).
    \end{equation*}
   This follows immediately from the fact that $\tilde{\bu}_{N} \to \bu$ and $\tilde{v}_{N} \bd{e}_{z} \to v \bd{e}_{z}$ in $C(0, T; L^{2}(\mathcal{O}_{\alpha}))$, $\tilde{\mathbb{P}}$-almost surely, and that  $\mathcal{T}_{\Delta t}\tilde{\eta}^{*}_{N} \to \eta^*$ in $C(0, T; C(\Gamma))$ {{(and similarly for $\tilde{\eta}^{*}_{N}$)}}, $\tilde{\mathbb{P}}$-almost surely, which implies for the indicator functions
    \begin{equation*}
        \mathbbm{1}_{T^{\delta}_{\mathcal{T}_{\Delta t}\tilde{\eta}^{*}_{ N}}} \to \mathbbm{1}_{T^{\delta}_{\eta^*}}, \quad \text{ in } C(0, T; L^{2}(\mathcal{O}_{\alpha})),\qquad\tilde\bP\text{-almost surely.}
    \end{equation*}
    \item Next, we show convergence of the structure stochastic integral. Using classical results about convergence of stochastic integrals \cite{Ben} (see also Lemma 2.1 in \cite{DGHT} and Lemma 2.6.6 in \cite{BFH18}), it suffices to show that 
    \begin{equation*}
    \int_{\Gamma} G(\tilde{\eta}_{N}, \tilde{v}_{N}) \psi \to \int_{\Gamma} G(\eta, v) \psi, \quad \text{ in probability in $L^{2}(0, T; L_{2}(\mathcal{U}_{0}; \R))$}.
    \end{equation*}
    By the growth assumption \eqref{gassumption} and Theorem \ref{skorohod_time}, we can verify this convergence, as we have:
    \begin{equation*}
        \tilde{\mathbb{E}} \int_{0}^{T} \sum_{k = 1}^{\infty} (g_{k}(\eta, v) - g_{k}(\tilde{\eta}_{N}, \tilde{v}_{N}), \psi)^{2} \le C\|\psi\|^{2}_{L^{\infty}(\Gamma)} \tilde{\mathbb{E}} \int_{0}^{T} \int_{\Gamma} |\eta - \tilde{\eta}_{N}|^{2} + |v - \tilde{v}_{N}|^{2} \to 0.
    \end{equation*}

\end{itemize}

Finally, we show convergence of the stochastic integrals for the fluid equations, which is the most involved convergence. Before showing convergence of these stochastic integrals, we first show the following convergence result: 

\begin{lemma}\label{Nconvergence}
For almost every $(\tilde{\omega}, t) \in \tilde{\Omega} \times [0, T]$:
\begin{equation}\label{sumsquaredelta2}
\sum_{k = 1}^{\infty} \|f_{n, k}(\rho, \rho \bu) - f_{n, k}(\tilde{\rho}_{N}, \tilde{\rho}_{N} \tilde{\bu}_{N})\|_{L^{2}(\mathcal{O}_{\alpha})}^{2} \to 0, \quad \text{ as $N \to \infty$}.
\end{equation}
\end{lemma}

\begin{proof}
We start with the following observation that we will use throughout the proof. Since $\tilde{\bu}_{N} \to \tilde{\bu}$ in $C(0, T; X^f_{n})$, $\tilde{\mathbb{P}}$-almost surely, 
\begin{equation}\label{asun}
\sup_{N}\|\bu_{N}(\tilde{\omega})\|_{C(0, T; X^f_{n})} \le C(\tilde{\omega})
\end{equation}
for $C(\tilde{\omega})$ depending on $\tilde{\omega} \in \tilde{\Omega}$. By equivalence of norms in $X^f_{n}$ and the comparison principle \eqref{comparison},
\begin{equation}\label{asdensity}
0 < c(\tilde{\omega}) \le \rho_{N}(\tilde{\omega}) \le C(\tilde{\omega}), \quad 0 < c(\tilde{\omega}) \le \rho(\tilde{\omega}) \le C(\tilde{\omega}) \quad \text{ for all $(t, x) \in [0, T] \times \mathcal{O}_{\alpha}$,}
\end{equation}
for positive constants $c(\tilde{\omega})$ and $C(\tilde{\omega})$ (independent of $N$, where we recall that the Galerkin parameter $n$ is fixed), which depend only on the outcome $\tilde{\omega} \in \tilde{\Omega}$.

We recall the definition of $\displaystyle f_{n, k}(\rho, \bd{q}) := \mathcal{M}^{1/2}[\rho] \left(P_{n}^{f}\left(\frac{f_{k}(\rho, \bd{q})}{\rho^{1/2}}\right)\right)$ from \eqref{Fndef} and show the desired convergence by handling each element of the nonlinearity in the definition of  $f_{n, k}(\rho, \rho \bd{u})$ one at a time. We sketch the proof below:

\medskip

\noindent\textbullet{} By Lemma A.1 in \cite{BreitHofmanova}, for $\rho_{1}, \rho_{2} \in L^{2}(\mathcal{O}_{\alpha})$ with $\rho_{1}, \rho_{2} \ge \kappa > 0$ and a constant depending only on the Galerkin parameter $n$ and the positive lower bound on density $\kappa$:
\begin{equation}\label{M12Lipschitz}
\|\mathcal{M}^{1/2}[\rho_{1}] - \mathcal{M}^{1/2}[\rho_{2}]\|_{\mathcal{L}(X^f_n, X^f_n)} \le C(n, \kappa) \|\rho_{1} - \rho_{2}\|_{L^{2}(\mathcal{O}_{\alpha})}.
\end{equation}
We can apply this result due to \eqref{asdensity} to get that $\tilde{\mathbb{P}}$-almost surely for all $t \in [0, T]$:
\begin{equation*}
\sum_{k = 1}^{\infty} \left\|\Big(\mathcal{M}^{1/2}[\rho] - \mathcal{M}^{1/2}[\rho_{N}]\Big) \left(P_n^{f}\left(\frac{f_{k}(\rho, \bd{q})}{\rho^{1/2}}\right) \right)\right\|_{L^{2}(\mathcal{O}_{\alpha})}^{2} \le \sum_{k = 1}^{\infty} C(\tilde{\omega}) \|\rho - \rho_{N}\|_{L^{2}(\mathcal{O}_{\alpha})}^{2} \left\|\frac{f_{k}(\rho, \bd{q})}{\rho^{1/2}}\right\|_{L^{2}(\mathcal{O}_{\alpha})}^{2},
\end{equation*}
for a random constant $C(\tilde{\omega})$. Since $\displaystyle \sum_{k = 1}^{\infty} \left\|\frac{f_{k}(\rho, \bd{q})}{\rho^{1/2}}\right\|_{L^{2}(\mathcal{O}_{\alpha})}^{2}$ is bounded by energy estimates and the assumption \eqref{fassumption} on the noise and $\tilde{\rho}_{N} \to \rho$ in $C(0, T; L^{\beta}(\mathcal{O}_{\alpha}))$ (where we recall that $\beta \ge 4$) $\tilde{\mathbb{P}}$-almost surely, we have that for almost every $(\tilde{\omega}, t) \in \tilde{\Omega} \times [0, T]$:
\begin{equation}\label{sumsquaredeltat4}
\sum_{k = 1}^{\infty} \left\|\Big(\mathcal{M}^{1/2}[\rho] - \mathcal{M}^{1/2}[\rho_{N}]\Big) \left(P_n^{f}\left(\frac{f_{k}(\rho, \bd{q})}{\rho^{1/2}}\right)\right)\right\|_{L^{2}(\mathcal{O}_{\alpha})}^{2} \to 0.
\end{equation}
\noindent\textbullet{}  Since $\mathcal{M}[\rho]$ for $\rho > 0$ is a positive definite, symmetric operator on the finite-dimensional space $X_{n}$ with $\|\mathcal{M}[\rho]\|_{\mathcal{L}(X^f_{n}, X^f_{n})} \le \|\rho\|_{L^{\infty}(\mathcal{O}_{\alpha})}$, we have that $\mathcal{M}^{1/2}[\rho]: X^f_{n} \to X^f_{n}$ is well-defined with the bound $\|\mathcal{M}^{1/2}[\rho]\|_{\mathcal{L}(X^f_n, X^f_n)} \le \|\rho\|_{L^{\infty}(\mathcal{O}_{\alpha})}^{1/2}$. So by the estimate \eqref{asdensity} which is independent of $N$, we have that $\tilde{\mathbb{P}}$-almost surely for all $t \in [0, T]$:
\begin{multline*}
\sum_{k = 1}^{\infty} \left\|\mathcal{M}^{1/2}[\rho_N]\left(P_n^{f}\left(\frac{f_{k}(\rho, \rho \bu)}{\rho^{1/2}}\right) - P_n^{f}\left(\frac{f_{k}(\tilde{\rho}_{N}, \tilde{\rho}_{N} \tilde{\bu}_{N})}{\tilde{\rho}_{N}^{1/2}}\right)\right)\right\|_{L^{2}(\mathcal{O}_{\alpha})}^{2} \\
\le C(\tilde{\omega}) \sum_{k = 1}^{\infty} \left\|\frac{f_{k}(\rho, \rho\bu)}{\rho^{1/2}} - \frac{f_{k}(\tilde{\rho}_{N}, \tilde{\rho}_{N}\tilde{\bu}_{N})}{\tilde{\rho}_{N}^{1/2}}\right\|_{L^{2}(\mathcal{O}_{\alpha})}^{2}
\end{multline*}
for a random constant $C(\tilde{\omega})$. Then, by \eqref{fassumption}, \eqref{asun}, \eqref{asdensity}, and the $\tilde{\mathbb{P}}$-almost sure convergences of $\tilde{\rho}_{N} \to \rho$ in $C(0, T; L^{\beta}(\mathcal{O}_{\alpha}))$ and $\tilde{\bd{u}}_{N} \to \bd{u}$ in $C(0, T; X^f_{n})$, we have that $\tilde{\mathbb{P}}$-almost surely, for all $t \in [0, T]$:
\begin{equation*}
\sum_{k = 1}^{\infty} \left\|\frac{f_{k}(\rho, \rho \bu) - f_{k}(\tilde{\rho}_{N}, \tilde{\rho}_{N} \tilde{\bu}_{N})}{\tilde{\rho}_{N}^{1/2}}\right\|_{L^{2}(\mathcal{O}_{\alpha})}^{2} \le C(\tilde{\omega}) \sum_{k = 1}^{\infty} \|f_{k}(\rho, \rho\bu) - f_{k}(\tilde{\rho}_{N}, \tilde{\rho}_{N} \tilde{\bu}_{N})\|_{L^{2}(\mathcal{O}_{\alpha})}^{2} \to 0.
\end{equation*}
By using \eqref{fassumption} \eqref{asun}, \eqref{asdensity}, and the mean value theorem to estimate $\rho^{-1/2} - \tilde{\rho}_{N}^{-1/2}$, one can verify that for all $t \in [0, T]$:
\begin{equation*}
\sum_{k = 1}^{\infty} \left\|f_{k}(\rho, \rho \bu) \left(\frac{1}{\rho^{1/2}} - \frac{1}{\tilde{\rho}_{N}^{1/2}}\right)\right\|_{L^{2}(\mathcal{O}_{\alpha})}^{2} \to 0, \quad \tilde{\mathbb{P}}\text{-almost surely.}
\end{equation*}
So for almost every $(\tilde{\omega}, t) \in \tilde{\Omega} \times [0, T]$, $\displaystyle \sum_{k = 1}^{\infty} \left\|\frac{f_{k}(\rho, \rho \bu)}{\rho^{1/2}} - \frac{f_{k}(\tilde{\rho}_{N}, \tilde{\rho}_{N}\tilde{\bu}_{N})}{\rho_{N}^{1/2}}\right\|_{L^{2}(\mathcal{O}_{\alpha})}^{2} \to 0$ and also:
\begin{equation}\label{sumsquaredeltat5}
\sum_{k = 1}^{\infty} \left\|\mathcal{M}^{1/2}[\rho_N] \left(P_n^{f}\left(\frac{f_{k}(\rho, \rho \bu)}{\rho^{1/2}}\right) - P_n^{f}\left(\frac{f_{k}(\tilde{\rho}_{N}, \tilde{\rho}_{N}\tilde{\bu}_{N})}{\tilde{\rho}_{N}^{1/2}}\right)\right)\right\|^{2}_{L^{2}(\mathcal{O}_{\alpha})} \to 0.
\end{equation}
Together, \eqref{sumsquaredeltat4} and \eqref{sumsquaredeltat5} prove the desired convergence \eqref{sumsquaredelta2}, which completes the proof. 
\end{proof}

Using the convergence stated in the immediately preceding lemma, we can thus verify the following convergence result for the stochastic integrals as $N \to \infty$.

\begin{proposition}
For a given deterministic pair $(\bd{q}, \psi) \in X_{n}$ for a fixed Galerkin parameter $n$, 
\begin{equation*}
 \int_{0}^{T} \Big(\mathbbm{1}_{\sO_{\tilde\eta^*_N}}\bd{F}_{n}(\tilde{\rho}_{N}, \tilde{\rho}_{N}\tilde{\bu}_{N}) d\tilde W_N^1, \bd{q}\Big) \to \int_{0}^{T} \Big(\mathbbm{1}_{\sO_{\eta^*}}\bd{F}_{n}(\rho, \rho \bu) d\tilde W^1, \bd{q}\Big), \quad \text{ $\tilde{\mathbb{P}}$-almost surely,}
\end{equation*}
where $\bd{F}_n$ is defined in \eqref{Fndef}. 
\end{proposition}

\begin{proof}

To do this, by classical methods of \cite{Ben} (for proofs see Lemma 2.1 in \cite{DGHT} and Lemma 2.6.6 in \cite{BFH18}), it suffices to show that
\begin{equation}\label{probstochdeltat}
(\mathbbm{1}_{\sO_{\tilde\eta^*_N}}\bd{F}_{ n}(\tilde{\rho}_{N}, \tilde{\rho}_{N}\tilde{\bu}_{N}) , \bd{q}) \to (\mathbbm{1}_{\mathcal{O}_{\eta^{*}}} \bd{F}_{n}(\rho, \rho \bu) , \bd{q}) \quad \text{ in probability in } L^{2}(0, T; L_{2}(\mathcal{U}_0; \mathbb{R})),
\end{equation}
 where $L_2(X,Y)$ is the Hilbert-Schmidt norm of operators mapping Hilbert space $X$ to $Y$.
We do this by showing the following:
\begin{align}\label{aestochdeltat}
&(\mathbbm{1}_{\sO_{\tilde\eta^*_N}}\bd{F}_{n}(\tilde{\rho}_{N}, \tilde{\rho}_{N} \tilde{\bu}_{N}) , \bd{q}) \to (\mathbbm{1}_{\mathcal{O}_{\eta^{*}}} \bd{F}_{n}(\rho, \rho \bu) , \bd{q}) \quad \text{in $L_{2}(\mathcal{U}_0; \R)$, for a.e. } (\tilde{\omega}, t) \in \tilde{\Omega} \times [0, T],\\
&\tilde{\mathbb{E}} \int_{0}^{T} \| \mathbbm{1}_{\sO_{\tilde\eta^*_N}}\bd{F}_{n}(\tilde{\rho}_{N}, \tilde{\rho}_{N} \tilde{\bu}_{N}) , \bd{q}) - (\mathbbm{1}_{\mathcal{O}_{\eta^{*}}} \bd{F}_{n}(\rho, \rho \bu) , \bd{q})\|_{L_{2}(\mathcal{U}_0; \R)}^{p} \le C_{p}, \ \text{ for $p \ge 2$},\label{momentstochdeltat}
\end{align}
where $C_{p}$ depends only on $p \ge 2$. These two facts imply by the Vitali convergence theorem that 
\begin{equation*}
\tilde\bE \int_{0}^{T} \|\mathbbm{1}_{\sO_{\tilde\eta^*_N}} \bd{F}_{n}(\tilde{\rho}_{N}, \tilde{\rho}_{N}\tilde{\bu}_{N}) , \bd{q}) - (\mathbbm{1}_{\mathcal{O}_{\eta^{*}}} \bd{F}_{n}(\rho, \rho \bu) , \bd{q})\|_{L_{2}(\mathcal{U}_0; \R)}^{2} \to 0 \quad \text{ as } N \to \infty,
\end{equation*}
from which \eqref{probstochdeltat} immediately follows.

\medskip

\noindent \textbf{Proof of \eqref{aestochdeltat}.} 
Recalling the definition of $\displaystyle f_{n, k}(\rho, \bd{q}) := \mathcal{M}^{1/2}[\rho] \left(P_{n}^{f}\left(\frac{f_{k}(\rho, \bd{q})}{\rho^{1/2}}\right)\right)$ from \eqref{Fndef} we prove the following convergences:
\begin{align}
&\sum_{k = 1}^{\infty} \Big(\mathbbm{1}_{\mathcal{O}_{\eta^{*}}}f_{n, k}(\rho, \rho \bu) - \mathbbm{1}_{\mathcal{O}_{ \tilde{\eta}_{N}^{*}}} f_{n, k}(\rho, \rho \bu), \bd{q}\Big)^{2} \to 0, \quad \text{ for a.e. $(\tilde{\omega}, t) \in \tilde{\Omega} \times [0, T]$}.\label{sumsquaredeltatc}
\\ 
&\sum_{k = 1}^{\infty} \Big(\mathbbm{1}_{\mathcal{O}_{{\tilde{\eta}_{N}^{*}}} }(f_{n, k}(\rho, \rho \bd{u}) - f_{n, k}(\tilde{\rho}_{N}, \tilde{\rho}_{N}\tilde{\bd{u}}_{N})), \bd{q}\Big)^{2} \to 0, \quad \text{ for a.e. $(\tilde{\omega}, t) \in \tilde{\Omega} \times [0, T]$},\label{sumsquaredeltatb}
\end{align}
The result \eqref{sumsquaredeltatc} follows from the fact that $\mathbbm{1}_{\sO_{\tilde\eta_N^*}} \to \mathbbm{1}_{\sO_{\eta^*}}$ in $L^\infty(0,T;L^p(\sO_\alpha))$, $\tilde\bP$-almost surely for any $1 \le p < \infty$ whereas
\eqref{sumsquaredeltatb} follows from \eqref{sumsquaredelta2}.

\medskip

\noindent \textbf{Proof of \eqref{momentstochdeltat}.} Recalling the definition in \eqref{Fndef}, we estimate that for $p \ge 2$:
\begin{multline*}
\tilde{\mathbb{E} }\int_{0}^{T} \|(\bd{F}_{n}(\tilde{\rho}_{N}, \tilde{\rho}_{N} \tilde{\bu}_{N}), \bd{q})) - (\bd{F}_{n}(\rho, \rho \bu) , \bd{q})\|_{L_{2}(\mathcal{U}_0; \R)}^{p} \\
\le C\|\bd{q}\|^p_{L^{\infty}(\mathcal{O}_{\alpha})} \left[\tilde{\mathbb{E} }\int_{0}^{T} \left(\sum_{k = 1}^{\infty} \|f_{n, k}(\tilde{\rho}_{N}, \tilde{\rho}_{N} \tilde{\bu}_{N})\|_{L^{2}(\mathcal{O}_{\alpha})}^{2}\right)^{p/2} + \tilde{\mathbb{E} }\int_{0}^{T} \left(\sum_{k = 1}^{\infty} \|f_{n, k}(\rho, \rho \bu)\|_{L_{2}(\mathcal{O}_{\alpha})}^{2}\right)^{p/2}\right].
\end{multline*}
This is uniformly bounded independently of $N$ as a result of the definition of $f_{n, k}(\rho, q)$ in \eqref{Fndef}, the estimate $\|\mathcal{M}^{1/2}[\rho]\|_{\mathcal{L}(X^f_n, X^f_n)} \le C_{n}\|\rho\|_{L^{1}(\mathcal{O}_{\alpha})}^{1/2}$, the assumption \eqref{fassumption} on the noise, and the uniform moment estimates on the energy (which are independent of $N$) in \eqref{moment_timedis}.

\end{proof}

\section{The Galerkin approximation: Passing to the limit $n \to \infty$}\label{sec:galerkin}
In this section, we will emphasize the dependence of the martingale solution constructed on the probability space $(\tilde\Omega,\tilde\sF,\tilde\bP)$ in the previous section on the parameter $n$ and we hence denote it by $(\rho_n,\bu_n,\eta_n,\eta^*_n,v_n,W_n^1,W_n^2)$. Notice that for that simplicity of notation, we temporarily suppress the dependence of this solution on the other parameters $\ep,\delta$. 
The aim of this section is to obtain bounds, uniformly in $n$, for these approximate solutions with the intent of passing $n$ to $\infty$.
 Recall that the random variables constructed in the previous section satisfy the following equation
\begin{multline}\label{galerkin}
 \int_{\mathcal{O}_{\alpha}} \rho_{n}(t) \bu_{n}(t) \cdot \bd{q}+ \int_{\Gamma} v_n(t) \psi= \int_{\mathcal{O}_{\alpha}} \bp_{0,\delta,\ep} \cdot \bd{q} + \int_{\Gamma} v_0 \psi\\
 +\int_{0}^{t} \int_{\mathcal{O}_{\alpha}} (\rho_{n} \bu_{n} \otimes \bu_{n}) : \nabla \bd{q} + \int_{0}^{t} \int_{\mathcal{O}_{\alpha}} \Big(a\rho_{n}^{\gamma} + \delta \rho_{n}^{\beta}\Big) (\nabla \cdot \bd{q}) - \int_{0}^{t} \int_{\mathcal{O}_{\alpha}} {\mu^{\eta^*_n}_\delta } \nabla \bu_{n} : \nabla \bd{q} \\
{+\ep\int_0^t\int_{\sO_\alpha}\ \bu_n\rho_n\cdot \Delta\bq }	-  \int_{0}^{t} \int_{\mathcal{O}_{\alpha}}{\lambda^{\eta^*_{n}}_\delta }\text{div}(\bu_{n}) \text{div}(\bd{q}) - \frac{1}{\delta} \int_{0}^{t} \int_{{T^\delta_{\eta^*_n}}} (\bu_{n} - v_{n} \bd{e}_{z}) \cdot (\bd{q} - \psi \bd{e}_{z})\\
	 - \int_{0}^{t} \int_{\Gamma} \nabla v_{n} \cdot \nabla \psi  - \int_{0}^{t} \int_{\Gamma} \nabla \eta_{n} \cdot \nabla \psi - \int_{0}^{t} \int_{\Gamma} \Delta \eta_{n} \Delta \psi  \\
+ \int_{0}^{t} \int_{\mathcal{O}_{\alpha}} \mathbbm{1}_{\sO_{\eta_n^*}}\bd{F}_{n}(\rho_{n}, \rho_{n}\bu_{n}) \cdot \bd{q}dW_n^1 +\int_{0}^{t} \int_{\Gamma}G_n(\eta_{n}, v_{n}) \psi dW_n^2,
\end{multline}
 $\tilde{\bP}$-almost surely for any $t\in[0,T]$ and for any test function $\bd{q} \in X_n^f$ and $\psi \in X^{st}_n$. Here,
$$\partial_t\eta_n=v_n.$$
Moreover, we have that the continuity equation is satisfied in a distributional sense, $\tilde{\bP}$-almost surely as follows:
 \begin{equation}\label{cont_ep}
\partial_t\rho_{n} + \text{div}(\rho_n \bu_n) = \varepsilon \Delta \rho_n, \,\, \text{ in } \mathcal{O}_{\alpha},\qquad
 \nabla \rho_n \cdot \bd{n}|_{\partial \mathcal{O}_{\alpha}} = 0, \qquad {{\rho_{n}(0) = \rho_{0, \delta, \ep}}},
\end{equation} 
where $\bd{n}$ is the unit normal to the boundary of the fixed maximal domain $\sO_\alpha$.
Thanks to weak lower semicontinuity of the norm, the energy estimates found in Section \ref{sec:timedis} hold true for the Galerkin approximations which gives us the following result.
\begin{lemma}\label{energy_galerkin}
The sequence of solutions $(\rho_n,\bu_n,\eta_n,\eta^*_n,v_n)$ to \eqref{galerkin} satisfies the following bounds: For any $p\geq 1$, there exists a constant $C>0$, independent of $n$, such that
    \begin{enumerate}
\item $\tilde{\bE}\|\bu_n\|^p_{L^2(0,T;H^1_0(\sO_{\alpha}))} \leq C(\delta)$,
\item $\tilde{\bE}\|\sqrt{\rho_n}\bu_n \|^p_{L^\infty(0,T;L^2(\sO_{\alpha}))}\leq C$,
\item $\tilde{\bE}\|\delta^{\frac1\beta}\rho_n\|^p_{L^\infty(0,T;L^\beta(\sO_{\alpha}))}\leq C$,
\item $\tilde{\bE}\|\sqrt{\ep\delta}\rho_n^{\frac{\beta}2}\|^p_{L^2(0,T;H^1(\sO_{\alpha}))}\leq C,\quad \tilde{\bE}\| \sqrt{\ep}\rho_n^{\frac{\gamma}2}\|^p_{L^2(0,T;H^1(\sO_{\alpha}))}\leq C$,
\item $\tilde{\bE}\|\sqrt{\ep}\rho_n\|^p_{L^2(0,T;H^1(\sO_\alpha))}\leq C$,
\item $\tilde{\bE}\|v_n \|^p_{L^2(0,T;H^1(\Gamma))} \leq C$, where $\partial_t\eta_n=v_n $, 
\item $\tilde{\bE}\|\eta_n\|^p_{{W}^{1,\infty}(0,T;L^2(\Gamma))\cap L^{2}(0,T;H^2(\Gamma))} \leq C$, \quad $\tilde{\bE}\|\eta^*_n\|^p_{ L^{2}(0,T;H^2(\Gamma))} \leq C$.
\item $\tilde{\bE}\|\bu_n-v_n{\bf e}_z\|_{L^2(0,T;L^2(T^\delta_{\eta^*_n}))}^p \le C\delta^{\frac{p}2}$.
\end{enumerate}
\end{lemma}
\begin{proof}
Statement (5) is the only one that requires further explanation as the remaining statements are a direct consequence of the energy estimate \eqref{moment_timedis}. For that purpose, we consider the continuity equation \eqref{cont_ep} and test it with $\rho_n$ to obtain that
\begin{align}
{{2}}\ep	\tilde{\bE}\int_0^T\int_{\sO_\alpha}|\nabla\rho_n|^2 &+\tilde{\bE}\int_{\sO_\alpha}|\rho_n|^2=\|\rho_{0,\delta,\ep}\|^2_{L^2(\sO_\alpha)}-\tilde{\bE}\int_0^T\int_{\sO_\alpha}\nabla\cdot \bu_n|\rho_n|^2\label{rhoH1}\\
&\leq {{\|\rho_{0, \delta, \ep}\|^{2}_{L^{2}(\mathcal{O}_{\alpha})}}} + \tilde{\bE}\|\bu_n\|^2_{L^2(0,T;H^1_0(\sO_\alpha))} + \tilde{\mathbb{E}} \int_0^T\int_{\sO_\alpha}|\rho_n|^4 \leq C,\notag
\end{align}
where we recall in the final inequality that the artificial pressure exponent satisfies $\beta \ge 4$, and where we use the uniform bound in Lemma \ref{energy_galerkin}(3).
\end{proof}

We quickly notice the following implications of the energy estimates Lemma \ref{energy_galerkin}.
First, for any $p>1$ there exists a constant $C_p>0$ independent of $n$:
$$\tilde{\bE}\|\rho_n\bu_n\|^p_{L^\infty(0,T;L^{\frac{2\beta}{\beta+1}}(\sO_\alpha))} \leq C_p.$$
Indeed, we apply the H\"older inequality to $\rho_n\bu_n$ with $\beta+1$ and $\frac{\beta+1}{\beta}$ to obtain
\begin{equation}
    \begin{split}\label{urho}
        \tilde{\bE}\sup_{0\leq t\leq T}\int_{\sO_\alpha}(\sqrt{\rho_n}|	\sqrt{\rho_n}\bu_n|)^{\frac{2\beta}{\beta+1}} &\leq \left(\tilde{\bE}\sup_{0\leq t\leq T} \int_{\sO_\alpha}|\rho_n|^{\beta}\right)^{\frac1{\beta+1}}  \left(\tilde{\bE}\sup_{0\leq t\leq T} \int_{\sO_\alpha}|\sqrt{\rho_n}\bu_n|^{2}\right)^{\frac{\beta}{\beta+1}}  \\
        &\leq C.
    \end{split}
\end{equation}
Consequently, we see that div$(\rho_n\bu_n)$ is bounded in $L^{p}(\tilde\Omega ;L^\infty(0,T;W^{-1,\frac{2\beta}{\beta+1}}(\sO_\alpha)))$ for all $p \in [1, \infty)$. We also know that $\ep\Delta\rho_n $ is bounded independently of $n$ in $ L^{{p}}(\tilde\Omega;L^\infty(0,T;W^{-2,2}(\sO_\alpha)))$. 
Hence we conclude by using the continuity equation that for any $p \geq 1$, there exists $C>0$ independent of $n$ for which
\begin{align}\label{rho_t}
    \tilde{\bE}\|\rho_n\|^p_{W^{1,\infty}(0,T;W^{-2,\frac{2\beta}{\beta+1}}(\sO_\alpha))} \leq C.
\end{align}

Now we will derive uniform estimates for the fractional time derivative of order $<\frac12$ for the structural velocity which will enable us to obtain tightness of its laws in $L^2_tL^2_x$. While the ideas used in the previous section are still applicable, we will derive estimates that are independent not only of $n$ but also of $\ep$ and $\delta$ so that these estimates can be used in the subsequent sections as well.
\begin{lemma}\label{lem:vtight}
For some $0<\bar{\kappa}<\frac12$ and $C>0$ independent of $n,\ep$ and $\delta$ we have that
   \begin{align}\label{nikolski_v}
 \tilde{\bE}\left[\sup_{0<h< T} \frac1{h^{\bar{\kappa}}}\|\mathcal{T}_hv_n-v_n\|_{L^2(h,T;L^2(\Gamma))}\right]\leq C,
\end{align}
where $\mathcal{T}_hf(t) = f(t-h)$.
\end{lemma}
\begin{proof}
To obtain the left-hand side term in \eqref{nikolski_v} we will test the weak formulation \eqref{galerkin} with the time integral from $t-h$ to $t$ of a modification of the solutions $\bu_n$ and $v_n$. This modification is necessary since $\bu_n$ and $v_n$ do not possess the required spatial regularity of a test function. Hence, this modification will be obtained by regularizing the approximate fluid and structure velocities.

To that end, we consider the following $H^1$-extension of the fluid velocity to $\Gamma\times(0,\infty)$ {constructed by interpolating between the fluid velocity $\bu_n$ and the structure velocity $v_n$ in a transition region of width $\lambda_1\delta^{\frac12-\frac1\beta}$ near the top boundary of the tubular region $T^\delta_{\eta^*_n}$ for some $0<\lambda_1<1$}:
\begin{equation*}
{\bu}_{n, \text{ext}}(t) = {\bu}_{n}(t) + \left[\min\left(\frac{z - (1+{\eta}^{*}_{n}(t)+\delta^{\frac12-\frac1\beta})}{\lambda_1\delta^{\frac{1}{2} - \frac{1}{\beta}}}+1, 1  \right)\right]^{+} \Big({v}_{n}(t) \bd{e}_{z} - {\bu}_n(t)\Big).
\end{equation*}

Since $\bu_{n,ext} = \bu_n$ in $\sO_{\eta^*_n+(1-\lambda_1)\delta^{\frac12-\frac1\beta}}$ (see Definition \ref{Oeta}) for any $q<6$ and $p\ge 1$, we observe that
\begin{align}\label{ext-n}
 \tilde\bE\|\bu_{n,ext} -\bu_n\|^p_{L^2(0,T;L^q(\sO_{\eta^*_n}\cup T^\delta_{\eta^*_n}))} \leq \tilde\bE\|\bu_n-v_n\bd{e}_z\|^p_{L^2(0,T;L^6(T^\delta_{\eta^*_n}))}|\lambda_1 \delta^{\frac12-\frac1\beta}|^{\frac{6-q}{6q}}\leq C|\lambda_1|^{\frac{6-q}{6q}}.
\end{align}
 Now for some $1<\sigma<2$ we will squeeze this extended function vertically as follows,
\begin{align*}
    \bu_{n,\sigma}(t,x,y,z)= \bu_{n,ext}(t,x,y,\sigma z).
\end{align*}
Note that for some $C>0$ independent of $n,\ep$ and $\delta$ (and $\sigma,\lambda_1$) we have any $q\in (1,\infty)$ that
\begin{align}\label{sigma_ext_l6}
\|\bu_{n,\sigma}\|_{L^2(0,T;L^q(\sO_\alpha))} \le \frac1\sigma\|\bu_{n,ext}\|_{L^2(0,T;L^q(\sO_\alpha))} \leq (\|\bu_{n}\|_{L^2(0,T;L^q(\sO_{\eta^*_n}\cup T^\delta_{\eta^*_n}))}+\|v_{n}\|_{L^2(0,T;L^q(\Gamma))}) ,
\end{align}
and for any $r<2$ and $p\geq 1$ that,
\begin{align}
 \tilde\bE \|\bu_{n,\sigma}\|^p_{L^2(0,T;W^{1,r}(\sO_\alpha))} &\leq C\tilde\bE \|\bu_{n,ext}\|^p_{L^2(0,T;W^{1,r}(\sO_\alpha))} \notag\\
 &\leq \frac{C}{\lambda_1^p}\tilde\bE (\|\bu_n\|_{L^2(0,T;H^1(\sO_{\eta^*_n}\cup T_{\eta^*_n}^\delta))} + \|v_n\|_{L^2(0,T;H^1(\Gamma))})^p \leq \frac{C}{\lambda_1^p},\label{u_sq}
\end{align}
and for some $C$ depending only on $0<m<1, l>1$ and $\alpha$ (see (47) in \cite{TC23}) we have
\begin{align}\label{sigma-ext}
    \|\bu_{n,\sigma}-\bu_{n,ext}\|^p_{L^2(0,T;L^l(\sO_\alpha))} \leq C (1-\sigma)^{(\frac3l+m)p}\|\bu_{n,ext}\|^p_{L^2(0,T;W^{m,l}(\sO_\alpha))} .
\end{align}
Note also that,
\begin{align*}
    \bu_{n,\sigma}(x,y,z)-v_n(x,y) \bd{e}_z &= \bu_n(x,y,\sigma z)- v_n(x,y) \bd{e}_z\quad \text{if} \quad z \leq \frac{1+\eta^*_n+(1-\lambda_1)\delta^{\frac12-\frac1\beta}}{\sigma} \\
    & = \left(\frac{ (1+{\eta}^{*}_{n}+\delta^{\frac12-\frac1\beta})-z}{\lambda_1\delta^{\frac{1}{2} - \frac{1}{\beta}}}\right)(\bu_n(x,y,\sigma z)- v_n(x,y)\bd{e}_z)\\
    &\hspace{1in}\text{if}  \quad \frac{1+\eta^*_n+(1-\lambda_1)\delta^{\frac12-\frac1\beta}}{\sigma}\leq z \leq \frac{1+\eta^*_n+\delta^{\frac12-\frac1\beta}}{\sigma}\\
    &=0\qquad\text{otherwise}.
\end{align*}
Hence,
\begin{align}\label{usig-v}
\int_{\Gamma}\int_{\frac{1+\eta^*_n}{\sigma}}^{\frac1\alpha}|\bu_{n,\sigma}-v_n\bd{e}_z|^2 \leq  \int_\Gamma\int_{1+\eta^*_n}^{1+\eta^*_n+\delta^{\frac12-\frac1\beta}}|\bu_n-v_n\bd{e}_z|^2 \leq \|\bu_n-v_n\bd{e}_z\|^2_{L^2(T^\delta_{\eta^*_n})}.
\end{align}
Now we will choose $\lambda\sim \sigma-1$ such that 
\begin{align}\label{distsq}
\text{dist}(\Gamma_{1+\eta^*_n},\Gamma_{\frac{1+\eta^*_n}{\sigma}})\geq \lambda.
\end{align}
Next, we let
$$\mathcal{S}^{\lambda,\delta}_\eta=\{(x,y,z):0<z<1+\eta+\delta^{\frac12-\frac1\beta}+\lambda\},$$
and define
\begin{align*}
    \bu_n^\lambda= \left(\mathbbm{1}_{\mathcal{S}^{\lambda,\delta}_{\eta^*_n}}\right)^\lambda(\bu_{n,\sigma})^\lambda,\qquad v_n^\lambda=(v_n)^\lambda,
\end{align*}
where for any $f\in L^p(\sO_\alpha)$, $p\geq 1$ we denote its space regularization, using the standard 3D mollifiers, by $(f)^\lambda \in C^\infty(\sO_\alpha)$.
Finally, we will take 
\begin{equation}\label{testvtight}
\bq_n=\int_{t-h}^t\bu_n^\lambda, \qquad\psi_n=\int_{t-h}^tv_n^\lambda,
\end{equation}
We fix an $h>0$ and let 
$\lambda\sim \sigma-1\sim h^{\theta}$ for some appropriately chosen $\theta<\frac15$, and we pick $\lambda_1 \sim \lambda^{\frac14(1-\frac3{2\gamma})}$; the reasons behind these choices will be apparent later in our calculations.

We test the coupled momentum equation \eqref{galerkin} with $(\bq_n,\psi_n)$ by applying the variant of the Ito formula given in Lemma 5.1 of \cite{BO13}. 
This yields,
\begin{align}
&	-\int_h^{T}\int_{\sO_\alpha}\rho_n\bu_n\cdot\partial_t\bq_n -\int_h^T\int_{\Gamma} v_{n} \partial_t\psi_n	= \int_{h}^{T} \int_{\mathcal{O}_{\alpha}} (\rho_{n} \bu_{n} \otimes \bu_{n}) : \nabla \bd{q}_n 	+ \int_{h}^{T} \int_{\mathcal{O}_{\alpha}} \Big(a\rho_{n}^{\gamma} + \delta \rho_{n}^{\beta}\Big) (\nabla \cdot \bd{q}_n)\notag \\
& - \int_{{h}}^{T} \int_{\mathcal{O}_{\alpha}} \mu^{\eta^*_n}_\delta  \nabla \bu_{n} : \nabla \bd{q}_n	-  \int_{h}^{T} \int_{\mathcal{O}_{\alpha}}\lambda^{\eta^*_{ n}}_\delta \text{div}(\bu_{n}) \text{div}(\bd{q}_n)+\ep\int_{h}^{T}\int_{\sO_\alpha}\bu_n \rho_n\cdot \Delta \bq_n \notag \\
& - \frac{1}{\delta} \int_{h}^{T} \int_{{T^\delta_{\eta^*_n}}} (\bu_{n} - v_{n} \bd{e}_{z}) \cdot (\bd{q}_n - \psi_n \bd{e}_{z}) - \int_{h}^{T} \int_{\Gamma} \nabla v_{n} \cdot \nabla \psi_{n} - \int_{h}^{T} \int_{\Gamma} \nabla \eta_{n} \cdot \nabla \psi_n- \int_{h}^{T} \int_{\Gamma} \Delta \eta_{n} \Delta \psi_n \notag \\
&	 + \int_{0}^{T} \int_{\Gamma}G_n(\eta_{n}, v_{n}) \psi_n dW_n^2(t) +\int_{h}^{T} \int_{\mathcal{O}_{\alpha}} \mathbbm{1}_{\sO_{\eta_n^*}}\bd{F}_{n}(\rho_{n} , \rho_{n} \bu_{n} ) \cdot \bd{q}_n dW_n^1(t)\notag \\
&	=I_1+...+I_{10}.\label{tightv_test}
	\end{align}
We will repeatedly use the following properties of mollification: For any $s,m\in\mathbb{R}$ and $p>1$,
\begin{equation}
\begin{split}\label{mollifyprop}
\|f-(f)^\lambda\|_{W^{s,p}} &\leq C\lambda^{m - s}\|f\|_{W^{m,p}},\\
\|(f)^\lambda\|_{W^{m+k,p}} &\leq C\lambda^{-k}\|f\|_{W^{m,p}},\quad \forall k \geq 0,
\end{split}
\end{equation}
where $(f)^\lambda$ denotes the space regularization of $f \in L^p(\sO_\alpha)$.
We will first analyze the two terms on the LHS. We begin with the more critical term: 
\begin{align*}
\int_h^{T}\int_{\sO_\alpha}&\rho_n\bu_n\cdot\partial_t\bq_n =	\int_0^{T}\int_{\sO_\alpha}\rho_n\bu_n\cdot\left( \bu_n^\lambda(t)-\bu_n^\lambda(t-h)\right)\\
&= \int_h^{T}\int_{\sO_\alpha}\rho_n^\lambda\bu_n^\lambda\cdot\left( \bu_n^\lambda(t)-\bu_n^\lambda(t-h)\right)+\rho_{n}(\bu_n-\bu_n^\lambda)\left( \bu_n^\lambda(t)-\bu_n^\lambda(t-h)\right)\\
&+\int_h^{T}\int_{\sO_\alpha}(\rho_n-\rho_n^\lambda)\bu_n^\lambda \left( \bu_n^\lambda(t)-\bu_n^\lambda(t-h)\right)\\
&= \frac12\int_h^{T}\int_{\sO_\alpha} \rho_{n}^\lambda(t)|\bu_n^\lambda(t)-\bu_n^\lambda(t-h)|^2 + \frac12\int_h^T \int_{\sO_\alpha} \rho_n^\lambda(t) \Big(|\bu_n^\lambda(t)|^2-|\bu_n^\lambda(t-h)|^2\Big) \\
&+\int_h^{T}\int_{\sO_\alpha}\rho_n(\bu_n-\bu_n^\lambda)\left( \bu_n^\lambda(t)-\bu_n^\lambda(t-h)\right)+\int_0^{T}\int_{\sO_\alpha}(\rho_n-\rho_n^\lambda)\bu_n^\lambda \left( \bu_n^\lambda(t)-\bu_n^\lambda(t-h)\right)\\
&=I_0^1+...+I_0^4.
\end{align*}

Observe that, since $\rho_n \geq 0$, the first term is nonnegative:
\begin{align*}
	0\leq I^1_0=\frac12 \int_0^{T}\int_{\sO_\alpha}\rho_n^\lambda(t)	|\bu_n^\lambda(t)-\bu_n^\lambda(t-h)|^2.
\end{align*}
Now, we will consider the second term $I_0^2$. 
\begin{align*}
&\tilde{\bE}	\int_h^{T}\int_{\sO_\alpha}\rho_n^\lambda(t)\left( |\bu_n^\lambda(t)|^2 - |\bu_n^\lambda(t-h)|^2\right)\\
&= 	\tilde{\bE}\int_h^{T}\int_{\sO_\alpha}\rho_n^\lambda(t) |\bu_n^\lambda(t)|^2 - \rho_n^\lambda(t-h) |\bu_n^\lambda(t-h)|^2
	-\tilde{\bE}\int_0^{T}\int_{\sO_\alpha}(\rho_n^\lambda(t)-\rho_n^\lambda(t-h))|\bu_n^\lambda(t-h)|^2.
 \end{align*}
Notice for the first term on the right hand side that
\begin{align*}
\left|\tilde{\bE}\int_h^{T}\int_{\sO_\alpha}\rho_n^\lambda(t) |\bu_n^\lambda(t)|^2 - \rho_n^\lambda(t-h) |\bu_n^\lambda(t-h)|^2\right|&=\left|-\tilde{\bE}\int_{0}^h\int_{\sO_\alpha}\rho_n^\lambda|\bu_n^\lambda(t)|^2+\int_{T-h}^T\int_{\sO_\alpha}\rho_n^\lambda|\bu_n^\lambda(t)|^2 \right|\\
&\leq  Ch\tilde\bE\|\sqrt{\rho^\lambda_n}\bu_n^\lambda\|_{L^\infty(0,T;L^2(\sO_\alpha))}^2\leq Ch.
\end{align*}
We will next treat the second term on the right hand side. Recall the bounds \eqref{rho_t} for $\partial_t\rho_n$. Note that $|\left(\mathbbm{1}_{T_{\eta^*_n+\lambda}}\right)^\lambda| \leq 1$. Moreover since $\beta>3$ implies that $\frac{4\beta}{\beta-1} <6$ and we have \eqref{u_sq}, we obtain,
 \begin{align*}
&|\tilde{\bE}\int_0^{T}\int_{\sO_\alpha}(\rho_n^\lambda(t)-\rho_n^\lambda(t-h))|\bu_n^\lambda(t-h)|^2|	\leq Ch\tilde{\bE}\left( {\|\partial_t \rho_n^\lambda\|_{L^\infty(0,T;L^{\frac{2\beta}{\beta+1}}(\sO_\alpha))}} \|(\bu_{n,\sigma})^\lambda\|^2_{L^2(0,T;L^{\frac{4\beta}{\beta-1}}(\sO_\alpha))}\right) \\
		&\leq Ch\tilde{\bE}\left( {\lambda^{-2}\|\partial_t \rho_n\|_{L^\infty(0,T;W^{-2,{\frac{2\beta}{\beta+1}}}(\sO_\alpha))}} {{\|\bu_{n,\sigma}\|^2_{L^2(0,T;H^1(\sO_\alpha))}}}\right) \\
  &\leq Ch\tilde{\bE}\left( {\lambda^{-2}\|\partial_t \rho_n\|_{L^\infty(0,T;W^{-2,{\frac{2\beta}{\beta+1}}}(\sO_\alpha))}} \lambda_1^{-2}(\|\bu_n\|^2_{L^2(0,T;H^1(\sO_{\eta^*_n}\cup T_{\eta^*_n}^\delta))} + \|v_n\|^2_{L^2(0,T;H^1(\Gamma))})\right) \\
	&\leq Ch^{\kappa_1}.
\end{align*}

We note here that, in the subsequent sections (e.g. Section \ref{sec:delta}), when we do not have the $\beta$-integrability of the pressure, this estimate for $I_0^2$ will change slightly as we will find bounds in terms of expectation of the $W^{1,\infty}(0,T; W^{-1,{\frac{2\gamma}{\gamma+1}}}(\sO_\alpha))$ norm of the density. Thus the final estimate will possibly contain a smaller power of $\lambda$, depending on $\gamma$, arising from the inverse estimate for the fluid velocity.  

Noting that $\bu_n^\lambda$ has support in $\mathcal{S}_{\eta^*_n}^{2\lambda,\delta}$, we write the third term $I^3_0$ as,
\begin{align}
&\tilde{\bE}	\left|\int_h^{T}\int_{\sO_\alpha}\rho_n(\bu_n-\bu_n^\lambda)\left( \bu_n^\lambda(t)-\bu_n^\lambda(t-h)\right) \right| 
\notag\\
&\leq \tilde{\bE}	\int_h^{T}\int_{\sO_\alpha}\rho_n\left(|\bu_{n,\sigma} -(\bu_{n,\sigma})^\lambda|+|\bu_{n,ext} -\bu_{n,\sigma}|\right)\left(\mathbbm{1}_{\mathcal{S}^{\lambda,\delta}_{\eta^*_n}}\right)^\lambda\left|\left( \bu_n^\lambda(t)-\bu_n^\lambda(t-h)\right) \right|\notag\\
&\qquad +\tilde{\bE}	\int_h^{T}\int_{\mathcal{S}_{\eta^*_n}^{2\lambda,\delta}}\rho_n|\bu_{n,ext}|\left|\left(\mathbbm{1}_{\mathcal{S}^{\lambda,\delta}_{\eta^*_n}}\right)^\lambda - \mathbbm{1}_{\mathcal{O}_{\eta^*_n}\cup T^\delta_{\eta^*_n}}\right|\left|\left( \bu_n^\lambda(t)-\bu_n^\lambda(t-h)\right) \right|\notag\\
&\qquad +\tilde{\bE}	\int_h^{T}\int_{\mathcal{O}_{\eta^*_n}\cup T^\delta_{\eta^*_n}}\rho_n|\bu_{n,ext}-\bu_{n}|\left|\left( \bu_n^\lambda(t)-\bu_n^\lambda(t-h)\right) \right|\notag.
\end{align}
Now we choose $q<6$ and $t>0$ such that $1=\frac1q+\frac1{\gamma}+\frac1t+ \frac16$ which is possible since, by assumption $\gamma>\frac32$. 
Using \eqref{ext-n}, \eqref{sigma-ext},    and \eqref{mollifyprop} we can then bound the right hand side as follows:
\begin{align*}
 &\leq \tilde{\bE}\Big[\|\rho_n\|_{L^\infty(0,T;L^\gamma(\sO_\alpha))}\left(\|(\bu_{n,\sigma})^\lambda-\bu_{n,\sigma}\|_{L^2(0,T;L^{\frac{6\gamma}{5\gamma-6}}(\sO_\alpha))}+\|\bu_{n,ext}-\bu_{n,\sigma}\|_{L^2(0,T;L^{\frac{6\gamma}{5\gamma-6}}(\sO_\alpha))}\right)\\
 &\qquad\qquad\qquad\times\|(\bu_{n,\sigma})^\lambda\|_{L^2(0,T;L^{6}(\sO_\alpha))}\Big]\notag\\
 &\qquad+\tilde\bE\left( \|\rho_n\|_{L^\infty(0,T;L^\gamma(\sO_\alpha))}\|\bu_{n,ext}\|_{L^2(0,T;L^q(\sO_\alpha))}|\mathcal{S}_{\eta^*_n}^{2\lambda,\delta}\setminus T^\delta_{\eta^*_n}|^{\frac1t}
 \|(\bu_{n,\sigma})^\lambda\|_{L^2(0,T;L^{6}(\sO_\alpha))}
 \right) \\
  &\qquad+\tilde\bE\left( \|\rho_n\|_{L^\infty(0,T;L^\gamma(\sO_\alpha))}\|\bu_{n,ext}-\bu_n\|_{L^2(0,T;L^{\frac{6\gamma}{5\gamma-6}}(\mathcal{O}_{\eta^*_n}\cup T^\delta_{\eta^*_n}))}
 \|(\bu_{n,\sigma})^\lambda\|_{L^2(0,T;L^{6}(\sO_\alpha))}
 \right) \\
&\leq C\tilde{\bE}\Big[\|\rho_n\|_{L^\infty(0,T;L^\gamma(\sO_\alpha))}\left(\lambda^{1-\frac3{2\gamma}}\|\bu_{n,\sigma}\|_{L^2(0,T;W^{1,r_1}(\sO_\alpha))}+(1-\sigma)^{\frac{7\gamma-9}{2\gamma}}\|\bu_{n,ext}\|_{L^2(0,T;W^{1,r_1}(\sO_\alpha))}\right)\\
&\qquad\qquad\qquad\times\|\bu_{n,ext}\|_{L^2(0,T;L^{6}(\sO_\alpha))}\Big]\notag\\
 &\qquad+C\tilde\bE\left( \|\rho_n\|_{L^\infty(0,T;L^\gamma(\sO_\alpha))}\|\bu_{n,ext}\|_{L^2(0,T;W^{1,r_2}(\sO_\alpha))}|\lambda|^{\frac1{t}}
 \|\bu_{n,\sigma}\|_{L^2(0,T;L^{6}(\sO_\alpha))}
 \right) \\
  &\qquad+C\tilde\bE\left( \|\rho_n\|_{L^\infty(0,T;L^\gamma(\sO_\alpha))}|\lambda_1|^{\frac23-\frac1\gamma}\|\bu_{n}-v_n\bd{e}_z\|_{L^2(0,T;H^{1}( T^\delta_{\eta^*_n}))}
\|\bu_{n,\sigma}\|_{L^2(0,T;L^{6}(\sO_\alpha))}
 \right),
\end{align*}
where we used the Sobolev embedding $W^{1,r_1}(\mathcal{O}_{\alpha}) \subset W^{1-\frac3{2\gamma},\,\frac{6\gamma}{5\gamma-6}}(\mathcal{O}_{\alpha})$ where $r_1=\frac{6\gamma}{5\gamma-3}<2$ and $W^{1,r_2}(\sO_\alpha)\hookrightarrow L^q(\sO_\alpha)$ for $q<6$, $r_2<2$ satisfying $q=\frac{3r_2}{3-r_2}$. 
To solidify our estimates, we will pick $q=\frac{12\gamma}{4\gamma-3}$ and $\frac1t=\frac12(1-\frac{3}{2\gamma})$.
Hence using \eqref{sigma_ext_l6} and \eqref{u_sq}, we obtain for some $\kappa_2>0$ that,
\begin{align}
 \notag   \tilde{\bE}	|I^3_0|&=\tilde\bE\left|\int_h^{T}\int_{\sO_\alpha}\rho_n(\bu_n-\bu_n^\lambda)\left( \bu_n^\lambda(t)-\bu_n^\lambda(t-h)\right) \right|\\
    & \leq C(\lambda^{1-\frac3{2\gamma}} \lambda_1^{-1} + (1-\sigma)^{\frac{7\gamma-9}{2\gamma}}\lambda_1^{-1}
    + \lambda^{\frac1t}\lambda_1^{-1} + |\lambda_1|^{\frac23-\frac1\gamma})
 \leq Ch^{\kappa_2}.\label{lambdas}
\end{align}

Next we will consider the term $I_0^4$. We recall \eqref{mollifyprop}, \eqref{u_sq}, and \eqref{sigma_ext_l6} and the Sobolev embedding $ W^{1,r}\hookrightarrow H^{1-\epsilon_1}$ for $\epsilon_1>0$ and $r=\frac{6}{3+\epsilon_1}<2$. Then there exists $\kappa_3>0$ for an appropriately small $0<\epsilon_1\ll \frac14(2-\frac3\gamma)$
such that, 
\begin{align*}
&\tilde\bE|I_0^4|\leq \tilde{\bE}	\int_h^{T}\int_{\sO_\alpha}(\rho_n-\rho_n^\lambda)|(\bu_{n,\sigma})^\lambda| \left( (\bu_{n,\sigma})^\lambda(t)|+|(\bu_{n,\sigma})^\lambda(t-h)|\right)\\
&\leq\tilde{\bE}\Big[\|\rho_n-\rho_n^\lambda\|_{L^\infty(0,T;W^{-1,\gamma}(\sO_\alpha))}\Big(\||(\bu_{n,\sigma})^\lambda|^2\|_{L^1(0,T;W^{1,\frac{\gamma}{\gamma-1}}(\sO_\alpha))} + \|(\bu_{n,\sigma})^\lambda(t) (\bu_{n,\sigma})^\lambda(t - h)\|_{L^{1}(h, T; W^{1, \frac{\gamma}{\gamma - 1}}(\mathcal{O}_{\alpha})}\Big)\Big]\\
	&\leq\tilde{\bE}(\|\rho_n-\rho_n^\lambda\|_{L^\infty(0,T;W^{-1,\gamma}(\sO_\alpha))}\|(\bu_{n,\sigma})^\lambda\|_{L^2(0,T;W^{1,\frac{6\gamma}{5\gamma-6}}(\sO_\alpha))}\|(\bu_{n,\sigma})^\lambda\|_{L^2(0,T;L^{6}(\sO_\alpha))})\\		
	&\leq \tilde{\bE}(\lambda\|\rho_n\|_{L^\infty(0,T;L^{\gamma}(\sO_\alpha))}\|(\bu_{n,\sigma})^\lambda\|_{L^2(0,T;H^{\frac{3}{\gamma}}(\sO_\alpha))}\|\bu_{n,\sigma}\|_{L^2(0,T;L^{6}(\sO_\alpha))})\\
		&\leq \lambda\tilde{\bE}(\|\rho_n\|_{L^\infty(0,T;L^{\gamma}(\sO_\alpha))}\lambda^{1-{\frac{3}{\gamma}}-\epsilon_1}\|\bu_{n,\sigma}\|_{L^2(0,T;H^{1-\epsilon_1}(\sO_\alpha))}\|\bu_{n,\sigma}\|_{L^2(0,T;L^{6}(\sO_\alpha))})\\
  	&\leq \lambda\tilde{\bE}(\|\rho_n\|_{L^\infty(0,T;L^{\gamma}(\sO_\alpha))}\lambda^{1-{\frac{3}{\gamma}}-\epsilon_1}\|\bu_{n,\sigma}\|_{L^2(0,T;W^{1,\frac{6}{3+\epsilon_1}}(\sO_\alpha))}\|\bu_{n,\sigma}\|_{L^2(0,T;L^{6}(\sO_\alpha))})\\
			&\leq C\lambda^{2-{\frac{3}{\gamma}}-\epsilon_1} \lambda_1^{-1}\leq Ch^{\kappa_3}.
\end{align*}
Similarly, the second term on the left-hand side of \eqref{tightv_test} can be written as
\begin{align*}
	&\int_h^T\int_\Gamma v_n(t)\partial_t\psi_n = \int_h^T\int_{\Gamma} v_{n} (v_n^\lambda(t)-v_n^\lambda(t-h))\\
	&=\int_h^T\int_{\Gamma} v_{n}(t) (v_n(t)-v_n(t-h))+\int_0^T\int_{\Gamma} v_{n}(t) (v_n^\lambda(t)-v_n(t)-v_n^\lambda(t-h)+v_n(t-h))\\
	&=\frac12\int_h^T\int_{\Gamma} (|v_{n}(t)|^2-|v_n(t-h)|^2) + \frac12\int_h^T\int_{\Gamma} |v_n(t)-v_n(t-h)|^2 \\
	&+\int_h^T\int_{\Gamma} v_{n}(t) (v_n^\lambda(t)-v_n(t)-v_n^\lambda(t-h)+v_n(t-h))\\
	&=: J^0_1+J^0_2+J^0_3.
\end{align*}
The second term $J^0_2$ is the term of our interest, i.e. the left-hand side of the desired inequality \eqref{nikolski_v}.
For the third term we immediately obtain,
\begin{align}\label{v1}
\tilde{\bE}|J^0_3| \leq \tilde{\bE}(\|v_n\|_{L^2(0,T;L^2(\Gamma))}	\|v_{n}^{\lambda}-v_n\|_{L^2(0,T;L^2(\Gamma))}) \leq C \lambda \tilde{\bE} \| v_{n}\|^{2}_{L^2(0,T;H^1(\Gamma))} \leq C\lambda.
\end{align}
Moreover, the first term can be treated as follows,
\begin{align*}
\tilde{\bE}	|J_1^0| =\left|\frac12\tilde{\bE}	\int_h^T\int_\Gamma |v_n(t)|^2-|v_{n}(t-h)|^2\right| &= \left|-\frac12\tilde{\bE} \int_{0}^{h} \|v_n\|^{2}_{L^{2}(\Gamma)} + \frac12\tilde{\bE} \int_{T-h}^{T}\|v_n\|^2_{L^2(\Gamma)}\right| \nonumber \\
&\leq Ch\|v_{n}\|_{L^{\infty}(0, T; L^{2}(\Gamma))}^{2}.
\end{align*} 
Now we will consider the terms on the right-hand side of \eqref{tightv_test} starting with $I_1$: Since $\lambda=h^{\theta}$, for some $\theta<\frac15$, for any $s<\frac12$ and $m<1$ we apply Theorem \ref{thm:extension} and choose an $0<\epsilon_1\ll\frac{s}4$ which is appropriately small  to guarantee the existence of $\kappa_4>0$ such that
\begin{align*}
&\tilde{\bE}	\left|\int_{h}^{T} \int_{\mathcal{O}_{\alpha}} (\rho_{n} \bu_{n} \otimes \bu_{n}) : \nabla \int_{t-h}^t\bu_n^\lambda\right| \\
&= \tilde{\bE}	\left|\int_{h}^{T} \int_{\mathcal{O}_{\alpha}} (\rho_{n} \bu_{n} \otimes \bu_{n}) :  \int_{t-h}^t\left(\nabla\left(\mathbbm{1}_{\mathcal{S}^{\lambda,\delta}_{\eta^*_n}}\right)^\lambda(\bu_{n,\sigma})^\lambda+\left(\mathbbm{1}_{\mathcal{S}^{\lambda,\delta}_{\eta^*_n}}\right)^\lambda\nabla(\bu_{n,\sigma})^\lambda\right)\right|\\
 &\leq {h} \tilde{\bE}\left[\|\sqrt{\rho_n}\bu_n\|_{L^\infty(0,T;L^2(\sO_\alpha))}^2\left(\|(\bu_{n,\sigma})^\lambda\|_{L^2(0,T;H^{\frac52+\epsilon_1}(\sO_\alpha))}\|(\mathbbm{1}_{\mathcal{S}^{\lambda,\delta}_{\eta^*_n}})^\lambda\|_{L^2(0,T;H^{\frac52+\epsilon_1}(\sO_\alpha))}\right)\right]\\
&\leq {h}\tilde{\bE}(\|\sqrt{\rho_n}\bu_n\|_{L^\infty(0,T;L^2(\sO_\alpha))}^2 \lambda^{-\frac52-\epsilon_1} \|\bu_{n,\sigma}\|_{L^2(0,T;L^2(\sO_\alpha))}^2\lambda^{s-\frac52-\epsilon_1} \|\mathbbm{1}_{\mathcal{S}^{\lambda,\delta}_{\eta^*_n}}\|_{L^2(0,T;H^{s}(\sO_\alpha))})\\
&\leq Ch^{\kappa_4}.
\end{align*}

The terms $I_{2,3,4,5}$ are treated identically. 
Next, recall that $\left(\mathbbm{1}_{\mathcal{S}^{\lambda,\delta}_{\eta^*_n}}\right)^\lambda=1$ in $T^\delta_{\eta^*_n}$. Moreover, since we have enforced \eqref{distsq}, the bounds \eqref{usig-v}  imply for the penalty term that
\begin{align*}
\tilde{\bE}|I_6|&=\tilde{\bE}\left|\frac1\delta \int_{0}^{t} \int_{{T^\delta_{\eta^*_n}}} (\bu_{n} - v_{n} \bd{e}_{z}) \cdot \int_{t-h}^t((\bd{u}_{n,\sigma})^\lambda - v^\lambda_n \bd{e}_{z})\right|\\
& \leq \frac{h}{\delta}\tilde{\bE}(\|\bu_n-v_n{\bf e}_z\|_{L^2(0,T;L^2(T^\delta_{\eta^*_n}))}\|(\bu_{n,\sigma})^\lambda-v_n^\lambda{\bf e}_z\|_{L^2(0,T;L^2(T^\delta_{\eta^*_n}))})\\
& \leq \frac{h}{\delta}\tilde{\bE}\|\bu_n-v_n{\bf e}_z\|_{L^2(0,T;L^2(T^\delta_{\eta^*_n}))}^2 \leq Ch,
\end{align*}
where we emphasize that, thanks to Lemma \ref{energy_galerkin} (8), the constant $C>0$ is independent of $\delta$. This is the reason behind choosing a non-trivial fluid test function.

Next, we see for $I_8$ that
\begin{align*}
\tilde{\bE}|I_8|=	\tilde{\bE}\left|\int_{h}^{T} \int_{\Gamma} \Delta \eta_{n}  \left( \int_{t-h}^t\Delta v^\lambda_n \right)\right| &\leq Ch\tilde{\bE}(\|\eta_n\|_{L^\infty(0,T;H^2(\Gamma))}\|v^\lambda_n\|_{L^\infty(0,T;H^2(\Gamma))})\\
&\leq Ch\lambda^{-2}\tilde{\bE}(\|\eta_n\|_{L^\infty(0,T;H^2(\Gamma))}\|v_n\|_{L^\infty(0,T;L^2(\Gamma))})\\
&\leq Ch^{\kappa_5}.
\end{align*}
Finally, for the stochastic integrals we find that by \eqref{gassumption}, 
\begin{align*}
	&\tilde{\bE}\left| \int_h^T	\int_\Gamma G_n(\eta_n,v_n) \psi_n dW_n^2\right| \le Ch\tilde{\bE} \left|\sum_{k = 1}^{\infty} \int_{0}^{T}  \left(\int_{\Gamma} |P_n^{st
    }g_{k}(\eta_n, v_n)|^{2} \right) \left(\sup_{t-h<s<t}\int_{\Gamma} |v_{n}(s)|^{2} \right)dt \right|^{1/2} \\
	& 	\le  Ch\left(1 +  \tilde{\bE} \|v_n\|_{L^\infty(0,T;L^2(\Gamma))}^2 +  \tilde{\bE} \|\eta_n\|^{2}_{L^\infty(0,T;L^2(\Gamma))} \right) \\
	& \leq Ch.
\end{align*}
The other stochastic integral is treated identically. Thus combining all the estimates we obtain \eqref{nikolski_v} for $\bar{\kappa}=\min_i\{\kappa_i\}$.
\end{proof}

These uniform bounds give us the existence of a subsequence of the Galerkin solutions that converges in weak/weak-* topologies of the corresponding energy spaces. 
Our next goal is to prove that a subsequence of the solutions to the Galerkin approximations \eqref{galerkin} is tight in an appropriate phase space which is a subspace of the energy space. This will aid us in establishing the necessary almost sure convergence of the approximate solutions in the phase space.
\subsection{Tightness of laws}\label{sec:galerkin_tight}
We begin by defining the appropriate spaces:
\begin{align*}
&	\mathcal{X}_{\rho} = C_{w}(0, T; L^{\beta}(\mathcal{O}_{\alpha}))
 \cap
 L^{4}(0, T; L^{4}(\mathcal{O}_{\alpha})) \cap \left(L^{2}(0, T; H^{1}(\mathcal{O}_{\alpha})) \cap {{L^{\beta+1}((0,T)\times \sO_\alpha), w}}\right). \\
&	\mathcal{X}_{\bu} = (L^{2}(0, T; H_0^{1}(\mathcal{O}_{\alpha})), w), \qquad \mathcal{X}_{\rho \bu} = C_{w}(0, T; L^{\frac{2\beta}{\beta+1}}(\mathcal{O}_{\alpha}))\cap C([0,T];H^{-l}(\sO_\alpha)) \quad \text{for }l>\frac52 ,\\
&	\mathcal{X}_{\eta} = C_{w}(0, T; H^2(\Gamma)) \cap C([0, T]; H^{s}(\Gamma)), \text{ for any }s<2, \\
&\mathcal{X}_{v}=L^2(0,T;L^2(\Gamma))\cap (L^2(0,T;H^1(\Gamma)),w),\quad   \mathcal{X}_{W} = C(0, T; \mathcal{U}_{0})^2.
\end{align*}
 where $w$ denotes weak topology.
The aim of this section is to prove the following proposition.
\begin{proposition}\label{tight_galerkin}
Define the family of random variables
\begin{equation*}
	\mathcal{U}_{n} := (\rho_{n}, \bu_{n}, P^f_{n}(\rho_n \bu_n), \eta_{n}, \eta_n^*, v_n, 
 W^1_n, W^2_n).
\end{equation*}
Then the sequence of measures $\{\tilde{\bP}\circ \mathcal{U}^{-1}_n\}_{n\geq 1}$ is tight in the phase space
\begin{equation*}
	\mathcal{X} = \mathcal{X}_{\rho} \times \mathcal{X}_{\bu} \times \mathcal{X}_{\rho \bu} \times \mathcal{X}_{\eta}  \times \mathcal{X}_{\eta} \times \mathcal{X}_{v} 
 \times \mathcal{X}_{W}.
\end{equation*}
\end{proposition}
\begin{proof}
We prove this proposition by proving tightness of the laws of $\rho_n, P_n(\rho_n\bu_n), \bu_n, v_n, \eta_n, \eta^*_n$ individually in the corresponding phase spaces. Tychonoff's theorem will then imply tightness of the joint laws of these random variables.
 
\medskip
\noindent \textbullet {\bf Tightness of $\bu_n$.}
We begin by defining the following set which is relatively compact in $\mathcal{X}_\bu$:
\begin{align*}
\mathcal{K}_M:=&\{\bu\in L^2(0,T;H_0^1(\sO_\alpha)): \|\bu\|_{ L^2(0,T;H_0^1(\sO_\alpha))}\leq M\}.\end{align*}

Observe that the uniform bounds in Lemma \ref{energy_galerkin} (1) and the Chebyshev inequality give us, for any $M>0$, that
	\begin{align*}
	\tilde{\bP}[\mathcal{K}_M^c] \leq \frac1{M^2}\bE[\| \bu_n\|^2_{L^2(0,T;H^1_0(\sO_\alpha))}] \leq \frac{C}{M^2}.
	\end{align*}
This proves the desired result of tightness of the laws of the approximate fluid velocity $\bu_n$. The tightness of laws of $v_n$ in $(L^2(0,T;H^1(\Gamma)),w)$ follows identically.

\medskip

\noindent \textbullet {\bf Tightness of laws of $\eta_n$ and $\eta_n^*$.} To prove this result we recall the uniform bounds for $\tilde\eta_{n}$ and $\tilde\eta_{n}^*$ derived in Lemma \ref{energy_galerkin} (6,7). Now, the Aubin-Lions theorem states that the set
	\begin{align*}\mathcal{B}_M:=&\{\eta\in L^\infty(0,T;H^2(\Gamma))\cap W^{1,\infty}(0,T;L^2(\Gamma)):\|\eta\|^2_{L^\infty(0,T;H^2(\Gamma))} + \|\eta\|^2_{ W^{1,\infty}(0,T;L^2(\Gamma))}\leq M\},\end{align*}
 is relatively compact in $C([0,T];H^{s}(\Gamma))$,  for any $0<s<2$. Using the Chebyshev inequality, we obtain for some $C>0$ independent of $n$ that the following holds for $\eta_n$ (and similarly for $\eta^*_n$):
	\begin{equation}
	\begin{aligned}
	\tilde{\bP}\left[\eta_{n}\not\in {\mathcal{B}}_M\right]&\le \tilde{\bP}\left[
	{\|\eta_{n}\|}^2_{L^\infty(0,T;H^2(\Gamma))}\ge \frac{M}{2}\right]+
	\tilde{\bP}\left[
	{\|\eta_{n}\|}^2_{W^{1,\infty}(0,T;L^2(\Gamma))}\ge \frac{M}{2}\right]	\\
	&\le  \frac{4}{M}\tilde{\bE}\left[
	{\|\eta_{n}\|}^2_{L^\infty(0,T;H^2(\Gamma))}+\|\eta_{n}\|^2_
	{W^{1,\infty}(0,T;L^2(\Gamma))}\right]
	\leq \frac{C}{M}.
	\end{aligned}
	\end{equation}

 \medskip
\noindent \textbullet {\bf Tightness of laws of $\rho_n$.}
Since, at this stage, we are working on the fixed maximal domain $\sO_\alpha$, this part of the proof follows closely the strategy in Proposition 4.2 in \cite{BreitHofmanova} applied toward obtaining a similar tightness result and in \cite{FeireislCompressible} for establishing compactness in the deterministic setting. Hence, in this section, we will summarize the technique briefly.
Recall that $\rho_n$ satisfies the bounds \eqref{rho_t}. Thus the compact embedding
\begin{equation*}
	L^{\infty}(0, T; L^{\beta}(\mathcal{O}_{\alpha})) \cap C^{0, 1}(0, T; W^{-2, \frac{2\beta}{\beta + 1}}(\mathcal{O}_{\alpha})) \subset \subset C_{w}(0, T; L^{\beta}(\mathcal{O}_{\alpha})),
\end{equation*}
immediately gives us that the measures  $\tilde{\bP}\circ\rho^{-1}_n$ are tight in $C_w(0,T;L^\beta(\sO_\alpha))$.

Observe that, since $\beta>3$ we have that $4<\frac{12\beta}{\beta+6}$, and thus that $W^{\frac12,\frac{4\beta}{\beta+2}}(\sO_\alpha)$ is compactly embedded in $L^4(\sO_\alpha)$. Hence, the Aubin-Lions theorem gives us that,
\begin{align*}
	L^4(0,T;W^{\frac12,\frac{4\beta}{\beta+2}}(\sO_\alpha) )\cap W^{1,\infty}(0,T;W^{-2,\frac{2\beta}{\beta+1}}(\sO_\alpha) )\subset\subset L^4((0,T)\times \sO_\alpha).
\end{align*}
Hence our next aim is to obtain uniform boundedness of $\rho_n$ in $L^p(\Omega;L^4(0,T;W^{\frac12,\frac{4\beta}{\beta+2}}(\sO_\alpha) )$.
First, observe that, due to the Brezis-Mironescu interpolation inequality, we have,
\begin{align*}
	\|\rho_n\|^2_{W^{\frac12,\frac{4\beta}{\beta+2}}(\sO_\alpha)} \leq \|\rho_n\|_{L^{\beta}(\sO_\alpha)}\|\rho_n\|_{W^{1,2}(\sO_\alpha)}.
\end{align*}
This implies that,
\begin{align*}
	\int_0^T	\|\rho_n\|^4_{W^{\frac12,\frac{4\beta}{\beta+2}}(\sO_\alpha)} \leq \int_0^T \|\rho_n\|^2_{L^{\beta}(\sO_\alpha)}\|\rho_n\|^2_{W^{1,2}(\sO_\alpha)}&\leq \sup_{0\leq t\leq T}\|\rho_n\|^2_{L^{\beta}(\sO_\alpha)} \int_0^T \|\rho_n\|^2_{W^{1,2}(\sO_\alpha)},
\end{align*}
thus establishing, for some constant $C>0$, that
\begin{align}\label{rhotightep}
	\tilde{\bE}\left(\int_0^T	\|\rho_n\|^4_{W^{\frac12,\frac{4\beta}{\beta+2}}(\sO_\alpha)} \right)^p \leq \tilde{\bE} \sup_{0\leq t\leq T}\|\rho_n\|^{4p}_{L^{\beta}(\sO_\alpha)} + \tilde{\bE}\left(\int_0^T \|\rho_n\|^2_{H^1(\sO_\alpha)}\right)^{2p} \leq C(\ep,\delta).
\end{align}
As earlier, an application of the Chebyshev inequality concludes the proof of tightness of the sequence $\{\tilde{\bP}\circ\rho^{-1}_n\}$ in $L^4((0,T)\times\sO_\alpha)$.

Finally, since, $\beta\geq 3$, interpolation and Lemma \ref{energy_galerkin} (3),(4) then give us,
\begin{align*}
 \tilde{\bE}\|\rho_n\|_{L^{\beta+1}((0,T)\times\sO_\alpha)} &\leq  \tilde{\bE} \|\rho^\beta_n\|_{L^\frac43(0,T;L^2(\sO_\alpha))}\leq \tilde{\bE}\left(\|\rho^\beta_n\|_{L^\infty(0,T;L^1(\sO_\alpha))}^{\frac14}\|\rho^\beta_n\|_{L^1(0,T;L^3(\sO_\alpha))}^{\frac34}\right)\\
&\le \tilde{\bE}\left(\|\rho^\beta_n\|_{L^\infty(0,T;L^1(\sO_\alpha))}^{\frac14}\|\rho^{\frac{\beta}2}_n\|_{L^1(0,T;L^6(\sO_\alpha))}^{\frac34}\right)\leq C,
\end{align*}
concluding the tightness of the sequence $\{\tilde{\bP}\circ\rho^{-1}_n\}$ in $\mathcal{X}_\rho.$ 


\medskip

\noindent \textbullet {\bf Tightness of laws of $P^f_n(\rho_n\bu_n)$.} 
For this part of the proof we will utilize the following compact embedding for any $r>0$ and $l>\frac52$:
\begin{align}\label{compact_urho}
  L^\infty(0,T;L^{\frac{2\beta}{\beta+1}}(\mathcal{O}_{\alpha})) \cap C^{r}(0,T;H^{-l}(\mathcal{O}_{\alpha}))\subset\subset C_{w}(0, T; L^{\frac{2\beta}{\beta+1}}(\mathcal{O}_{\alpha})).
\end{align}
Recall that we already have, thanks to \eqref{urho}, that 
$$\tilde{\bE}\|P_n^f(\rho_n\bu_n)\|^p_{L^\infty(0,T;L^{\frac{2\beta}{\beta+1}}(\mathcal{O}_{\alpha})) }\leq C_p,$$ for any $p \geq 1$.
Next we will show that for any $l>\frac52$ and any $0\leq r<\frac12$, there exists some $C>0$ independent of $n$, such that
\begin{align}\label{urho_cont}
\tilde{\bE}\|P_n^f(\rho_n\bu_n)\|_{C^{r}(0,T;H^{-l}(\mathcal{O}_{\alpha}))} \leq C.
\end{align}
To obtain the aforementioned uniform bounds we consider \eqref{galerkin} with $\psi=0$ and write:
\begin{align*}
& \int_{\mathcal{O}_{\alpha}}P_n^f( \rho_{n}(t) \bu_{n}(t) )\cdot \bd{q}  = \int_{\mathcal{O}_{\alpha}}  P_n^f(\bp_{0,\delta,\ep})\cdot \bd{q} \\
&+\int_{0}^{t} \int_{\mathcal{O}_{\alpha}} P_n^f(\rho_{n} \bu_{n} \otimes \bu_{n}) : \nabla \bd{q} + \int_{0}^{t} \int_{\mathcal{O}_{\alpha}} { P_n^f}\Big(a\rho_{n}^{\gamma} + \delta \rho_{n}^{\beta}\Big) (\nabla \cdot \bd{q}) - \int_{0}^{t} \int_{\mathcal{O}_{\alpha}} { P_n^f\left(\mu^{\eta^*_n}_\delta  \nabla \bu_{n} \right)} : \nabla \bd{q} \\
& 	-  \int_{0}^{t} \int_{\mathcal{O}_{\alpha}}{ P_n^f\left(\lambda_\delta^{\eta^*_{ n}}\text{div}(\bu_{n}) \right)} \text{div}(\bd{q}) +\ep\int_0^t\int_{\sO_\alpha}P_n^f(\bu_n \rho_n)\cdot \Delta \bq- \frac{1}{\delta} \int_{0}^{t} \int_{\sO_\alpha} { P_n^f(\mathbbm{1}_{T^\delta_{\eta^*_n}}(\bu_{n} - v_{n} \bd{e}_{z}))} \cdot \bd{q} \\
& + \int_{0}^{t} \int_{\mathcal{O}_{\alpha}}{ P_n^f(\mathbbm{1}_{\sO_{\eta_n^*}}\bd{F}_{n}(\rho_{n} , \rho_{n} \bu_{n} ) )}\cdot \bd{q} dW_n^1(t) \\
&:=I_0+...+I_7,
\end{align*}
for any $\bq\in H^l_0(\sO_\alpha)$. 
We will take $\sup_{\bq\in H^l_0(\sO_\alpha)}$, $l>\frac52$ on both sides and analyze each term appearing in this weak form individually. 

We start with the first term $\displaystyle I_1 :=\int_0^t\sup_{\bq\in H^l_0(\sO_\alpha)}\int_{\sO_{\alpha}} P_n^f(\rho_n\bu_n\otimes\bu_n):\nabla\bq$. 
For any $p\geq1$, we obtain
\begin{align*}
  \tilde{\bE}  \|\partial_tI_1\|^p_{L^2(0,T)} \leq C(T)\tilde{\bE}\left(\|\rho_n\bu_n\|_{L^\infty(0,T;L^{\frac{2\beta}{\beta+1}}(\sO_\alpha))}\|\bu_n\|_{L^2(0,T;H_0^1(\sO_\alpha))}\right)^p \leq C.
\end{align*}
For the next term, $
		I_2=	\int_0^t\sup_{\bq \in H^l_0(\sO_\alpha)}|\int_{\sO_\alpha}P_n^f (a\rho^{\gamma}_n+\delta\rho^\beta_n) \nabla\cdot\bq|,
$
we choose $l>\frac52$ to obtain the following estimate
\begin{align*}
\tilde{\bE}\|\partial_tI_2\|_{L^2(0,T)}\le \tilde{\bE}\|\rho_n^\beta\|_{L^\infty(0,T;L^1(\sO_\alpha))}<C.
\end{align*}
Next, we have $
I_3=	\int_0^t\sup_{\bq \in H^l_0(\sO_\alpha)}\int_{\sO_\alpha}|P_n^f(\mu_\delta^{\eta^*_n}\nabla\bu_n ):\nabla\bq| $. Then Lemma \ref{energy_galerkin}
 readily gives us that 
\begin{align*}
\tilde{\bE}\|\partial_tI_3\|^2_{L^2(0,T)} \leq \tilde{\bE}\|\bu_n\|^2_{L^2(0,T;H^1_0(\sO_\alpha))}<C.
\end{align*}
The term $I_4$ is treated identically.
Now, for $I_5=	\ep\int_0^t\sup_{\bq \in H^l_0(\sO_\alpha)}|\int_{\sO_\alpha} P_n^f(\bu_n\rho_n)\cdot \Delta\bq| $ we write,
\begin{align*}
&	\|\partial_tI_5\|^2_{L^2(0,T)}\leq \int_0^T\|\bu_n\rho_n\|^2_{L^\frac32(\sO_{\alpha})}\leq \int_0^T\|\bu_n\|_{L^6(\sO_{\alpha})}^2\|\rho_n\|_{L^2(\sO_{\alpha})}.
\end{align*}
Hence, we obtain,
\begin{align*}
\tilde{\bE}[\|\partial_tI_5\|_{L^2(0,T)}]\leq \tilde{\bE}\|\bu_n\|^2_{L^2(0,T;H^1(\sO_{\alpha}))}+\tilde{\bE}\|\rho_n\|^2_{L^\infty(0,T;L^2(\sO_{\alpha}))} \leq C.
\end{align*}
Next, we find bounds for the term $I_6=\frac1{\delta}\int_0^t\sup_{\bq\in H^l_0(\sO_\alpha)}|\int_{\sO_\alpha}P_n^f(\mathbbm{1}_{T^\delta_{\eta^*_n}} (\bu_{n} - v_{n} \bd{e}_{z}) )\cdot \bd{q} | $. Observe that, using the trace theorem we obtain
\begin{align*}
 \tilde{\bE} [\|\partial_tI_6\|^2_{L^2(0,T)}] \leq   \tilde{\bE}\int_0^T(\|\bu_{n}\|_{L^2(\sO_\alpha)}+\|v_n\|_{L^2(\Gamma)})^2 \leq C(\delta).
\end{align*}
We summarize our bounds so far as follows,
\begin{align}\label{i6}
  \tilde{\bE}\|I_0+...+I_6\|_{C^{\frac12}([0,T];H^{-l}(\sO_\alpha))} \leq   C\tilde{\bE}\|I_0+...+I_6\|_{H^1(0,T;H^{-l}(\sO_\alpha))} \leq C.
\end{align}

Finally, we use the Burkholder-Davis-Gundy inequality to find for the stochastic term $I_7$ the following bounds for any $t_1<t_2$,
\begin{align*}
\tilde{\bE}\left\|\int_{t_1}^{t_2}P_n^f\mathbbm{1}_{\sO_{\eta_n^*}}\bd{F}_{n}(\rho_n,\rho_n\bu_n)dW_n^1\right\|_{H^{-l}(\sO_\alpha)}^p 
&\leq C\bE\left( \int_{t_1}^{t_2} \sum_{k= 1}^{\infty}\|f_{n,k}(\rho_n,\rho_n\bu_n)\|^2_{H^{-1}(\sO_\alpha)}\right)^{\frac{p}2}\\
	&\leq C\tilde{\bE}\left( \int_{t_1}^{t_2}\int_{\sO_\alpha} \rho_n+\rho_n|\bu_n|^2\right) ^{\frac{p}2}\\
	&\leq C|t_2- t_1|^{{\frac{p}2}}\left(\tilde{\bE}\|\rho_n\|^{{\frac{ p}2}}_{L^\infty(0,T;L^\gamma(\sO_{\alpha}))}
	+\tilde{\bE}\|\sqrt{\rho_n}\bu_n\|^p_{L^2(0,T;L^2(\sO_{\alpha}))} \right) .
\end{align*}
Then the Kolmogorov continuity theorem gives us that for any $r<\frac12$
\begin{align}\label{i7}
    \tilde{\bE}\left\|\int_0^t\mathbbm{1}_{\sO_{\eta_n^*}}\bd{F}_{n}(\rho_n,\rho_n\bu_n)dW_n^1\right\|_{C^r([0,T];H^{-l}(\sO_\alpha))}<C.
\end{align}
Hence, combining \eqref{i6} and \eqref{i7} we obtain \eqref{urho_cont}.
A final application of the embedding \eqref{compact_urho} and Chebyshev's inequality then gives us that $\{\tilde{\bP}\circ (\rho_n\bu_n)^{-1}\}$ is tight in $C_w(0,T;L^{\frac{2\beta}{\beta+1}}(\sO_\alpha))$.
\medskip

\noindent \textbullet {\bf Tightness for $v_n$.}
We will use the following variant of the Aubin-Lions theorem (see e.g. \cite{S87}, \cite{Tem95}) that states:
 Assume that $\mathcal{Y}_0\subset\mathcal{Y}$ are Banach spaces such that $\mathcal{Y}_0$ and $\mathcal{Y}$ are reflexive with compact embedding of $\mathcal{Y}_0$ in $\mathcal{Y}$, 
then for any $\sigma>0$, the embedding	
\begin{align}\label{AL}
    \left\{v \in L^2(0,T;\mathcal{Y}_0):\sup_{0<h< T} \frac1{h^\sigma}\|T_hv-v\|_{L^2(h,T;\mathcal{Y})}<\infty \right\} \hookrightarrow_c L^2(0,T;\mathcal{Y}),
\end{align}	
is compact. We take $\mathcal{Y}_0= H^1(\Gamma)$ and $\mathcal{Y}=L^2(\Gamma)$, use the bounds obtained in Lemma \ref{energy_galerkin} (6) and Lemma \ref{lem:vtight}, and with the aid of Chebyshev's inequality, as earlier, we infer that the sequence of measures $\tilde{\bP}\circ v_n^{-1}$ is tight in $L^2(0,T;L^2(\Gamma)).$ This completes the proof of Proposition \ref{tight_galerkin}.
\end{proof}

\subsection{Almost sure convergence and identification of the limit}
Thanks to the tightness result obtained in Section \ref{tight_galerkin} and Jakubowski's version of the Skorohod representation theorem, we obtain the following convergence result.
\begin{theorem}\label{skorohod_galerkin} 
	There exists a probability space $(\tilde\Omega,\tilde\sF,\tilde\bP)$ and random variables  \\ $\displaystyle\bar{\mathcal{U}}_{n}:=(
\bar{\rho}_n,	\bar{\bu}_{n},\overline{(\rho\bu)}_n,\bar\eta_n, \bar\eta^*_n, \bar v_n, \bar W^1_n, \bar W^2_n)$ for $n\in\mathbb{N}$ and
	${\mathcal{U}}:=(\rho,
	{\bu},\overline{\rho\bu},\eta, \eta^*, v, \bar W^1,\bar W^2)$
	defined on this new probability space, such that 
\begin{enumerate}
\item $\bar{\mathcal{U}}_{n}$ has the same law in $\mathcal{X}$ as $\mathcal{U}_{n}$,
\item $\bar{\mathcal{U}}_{n} \to \mathcal{U}_{n}$ almost surely in $\mathcal{X}$.
\item $\bar \eta^*_n=\bar \eta_n$ for every $t<\tau^\eta_n$ where, for the fixed $s\in(\frac32,2)$,
\begin{align*}
			\tau^{\eta}_n &:=T\wedge \inf\left\{t> 0:\inf_{\Gamma}(1+{\bar\eta_n}(t))\leq \alpha \text{ or } \|\bar{\eta}_n(t)\|_{H^s(\Gamma)}\geq \frac1{\alpha}\right\}.
 \end{align*}
\item $\partial_{t}\eta = v$, $\tilde{\mathbb{P}}$-almost surely.
\end{enumerate}
	
	\end{theorem}
We will construct a complete, right-continuous filtration $(\bar\sF_t)_{t \geq 0}$ on the new probability space $(\tilde\Omega,\tilde\sF,\tilde \bP)$ given in Theorem \ref{skorohod_galerkin}, to which the noise processes and the solutions are adapted. However, since the solutions at this stage are not regular enough in time to be considered stochastic processes, we rely on the definition of random distribution and its so-called history as introduced in \cite{BFH18}. 
 Let $\bu$ be a random distribution  then we define its history as
\begin{align}\label{history}
    \sigma_t(\bu):=\bigcap_{s\geq t}\sigma\left(\bigcup_{\bq\in C^\infty_0((0,s)\times\sO_\alpha)}\{(\bu,\bq)<1\}\cup \mathcal{N} \right),
\end{align}
 where $\mathcal{N}=\{\mathcal{A}\in \tilde\sF |  \tilde\bP(\mathcal{A})=0\}.$
 Let $\bar\sF_t'$ be the $\sigma-$ field generated by the random variables $\eta(s),\bar W^1(s),\bar W^2(s)$ for all $0\leq s \leq t$.
Then we define
\begin{align}\label{Ft1}
\bar\sF^0_t:=\bigcap_{s\ge t}\sigma(\bar{\sF}_s' \cup \mathcal{N}),\qquad
\bar\sF_t :=\sigma(\sigma_t(\bu,v)\cup  \bar\sF^0_t).
\end{align}
This gives a complete, right-continuous filtration $(\bar\sF_t)_{t \geq 0}$, on the probability space $(\tilde\Omega,\tilde\sF,\tilde\bP)$, to which the noise processes and the solutions are adapted. 
The filtrations $(\bar\sF^n_t)_{t \geq 0}$ are defined identically by replacing $\bu,v,\eta,\bar W^1,\bar W^2$ with $\bar\bu_n,\bar v_n,\bar\eta_n,\bar W^1_n,\bar W^2_n$ in the definition above.

Observe that, due to equivalence of laws, we have
$$\overline{(\rho\bu)_n}=\bar\rho_n\bar\bu_n \quad\text{ and thus, } \overline{(\rho\bu)}=\rho\bu.$$
Observe also that, thanks to the first equality in \eqref{rhoH1}, the strong almost sure convergence of the approximate densities in $L^4((0,T)\times\sO_\alpha)$ implies that$$\nabla \bar\rho_n\to \nabla \rho\quad\text{ in } L^2(0,T;L^2(\sO_\alpha)).$$ 
Moreover, as shown in \eqref{etasequal1}, we have that
		\begin{align}
			{\eta}^*(t)={\eta}(t) \quad \text{ for any } t<\tau^{\eta}, \quad \tilde\bP\text{-almost surely,}
		\end{align}	
		where, for any given $\alpha$, $\tau^\eta$ is an $(\bar\sF_t)_{t\in[0,T]}$-stopping time given by:
	       $$	\tau^{\eta}  := T\wedge\inf\left\{t>0:\inf_{\Gamma}(1+\eta(t))\leq \alpha \text{ or } \|{\eta}(t)\|_{H^s(\Gamma)}\geq \frac1{\alpha} \right\}.$$
        
The new random variables $\bar{\mathcal{U}}_n$ also satisfy the weak formulation \eqref{galerkin} and the uniform bounds found in Lemma \ref{energy_galerkin}.
Using these uniform bounds and the convergence result in Theorem \ref{skorohod_galerkin} we can pass $n\to\infty$ in \eqref{galerkin} and prove that $\bu,\rho,v,\eta,\eta^*$ solve the desired weak formulation \eqref{viscous}. We will only discuss the passage of $n\to\infty$ in the stochastic integral in Proposition \ref{stochint_galerkin} and in the advection term. For the convergence of the advection term in the weak formulation, we show that
\begin{align*}
	\int_0^t\int_{\sO_\alpha}{ P_n^f}(\bar\rho_n\bar\bu_n\otimes\bar\bu_n):\nabla\bq \to \int_0^t\int_{\sO_\alpha}\rho\bu\otimes\bu:\nabla \bq 
\end{align*}
$\tilde\bP$-a.s, for any $t \in[0,T]$ for any $\bq\in L^\infty(0,T;W^{1,\infty}(\sO_\alpha))$.
Recall, due to that Theorem \ref{skorohod_galerkin}, we have $\bar\rho_n\bar\bu_n \to \rho\bu$ in $C_w(0,T;L^{\frac{2\beta}{\beta+1}}(\sO_\alpha))$ almost surely. Now since the embedding $L^{\frac{2\beta}{\beta+1}}(\sO_\alpha) \hookrightarrow H^{-1}(\sO_\alpha)$ is compact, we can deduce that $\bar\rho_n\bar\bu_n \to \rho\bu$ in $L^p(0,T;H^{-1}(\sO_\alpha))$, $1\leq p < \infty$ almost surely, for example by Lemma B.2 in \cite{Smith}. This strong convergence together with the almost sure convergence of $\bar\bu_n$ weakly in $L^2(0,T;H^1(\sO_\alpha))$ to $\bu$ gives us the desired result.

Finally, we use the convergences in Theorem \ref{skorohod_galerkin} to establish $\tilde\bP$-almost sure convergence of the stochastic integrals in the weak formulation, for the fluid momentum equation. To help with this, we first prove the following lemma, which we will use to show convergence of the stochastic integrals.

\begin{lemma}\label{L2fmconv}
For every positive integer $k$ and for almost every $(\tilde\omega, t) \in \tilde\Omega \times [0, T]$,
\begin{equation}\label{convnoisecoeff}
\int_{\mathcal{O}_{\alpha}} \left|\frac{f_{k}(\bar\rho_n, \bar\rho_n\bar\bu_n)}{\bar\rho_n^{1/2}} - \frac{f_{k}(\rho, \rho \bu)}{\rho^{1/2}}\right|^{2} \to 0.
\end{equation}
\end{lemma}

\begin{proof}
This follows by an argument as in {\color{blue}{\cite{BFH18}}} involving Egorov's theorem and uniform estimates. To apply Egorov's theorem, the $\tilde\bP$-almost sure convergences $\bar\rho_{n} \to \rho$ (strongly) in $L^{2}(0, T; H^{1}(\mathcal{O}_{\alpha}))$ and $\bar\bu_{n} \rightharpoonup \bu$ (weakly) in $L^{2}(0, T; H^{1}_{0}(\mathcal{O}_{\alpha}))$ from Theorem \ref{skorohod_galerkin} imply that $\bar\rho_{n} \bar\bu_{n} \to \rho \bu$ (strongly) $\tilde\bP$-almost surely in $L^{1}(0, T; L^{1}(\mathcal{O}_{\alpha}))$. So by Egorov's theorem, there exists for every $\epsilon > 0$, a measurable set $A_{\epsilon} \subset \tilde\Omega \times [0, T] \times \mathcal{O}_{\alpha}$ such that $|A_{\epsilon}^{c}| \le \epsilon$, and both $\bar\rho_{n} \to \rho$ and $\bar\rho_{n}\bar \bu_{n} \to \rho \bu$ uniformly on $A_{\epsilon}$.

To avoid problems with vacuum, arising from dividing by the square root of the density, we define the set $A_{\epsilon, 1} = \{(\omega, t, x) \in A_{\epsilon} : 0 \le \rho \le \epsilon\}$ and $A_{\epsilon, 2} = A_{\epsilon} \setminus A_{\epsilon, 1}$. Then, \eqref{convnoisecoeff} will follow once we observe that:
\begin{itemize}
    \item $\displaystyle \int_{A_{\epsilon, 2}} \left|\frac{f_{k}(\bar\rho_{n}, \bar\rho_{n}\bar\bu_{n})}{\bar\rho_{n}^{1/2}} - \frac{f_{k}(\rho, \rho \bu)}{\rho^{1/2}}\right|^{2} \to 0 \text{ on } A_{\epsilon}^{c} \text{ as $n \to \infty$}$, since we are away from vacuum by the definition of $A_{\epsilon, 2}$ and by the uniform convergence of $\bar\rho_n \to \rho$ and $\bar\rho_n \bar\bu_n \to \rho\bu$ on $A_{\epsilon}$. 
    \item Next, by \eqref{fassumption}:
\begin{equation*}
\int_{A_{\epsilon, 1}} \left|\frac{f_{k}(\bar\rho_{n}, \bar\rho_{n}\bar\bu_{n})}{\bar\rho_{n}^{1/2}} - \frac{f_{k}(\rho, \rho \bu)}{\rho^{1/2}}\right|^{2} \le C \int_{A_{\epsilon, 1}} (\bar\rho_{n} + \bar\rho_{n}|\bar\bu_{n}|^{2} + \rho + \rho|\bu|^{2}),
\end{equation*}
for a constant $C$ depending only on $k$, and for sufficiently large $n$, this can be estimated by using the definition of $A_{\epsilon, 1}$ as:
\begin{equation*}
\limsup_{n \to \infty} \int_{A_{\epsilon, 1}} \left|\frac{f_{k}(\bar\rho_{n}, \bar\rho_{n}\bar\bu_{n})}{\bar\rho_{n}^{1/2}} - \frac{f_{k}(\rho, \rho \bu)}{\rho^{1/2}}\right|^{2} \le 2C\left(\limsup_{n \to \infty} \tilde\bE \int_{0}^{T} \int_{\mathcal{O}_{\alpha}} (2 + |\bar\bu_{n}|^{2} + |\bu|^{2})\right)\epsilon \le C\epsilon,
\end{equation*}
for $C$ depending only on $k$, where this estimate holds for sufficiently large $n$ depending on $(\omega, t, x)$ since eventually, $0 \le \rho_{n} \le 2\epsilon$ by the uniform convergence of $\bar\rho_{n} \to \rho$ on $A_{\epsilon}$ and the fact that $0 \le \rho \le \epsilon$ on $A_{\epsilon, 1}$ by definition. We also used the uniform estimate of $\bar\bu_{n} \in L^{2}(\tilde\Omega \times [0, T] \times \mathcal{O}_{\alpha})$ independently of $n$, which follows from Poincar\'{e}'s inequality. 
\item By uniform bounds of $\bar\rho_{n} \in L^{p}(\tilde\Omega; L^{\infty}(0, T; L^{\gamma}(\mathcal{O}_{\alpha})))$, $\bar\rho_{n} |\bar\bu_{n}|^{2} \in L^{p}(\tilde\Omega; L^{2}(0, T; L^{\frac{6\gamma}{4\gamma + 3}}(\mathcal{O}_{\alpha})))$ for all $1 \le p < \infty$, the assumption \eqref{fassumption}, and $|A_{\epsilon}^{c}| \le \epsilon$, for a constant $C$ depending only on $k$:
\begin{equation*}
\limsup_{n \to \infty} \int_{A_{\epsilon}^{c}} \left|\frac{f_{k}(\bar\rho_{n}, \bar\rho_{n}\bar\bu_{n})}{\bar\rho_{n}^{1/2}} - \frac{f_{k}(\rho, \rho \bu)}{\rho^{1/2}} \right|^{2} \le C \limsup_{n \to \infty} \int_{A_{\epsilon}^{c}} (\bar\rho_{n} + \bar\rho_{n}|\bar\bu_{n}|^{2} + \rho + \rho|\bu|^{2}) \le C \epsilon^{\frac{2\gamma - 3}{6\gamma}}.
\end{equation*}
\end{itemize}
So for a constant $C$ independent of $\epsilon$ and $n$ (depending only on $k$) and for arbitrary $\epsilon > 0$,
\begin{equation*}
\limsup_{n \to \infty} \tilde\bE \int_{0}^{T} \int_{\mathcal{O}_{\alpha}} \left|\frac{f_{k}(\bar\rho_{n}, \bar\rho_{n}\bar\bu_{n})}{\bar\rho_{n}^{1/2}} - \frac{f_{k}(\rho, \rho \bu)}{\rho^{1/2}}\right|^{2} \le C \epsilon,
\end{equation*}
which completes the proof.
\end{proof}
We now use Lemma \ref{L2fmconv} to show convergence of the stochastic integral as the Galerkin parameter $n \to \infty$. In contrast to the other stochastic integral convergences (see Lemma \ref{Nconvergence} and Lemma \ref{stochint}), we will make use of Egorov's theorem to establish this convergence, in order to control the term $\rho^{1/2}$ in the denominator of the definition \eqref{Fndef} of $f_{n, k}$, which needs to be appropriately estimated near vacuum. Note that unlike for the convergence argument for the stochastic integral in the time discretization passage $N \to \infty$ in Lemma \ref{Nconvergence}, see in particular the estimate \eqref{asdensity}, we do not have bounds of the densities $\rho_{n}$ away from vacuum $\tilde\bP$-almost surely that are independent of the Galerkin parameter $n$. This complicates the argument and requires the use of Egorov's theorem to rigorously pass to the limit in the stochastic integral term as $n \to \infty$.
\begin{lemma}\label{stochint_galerkin}
For $\bd{q} \in \bigcup_{n} X_{n}$,
\begin{equation*}
\int_{0}^{t} \int_{\mathcal{O}_{\alpha}} \mathbbm{1}_{\sO_{\bar\eta_n^*}}\bd{F}_{n}(\bar\rho_{n}, \bar\rho_{n}\bar{\bu}_{n}) \cdot \bd{q} d\bar W_n^1(t) \to \int_{0}^{t} \int_{\mathcal{O}_{\alpha}} \mathbbm{1}_{\sO_{\eta^*}}\bd{F}(\rho, \rho \bu) \cdot \bd{q} d\bar W^1(t), \quad \tilde\bP\text{-almost surely.}
\end{equation*}
\end{lemma}

\begin{proof}
By using classical ideas \cite{Ben} (see Lemma 2.1 of \cite{DGHT}, Lemma 2.6.6 in \cite{BFH18}) it suffices to prove
\begin{equation*}
\int_{\mathcal{O}_{\alpha}} \mathbbm{1}_{\sO_{\eta_n^*}}\bd{F}_{n}(\bar\rho_{n}, \bar\rho_{n}\bar{\bu}_{n}) \cdot \bd{q} \to \int_{\mathcal{O}_{\alpha}} \mathbbm{1}_{\sO_{\eta^*}} \bd{F}(\rho, \rho \bu) \cdot \bd{q} \text{ in probability in $L^{2}(0, T; L_{2}(\mathcal{U}_0; \R))$},
\end{equation*}
which follows if we show:
\begin{align}\label{Galerkinstoch0}
\tilde{\mathbb{E}} \int_{0}^{T} \left(\int_{\mathcal{O}_{\alpha}} \Big(\mathbbm{1}_{\mathcal{O}_{\eta^{*}}} - \mathbbm{1}_{\mathcal{O}_{\eta_{n}^{*}}}\Big) f_{k}(\rho, \rho \bu) \cdot \bd{q}\right)^{2} \to 0 \quad \text{ as $n \to \infty$, for each $k$},
\\
\label{Galerkinstoch1}
\tilde{\mathbb{E}} \int_{0}^{T} \left(\int_{\mathcal{O}_{\alpha}} \mathbbm{1}_{\mathcal{O}_{\eta_{n}^{*}}} \Big(f_{k}(\rho, \rho \bu) - f_{n, k}(\rho_{n}, \bar\rho_{n} \bar{\bu}_{n})\Big) \cdot \bd{q}\right)^{2} \to 0 \quad \text{ as $n \to \infty$, for each $k$,}
\\
\label{Galerkinstoch2}
\lim_{m \to \infty} \left(\sup_{n} \tilde\bE \int_{0}^{T} \sum_{k = m}^{\infty} \left(\int_{\mathcal{O}_{\alpha}} |f_{n, k} (\bar\rho_{n}, \bar\rho_{n}\bar{\bu}_{n}) \cdot \bd{q}| + |f_{k}(\rho, \rho \bu) \cdot \bd{q}|\right)^{2}\right) = 0.
\end{align}
We observe that \eqref{Galerkinstoch0} follows from the $\tilde\bP$-almost sure strong convergence of $\mathbbm{1}_{\mathcal{O}_{\eta^{*}_{n}}} \to \mathbbm{1}_{\mathcal{O}_{\eta^{*}}}$ in $L^{\infty}(0, T; L^{p}(\mathcal{O}_{\alpha}))$ for any $1 \le p < \infty$. So it suffices to verify \eqref{Galerkinstoch1} and \eqref{Galerkinstoch2}.

\medskip

\noindent \textbf{Proof of \eqref{Galerkinstoch1}:} Using Sobolev embedding of $L^{1}(\mathcal{O}_{\alpha}) \subset H^{-l}(\mathcal{O}_{\alpha})$ for $l > 3/2$:
\begin{equation*}
    \tilde\bE \int_{0}^{T} \left(\int_{\mathcal{O}_{\alpha}} \Big(f_{k}(\rho, \rho \bu) - f_{n, k}(\bar\rho_{n}, \bar\rho_{n} \bar{\bu}_{n})\Big) \cdot \bd{q}\right)^{2} \le \tilde\bE \int_{0}^{T} \|f_{k}(\rho, \rho\bu) - f_{n, k}(\bar\rho_{n}, \bar\rho_{n}\bar{\bu}_{n})\|_{H^{-l}(\mathcal{O}_{\alpha})}^{2},
\end{equation*}
where we recall from \eqref{Fndef} that $\displaystyle f_{n, k}(\bar\rho_{n}, \bar\rho_{n} \bar\bu_{n}) := \mathcal{M}^{1/2}[\bar\rho_{n}] P_{n}^{f} \left(\frac{f_{k}(\bar\rho_{n}, \bar\rho_{n}\bar\bu_{n})}{\bar\rho_{n}^{1/2}}\right)$. This estimate motivates us to show first that
\begin{equation}\label{negativeSobolev}
\|f_{k}(\rho, \rho \bu) - f_{n, k}(\bar\rho_{n}, \bar\rho_{n}\bar\bu_{n})\|_{H^{-l}(\mathcal{O}_{\alpha})} \to 0, \quad \text{ for {a.e. $(\tilde{\omega}, t) \in \tilde\Omega\times [0, T]$}, }
\end{equation}
which is established once we obtain these convergences for arbitrary $\psi \in H^l_0(\mathcal{O}_{\alpha})$ with $l > 3/2$:
\begin{enumerate}
\item $\displaystyle \left|\int_{\mathcal{O}_{\alpha}} \mathcal{M}^{1/2}[\bar\rho_{n}]P_{n}^{f}\left(\frac{f_{k}(\rho, \rho \bu)}{\rho^{1/2}} - \frac{f_{k}(\bar\rho_{n}, \bar\rho_{n}\bar\bu_{n})}{\bar{\rho}_{n}^{1/2}}\right) \psi \right| \to 0$ for a.e. $(\tilde{\omega}, t) \in \tilde\Omega \times [0, T]$.
\item $\displaystyle 
\left|\int_{\mathcal{O}_{\alpha}} \mathcal{M}^{1/2}[\bar\rho_{n}] \left(\frac{f_{k}(\rho, \rho\bu)}{\rho^{1/2}} - P^{f}_{n}\left(\frac{f_{k}(\rho, \rho \bu)}{\rho^{1/2}}\right)\right) \psi \right| \to 0$, for a.e. $(\tilde\omega, t) \in \tilde\Omega \times [0, T]$
\item $\displaystyle \left|\int_{\mathcal{O}_{\alpha}} \left(f_{k}(\rho, \rho \bu) - \mathcal{M}^{1/2}[\bar\rho_{n}] \left(\frac{f_{k}(\rho, \rho \bu)}{\rho^{1/2}}\right)\right) \psi \right| \to 0$ for a.e. $(\tilde\omega, t) \in \tilde\Omega \times [0, T]$.
\end{enumerate}

\medskip

\noindent \textbf{Proof of statement 1.} By the symmetry of $\mathcal{M}^{1/2}[\bar\rho_{n}]$ with respect to the $L^{2}(\mathcal{O}_{\alpha})$ inner product and by continuity of the projection operator on $L^{2}(\mathcal{O}_{\alpha})$, for $\psi \in H^l_0(\mathcal{O}_{\alpha})$ with $l > 3/2$:
\begin{footnotesize}
\begin{equation*}
\left|\int_{\mathcal{O}_{\alpha}} \mathcal{M}^{1/2}[\bar\rho_n] P_{n}^{f} \left(\frac{f_{k}(\rho, \rho \bu)}{\rho^{1/2}} - \frac{f_{k}(\bar\rho_n, \bar\rho_n\bu_{n})}{\bar\rho_n^{1/2}}\right)\psi\right| \le \left\|\frac{f_{k}(\rho, \rho \bu)}{\rho^{1/2}} - \frac{f_{k}(\bar\rho_n, \bar\rho_n\bu_{n})}{\bar\rho_n^{1/2}}\right\|_{L^{2}(\mathcal{O}_{\alpha})} \|\mathcal{M}^{1/2}[\bar\rho_n]\psi\|_{L^{2}(\mathcal{O}_{\alpha})},
\end{equation*}
\end{footnotesize}
which goes to zero as $n \to \infty$, by estimate (3.15) from \cite{BreitHofmanova} that 
\begin{equation}\label{BH314}
\|\mathcal{M}^{1/2}[\rho]\psi\|_{L^{2}(\mathcal{O}_{\alpha})} \le C\|\psi\|^{1/2}_{H^l_0(\mathcal{O}_{\alpha})} \|\rho\|_{L^{2}(\mathcal{O}_{\alpha})}^{1/2}
\end{equation}
and Lemma \ref{L2fmconv}\footnote{As a technical detail, we emphasize that the constant $C$ in estimate (3.14) from \cite{BreitHofmanova} depends only on the Sobolev embedding from $H^l_0(\mathcal{O}_{\alpha})$ to $L^{2}(\mathcal{O}_{\alpha})$, and is hence independent of the Galerkin parameter $n$. In particular, we note that the inequality that they use $\|P_{n}\psi\|_{H^l_0(\mathcal{O}_{\alpha})} \le \|\psi\|_{H^l_0(\mathcal{O}_{\alpha})}$ to establish this inequality does not change the constant $C$ since the Galerkin basis was also chosen to be orthonormal in $H^l_0(\mathcal{O}_{\alpha})$.}. So the result follows from the fact that $\limsup_{n \to \infty} \|\bar\rho_n\|_{L^{2}(\mathcal{O}_{\alpha})}$ is bounded for almost every $(t, \bar\omega)$ by the convergence $\tilde\bP$-almost surely of $\bar\rho_n \to \rho$ in $L^{p}(0, T; W^{1, p}(\mathcal{O}_{\alpha}))$ for $p > 2$. 

\medskip

\noindent \textbf{Proof of statement 2.} This follows analogously to Statement 1, since $\displaystyle P_{n}^{f}\left(\frac{f_{k}(\rho, \rho \bu)}{\rho^{1/2}}\right) \to \frac{f_{k}(\rho, \rho \bu)}{\rho^{1/2}}$ in $L^{2}(\mathcal{O}_{\alpha})$ as $n \to \infty$ for almost every $(t, \bar\omega)$ by the fact that $\displaystyle \frac{f_{k}(\rho, \rho \bu)}{\rho^{1/2}} \in L^{2}(\mathcal{O}_{\alpha})$ and by properties of the projection operator.

\medskip

\noindent \textbf{Proof of statement 3.} Finally, for the last convergence, we use symmetry of $\mathcal{M}^{1/2}[\bar\rho_n]$ to write
\begin{equation*}
\left|\int_{\mathcal{O}_{\alpha}} \left(f_{k}(\rho, \rho \bu) - \mathcal{M}^{1/2}[\bar\rho_n]\left(\frac{f_{k}(\rho, \rho \bu)}{\rho^{1/2}}\right)\right) \psi\right| = \left|\int_{\mathcal{O}_{\alpha}} \frac{f_{k}(\rho, \rho \bu)}{\rho^{1/2}} \left(\rho^{1/2} \psi - \mathcal{M}^{1/2}[\bar\rho_n]\psi\right)\right|,
\end{equation*}
which converges to zero as $n \to \infty$ for almost every $(t, \tilde\omega)$ for arbitrary $\psi \in H^l_0(\mathcal{O}_{\alpha})$ ($l > 3/2$), by (4.29) from \cite{BreitHofmanova}, which states that given $\bar\rho_n \to \rho$ in $H^l_0(\mathcal{O}_{\alpha})$: 
\begin{equation*}
\|\mathcal{M}^{1/2}_{n}[\bar\rho_n]\bd{p} - \sqrt{\rho} \bd{p}\|_{L^{2}(\mathcal{O}_{\alpha})} \to 0 \quad \text{ for all $\bd{p} \in H^l_0(\mathcal{O}_{\alpha})$},
\end{equation*}
and the fact that $\displaystyle \frac{f_{k}(\rho, \rho \bu)}{\rho^{1/2}} \in L^{2}(\mathcal{O}_{\alpha})$ for almost every $(t, \tilde\omega)$, by \eqref{fassumption} and uniform estimates.

\medskip

\noindent \textbf{Conclusion of proof.} Combining the convergences in statements 1-3 implies \eqref{negativeSobolev} for almost every $(t, \tilde\omega)$, so  the desired convergence \eqref{Galerkinstoch1} follows from \eqref{negativeSobolev} by the Vitali convergence theorem once we show that  $\displaystyle \tilde\bE \int_{0}^{T} \left(\|f_{k}(\rho, \rho \bu) - f_{n, k}(\bar\rho_n, \bar\rho_n\bar\bu_{n})\|^{2}_{H^{-l}(\mathcal{O}_{\alpha})}\right)^{r}$ is bounded independently of $n$ for some uniform constant $C$ for some $r > 1$. To derive this uniform bound, we use \eqref{fassumption}, \eqref{BH314}, and the embedding of $H^{-l}(\mathcal{O}_{\alpha})$ into $L^{1}(\mathcal{O}_{\alpha})$ to estimate:
\begin{align}\label{negativeSobolevcalc}
\|f_{k}(\rho&, \rho \bu)\|_{H^{-l}(\mathcal{O}_{\alpha})}^{2} + \|f_{n, k}(\bar\rho_n, \bar\rho_n\bar\bu_{n})\|_{H^{-l}(\mathcal{O}_{\alpha})}^{2} \nonumber \\
&\le \left[\left(\int_{\mathcal{O}_{\alpha}} |f_{k}(\rho, \rho \bu)|\right)^{2} + \sup_{\|\psi\|_{H^l_0(\mathcal{O}_{\alpha})} \le 1} \left|\int_{\mathcal{O}_{\alpha}} P^{n}_{f}\left(\frac{f_{k}(\bar\rho_n, \bar\rho_{n}\bar\bu_{n})}{\bar\rho_n^{1/2}}\right) \mathcal{M}^{1/2}[\bar\rho_n]\psi\right|^{2} \right] \nonumber \\
&\le c_{k}^{2} \left(\int_{\mathcal{O}_{\alpha}} \rho + \rho|\bu|\right)^{2} + \left\|\frac{f_{k}(\bar\rho_n, \bar\rho_n\bar\bu_{n})}{\bar\rho_n^{1/2}}\right\|_{L^{2}(\mathcal{O}_{\alpha})}^{2} \sup_{\|\psi\|_{H^l_0(\mathcal{O}_{\alpha})} \le 1} \|\mathcal{M}^{1/2}[\bar\rho_n]\psi\|^{2}_{L^{2}(\mathcal{O}_{\alpha})} \nonumber \\
&\le c_{k}^{2} \left(\int_{\mathcal{O}_{\alpha}} \rho + \rho|\bu|\right)^{2} + c_{k}^{2} \|\bar\rho_n\|_{L^{2}(\mathcal{O}_{\alpha})} \left(\int_{\mathcal{O}_{\alpha}} \bar\rho_n + \bar\rho_n|\bar\bu_{n}|^{2}\right) \nonumber \\
&\le 4c_{k}^{2}\left(\|\rho\|_{L^{1}(\mathcal{O}_{\alpha})}^{2} + \|\sqrt{\rho} \bu\|_{L^{2}(\mathcal{O}_{\alpha})}^{4} + \|\bar\rho_n\|_{L^{2}(\mathcal{O}_{\alpha})}^{2} + \|\bar\rho_n\|_{L^{1}(\mathcal{O}_{\alpha})}^{2} + \|\sqrt{\bar\rho_n}\bar\bu_{n}\|_{L^{2}(\mathcal{O}_{\alpha})}^{4}\right).
\end{align}
By using uniform bounds (independent of $n$) on the approximate solutions, this implies that $\displaystyle \tilde\bE \int_{0}^{T} \Big(\|f_{k}(\rho, \rho \bd{u}) - f_{n, k}(\bar\rho_n, \bar\rho_n \bar\bu_n)\|_{H^{-l}(\mathcal{O}_{\alpha})}^{2}\Big)^{r}$ for $r > 1$ is bounded independently of $n$ for each $k$, which together with \eqref{negativeSobolev} establishes \eqref{Galerkinstoch1}. The preceding calculation also establishes \eqref{Galerkinstoch2}, since $\displaystyle \sum_{k = 1}^{\infty} c_{k}^{2} < \infty$, see \eqref{fassumption}, and since the uniform bounds that we are using to bound the final expression in \eqref{negativeSobolevcalc} are independent of $n$.

\end{proof}


\section{Passing viscosity parameter $\ep \to 0$}\label{sec:viscousreg}
In this section, to emphasize its dependence on the parameter $\ep$, we denote the solutions constructed on $(\tilde\Omega,\tilde\sF,(\bar\sF^\ep_t)_{t\geq 0}, \tilde\bP)$ obtained in the previous section by $(\rho_\ep,\bu_\ep,\eta_\ep,\eta^*_\ep,v_\ep,W_\ep^1,W_\ep^2)$ respectively. Let us then recall that the weak formulation at this stage reads:
\begin{multline}\label{viscous}
\int_{\mathcal{O}_{\alpha}} \rho_{\ep}(t) \bu_\ep(t) \cdot \bd{q} + \int_{\Gamma} v_\ep(t) \psi = \int_{\mathcal{O}_{\alpha}} \bp_{0,\delta,\ep} \cdot \bd{q} + \int_{\Gamma} v_0 \psi
	+	\int_{0}^{t} \int_{\mathcal{O}_{\alpha}} (\rho_\ep \bu_\ep \otimes \bu_\ep) : \nabla \bd{q}\\
 + \int_{0}^{t} \int_{\mathcal{O}_{\alpha}} \Big(a\rho_\ep^{\gamma} + \delta \rho_\ep^{\beta}\Big) (\nabla \cdot \bd{q}) - \int_{0}^{t} \int_{\mathcal{O}_{\alpha}} \mu_\delta^{\eta^*_\ep}  \nabla \bu_\ep : \nabla \bd{q} +\ep\int_0^T\int_{\sO_\alpha}\rho_\ep\bu_\ep\cdot\Delta \bq \\
	-  \int_{0}^{t} \int_{\mathcal{O}_{\alpha}}\lambda_\delta^{\eta_{\ep}^*} \text{div}(\bu_\ep) \text{div}(\bd{q}) - \frac1\delta \int_{0}^{t} \int_{{T^\delta_{\eta^*_\ep}}} (\bu_{\ep} - v_{\ep} \bd{e}_{z}) \cdot (\bd{q} - \psi \bd{e}_{z}) - \int_{0}^{t} \int_{\Gamma} \nabla \eta_\ep \cdot \nabla \psi \\
	- \int_{0}^{t} \int_{\Gamma} \Delta \eta_\ep \Delta \psi - \int_{0}^{t} \int_{\Gamma} \nabla v_\ep \cdot \nabla \psi + \int_{0}^{t} \int_{\mathcal{O}_{\alpha}} \mathbbm{1}_{\sO_{\eta^*_\ep}} \bd{F}(\rho_\ep, \rho_\ep \bu_\ep) \cdot \bd{q} dW^1_\ep(t) + \int_{0}^{t} \int_{\Gamma} G(\eta_\ep, v_\ep) \psi dW_\ep^2(t),
\end{multline}
$\tilde\bP$-almost surely for any test function $\bd{q} \in C^{\infty}_c(\sO_{\alpha})$ and $\psi \in C^{\infty}(\Gamma)$ and for every $t\in[0,T]$.
Moreover, we have that the continuity equation reads $\tilde\bP$-almost surely as follows,
 \begin{equation}\label{cont_ep2}
\partial_t\rho_{\ep} + \text{div}(\rho_\ep \bu_\ep) = \varepsilon \Delta \rho_\ep, \quad \text{ in } \mathcal{O}_{\alpha}, \qquad{ \nabla \rho_\ep \cdot \bd{n}|_{\partial \mathcal{O}_{\alpha}} = 0,} \qquad {{\rho_{\ep}(0) = \rho_{0, \delta, \ep}.}}
\end{equation}

Thanks to the weak lower semicontinuity of norm, all the uniform bounds obtained Lemma \ref{energy_galerkin}, that are also independent of $\ep$, hold at this stage as well for the solution $(\rho_\ep,\bu_\ep,\eta_\ep,\eta^*_\ep, v_\ep)$. Hence, most of the calculations in Section \ref{sec:galerkin_tight}
are valid.
What is different is that we do not have bounds for the spatial derivatives of the density uniform in $\ep$. Hence, the tightness result for the density in $L^{\beta+1}$ in space requires a different approach which we explain in the next lemma. This result is required since the uniform energy estimates give us boundedness of $\rho_\ep^\beta$ merely in the non-reflexive space $L^1(\sO_\alpha)$ which is not sufficient to pass $\ep\to 0$ in the pressure term.

\begin{lemma}\label{viscous_pressure}
We have that
	\begin{align*}
		\tilde\bE\int_0^T\int_{\sO_\alpha} (\rho_\ep^{\gamma+1} + \delta\rho^{\beta+1}_\ep )\leq C.
	\end{align*}
\end{lemma}
\begin{proof}
The proof of this lemma is standard.
Let $\Delta^{-1}\rho$ denote the unique solution in $W^{2,\alpha}(\sO_\alpha)\cap W^{1,\gamma^*}(\sO_\alpha)$, where $\gamma^*$ is the Holder conjugate of $\gamma$, to the equation
$$-\Delta w =\rho,\quad w|_{\partial\sO_\alpha}=0.$$
We set $\psi=0$ in \eqref{viscous} and apply Ito's formula to $f_\chi( \rho,\bu)=\int_{\sO_\alpha}\bu\cdot \nabla\Delta^{-1}\rho$ in the spirit of Lemma 5.1 in \cite{BO13} to obtain
\begin{align*}
\int_0^t\int_{\sO_\alpha}  (a\rho^{\gamma+1}_\ep&+\delta\rho^{\beta+1}_\ep)=\int_0^t\int_{\sO_\alpha} \mu^{\eta^*_\ep}_\delta  \nabla \bu_\ep:\nabla^2\Delta^{-1}\rho_\ep + \int_0^t\int_{\sO_\alpha}  \mu^{\eta^*_\ep}_\delta\nabla\bu_\ep:\nabla \otimes\nabla\Delta^{-1}\rho_\ep \\
&-\int_0^t\int_{\sO_\alpha}(a\rho_\ep^\gamma+\delta\rho_\ep^\beta)\nabla \cdot\nabla\Delta^{-1}\rho_\ep		+\int_0^t\int_{\sO_\alpha}  \nabla\cdot\bu_\ep( \rho_\ep+\nabla \cdot\nabla\Delta^{-1}\rho_\ep)\\
&+\int_0^t\int_{\sO_\alpha}   \rho_\ep\bu_\ep\otimes \bu_\ep:\nabla^2\Delta^{-1}\rho_\ep		+\int_0^t\int_{\sO_\alpha} \bu_\ep\otimes \bu_\ep:\nabla \otimes\nabla\Delta^{-1}\rho_\ep\\
&+\ep\int_0^t\int_{\sO_\alpha} \nabla(\bu_\ep\rho_\ep) :  \nabla^2\Delta^{-1}\rho_\ep + \ep\int_0^t\int_{\sO_\alpha} \nabla(\bu_\ep\rho_\ep) :\nabla \otimes \nabla\Delta^{-1}\rho_\ep\\
&
+\sum_k\int_0^t\int_{\sO_\alpha}\mathbbm{1}_{\sO_{\eta^*_\ep}}f_k(\rho_\ep,\rho_\ep\bu_\ep)\cdot\nabla\Delta^{-1}\rho_\ep dW^1_\ep+ \frac1\delta \int_{0}^{t} \int_{{T^\delta_{\eta^*_\ep}}} (\bu_{\ep} - v_{\ep} \bd{e}_{z}) \cdot {(\nabla\Delta^{-1}\rho_\ep)}\\
&+\int_0^t\int_{\sO_\alpha} \rho_\ep\bu_\ep\nabla\Delta^{-1}\text{div}(\rho_\ep\bu_\ep+\ep\nabla\rho_\ep)+\int_0^t\int_{\sO_{\alpha}}\partial_t \rho_\ep\bu_\ep\cdot\nabla\Delta^{-1}\rho_\ep.
\end{align*}
The estimates then depend on the fact that, in $\mathbb{R}^3$, since $\beta>3$, we have the Sobolev embedding $W^{1,\beta} (\sO_\alpha)\hookrightarrow L^\infty(\sO_\alpha)$ which implies that
\begin{align*}
\|\Delta^{-1}\nabla\rho_\ep\|_{L^\infty(\sO_\alpha)} \leq \|\rho_\ep\|_{L^\beta(\sO_\alpha)}.
\end{align*}
Full details of these estimates will be provided for the more difficult proof of obtaining higher integrability of the fluid density uniformly in $\delta$ for the $\delta \to 0$ limit passage in the next section, see for example the interior pressure estimates of Proposition \ref{interior_pressure}.
\end{proof}

\subsection{Tightness of laws}\label{sec:viscous_tight}
We begin by defining the appropriate spaces
\begin{align*}
	&{\mathcal{X}_{\rho}} = C_{w}(0, T; L^{\beta}(\mathcal{O}_{\alpha})) 
\cap \left(L^2(0,T;H^1(\sO_\alpha))\cap L^{\beta+1}((0,T)\times\sO_\alpha),w\right)\\
	&{\mathcal{X}_{\bu}} = (L^{2}(0, T; H^{1}(\mathcal{O}_{\alpha})), w), \qquad \mathcal{X}_{\rho \bu} = C_{w}(0, T; L^{\frac{2\beta}{\beta+1}}(\mathcal{O}_{\alpha})) \cap C(0,T;H^{-l}(\sO_\alpha)) \quad \text{for }l>\frac52,\\
	&\mathcal{X}_{\eta} = C_{w}(0, T; H^2(\Gamma)) \cap C(0, T; H^{s}(\Gamma)), \text{ for any }s<2, \\
&\mathcal{X}_{v}=L^2(0,T;L^2(\Gamma))\cap (L^2(0,T;H^1(\Gamma)),w),\quad   \mathcal{X}_{W} = C(0, T; \mathcal{U}_{0})^2.
\end{align*}
We can then show the following tightness result that follows the proof of Proposition \ref{tight_galerkin} and by substituting Lemma \ref{viscous_pressure} to obtain tightness of the laws of density. 
\begin{proposition}\label{prop:tight_viscous}
Define the family of random variables
\begin{equation*}
	\mathcal{U}_{\ep} := (\rho_{\ep}, \bu_{\ep},\rho_\ep \bu_\ep, \eta_{\ep}, \eta_\ep^*, v_\ep, 
 W^1_\ep, W^2_\ep).
\end{equation*}
Then the sequence of the laws of $\mathcal{U}_{\ep}$ i.e. the measures $\{\tilde\bP\circ \mathcal{U}^{-1}_\ep\}_{\ep\geq 0}$ is tight in the phase space
\begin{equation*}
	{\mathcal{X}} = {\mathcal{X}_{\rho}} \times {\mathcal{X}_{\bu}} \times {\mathcal{X}_{\rho \bu}} \times {\mathcal{X}_{\eta}} \times {\mathcal{X}_{\eta}}\times {\mathcal{X}_{v} }
 \times {\mathcal{X}_{W}}.
\end{equation*}
\end{proposition}

\subsection{Skorohod convergence theorem}
 Hence,  by applying the Prohorov theorem and Theorem 1.10.4 in \cite{VW96}, which is a variant of the Skorohod representation theorem, we obtain the following convergence result.
\begin{theorem}\label{skorohod_viscous} 
There exists a filtered probability space $(\tilde\Omega,\tilde\sF,\tilde\bP)$ and random variables  \\$\tilde{\mathcal{U}}_{\ep}:=(\tilde\rho_{\ep}, \tilde\bu_\ep,\tilde{\rho_\ep \bu_\ep}, \tilde\eta_{\ep}, \tilde\eta_\ep^*, \tilde v_\ep, 
\tilde W_\ep^1, \tilde W_\ep^2)$,  and
	${\mathcal{U}}:=(\rho, \bu, \tilde{\rho\bu}, \eta, \eta^*, v,  \tilde W^1, \tilde W^2 )$
	defined on this new probability space, such that  
\begin{enumerate}
\item $\tilde{\mathcal{U}}_{\ep}$ has the same law in $\mathcal{X}$ as $\mathcal{U}_{\ep}$ 
\item $
\tilde{\mathcal{U}}_{\ep} \to \mathcal{U} \text{ in the topology of $\mathcal{X}$, $\tilde{\mathbb{P}}$-almost surely as $\ep \to \infty$},$
\item $\tilde{ \eta}^*_\ep =\tilde{ \eta}_\ep$ for every $t<\tau^\eta_\ep$ where, for the fixed $s\in(\frac32,2)$,
		\begin{align*}
			\tau^{\eta}_\ep &:=T\wedge \inf\left\{t> 0:\inf_{\Gamma}(1+{\tilde\eta_\ep}(t))\leq \alpha \text{ or } \|\tilde{\eta}_\ep(t)\|_{H^s(\Gamma)}\geq \frac1{\alpha}\right\}.
		\end{align*}
\item $\partial_{t}\eta = v$, $\tilde{\mathbb{P}}$-almost surely.
\end{enumerate}
 \end{theorem}
 The filtrations $(\tilde\sF_t)_{t\geq 0}$ and $(\tilde\sF^\ep_t)_{t\geq 0}$ are constructed as in the previous section ensuring adaptibility of the solutions and that $\tilde W^1, \tilde W^2$ are $(\tilde\sF_t)_{t\geq 0}$-Wiener processes and $\tilde W_\ep^1, \tilde W_\ep^2$ are $(\tilde\sF^\ep_t)_{t\geq 0}$-Wiener processes. The new random variables $\tilde{ \mathcal{U}}_\ep$ satisfy the continuity equation \eqref{cont_ep2} and the structure-fluid momentum equation \eqref{viscous} for every $\ep$.
 
 Using the same procedure as in the previous section, we can pass $\ep\to 0$ and prove that the random variables $(\bu,\rho,\eta,\eta^*)$ satisfy the desired weak formulation \eqref{delta} for a fixed $\delta>0$. We refer the reader to the proof of Lemma \ref{stochint} for details regarding passage of $\ep\to 0$ in the stochastic integral. We recall here that, due to its construction, $\rho_{0,\delta,\ep} \to \rho_{0,\delta}$ in $L^\gamma(\sO_\alpha)$ as $\ep\to 0$ (see \eqref{rho0conv}) where $\rho_{0,\delta}|_{\sO_\alpha\setminus\sO_{\eta_0}}=0$.

\section{Passage to the limit in $\delta \to 0$}\label{sec:delta}


As done previously, to emphasize its dependence on the parameter $\delta$, we denote the probability space and the solutions obtained in the previous section by $(\tilde\Omega,\tilde\sF,(\tilde\sF^\delta_t)_{t\in[0,T]},\tilde\bP)$
and $(\rho_\delta,\bu_\delta,\eta_\delta,\eta^*_\delta,v_\delta,W^1_\delta,W^2_\delta)$ respectively. Let us then recall that the weak formulation that we obtain after the limit in the artificial viscosity parameter $\varepsilon \to 0$ states that:
\begin{enumerate}
\item The following momentum equation, 
\begin{equation}
\begin{split}\label{delta}
&\int_{\mathcal{O}_{\alpha}} \rho_{\delta}(t) \bu_{\delta}(t) \cdot \bd{q} + \int_{\Gamma} v_{\delta}(t) \psi = \int_{\mathcal{O}_{\alpha}}  \bp_{0,\delta} \cdot \bd{q} + \int_{\Gamma} v_{0} \psi + \int_{0}^{t} \int_{\mathcal{O}_{\alpha}} (\rho_{\delta} \bu_{\delta} \otimes \bu_{\delta}) : \nabla \bd{q} \\
&+ \int_{0}^{t} \int_{\mathcal{O}_{\alpha}} \Big(a\rho_{\delta}^{\gamma} + \delta \rho_{\delta}^{\beta}\Big) (\nabla \cdot \bd{q}) - \int_{0}^{t} \int_{\mathcal{O}_{\alpha}} \mu^{\eta_{\delta}^*}_{\delta} \nabla \bu_{\delta} : \nabla \bd{q} - \int_{0}^{t} \int_{\mathcal{O}_{\alpha}} \lambda^{\eta_{\delta}^*}_{\delta} \text{div}(\bu_{\delta}) \text{div}(\bd{q}) \\
&- \frac1\delta \int_{0}^{t} \int_{{T^\delta_{\eta^*_\delta}}} (\bu_{\delta} - v_{\delta} \bd{e}_{z}) \cdot (\bd{q} - \psi \bd{e}_{z})  -\int_{0}^{t} \int_{\Gamma} \nabla v_{\delta} \cdot \nabla \psi- \int_{0}^{t} \int_{\Gamma} \nabla \eta_{\delta} \cdot \nabla \psi \\
&- \int_{0}^{t} \int_{\Gamma} \Delta \eta_{\delta} \Delta \psi + \int_{0}^{t} \int_{\mathcal{O}_{\alpha}} \mathbbm{1}_{\sO_{\eta^*_\delta}}\bd{F}(\rho_{\delta}, \rho_{\delta} \bu_{\delta}) \cdot \bd{q} dW^1_\delta + \int_{0}^{t} \int_{\Gamma} G(\eta_{\delta}, v_{\delta}) \psi d W^2_\delta,
\end{split}
\end{equation}
holds $\tilde\bP$-almost surely, for almost every $t\in[0,T]$ and for every test function $\bd{q} \in C_c^\infty(\mathcal{O}_{\alpha})$ and $\psi \in C^{\infty}(\Gamma)$. We recall that $\bp_{0,\delta}$, defined in \eqref{p0}, approximates the initial data $\bp_0$.
\item The continuity equation,
\begin{align*}
    \int_{\sO_\alpha}\rho_\delta(t)\phi = \int_{\sO_\alpha}\rho_{0,\delta}\phi + \int_0^t\int_{\sO_\alpha}\rho_\delta\bu_\delta\cdot\nabla\phi
\end{align*}
holds $\tilde\bP$-almost surely for every $\phi\in C^\infty(\bar\sO_\alpha)$ and $t\in[0,T]$.
\end{enumerate}

\begin{remark}[A discussion about numerology] 
Many of the chosen constructions in the existence proof, such as extension of viscosity coefficients and the use of the exterior tubular neighborhood $T^{\delta}_{\eta^{*}_{\delta}}$ {{defined in \eqref{tube}}} in the penalty term that enforces the kinematic coupling condition in the limit as $\delta \to 0$, rely on $\delta$ explicitly in important ways, where the specific numerology is carefully chosen. Before continuing with the existence proof, we remind the reader of the relevant $\delta$-dependent quantities. 

First, $\nu_0$ is a fixed parameter that will later be relevant to the vanishing of density outside the physical domain in the limit as $\delta \to 0$, see Proposition \ref{vacuum2}. Then, the viscosity coefficients are extended from the physical domain to the maximal domain via the equations \eqref{viscosityextension}:
\begin{equation*}
\mu^{\eta^{*}_{\delta}}_\delta = \chi^{\eta^{*}_{\delta}}_{\delta^{\nu_{0}}} \mu, \quad \lambda^{\eta^{*}_{\delta}}_\delta = \chi^{\eta^{*}_{\delta}}_{\delta^{\nu_{0}}}\lambda,
\end{equation*}
where $\chi^{\eta^{*}_{\delta}}_{\delta^{\nu_{0}}}$ is defined in \eqref{chi}, for $\kappa = \delta^{\nu_{0}}$. Note that by the properties of $\chi^{\eta^{*}_{\delta}}_{\delta^{\nu_{0}}}$ and of the bounding function $a^{\eta^{*}_{\delta}}_{\delta^{\nu_0}}$, see estimate \eqref{abound}, we have that
\begin{align}\label{muconstantremark}
&\mu^{\eta^{*}_{\delta}}_{\delta} = \mu, \quad \lambda^{\eta^{*}_{\delta}}_{\delta} = \lambda, \quad \text{ for $(x, y, z) \in \mathcal{O}_{\alpha} \quad$ such that $z \le \eta^{*}_{\delta}(x,y) + \left(C_{\alpha} + \frac{1}{4}\right)\delta^{\nu_{0}/2}$},
\\
&
\mu^{\eta^{*}_{\delta}}_{\delta} = \delta^{\nu_0}, \quad \text{ for $(x, y, z) \in \mathcal{O}_{\alpha} \quad$ such that $z \ge \eta^{*}_{\delta}(x,y) + \left(3C_{\alpha} + \frac{3}{4}\right)\delta^{\nu_{0}/2}$}.\label{muexteriorremark}
\end{align}
for a constant $C_{\alpha}$ that depends only on $\alpha$. More generally, note that
\begin{equation}\label{mulowerbound}
\mu^{\eta^{*}_{\delta}}_{\delta} \ge \delta^{\nu_0} \text{ and } \lambda^{\eta^{*}_{\delta}}_{\delta} \ge \delta^{\nu_0}, \quad \text{ for all } (x, y, z) \in \mathcal{O}_{\alpha}.
\end{equation}

Recall also that the penalty term will give a tubular neighborhood estimate $T^{\delta}_{\eta^{*}_{\delta}}$, see the estimate \eqref{penaltybound}, where the tubular neighborhood is defined in \eqref{tube} to be an exterior tubular neighborhood of width $\delta^{\left(\frac{1}{2} - \frac{1}{\beta}\right)}$ containing the points $(x, y, z) \in \mathcal{O}_{\alpha} \setminus \mathcal{O}_{\eta^{*}_{\delta}}$ such that
\begin{equation*}
0 < (z - 1 - \eta^{*}_{\delta}) < \delta^{\left(\frac{1}{2} - \frac{1}{\beta}\right)}.
\end{equation*}
We now make an important observation relating the width of the tubular neighborhood to the width of extension for the viscosity coefficients. Since $\nu_{0}$ will be chosen to be $\sim \left(\frac{1}{2} - \frac{1}{\beta}\right)^{2}$
(cf. \eqref{nu0} and Proposition \ref{vacuum2})
and since we are considering $\delta$ sufficiently small in the limit as $\delta \to 0$, we have that the width of the tubular neighborhood is less than the width of the region outside of the moving domain where the viscosity coefficients are still equal to their values; see \eqref{muconstantremark}, namely:
\begin{equation*}
\delta^{\left(\frac{1}{2} - \frac{1}{\beta}\right)} \le \left(C_{\alpha} + \frac{1}{4}\right)\delta^{\nu_{0}/2}. 
\end{equation*}
Since we have uniform bounds on the dissipation in terms of the extended viscosity coefficient $\mu^{\eta^{*}_{\delta}}_{\delta}$, as in Lemma \ref{energy_delta} (2), we can use this fact to conclude not just boundedness of $\tilde{\mathbb{E}} \|\nabla \bu_{\delta}\|^{p}_{L^2(0,T;L^{2}(\mathcal{O}_{\eta^{*}_{\delta}}))}$ but more generally of $\tilde{\mathbb{E}} \|\nabla \bu_{\delta}\|^{p}_{L^2(0,T;L^{2}(\mathcal{O}_{\eta^{*}_{\delta}} \cup T^{\delta}_{\eta^{*}_{\delta}}))}$.
\end{remark}

\medskip

\noindent \textbf{Uniform boundedness.} Thanks to the fact that the estimates obtained in Lemma \ref{energy_galerkin} are independent of $\delta$, we obtain the following uniform boundedness results. 

\begin{lemma}\label{energy_delta}
For any $p\geq 1$, there exists a constant $C$ that is independent of $\delta$ such that:
\begin{enumerate}
    \item $\tilde\bE\|\rho_{\delta}\|^p_{L^{\infty}(0, T; L^{\gamma}(\mathcal{O}_{\alpha}))} \leq C$.
  \item {{$\tilde\bE\|\sqrt{\mu^{\eta^*_\delta}_\delta}\nabla\bu_{\delta}\|^p_{ L^2(0,T;L^{2}(\mathcal{O}_{\alpha}))}\leq C$,\quad and $\tilde\bE\|\sqrt{\lambda^{\eta^*_\delta}_\delta}\nabla\cdot\bu_{\delta}\|^p_{ L^2(0,T;L^{2}(\mathcal{O}_{\alpha}))}\leq C$.}}
    \item $\tilde\bE\|\eta_{\delta}\|^p_{ L^{\infty}(0, T; H^2(\Gamma))}\leq C$ and  $\tilde\bE\|\eta^*_{\delta}\|^p_{ L^{\infty}(0, T; H^2(\Gamma))}\leq C$.
    \item $\tilde\bE\|v_{\delta}\|^p_{L^{\infty}(0, T; L^{2}(\Gamma))}\leq C$.
    \item $\tilde\bE\|\sqrt{\rho_{\delta}} \bu_{\delta}\|^p_{ L^{\infty}(0, T; L^{2}(\mathcal{O}_{\alpha}))}\leq C$.
    \item $\tilde\bE\|\rho_{\delta} \bu_{\delta}\|^p_{ L^{\infty}(0, T; L^{\frac{2\gamma}{\gamma + 1}}(\mathcal{O}_{\alpha}))}\leq C$.
    \item $\tilde\bE\|\bu_\delta\|^p_{L^2(0,T;H^{1}(\sO_{\eta^*_\delta}\cup{T}^\delta_{\eta^*_\delta}))} \leq C$. 
    \item $\tilde\bE\|\bu_\delta\|^p_{L^2(0,T;L^q(\sO_{\eta^*_\delta}\cup T^\delta_{\eta^*_\delta}))}<C(\alpha). \qquad \forall 0 \leq q<6$
      \item For some $0<\bar{\kappa}<\frac12$, $\tilde{\bE}\left[\sup_{0<h<T}\frac1{h^{\bar\kappa}}\|\mathcal{T}_hv_\delta-v_\delta\|_{L^{2}(h,T;L^2(\Gamma))}^2\right]\leq C$, where $\mathcal{T}_hf(t) = f(t-h)$.
\end{enumerate}
Moreover, we have that
\begin{align}\label{penaltybound}
 \tilde\bE\|\bu_{\delta} - v_{\delta} \bd{e}_{z}\|^p_{L^2(0,T;L^2(T^{\delta}_{\eta^*_\delta}))}\leq C\delta^{\frac{p}2},
\end{align}
\end{lemma}

We will next discuss how Statements (7) and (8) are derived. 
Observe that, since $\bu_\delta = 0$ on the bottom boundary of $\mathcal{O}_{\eta_\delta^*}$,  we have
\begin{multline*}
\int_{0}^{1 + \eta^*_\delta(t, x, y)} |\bu_\delta(t, x, y, z)|^{2} dz = \int_{0}^{1 + \eta^*_\delta(t, x, y, z)} \left(\int_{0}^{z} \partial_{z}\bu_\delta(t, x, y, w) dw\right)^{2} dz \\
\le (1 + \eta^*_\delta(t, x, y)) \int_{0}^{1 + \eta^*_\delta(t, x, y)} \int_{0}^{z} |\partial_{z}\bu_\delta(t, x, y, w)|^{2} dw dz 
\le (1 + \eta^*_\delta(t, x, y))^{2} \int_{0}^{1 + \eta^*_\delta(t, x, y)} |{{\partial_{z}\bu_\delta(t, x, y, z)}}|^{2} dz.
\end{multline*}
Hence we obtain the following Poincare inequality by using the fact that $\|\eta^*_\delta\|_{L^\infty((0,T)\times\Gamma)} \leq \frac1\alpha$ almost surely:
\begin{equation}\label{poincare}
\int_{0}^{1 + \eta^*_\delta(t, x, y)} |\bu_\delta(t, x, y, z)|^{2} dz \le C(\alpha^{-2} )\int_{0}^{1 + \eta^*_\delta(t, x, y)} |\nabla \bu_\delta(t, x, y, z)|^{2} dz.
\end{equation}
Similar calculations, obtained by considering the bounding function $a^{\eta^*_\delta}_\kappa$ defined in \eqref{abbounding} and by noticing that the extension coefficient $\mu_{\delta}^{\eta^*_\delta}$ is equal to $\mu$ in $\sO_{\eta^*_\delta}\cup{T}^\delta_{\eta^*_\delta}$, give us, for some $C>0$ independent of $\delta$, that
\begin{align}\label{bounduH1}
\tilde\bE\|\bu_\delta\|^p_{L^2(0,T:H^1(\sO_{\eta^*_\delta}\cup{T}^\delta_{\eta^*_\delta}))} \leq C.
\end{align}
Statement (8) follows from Statement (7) after an application of Corollary 2.9 in \cite{LengererRuzicka}. Finally,
using bounds \eqref{bounduH1} in the proof of Lemma \ref{lem:vtight} gives us Statement (9).

\subsection{Tightness of laws}\label{sec:delta_tight}
As done in previous sections, we will now show that the sequence of the laws of the approximate solutions is tight in an appropriate phase space. For that purpose, we define the following spaces
\begin{align}\label{Xnu}
	\mathcal{X}_{\rho} = C_{w}(0, T; L^{\gamma}(\mathcal{O}_{\alpha})),\qquad \mathcal{X}_{\rho \bu} ={   (L^\infty(0,T;L^{\frac{2\gamma}{\gamma+1}}(\mathcal{O}_{\alpha})),w^*)} \nonumber \\
{	 \mathcal{X}_{\bu} = (L^{2}(0, T; L^{q}(\mathcal{O}_{\alpha})), w)},\qquad \mathcal{X}_{\nabla\bu} = (L^2(0, T; L^2(\mathcal{O}_{\alpha}), w),\quad \mathcal{X}_{div\bu} = (L^{2}(0, T; L^{2}(\mathcal{O}_{\alpha})), w) \nonumber \\
	\mathcal{X}_{\eta} = C_{w}(0, T; H^2(\Gamma)) \cap C([0, T]; H^{s}(\Gamma)), \quad
{\mathcal{X}_{v}=L^2(0,T;L^2(\Gamma)) \cap (L^2(0,T;H^1(\Gamma),w)},\nonumber\\
  \mathcal{X}_{W} = C(0, T; \mathcal{U}_{0})^2, \quad
\mathcal{X}_{\nu} = (L^{\infty}([0, T] \times \mathcal{O}_{\alpha}; \mathcal{P}(\R^{13})), w^{*}),
\end{align}
for fixed $\frac32<s<2$ and $q<6$. 

\medskip

\noindent \textbf{A short exposition on Young measures.} In the definition of $\mathcal{X}_{\nu}$ in \eqref{Xnu}, $\mathcal{P}(\R^{13})$ denotes the space of probability measures on the space $\R^{13}$ so that $\mathcal{X}_{\nu}$ denotes the space for the \textit{Young measures} for the fluid density, the fluid velocity, and the gradient of the fluid velocity. The intuition for Young measures is as follows. If we have a deterministic fluid velocity and fluid velocity giving rise to the pointwise-defined function $(\rho, \bu, \nabla \bu): [0, T] \times \mathcal{O}_{\alpha} \to \R^{13}$, then for each point in spacetime, $(\rho, \bu, \nabla \bu)$ has a single value in $\R^{13}$. However, a Young measure is a rigorous way of dealing with ``multi-valued" functions from $[0, T] \times \mathcal{O}_{\alpha}$ to $\R^{13}$, where the value at each point $(t, x) \in [0, T] \times \mathcal{O}_{\alpha}$ is a weighted average of possible values in the range $\R^{13}$. 

To represent this weighted average, a \textbf{Young measure} in our current context is a measurable map $\nu: [0, T] \times \mathcal{O}_{\alpha} \to \mathcal{P}(\R^{13})$, that is a probability measure on the range $\R^{13}$ for each $(t, x) \in [0, T] \times \mathcal{O}_{\alpha}$, where this map from $(t, x) \in [0, T] \times \R^{13}$ to a probability measure associated to $(t, x)$, denoted by $\nu_{t, x} \in \mathcal{P}(\R^{13})$, is measurable in an appropriate sense, see Definition 2.8.4 in \cite{BFH18}. In this sense of Young measures representing weighted averages of potential function values, given any function $(\rho, \bd{u}, \nabla \bu): [0, T] \times \mathcal{O}_{\alpha} \to \R^{13}$ that is real-valued (rather than probability measure-valued as for general Young measures), we can associate a natural Young measure $\nu$ defined to have a probability-measured value at $(t, x) \in [0, T] \times \mathcal{O}_{\alpha}$ of
\begin{equation}\label{deltafunction}
\nu_{t, x} = \delta_{\small(\rho(t, x), \bu(t, x), \nabla \bu(t, x)\small)},
\end{equation}
which is the Dirac delta function on $\R^{13}$ supported at the specific point $(\rho(t, x), \bu(t, x), \nabla \bu(t, x)) \in \R^{13}$. The Dirac delta measure at each $(t, x)$ represents the fact that given a genuinely $\R^{13}$-valued function $(\rho, \bu, \nabla \bu)$, the value of this function at $(t, x)$ is a single determined value in $\R^{13}$ rather than a spread of potential values in $\R^{13}$. The space of general Young measures considered as measurable maps from $[0, T] \times \mathcal{O}_{\alpha}$ to probability measures on $\R^{13}$ that we consider for $(\rho, \bu, \nabla \bu)$ is the space $\mathcal{X}_{\nu}$ defined in \eqref{Xnu}, which has an appropriate weak-star topology of convergence, which can be precisely defined as in pg.~156 on \cite{BFH18}. Since the precise definition will not be important to the current exposition, we refer the interested reader to pg.~156 of \cite{BFH18} for the explicit definition of this space, and more generally to Section 2.8 of \cite{BFH18} for more information about Young measures. 


\medskip

We can now state the main result of this section is stated in the following proposition. 
\begin{proposition}\label{tight_delta}
Consider the sequence of random variables,

\begin{small}
\begin{multline*}
	\mathcal{U}_{\delta} := \left(\rho_{\delta},\mathbbm{1}_{\left(\sO_{\eta^*_\delta}\cup T^\delta_{\eta^*_\delta}\right)}\bu_\delta,
    \sqrt{\mu^{\eta^*_\delta}_\delta}\nabla \bu_\delta, \sqrt{\lambda^{\eta^{*}_{\delta}}_\delta}\nabla\cdot\bu_{\delta}, \rho_\delta \bu_\delta, \eta_{\delta}, \eta_\delta^*, v_\delta, W^1, W^2, \delta_{(\rho, \mathbbm{1}_{\mathcal{O}_{\eta^{*}_{\delta}}} \bu_\delta, \mathbbm{1}_{\mathcal{O}_{\eta^{*}_{\delta}} \nabla \bu_\delta}})\right)
\end{multline*}
\end{small}
where the last term $\delta_{(\rho, \mathbbm{1}_{\mathcal{O}_{\eta^{*}_{\delta}}} \bu_\delta, \mathbbm{1}_{\mathcal{O}_{\eta^{*}_{\delta}}} \nabla \bu_\delta)}$ is to be interpreted in the sense of a Dirac delta probability measure-valued function as discussed in \eqref{deltafunction}, since the random functions $(\rho, \mathbbm{1}_{\mathcal{O}_{\eta^{*}_{\delta}}\bu, \mathbbm{1}_{\mathcal{O}_{\eta^{*}_{\delta}}} \nabla \bu})$ for each $\omega \in \tilde\Omega$ are (single-valued) real-valued functions from $[0, T] \times \mathcal{O}_{\alpha} \to \R^{13}$. Then the sequence of measures $\{\tilde\bP\circ \mathcal{U}_\delta^{-1}\}_{\delta \geq 0}$
 is tight in the phase space
\begin{equation*}
	\mathcal{X} = \mathcal{X}_{\rho} \times\mathcal{X}_{\bu} \times \mathcal{X}_{\nabla\bu} \times \mathcal{X}_{div\bu}  \times \mathcal{X}_{\rho \bu} \times \mathcal{X}_{\eta}\times \mathcal{X}_{\eta} \times \mathcal{X}_{v} 
 \times \mathcal{X}_{W} \times \mathcal{X}_{\nu}.
\end{equation*}
\end{proposition}
The proof of this proposition is identical to that of the corresponding statements in Proposition \ref{tight_galerkin}. The only novelty is to show the tightness of the Young measures $\delta_{(\rho, \mathbbm{1}_{\eta^{*}_{\delta}} \bu_\delta, \mathbbm{1}_{\mathcal{O}_{\eta_{\delta}^{*}} \nabla \bu_\delta})}$, which is shown using the compactness criterion in Corollary 2.8.6 in \cite{BFH18}, Chebychev's inequality, and the fact that
\begin{equation*}
\tilde\bE\left(\|\rho_{\delta}\|_{L^{\infty}(0, T; L^{ \gamma}(\mathcal{O}_{\alpha}))} + \|\mathbbm{1}_{\mathcal{O}_{\eta^{*}_{\delta}}} \bu_{\delta}\|_{L^{2}(0, T; L^{2}(\mathcal{O}_{\alpha}))} + \|\mathbbm{1}_{\mathcal{O}_{\eta^{*}_{\delta}}\cup T^\delta_{\eta^*_\delta}} \nabla \bu_{\delta}\|_{L^{2}(0, T; L^{2}(\mathcal{O}_{\alpha}))}\right) < \infty,
\end{equation*}
as in the proof of Proposition 4.4.7 in \cite{BFH18}.

\subsection{Skorohod convergence theorem}
At this stage we apply the Skorohod representation theorem along with the result of \cite{J97}, to obtain the following convergence result.
\begin{theorem}\label{skorohod} 
	There exists a probability space $(\tilde\Omega,\tilde\sF,\tilde\bP)$ and random variables  ${ \hat\bu_\delta}$ and $\hat{\mathcal{U}}_{\delta}:=(\hat\rho_{\delta}, \mathbbm{1}_{\left(\sO_{\hat\eta^*_\delta}\cup T^\delta_{\hat\eta^*_\delta}\right)}\hat\bu_\delta, 
    \sqrt{\mu^{\hat\eta^*_\delta}_\delta}\nabla\hat\bu_{\delta}^*,\sqrt{\lambda^{\hat\eta^{*}_{\delta}}_\delta}\nabla\cdot\hat\bu_{\delta},{\hat\rho_\delta \hat\bu_\delta}, \hat\eta_{\delta}, \hat\eta_\delta^*, \hat v_\delta, 
\hat W^1_\delta, \hat W^2_\delta, \hat{\nu}_{\delta})$, and \\
	${\mathcal{U}}:=(\rho, \bu, \bu_{\nabla}^*, \bu_{div}^*, \bp, \eta, \eta^*, v,  \hat W^1, \hat W^2, \nu)$
	defined on this new probability space, such that 
	\begin{enumerate}
\item $\hat{\mathcal{U}}_{\delta}$ has the same law in $\mathcal{X}$ as $\mathcal{U}_{\delta}$,
\item $
\hat{\mathcal{U}}_{\delta} \to \mathcal{U} \text{ in the topology of $\mathcal{X}$, $\tilde{\mathbb{P}}$-almost surely as $\delta \to 0$},$
\item $\hat \eta^*_\delta=\hat \eta_\delta$ for every $t<\tau^\eta_\delta$ where, for the fixed $s\in(\frac32,2)$,
		\begin{align}\label{taudelta}
			\tau^{\eta}_\delta &:=T\wedge \inf\left\{t> 0:\inf_{\Gamma}(1+{\hat\eta_\delta}(t))\leq \alpha \text{ or } \|\hat{\eta}_\delta(t)\|_{H^s(\Gamma)}\geq \frac1{\alpha}\right\}.
		\end{align}
\item  $\partial_{t}\eta = v$, $\tilde{\mathbb{P}}$-almost surely.
\end{enumerate}
 In addition, for any Carath\'{e}odory function $H: [0, \infty) \times \R^{3} \times \R^{9} \to \R$ satisfying the growth condition:
 \begin{equation*}
|H(\rho, \bu, \bd{Q})| \le C(1 + \rho^{q_1} + |\bu|^{q_2} + |\bd{Q}|^{q_3})
\end{equation*}
for some constant $C$ and for some $q_1, q_2, q_3 \ge 1$, we have that $\tilde{\mathbb{P}}$-almost surely: for any $1 < p \le \min\left(\frac{\gamma}{q_1}, \frac{6}{q_2}, \frac{2}{q_3}\right)$,
\begin{equation}\label{caratheodoryconv}
H(\rho_{\delta}, \mathbbm{1}_{\mathcal{O}_{\hat{\eta}^{*}_{\delta}}} \hat{\bu}_{\delta}, \mathbbm{1}_{\mathcal{O}_{\hat{\eta}^{*}_{\delta}}} \nabla \hat{\bu}_{\delta}) \rightharpoonup \overline{H(\rho, \bu, \bd{Q})}, \ \ \text{weakly in $L^{p}((0, T) \times \mathcal{O}_{\alpha})$},
\end{equation}
where the weak limit is defined using the limiting Young measure:
{\begin{equation*}
\overline{H(\rho, \bu, \bd{Q})} = \int_{\R^{13}} H(\rho', \bu', \bd{Q}') d\nu(\rho',\bu',\bd{Q}'). 
\end{equation*}}
	\end{theorem}

 \begin{proof}
This is an application of the usual Skorohod representation theorem, in addition to a generalization of the Skorohod representation theorem to include the weak convergence of Carath\'{e}odory functions of random variables with bounded probability moments (see Theorems 2.8.1 and Corollary 2.8.3 in \cite{BFH18}). The numerology in \eqref{caratheodoryconv} is due to the result of Theorem 2.8.1 and Corollary 2.8.3 in \cite{BFH18} and the uniform bounds of $\tilde{\mathbb{E}}\left(\|\hat\rho_\delta\|_{L^{\infty}(0, T; L^{\gamma}(\mathcal{O}_{\alpha}))}\right)^{p}$, $\tilde{\mathbb{E}}\left(\|\mathbbm{1}_{\mathcal{O}_{\hat{\eta}^{*}_{\delta}}} \hat{\bu}_{\delta}\|_{L^{2}(0, T; L^{q}(\mathcal{O}_{\alpha}))}\right)^{p}$, $q<6$, and $\tilde{\mathbb{E}}\left(\|\mathbbm{1}_{\mathcal{O}_{\hat{\eta}^{*}_{\delta}}} \nabla \hat\bu_\delta\|_{L^{2}(0, T; L^{2}(\mathcal{O}_{\alpha}))}\right)^{p}$ for all $1 \le p < \infty$, independently of $\delta$.
 \end{proof}


Furthermore, we have that the following equality holds almost surely,
\begin{align}
\hat v_\delta = \partial_t\hat\eta_\delta \quad \text{ in the sense of distributions}.
\end{align}
Note in particular that,
\begin{equation}\label{etastarconv}
\hat\eta_{\delta} \to \eta \text{ and } \hat\eta^{*}_{\delta} \to \eta^{*} \quad \text{ in } C_{w}(0, T; H^2(\Gamma)) \cap C(0, T; H^{s}(\Gamma)), \text{ for } s \in (3/2, 2),
\end{equation}
implies that,
\begin{align}\label{etauniform}
    \hat\eta_{\delta} \to \eta \text{ and } \hat\eta^{*}_{\delta} \to \eta^{*} \quad \text{ in } C(0, T; C(\Gamma)),
\end{align}
which further implies that
\begin{align}\label{etasequal2}
			{\eta}^*(t)={\eta}(t) \quad \text{ for any } t<\tau^{\eta}, \quad \tilde\bP\text{-almost surely.}
		\end{align}	
   where for a given $\alpha$,
\begin{align}\label{tau}
    \tau^{\eta}  := T\wedge\inf\left\{t>0:\inf_{\Gamma}(1+\eta(t))\leq \alpha \text{ or } \|{\eta}(t)\|_{H^s(\Gamma)}\geq \frac1{\alpha} \right\}.
\end{align}
The proof of \eqref{etasequal2} is identical to that of \eqref{etasequal1}.

Thanks to Lemma 5.7 in \cite{TC23} we conclude that $\tau^\eta$ is almost surely strictly positive. In particular, given that the deterministic initial data $\eta_0$ satisfies \eqref{etainitial}, we have
\begin{align}\label{positivetau}
    \tilde\bP[\tau^\eta = 0] = 0.
\end{align}
This gives us the stopping time which is part of our martingale solution; see Definition \ref{def:martingale}. We will show that the candidate solution $(\rho,\bu,\eta)$ constructed in Theorem \ref{skorohod} satisfies the continuity equation \eqref{renorm_cont} in the renormalized sense and the momentum equation \eqref{weaksol} in the weak sense until this stopping time $\tau^\eta$.

Recalling the definition of the history of random variables $\sigma_t$  \eqref{history}, we will next construct a filtration to which these new random variables are adapted. As done in previous sections,
letting $\hat\sF_t'$ be the $\sigma-$ field generated by the random variables $\eta(s),\hat W^1(s),\hat W^2(s)$ for all $0\leq s \leq t$, and $\mathcal{N}=\{\mathcal{A}\in \tilde{\mathcal{F}} |  \tilde\bP(\mathcal{A})=0\}$, we define
\begin{align}\label{Ft1_delta}
\hat{\mathcal{F}}^0_t:=\bigcap_{s\ge t}\sigma(\hat{\sF}_s' \cup \mathcal{N}),\qquad
\hat{\mathcal{F}}_t :=\sigma(\sigma_t(\bu) \cup \sigma_t(v)\cup \hat{\mathcal{F}}^0_t).
\end{align} 
This is a complete, right-continuous filtration $(\hat\sF_t)_{t \geq 0}$, on the probability space $(\tilde\Omega,\tilde\sF,\tilde\bP)$, to which the limiting noise processes and the solutions are adapted. The filtration $(\hat\sF^\delta_t)_{t \geq 0}$ for the approximation weak formulation is constructed similarly.


Next, we know that the weak formulation \eqref{delta} is also satisfied by the new random variables $\hat{\mathcal{U}}_\delta$, i.e., for any pair of {{deterministic}} test functions $(\bq,\psi) \in C_c^\infty(\sO_\alpha)\times C^\infty(\Gamma)$ the following equation,
 \begin{multline}\label{newdelta}
\int_{\mathcal{O}_{\alpha}} \hat\rho_{\delta}(t) \hat{\bu}_{\delta}(t) \cdot \bd{q} + \int_{\Gamma} \hat v_{\delta}(t) \psi = \int_{\mathcal{O}_{\alpha}}  \bp_{0,\delta} \cdot \bd{q} + \int_{\Gamma} v_{0} \psi 
+ \int_{0}^{t} \int_{\mathcal{O}_{\alpha}} (\hat\rho_{\delta} \hat{\bu}_{\delta} \otimes \hat{\bu}_{\delta}) : \nabla \bd{q} \\
+ \int_{0}^{t} \int_{\mathcal{O}_{\alpha}} \Big(a\hat\rho_{\delta}^{\gamma} + \delta \hat\rho_{\delta}^{\beta}\Big) (\nabla \cdot \bd{q}) - \int_{0}^{t} \int_{\mathcal{O}_{\alpha}} \mu^{\hat\eta^{*}_{\delta}}_\delta  \nabla\hat{ \bu}_{\delta} : \nabla \bd{q} 
- \int_{0}^{t} \int_{\mathcal{O}_{\alpha}} \lambda^{\hat\eta^{*}_{\delta}}_\delta \text{div}(\hat{\bu}_{\delta}) \text{div}(\bd{q})\\
-\frac1\delta \int_{0}^{t} \int_{{T^\delta_{\hat\eta^*_\delta}}} (\hat\bu_{\delta} - \hat v_{\delta} \bd{e}_{z}) \cdot (\bd{q} - \psi \bd{e}_{z})
- \int_{0}^{t} \int_{\Gamma} \nabla \hat v_{\delta} \cdot \nabla \psi - \int_{0}^{t} \int_{\Gamma} \nabla \hat\eta_{\delta} \cdot \nabla \psi - \int_{0}^{t} \int_{\Gamma} \Delta \hat\eta_{\delta} \Delta \psi \\
+ \int_{0}^{t} \int_{\mathcal{O}_{\alpha}} \mathbbm{1}_{\sO_{\eta^*_\delta}}\bd{F}(\hat\rho_{\delta}, \hat\rho_{\delta} \hat{\bu}_{\delta}) \cdot \bd{q} d\hat W_\delta^1(t) + \int_{0}^{t} \int_{\Gamma} G(\hat\eta_{\delta}, \hat v_{\delta}) \psi d\hat W_\delta^2(t),
\end{multline}
holds $\tilde\bP$-almost surely and for almost every $t\in[0,T]$.
Furthermore, the continuity equation reads,
\begin{align}\label{contdelta}
    \int_{\sO_\alpha}\hat\rho_\delta(t)\phi(t) = \int_{\sO_\alpha}\rho_{0,\delta}\phi(0) + \int_0^t\int_{\sO_\alpha}\hat\rho_\delta(\partial_t\phi+\hat\bu_\delta\cdot\nabla\phi),
\end{align}
$\tilde\bP$-almost surely for any {$\phi\in C^\infty((0,T);C^\infty_c(\sO_\alpha))$} and $t\in[0,T]$.

\subsection{Vanishing bounds for density outside of the moving domain}
We show an estimate that shows that the integral of the density converges to zero in the exterior region $\left(\mathcal{O}_{\hat{\eta}_{\delta}} \cup T^{\delta}_{\hat{\eta}_{\delta}}\right)^{c}$, with an explicit bound on the rate of convergence of this integral, where we recall the definition of the exterior tubular neighborhood from \eqref{tube}. Due to the nature of the proof of this estimate, the moving domain involves $\hat{\eta}_{\delta}$ rather than $\hat{\eta}_{\delta}^{*}$, so that the resulting estimate holds only up to the stopping time $\tau^{\eta}_{\delta}$ defined in \eqref{taudelta}.
This vanishing of density argument takes advantage of the continuity equation \eqref{contdelta} and is proved in the spirit of Lemma 3.1 in \cite{SarkaHeatFSI} and Lemma 4.1 in \cite{FKNNS13}.
\begin{proposition}\label{vacuum2}
For any $0<\nu_*<(\frac12-\frac1\beta)^2$ there exists a constant $c$ independent of $\delta$, such that
\begin{equation*}
\tilde\bE\sup_{t\in[0,\tau^{\eta}_\delta]}\int_{\sO^c_{\hat\eta_\delta}\cap (T^\delta_{\hat\eta_\delta})^c(t)}\hat\rho_\delta(t) \leq c\delta^{\nu_*},
\end{equation*}
where the stopping time $\tau^\eta_\delta$ is defined in \eqref{taudelta}.
 We also have the pathwise bound
\begin{equation}\label{vacuumpathwise}
\sup_{t \in [0, \tau^{\eta}_\delta]} \int_{\mathcal{O}_{\hat{\eta}_{\delta}}^{c} \cap (T^{\delta}_{\hat{\eta}_{\delta}})^{c}(t)} \hat{\rho}_{\delta}(t) \le c(\tilde{\omega}) \delta^{\nu_*},
\end{equation}
for a constant $c(\tilde\omega)$ (independent of $\delta > 0$) depending only on the outcome $\tilde\omega \in \tilde\Omega$. 
\end{proposition}
\begin{proof}
    We take $g_\delta(t,x,y,z)=z-(1+\hat\eta_\delta(t,x,y)).$ Then $g_\delta$ satisfies,
    $$\partial_t g_\delta +(\hat v_\delta\bd{e}_z)\cdot \nabla g_\delta  = 0.$$
    Now 
    let
    $$ \varphi_\delta = \left[\min\left\{\frac{g_\delta}{\delta^{(\frac12-\frac1\beta)}},1\right\}\right]^+,$$
    and we use it as a test function in \eqref{contdelta}. This yields
    \begin{align*}
\int_{\sO_\alpha}\hat\rho_\delta(t)\varphi_\delta(t) -{\rho_{0,\delta}\varphi_\delta(0)}=   \frac1{\delta^{(\frac12-\frac1\beta)}}\int_0^t\int_{T^\delta_{\hat\eta_\delta}}\hat\rho_\delta(\partial_tg_\delta+\hat\bu_\delta\cdot\nabla g_\delta),
    \end{align*}
    {{for all $t \in [0, T]$, by the definition of the tubular neighborhood $T^{\delta}_{\hat{\eta}_{\delta}}$ in \eqref{tube}.}}
Hence,
  \begin{align*}
\int_{\sO_\alpha}\hat\rho_\delta(t)\varphi_\delta(t) =  \frac1{\delta^{(\frac12-\frac1\beta)}}\int_0^t\int_{T^\delta_{\hat\eta_\delta}}\hat\rho_\delta(\hat\bu_\delta-\hat v_\delta\bd{e}_z)\cdot\nabla g_\delta,
    \end{align*}
since the integral $\displaystyle \int_{\mathcal{O}_{\alpha}} \rho_{0, \delta} \varphi_{\delta}(0)$ is equal to 0, because $\varphi_\delta(0)$ is supported outside of $\sO_{\eta_0}$ whereas $\rho_{0,\delta}$, by definition in \eqref{p0}, is supported inside of $\sO_{\eta_0}$.
    Now recall the definition \eqref{taudelta} which implies that $\tilde\bE\left\|\frac1{\sqrt\delta}(\hat\bu_\delta-\hat v_\delta\bd{e}_z) \right\|^p_{  L^2(0,\tau^{\eta}_\delta;L^2(T^\delta_{\hat{\eta}_\delta}))}\leq c$. Moreover, we have $\tilde\bE\|\delta^{\frac1{\beta}}\hat\rho_\delta\|^p_{ L^\infty(0,T;L^\beta(\sO_\alpha))} \leq c$ and $\tilde\bE\|\nabla g_\delta\|^p_{ L^\infty(0,T;L^q(\sO_\alpha))}\leq c$ for any $p,q<\infty$. Hence, we obtain for some $1<k<2$ that
{{\begin{equation}\label{expectationexterior} \tilde\bE\left\|\frac1{\delta^{(\frac12-\frac1\beta)}}\hat\rho_\delta(\hat\bu_\delta-\hat v_\delta\bd{e}_z)\cdot\nabla g_\delta \right\|_{L^2(0,\tau^{\eta}_\delta;L^k(T^\delta_{\hat\eta_\delta}))} \leq c.
\end{equation}}}
Since { $\varphi_\delta(t)=1$ on $\sO^c_{\hat\eta_\delta}\cap (T^\delta_{\hat\eta_\delta})^c(t)$ and 0 on $\sO_{\hat\eta_\delta}$}, we use {{H\"{o}lder's inequality and \eqref{expectationexterior} to}} thus obtain for some $\nu_*>0$ that
\begin{align*}
\tilde\bE\sup_{t\in[0,\tau^{\eta}_\delta]}\int_{\sO^c_{\hat\eta_\delta}\cap (T^\delta_{\hat\eta_\delta})^c(t)}\hat\rho_\delta(t)&\le  \tilde\bE\sup_{t\in[0,\tau^{\eta}_\delta]}\frac1{\delta^{(\frac12-\frac1\beta)}}\int_{T^\delta_{\hat{\eta}_\delta}(t)}\hat\rho_\delta(t)g_{\delta}+ \tilde\bE\sup_{t\in[0,\tau^{\eta}_\delta]}\int_{\sO^c_{\hat\eta_\delta}\cap (T^\delta_{\hat\eta_\delta})^c(t)}\hat\rho_\delta(t) \\
&=\tilde\bE\sup_{t\in[0,\tau^{\eta}_\delta]}\int_{\sO_\alpha}\hat\rho_\delta(t)\varphi_\delta(t) \leq c\sup_{t\in[0,T]}|T^\delta_{\hat\eta_\delta(t)}|^{1-\frac1k}\leq c\delta^{(\frac12-\frac1\beta)(1-\frac1{k})} \leq \delta^{\nu_*}. 
\end{align*}

The pathwise form of the estimate follows using exactly the same estimates, where we instead observe that by the weak convergences provided by Theorem \ref{skorohod}, we have that the norms $\|\delta^{\frac{1}{\beta}} \hat{\rho}_{\delta}\|_{L^{\infty}(0, T; L^{\beta}(\mathcal{O}_{\alpha}))}$, $\left\|\frac{1}{\sqrt{\delta}}(\hat{\bu}_\delta - \hat{v}_\delta \bd{e}_{z})\right\|_{L^{2}(0, \tau^\eta_\delta; L^{2}(T^{\delta}_{\hat{\eta}_{\delta}}))}$, and $\|\nabla g_{\delta}\|_{L^{\infty}(0, T; L^{q}(\mathcal{O}_{\alpha}))}$ for any $1 \le q < \infty$ are $\tilde{\bP}$-almost surely bounded by a constant depending on the outcome $\tilde{\omega} \in \tilde{\Omega}$. 
\end{proof}

\subsection{Weak convergence of the pressure.} Our aim is to now pass $\delta\to 0$ in this weak formulation and prove that the new random variable $\mathcal{U}$ is a candidate solution. Note that the energy estimates found in Lemma \ref{energy_delta} hold true for the new random variables as well. 
The challenge, as is the case of most compressible flow problems is to pass to the limit in the pressure term $a\hat\rho_{\delta}^{\gamma} + \delta\hat\rho_{\delta }^{\beta}$. From the uniform bounds, we only have that this quantity is bounded in $L^{p}(\tilde\Omega; L^{\infty}(0, T; L^{1}(\sO_{\alpha})))$ but the function space $L^{1}$ is not amenable to weak compactness arguments, from just uniform boundedness. The ultimate goal of this subsection is to show weak convergence of the pressures to some limiting function $\bar{p}$, see Proposition \ref{L1weakpressure}:
\begin{equation*}
a\hat{\rho}^{\gamma}_{\delta} + \delta\hat{\rho}^{\beta}_{\delta} \rightharpoonup \bar{p} \text{ weakly in } L^{1}(\tilde{\Omega}; L^{1}(0, T; L^{1}(\mathcal{O}_{\alpha}))).
\end{equation*}
This will be established if we can show equi-integrability of the pressure uniformly in the parameter $\delta$, as a result of the following classical weak $L^{1}$ convergence criterion.

\begin{lemma}[Dunford-Pettis Theorem] \label{L1compact}
	Consider a bounded family of functions $\sF$ in $L^1(\sO)$. Then it contains a subsequence that converges weakly in $L^1(\sO)$ if and only if for every $\epsilon>0$ there exists $\theta>0$ such that for every measurable set $E$, with $\mu(E)<\theta$ we have
	$$\int_E|f|<\epsilon,\qquad\forall f \in \sF.$$
\end{lemma}
Thus, the goal of this subsection will be to show that the pressure functions $a\hat{\rho}^{\gamma}_{\delta} + \delta \hat{\rho}_{\delta}^{\beta}$ satisfies the equi-integrability criterion uniformly in $\delta$. To do this, we will have three estimates on the pressure that together will imply equi-integrability. The first two estimates will be on the time-dependent random moving fluid domain $\mathcal{O}_{\hat{\eta}_{\delta}^{*}}$, and the last estimate will be on the exterior tubular neighborhood of the random moving fluid domain. We hence show the following estimates\footnote{We remark that the remaining exterior portions of the fluid domain in $\mathcal{O}_{\alpha} \setminus (\mathcal{O}_{\hat{\eta}^{*}_{\delta}} \cup T^{\delta}_{\hat{\eta}^{*}_{\delta}})$ can be handled using the same ideas in Estimates 1 and 2.}:

\medskip

\noindent \textbf{Estimate 1.} Given by Proposition \ref{interior_pressure}, this is a higher integrability estimate on the density that will imply equi-integrability of the pressure in the region within $\mathcal{O}_{\hat{\eta}^{*}_{\delta}}$, defined for $l > 0$:
    \begin{equation}\label{Alset}
    A^{l}_{\hat{\eta}^{*}_\delta}(t) = \{(x, y, z) \in 
    \mathbb{R}^3: (x,y)\in\Gamma,\ l < z < 1 + \hat{\eta}^{*}_{\delta}(t, x, y) - l\}.
    \end{equation}
\medskip

\noindent \textbf{Estimate 2.} Given by Proposition \ref{bdry_pressure}, we show that the pressure has arbitrarily small integral within $\mathcal{O}_{\hat{\eta}^{*}_{\delta}}$ sufficiently near enough to the fluid-structure boundary $\partial \mathcal{O}_{\hat{\eta}^{*}_{\delta}}$. This will be an estimate on the integral of the pressure on the interior boundary set $\sO_{\hat\eta^*_\delta}\cap (A^l_{\hat\eta^*_\delta})^c$.
\medskip

\noindent \textbf{Estimate 3.} Finally, in Proposition \ref{tube_pressure}, we obtain an estimate on the integral of the pressure in the exterior tubular neighborhood $T^{\delta}_{\hat{\eta}^{*}_{\delta}}$, and show an explicit bound on the rate of convergence to zero of the integral of the pressure on this exterior tubular neighborhood.

\medskip

The crucial estimate and fundamental step is to \textbf{improve the integrability of the pressure} in the interior of the moving domain, given in Estimate 1 in Proposition \ref{interior_pressure}. This will allow us to interpret the term $a\hat\rho_\delta^\gamma+\delta\hat\rho_\delta^\beta$ in an $L^{p}$ space with $p > 1$, which will allow us to attain interior equi-integrability of the pressure. This is traditionally done by using the Bogovski operator to construct an appropriate test function for the momentum equation which requires that the fluid domain is at least Lipschitz continuous. However, since the fluid domain in our case is not regular enough, we will follow the ideas that were first introduced, in the fixed irregular domain case, in \cite{K09} and later extended to the time-varying geometries in \cite{BreitSchwarzacherNSF}. This procedure of obtaining higher integrability of the fluid density in the interior region $A^{l}_{\hat{\eta}^{*}_{\delta}}$ is the content of the following proposition on interior pressure estimates.

\begin{proposition}[Interior pressure estimates]\label{interior_pressure}
Recall the definition of $A^{l}_{\hat{\eta}^{*}_{\delta}}$ from \eqref{Alset}. For a sufficiently small $\Theta > 0$ depending only on $\gamma$, and every $l>0$, there exists a constant $C_l$ depending only on $l$ and $\alpha$ (and independent of $\delta$) such that:
\begin{equation}\label{intpressure}
\tilde\bE\int_{0}^{T} \int_{A^{l}_{\hat\eta^{*}_{\delta}}(t)} (\hat\rho_{\delta}^{\gamma + \Theta} + \delta \hat\rho_{\delta}^{\beta + \Theta}) \le C_{l}.
\end{equation}
\end{proposition}

\begin{proof}
The idea here is to ``test" the momentum equation with 
\begin{equation}\label{qdeltaTheta}
\bq_{\delta,\Theta} = \chi_\delta\bd{p}_{\delta,\Theta},\quad\text{ such that } \nabla\cdot\bd{p}_{\delta,\Theta} = \hat\rho_\delta^\Theta,
\end{equation}
and where $\chi_\delta$ is an appropriate compactly supported function that localizes away from the structure boundary so that the penalized boundary term does not come into the picture. We note here that $\Theta$ cannot be arbitrarily large (i.e. we do not obtain arbitrarily high integrability of the density) due to low integrability of the advection term.

To that end, we will take 
$$\bd{p}_{\delta,\Theta}:= \nabla\Delta^{-1}\hat\rho_\delta^\Theta,$$
where $\Delta^{-1}\hat\rho^\Theta_\delta$ is the unique solution in $W^{2,\gamma/\Theta}(\sO_\alpha)\cap W^{1,(\gamma/\Theta)^*}_0(\sO_\alpha)$ to
$$-\Delta w =\hat\rho_\delta^\Theta,$$ where $(\gamma/\Theta)^{*}$ denotes the dual exponent to $\gamma/\Theta$ satisfying $\displaystyle \frac{1}{(\gamma/\Theta)^{*}} + \frac{\Theta}{\gamma} = 1$.

Next, we will construct a smooth random process $\chi_\delta$ on $\sO_\alpha$ such that $\chi_\delta(t)=0$ on $\sO_\alpha\setminus A^{\frac{l}4}_{\hat\eta^*_\delta}(t)$ and such that $\chi_\delta(t) = 1$ in $A^{l}_{\hat\eta^*_\delta}(t)$ for almost every $\omega\in\tilde\Omega$ and $t\in[0,T]$, where the sets $A_{\eta}^l$ are defined in \eqref{Alset}. Due to the random motion of the fluid domain, it will be helpful to explicitly construct this process and obtain appropriate bounds for time and space derivatives of $\chi_\delta$ in terms of $\hat\eta_\delta^*$.
For that purpose, we will first construct an appropriate smooth {\it deterministic} function $g$ on the fixed maximal domain $\sO_\delta$ and then squeeze in such a way that the abovementioned conditions for $\chi_\delta$ are satisfied. The squeezing operator, however, inherits the regularity of the structure displacement via the ALE map. Hence, to obtain a function $\chi_\delta$ which is smooth enough, we will also spatially regularize the squeezed function.

Hence, let $g\in C^\infty_0(\sO_\alpha)$ be  such that $g(x,y,z)=1$ when $\frac{l}{2\alpha}\leq z\leq 1-\frac{l}{2\alpha}$ 
and transitions smoothly from 0 to 1 when $\frac{l}{4\alpha}<z<\frac{l}{2\alpha}$ and $1-\frac{l}{2\alpha}<z<1-\frac{l}{4\alpha}$. 
 Then we define
 \begin{equation}
     \begin{split}
         \zeta_\delta(t, x, y, z) &= g\left(x,y,\frac{z}{(1+\hat\eta^*_\delta(t,x,y))}\right),\quad \text{ in } \sO_{\hat\eta^*_\delta}(t)\\
         &=0 ,\text{ in } \mathbb{R}^3\setminus\sO_{\hat\eta^*_\delta}(t).
     \end{split}
 \end{equation}
 Observe that, by construction, we have $\zeta_\delta(t,x,y,z) =1$ when $\frac{l}2\leq z\leq 1+\hat\eta^*_\delta(t,x,y) -\frac{1+\hat\eta^*_\delta(t,x,y)}{\alpha}\frac{l}2 < (1+\hat\eta^*_\delta(t,x,y)) -\frac{l}2$ i.e. in $A^{\frac{l}2}_{\hat\eta^*_\delta}(t)$ which contains the set $ A^{{l}}_{\hat\eta^*_\delta}(t)$. 
 
It is easy to see that $\|\zeta_\delta\|_{L^\infty(\tilde\Omega\times(0,T)\times\sO_\alpha)}\leq 1$. Moreover, since $H^1(\Gamma)\hookrightarrow L^p(\Gamma)$ for every $1\leq p<\infty$, we have
\begin{align}\label{zeta_lp}
\|\nabla\zeta_\delta\|_{L^\infty(0,T;L^p(\sO_\alpha))}\leq C(l,\alpha) \|\nabla g\|_{L^\infty(\sO_\alpha)}\|\nabla_\Gamma \hat\eta^*_\delta\|_{L^\infty(0,T; L^p(\Gamma))} \leq C\|\hat\eta^*_\delta\|_{L^\infty(0,T;H^2(\Gamma))},
\end{align}
and similarly
\begin{align}\label{zeta_t}
\|\partial_t\zeta_\delta\|_{L^2(0,T;L^p(\sO_\alpha))} \le C(\alpha)\|\partial_{z} g\|_{L^\infty(\sO_{\alpha})} \|\partial_t\hat\eta^*_\delta\|_{L^2(0,T;L^p(\Gamma))} \leq C\|\partial_t\hat\eta^*_\delta\|_{L^2(0,T;H^1(\Gamma))},
\end{align}
where the constant $C$ depends only on $\Gamma$, $\alpha$ and $l$.

Observe, that the regularity \eqref{zeta_lp} is not enough. Hence we mollify it with a standard 3D mollifier and denote this space mollification by $\chi_{\delta}$. We choose the radius of mollification {{$\sigma \ll \frac{l}8$}}, appropriately small (depending only on $\alpha$ and $l$ and not on $\delta$), so that 
$$\chi_\delta(t)=1 \text{ in } A^l_{\hat\eta^*_\delta}(t) \text{ and } \chi_\delta(t)=0\text{ in } \sO_\alpha\setminus A^{\frac{l}8}_{\hat\eta^*_\delta}(t).$$ 
Observe, due to the properties of mollification  $\|f^\lambda\|_{W^{m+k, p}} \le C\lambda^{-k}\|f\|_{W^{m,p}}$ for $k \ge 0$, and \eqref{zeta_lp} that for any $p>3$ we have
\begin{equation}\label{chi_lp}
\begin{split}
    \|\chi_\delta\|_{L^\infty(0,T;W^{1,\infty}(\sO_\alpha))}&\leq C\|\chi_\delta\|_{L^\infty(0,T;W^{2,p}(\sO_\alpha))} \leq \frac{C(\alpha)}{l}\|\zeta_\delta\|_{L^\infty(0,T;W^{1,p}(\sO_\alpha))}\\
&\leq C(l,\alpha) \|\nabla g\|_{L^\infty(\sO_\alpha)}\|\nabla_\Gamma \hat\eta^*_\delta\|_{L^\infty(0,T;L^p(\Gamma))}  \\
&\leq C(l,\alpha)\|\hat\eta^*_\delta\|_{L^\infty(0,T;H^2(\Gamma))},
\end{split}
\end{equation}
where the constant $C$ depends on (negative powers of) $l$ and $\alpha$ but is independent of $\delta$. Similarly, thanks to \eqref{zeta_t}, we can see that 
\begin{align}\label{chi_t}
\|\partial_t\chi_\delta \|_{L^2(0,T;L^p(\sO_\alpha))}\leq C \|\partial_t\zeta_\delta\|_{L^\infty(0,T;L^p(\sO_\alpha))}  \leq C\|\partial_t\hat\eta^*_\delta\|_{L^2(0,T;H^1(\Gamma))},
\end{align}
where the constant $C$ depends only on $\Gamma$ and $g$.
 Moreover, this explicit construction of the process $\chi_\delta$ ensures that it is $\{\hat\sF^\delta_t\}_{t\geq 0}$-adapted. 

We test \eqref{delta} with $\bq_{\delta,\Theta}$ i.e. we apply It\^{o}'s formula to $f_\chi( \rho,\bu)=\int_0^T\int_{\sO_\alpha}\bu\cdot  \chi_\delta\nabla\Delta^{-1}\rho^\Theta$ in the spirit of Lemma 5.1 in \cite{BO13} and take the structure test function $\bd{\psi}=0$, to obtain
\begin{align*}
		\int_0^t\int_{\sO_\alpha} \chi_\delta(a\hat\rho_\delta^{\gamma+\Theta}&+\delta\hat\rho^{\beta+\Theta}_\delta)=\int_0^t\int_{\sO_\alpha}  \chi_\delta \mu^{\hat\eta^*_\delta}_\delta \nabla \hat\bu_\delta:\nabla^2\Delta^{-1}\hat\rho_\delta^\Theta + \int_0^t\int_{\sO_\alpha} \mu^{\hat\eta^{*}_{\delta}}_\delta \nabla\hat\bu_\delta:\nabla \chi_\delta\otimes\nabla\Delta^{-1}\hat\rho^\Theta_\delta \\
		&-\int_0^t\int_{\sO_\alpha}(a\hat\rho_\delta^\gamma+\delta\hat\rho_\delta^\beta)\nabla \chi_\delta\cdot\nabla\Delta^{-1}\hat\rho^\Theta_\delta
		+\int_0^t\int_{\sO_\alpha}  \lambda^{\hat\eta^{*}_{\delta}}_\delta \text{div}\hat\bu_\delta( \chi_\delta\hat\rho^\Theta_\delta+\nabla \chi_\delta\cdot\nabla\Delta^{-1}\hat\rho_\delta^\Theta)\\
		&+\int_0^t\int_{\sO_\alpha}  \chi_\delta \hat\rho_\delta\hat\bu_\delta\otimes \hat\bu_\delta:\nabla^2\Delta^{-1}\hat\rho^\Theta_\delta
		+\int_0^t\int_{\sO_\alpha} \hat\bu_\delta\otimes \hat\bu_\delta:\nabla \chi_\delta\otimes\nabla\Delta^{-1}\hat\rho^\Theta_\delta\\
		&+\sum_k\int_0^t\int_{\sO_\alpha} \chi_\delta f_k(\hat\rho_\delta,\hat\rho_\delta\hat\bu_\delta)\cdot\nabla\Delta^{-1}\hat\rho^\Theta_\delta d\hat W_\delta^1\\
		&+{{\int_0^t\int_{\sO_\alpha} \chi_\delta\hat\rho_\delta\hat\bu_\delta\nabla\Delta^{-1}[\text{div}(\hat\rho_\delta^\Theta\hat\bu_\delta)+(\Theta-1)\hat\rho_\delta^\Theta\text{div}\hat\bu_\delta]}}+\int_0^t\int_{\sO_{\alpha}}{\partial_t\chi_\delta\hat\rho_\delta}\hat\bu_\delta\cdot\nabla\Delta^{-1}\hat\rho^\Theta_\delta.
\end{align*}
{ Here we additionally used the fact that $\partial_t\bd{p}_{\delta,\Theta} = \nabla\Delta^{-1}\partial_t\hat\rho_\delta^\Theta=\nabla\Delta^{-1}[\text{div}(\hat\rho_\delta^\Theta\hat\bu_\delta)+(\Theta-1)\hat\rho_\delta^\Theta\text{div}\hat\bu_\delta]$.}
Now, due to the positivity of the density, the left-hand side term of the equation above is an upper bound for the desired term that we want to bound i.e. the left-hand side term of \eqref{intpressure}. Hence, we will now find bounds for each term that appears on the right-hand side of the equation above by noting that the energy estimates Lemma \ref{energy_delta} hold for the new random variables $\hat{\mathcal{U}}_\delta$ found in Theorem \ref{skorohod} as well.

\noindent\textbullet{} We start with the most critical term. Depending on $\gamma>\frac32$, we choose $q<6$ and $0<\Theta <\frac{(q-2)\gamma-q}{q}$ and utilize the bounds in Lemma \ref{energy_delta}(8). The definition of $\chi_\delta$ and Lemma \ref{energy_delta} give us
	\begin{align*}
\Bigg|\tilde\bE	\int_0^T\int_{\sO_\alpha}  \chi_\delta & \hat\rho_\delta\hat\bu_\delta\otimes \hat\bu_\delta:\nabla^2\Delta^{-1}\hat\rho^\Theta_\delta \Bigg|\leq C\tilde\bE	\int_0^T\int_{\sO_{\hat\eta^*_\delta}}  \left|\hat\rho_\delta\hat\bu_\delta\otimes \hat\bu_\delta:\nabla^2\Delta^{-1}\hat\rho^\Theta_\delta\right|\\
&\leq C\tilde\bE\left|\Big(\|\hat\rho_\delta\|_{L^\infty(0,T;L^\gamma(\sO_\alpha))}\|\hat\bu_\delta\|_{L^2(0,T;L^{q}(\sO_{\hat\eta^*_\delta}))}^2\|\hat\rho^\Theta_\delta\|_{L^\infty(0,T;L^{\frac{q\gamma}{(q-2)\gamma-q}}(\sO_\alpha))}\Big)\right| \leq C,
	\end{align*}
 The term $\displaystyle \int_0^T\int_{\sO_\alpha} \chi_\delta\hat\rho_\delta\hat\bu_\delta\nabla\Delta^{-1}\text{div}(\hat\rho^\Theta_\delta\hat\bu_\delta)$ is treated identically. We remark that the assumption on the adiabatic constant $\gamma > 3/2$ is essential at this step to ensure that, $\Theta > 0$ and $q<6$ satisfying $\Theta <\frac{(q-2)\gamma-q}{q}$ can be appropriately chosen in order to improve the integrability of the fluid density.
 
\noindent\textbullet{} Next, by using \eqref{chi_lp}, 
 and taking $\Theta<\frac{\gamma}{3}$, which gives us the embedding $W^{1,\frac{\gamma}{\Theta}}(\sO_\alpha)\hookrightarrow L^\infty(\sO_\alpha)$, we obtain
 \begin{align*}
  \Bigg|   \tilde\bE\int_0^T\int_{\sO_\alpha} \mu^{\hat\eta^{*}_{\delta}}_\delta \nabla\hat\bu_\delta:\nabla &\chi_\delta\otimes\nabla\Delta^{-1}\hat\rho^\Theta_\delta \Bigg| \\
  &\leq C\tilde\bE\left( \| \mu^{\hat\eta^{*}_{\delta}}_\delta  \nabla \hat\bu_\delta\|_{L^2(0,T;L^2(\sO_\alpha))}\|\nabla\chi_\delta\|_{L^\infty(0,T;L^{2}(\sO_\alpha))}\|\nabla\Delta^{-1}\hat\rho_\delta^\Theta\|_{L^2(0,T;L^\infty(\sO_\alpha))}\right)\\
   &\leq C \tilde\bE\left( \| \mu^{\hat\eta^{*}_{\delta}}_\delta  \nabla \hat\bu_\delta\|_{L^2(0,T;L^2(\sO_\alpha))}\|\hat\eta^*_\delta\|_{L^\infty(0,T;H^{2}(\sO_\alpha))}\|\hat\rho_\delta\|_{L^\infty(0,T;L^\gamma(\sO_\alpha))}^{{\Theta}}\right) \leq C.
 \end{align*}
 
\noindent\textbullet{} Similarly, by choosing $\Theta<\frac{\gamma}2$, we have
\begin{align*}
 \Bigg| \tilde\bE\int_0^T\int_{\sO_\alpha}  \chi_\delta\mu^{\hat\eta^{*}_{\delta}}_\delta  \nabla \hat\bu_\delta:\nabla^2\Delta^{-1}\hat\rho_\delta^\Theta\Bigg| &\leq C\tilde\bE\left( \| \mu^{\hat\eta^{*}_{\delta}}_\delta  \nabla \hat\bu_\delta\|_{L^2(0,T;L^2(\sO_\alpha))}\|\nabla^2\Delta^{-1}\hat\rho_\delta^\Theta\|_{L^2(0,T;L^2(\sO_\alpha))}\right)\\
 &\leq C\tilde\bE\left( \| \mu^{\hat\eta^{*}_{\delta}}_\delta  \nabla \hat\bu_\delta\|_{L^2(0,T;L^2(\sO_\alpha))}\|\hat\rho_\delta\|_{L^\infty(0,T;L^\gamma(\sO_\alpha))}^{{\Theta}}\right) \leq C.
\end{align*}
The term $\int_0^t\int_{\sO_\alpha}  \lambda^{\hat\eta^{*}_{\delta}}_\delta \text{div}\hat\bu_\delta( \chi_\delta\hat\rho^\Theta_\delta+\nabla \chi_\delta\cdot\nabla\Delta^{-1}\hat\rho_\delta^\Theta)$ is treated identically. 

\noindent\textbullet{} Next, using \eqref{chi_lp}, we have
\begin{align*}
    &\Bigg|\tilde\bE\int_0^T\int_{\sO_\alpha}(a\hat\rho_\delta^\gamma+\delta\hat\rho_\delta^\beta)\nabla \chi_\delta\cdot\nabla\Delta^{-1}\hat\rho^\Theta_\delta\Bigg| \\
    &\leq C\tilde\bE\left[\left(\|\hat\rho_\delta\|^\gamma_{L^\infty(0,T;L^\gamma(\sO_\alpha))}+\delta\|\hat\rho_\delta\|^\beta_{L^\infty(0,T;L^\beta(\sO_\alpha))}\right)\|\chi_\delta\|_{L^\infty(0,T;W^{1,\infty}(\sO_\alpha))}\|\nabla\Delta^{-1}\hat\rho_\delta^\Theta\|_{L^\infty(0,T;L^\infty(\sO_\alpha))}\right]\\
     &\leq C\tilde\bE\left[\left(\|\hat\rho_\delta\|^\gamma_{L^\infty(0,T;L^\gamma(\sO_\alpha))}+\delta\|\hat\rho_\delta\|^\beta_{L^\infty(0,T;L^\beta(\sO_\alpha))}\right)\|\hat\eta^*_\delta\|_{L^\infty(0,T;H^2(\Gamma))}\|\hat\rho_\delta\|_{L^\infty(0,T;L^\gamma(\sO_\alpha))}^{{\Theta}}\right] \leq C(l,\alpha),
\end{align*}
for any $\Theta<\frac{\gamma}{3}$, by the continuous embedding of $W^{1, \gamma/\Theta}(\mathcal{O}_{\alpha}) \subset L^{\infty}(\mathcal{O}_{\alpha})$.

\noindent\textbullet{} The last term is treated using inequality \eqref{chi_t} for any $\Theta < \frac{\gamma}{3}$, as follows:
\begin{align*}
   & \Bigg|\tilde\bE\int_0^t\int_{\sO_{\alpha}}{\partial_t\chi_\delta\hat\rho_\delta}\hat\bu_\delta\cdot\nabla\Delta^{-1}\hat\rho^\Theta_\delta\Bigg| \\
    &\leq C\tilde\bE\left({{\|\partial_t\chi_\delta\|_{L^2(0,T;L^\infty(\sO_\alpha))}}}\|\hat\rho_\delta\hat\bu_\delta\|_{L^\infty(0,T;L^{\frac{2\gamma}{\gamma+1}}(\sO_\alpha))}\|\nabla\Delta^{-1}\hat\rho_\delta^\Theta\|_{L^\infty(0,T;L^\infty(\sO_\alpha))}\right)\\
     &\leq C\tilde\bE\left(\|\partial_t\hat\eta^*_\delta\|_{L^2(0,T;H^1(\Gamma))}\|\hat\rho_\delta\hat\bu_\delta\|_{L^\infty(0,T;L^{\frac{2\gamma}{\gamma+1}}(\sO_\alpha))}\|\hat\rho_\delta\|_{L^\infty(0,T;L^\gamma(\sO_\alpha))}^{\Theta}\right) \leq C.
\end{align*}
\item Finally, we note that, since the stochastic integral is an $(\hat\sF^\delta_t)_{t\in [0,T]}$-martingale, its expectation is zero. This concludes the proof of Lemma \ref{interior_pressure}.
\end{proof}

Next, we will need an estimate on the boundary, since we need to avoid the boundary for the interior estimates due to the badly behaved $1/\delta$ penalty term and low spatial regularity of the structure displacements.

\begin{proposition}[Boundary estimate for the pressure]\label{bdry_pressure}
 For any $\epsilon > 0$, {there exists an $l>0$} and a $\delta_{0}$ sufficiently small so that for all $\delta \in (0, \delta_{0}]$, 
\begin{equation*}
\tilde\bE\int_{0}^{T} \int_{\sO_{\hat\eta^*_\delta}\cap (A^{l}_{\hat\eta^{*}_{\delta}})^c} (a\hat\rho^{\gamma}_{\delta} + \delta \hat\rho^{\beta}_{\delta}) \le \epsilon.
\end{equation*}

\end{proposition}

\begin{proof}
For any parameter $K > 0$, we define:
{$$\phi^K_\delta=-[\min\left(K(1 + \hat\eta^{*}_{\delta}(t, x, y)-z), 1\right)]^{+} \bd{e}_{z},$$}
and we will use $(\bd{q}, \psi) = (\phi^{K}_\delta, 0)$ as test functions in the weak formulation \eqref{newdelta} which is justified thanks to Lemma 5.1 in \cite{BO13}. Note that since 
$${ \nabla \cdot \phi^{K}_{\delta} =  K }\ \text{ on }\ B^{\frac1K}_{\hat\eta^*_\delta}(\omega):= \{(x,y,z)\in\mathbb{R}^3: (x,y) \in \Gamma,\ 0<(1+\hat\eta^*_\delta(t,x,y))-z<l\},$$
we have that 
\begin{multline*}
K\int_{0}^{T} \int_{B^{\frac1K}_{\hat\eta^*_\delta}(\omega)} (a\hat\rho^{\gamma}_{\delta} + \delta \hat\rho^{\beta}_{\delta}) 
= | \int_{0}^{T} \int_{\sO_\alpha} (a\hat\rho_{\delta}^{\gamma} + \delta \hat\rho^{\beta}_{\delta})(\nabla \cdot \phi^{K}_{\delta}) | \\
 =|\int_{0}^{T} \int_{\mathcal{O}_{\alpha}} \hat\rho_{\delta} \hat\bu_{\delta} \partial_{t} \phi^{K}_{\delta} + \int_{\mathcal{O}_{\alpha}} \hat\rho_{\delta}(t) \hat\bu_{\delta}(t) \cdot \phi^{K}_{\delta}(t) - \int_{\mathcal{O}_{\alpha}}  \bp_{0,\delta} \cdot \phi^{K}_{\delta}(0) \\
- \int_{0}^{T} \int_{\mathcal{O}_{\alpha}} (\hat\rho_{\delta} \hat\bu_{\delta} \otimes \hat\bu_{\delta}) : \nabla \phi^{K}_{\delta} + \int_{0}^{T} \int_{\mathcal{O}_{\alpha}} \mu_{\delta}^{\hat\eta_{\delta}^*} \nabla \hat\bu_{\delta} : \nabla \phi^{K}_{\delta} + \int_{0}^{T} \int_{\mathcal{O}_{\alpha}} \lambda^{\hat\eta_{\delta}^*}_{\delta} \text{div}(\hat\bu_{\delta}) \text{div}(\phi^{K}_{\delta}) \\
+ { \frac1\delta \int_{0}^{T} \int_{{T^\delta_{\hat\eta^*_\delta}}} (\hat\bu_{\delta} - \hat v_{\delta} \bd{e}_{z}) \cdot (\phi_\delta^K )}+ \int_{0}^{T} \int_{\mathcal{O}_{\alpha}} \mathbbm{1}_{\sO_{\hat\eta^*_\delta}} \bd{F}(\hat\rho_{\delta} , \hat\rho_{\delta} \hat\bu_{\delta}) \cdot \phi^{K}_{\delta} d\hat W^1_\delta(t)|.
\end{multline*}
Observe that by construction,  $\phi^{K}_{\delta}|_{\sO_\alpha\setminus\sO_{{\hat\eta^*_\delta}}} = 0$ so the penalty term vanishes.
We estimate the terms on the right-hand side, based on the fact that the bounds in Lemma \ref{energy_delta} hold for the new random variables $\hat{\mathcal{U}}_\delta$ found in Theorem \ref{skorohod} as well, as follows:

\noindent\textbullet{} Outside of $B^{\frac1K}_{\hat\eta^*_\delta}$, we have that $|\nabla \cdot \phi^{K}_{\delta}| \le C$, and hence, 
we obtain the estimate:
\begin{equation*}
\left|\tilde\bE\int_{0}^{T} \int_{\mathcal{O}_{\alpha} \cap A^{\frac1K}_{\hat{\eta}^{*}_{\delta}}(\tilde\omega)^{c}} (a\hat\rho^{\gamma}_{\delta} + \delta \hat\rho^{\beta}_{\delta}) (\nabla \cdot \phi^{K}_{\delta})\right| \le C.
\end{equation*}

\noindent\textbullet{} Note that $|\partial_{t}\phi^{K}_{\delta}| \le K |\partial_{t}\hat\eta^{*}_{\delta}|$. So since \begin{equation*}
\partial_{t}\hat\eta^{*}_{\delta} \in L^{2}(\tilde\Omega; L^{2}(0, T; H_{0}^{1}(\Gamma))) \subset L^{2}(\tilde\Omega; L^{2}(0, T; L^{p}(\Gamma))), \quad \text{ for all $1 < p < \infty$},
\end{equation*}
we have that $\partial_{t} \phi^{K}_{\delta} \in L^{2}(\tilde\Omega; L^{2}(0, T; L^{p}(\mathcal{O}_{\alpha})))$, for all $1 < p < \infty$, with
\begin{equation*}
\|\partial_{t}\phi^{K}_{\delta}\|_{L^{2}(\tilde\Omega; L^{2}(0, T; L^{p}(\mathcal{O}_{\alpha})))} \le CK.
\end{equation*}
Since $\hat\rho_{\delta} \hat\bu_{\delta}$ is uniformly bounded in $L^{2}(\tilde\Omega; L^{2}(0, T; L^{\frac{2\gamma}{\gamma + 1}}(\mathcal{O}_{\alpha})))$ and because $\partial_{t}\phi^{K}_{\delta}$ is zero outside of $B^{\frac1K}_{\hat\eta^*_\delta}$, we conclude that
\begin{multline*}
\left|\tilde\bE\int_{0}^{T} \int_{\mathcal{O}_{\alpha}} \hat\rho_{\delta} \hat\bu_{\delta} \partial_{t} \phi^{K}_{\delta}\right| = \left|\tilde\bE\int_{0}^{T} \int_{B^{\frac1K}_{\hat\eta^*_\delta}(\omega)} \hat\rho_{\delta} \hat\bu_{\delta} \partial_{t} \phi^{K}_{\delta}\right| \\
\le \|\hat\rho_{\delta} \hat\bu_{\delta}\|_{L^{2}(\tilde\Omega; L^{2}(0, T; L^{\frac{2\gamma}{\gamma + 1}}(\mathcal{O}_{\alpha})))} \|1_{B^{\frac1K}_{\hat\eta^*_\delta}(\omega)}\|_{L^{\infty}(\tilde\Omega; L^{\infty}(0, T; L^{\frac{4\gamma}{\gamma - 1}}(\mathcal{O}_{\alpha})))} \|\partial_{t}\phi^{K}_{\delta}\|_{L^{2}(\tilde\Omega; L^{2}(0, T; L^{\frac{4\gamma}{\gamma - 1}}(\mathcal{O}_{\alpha})))} \\
\le CK^{1 - \frac{\gamma - 1}{4\gamma}}.
\end{multline*}

\noindent\textbullet{} To estimate the terms {{$\displaystyle \tilde{\mathbb{E}}\left|\int_{\mathcal{O}_{\alpha}} \hat\rho_{\delta}(t) \hat\bu_{\delta}(t) \cdot \phi^{K}_\delta(t)\right|$ and $\displaystyle\tilde{\mathbb{E}}\left| \int_{\mathcal{O}_{\alpha}} \bp_{0,\delta} \cdot \phi^{K}_\delta(0)\right|$}}, we use the fact that $\hat\rho_{\delta} \hat\bu_{\delta} \in C_{w}(0, T; L^{\frac{2\gamma}{\gamma + 1}}(\mathcal{O}_{\alpha}))${ $\tilde\bP$-almost surely so that $\tilde\bE\|\hat\rho_\delta\hat\bu_\delta\|^p_{L^\infty(0,T;L^{\frac{2\gamma}{2\gamma+1}}(\sO_\alpha))} \leq C$,}
and the fact that $|\phi_{K}(t, \cdot)| \le 1$ for all $t \in [0, T]$. Hence, we have the immediate estimate that these terms are bounded by a uniform constant $C$. 

\noindent\textbullet{} By definition we have
    \begin{equation}\label{nablaphiK}
        |\nabla \phi^{K}_{\delta}| \le CK(1 + |\nabla \hat\eta_{\delta}^{*}|), \qquad \text{ on $B^{\frac1K}_{\hat\eta^*_\delta}$.}
    \end{equation}
 and then we use the fact that $\nabla \hat\eta^{*}_{\delta} \in L^{2}(\tilde\Omega; L^{\infty}(0, T; H_{0}^{1}(\Gamma))) \subset L^{2}(\tilde\Omega; L^{\infty}(0, T; L^{p}(\Gamma)))$ for all {{$1 \le p < \infty$}} along with Lemma \ref{energy_delta} (8) for any $q<6$ to conclude that:
\begin{align}
\notag\tilde\bE\left|\int_{0}^{T} \int_{B^{\frac1K}_{\hat\eta^*_\delta}} (\hat\rho_{\delta} \hat\bu_{\delta} \otimes \hat\bu_{\delta}) \cdot \nabla \phi^{K}_{\delta}\right|&\le \|\hat\rho_{\delta} \hat\bu_{\delta} \otimes \hat\bu_{\delta}\|_{L^{2}(\tilde\Omega; L^1(0,T;L^{\frac{q\gamma}{q+2\gamma}}(B^{\frac1K}_{\hat\eta^*_\delta}))} \\
\cdot &\|1_{B^{\frac1K}_{\hat\eta^*_\delta}}\|_{L^{\infty}(\tilde\Omega; L^{\infty}(0, T; L^{\frac{2q\gamma}{q(\gamma-1)-2\gamma}}(\mathcal{O}_{\alpha})))} \cdot \|\nabla \phi^{K}_{\delta}\|_{L^{2}(\tilde\Omega; L^{\infty}(0, T; L^{\frac{2q\gamma}{q(\gamma-1)-2\gamma}}(\mathcal{O}_{\alpha})))} \notag\\
&\le {{CK^{1 - {\frac{q(\gamma-1)-2\gamma}{2q\gamma}}}}}.\label{schoice}
\end{align}

\noindent\textbullet{} For the next term, we do a similar estimate {{using \eqref{nablaphiK}}}, noting that 
$\mu_{\delta}^{\hat\eta_{\delta}^*}\nabla \hat\bu_{\delta}$ is uniformly bounded in $L^{2}(\tilde\Omega; L^{2}(0, T; L^{2}(\mathcal{O}_{\alpha})))$. Hence,
\begin{multline*}
\tilde\bE\left|\int_{0}^{T} \int_{B^{\frac1K}_{\hat\eta^*_\delta}(\omega)} \mu_\delta^{\hat\eta_{\delta}^*} \nabla \hat\bu_{\delta} : \nabla \phi^{K}_{\delta}\right| \le C\|{\mu_\delta^{\hat\eta_{\delta}^*}}\nabla \hat\bu_{\delta}\|_{L^{2}(\tilde\Omega; L^{2}(0, T; L^{2}(\mathcal{O}_{\alpha})))} \\
\cdot \|1_{B^{\frac1K}_{\hat\eta^*_\delta}}\|_{L^{\infty}(\tilde\Omega; L^{\infty}(0, T; L^{4}(\mathcal{O}_{\alpha})))} \cdot \|\nabla \phi^{K}_{\delta}\|_{L^{2}(\tilde\Omega; L^{2}(0, T; L^{4}(\mathcal{O}_{\alpha})))} \le {{CK^{3/4}}}.
\end{multline*}
We then have a similar estimate for the other viscosity term:
\begin{equation*}
\tilde\bE\left|\int_{0}^{T} \int_{\mathcal{O}_{\alpha}} \lambda^{\hat\eta_{\delta}^*}_{\delta} \text{div}(\hat\bu_{\delta}) \text{div}(\phi^{K}_{\delta})\right| \le {{CK^{3/4}}}.
\end{equation*} 

 \noindent\textbullet{} Finally, the expectation of the stochastic integral term is zero.
Putting together all of these estimates, we have for any $q<6$, that
\begin{equation*}
K\tilde\bE\int_{0}^{T} \int_{B^{\frac1K}_{\hat\eta^*_\delta}} (a\hat\rho_{\delta}^{\gamma} + \delta\hat\rho_{\delta}^{\beta}) \le C(K^{1 - \frac{\gamma - 1}{4\gamma}} + K^{1 - {\frac{q(\gamma-1)-2\gamma}{2q\gamma}}} + K^{3/4}).
\end{equation*}
So given $\epsilon > 0$, there exists $K$ sufficiently large such that
\begin{equation}\label{bound:Adelta}
\tilde\bE\int_{0}^{T} \int_{B^{\frac1K}_{\hat\eta^*_\delta}} (a\hat\rho^{\gamma}_{\delta} + \delta \hat\rho^{\beta}_{\delta}) \le \epsilon.
\end{equation}
{The same estimate near the bottom boundary is obtained using similar arguments on the set $\Gamma \times (0,l)$ which, combined with \eqref{bound:Adelta}, gives us the desired result.}
\if 1 = 0
{\bf ??} However, in the statement of the proposition, the spatial set over which the integral is taken is independent of $\delta$ and $\kappa$, whereas $A^{K}_{\delta}(\omega)$ depends explicitly on $\delta$ and $\kappa$. So the strategy is to integrate over $A^{K}(\omega)$ instead and use the almost sure convergence of $\hat\eta^{*}_{\delta}$ to $\eta^{*}$ in $C(0, T; C(\Gamma))$ to show that the $A^{K}_{\delta}(\omega)$ and $A^{K}(\omega)$ are ``almost the same" with arbitrarily high probability. In particular, we estimate that
\begin{equation}\label{symdiff}
\tilde\bE\int_{0}^{T} \int_{A^{K}} (a\rho_{\delta}^{\gamma} + \delta \rho_{\delta}^{\beta})
\le \tilde\bE\int_{0}^{T} \int_{A^{K}_{\delta}} (a\hat\rho^{\gamma}_{\delta} + \delta \hat\rho^{\beta}_{\delta}) + \tilde\bE\int_{0}^{T} \int_{A^{K} \Delta A^{K}_{\delta}} (a\hat\rho^{\gamma}_{\delta} + \delta \hat\rho^{\beta}_{\delta}),
\end{equation}
where $A^{K} \Delta A^{K}_{\delta} = (A^{K} \cap (A^{K}_{\delta})^{c}) \cup ((A^{K})^{c} \cap A^{K}_{\delta})$ denotes the symmetric difference. As previously established earlier in this proof, we can make the first term on the right-hand side arbitrarily small by choosing $K$ sufficiently large, and we will fix a $K$ so that this first term on the right-hand side is less than $\theta/2$ for some given $\theta > 0$.  

It remains to bound the second term on the right-hand side of \eqref{symdiff}. To estimate the term involving $A^{K} \Delta A^{K}_{\delta}$ we want to combine the following two facts: \begin{itemize}
\item For sufficiently small $\delta, \kappa$, $A^{K} \Delta A^{K}_{\delta}$ has small measure and does not include $\Gamma_{\delta}^{*}$ with high probability by the almost sure convergence of $\eta_{\delta}^*$ to $\eta^*$ in $C(0, T; C(\Gamma))$. 
\item The interior estimates where for all $Q \subset (\mathcal{O}_{\alpha} \cap \Gamma_{\eta_{\delta}^*}^{c}) \times [0, T]$, we have that for some $\Theta > 0$:
\begin{equation*}
    \tilde\bE\int_{Q} (a\rho^{\gamma +     \Theta}_{\delta} + \delta \rho^{\beta + \Theta}_{\delta}) \le C.
\end{equation*}
\end{itemize}
To make this argument rigorous, consider the set $\mathcal{O}_{\alpha} \cap A^{2K}_{\delta}(t, \omega)^{c}$ and note that uniformly in $\delta$, we have that
\begin{equation*}
\tilde\bE\int_{0}^{T} \int_{\mathcal{O}_{\alpha} \cap A^{2K}_{\delta}(t, \omega)^{c}} (a\rho^{\gamma + \frac{\gamma}{\beta} \Theta}_{\delta} + \delta \rho^{\beta + \Theta}_{\delta}) \le C_{0}.
\end{equation*}
Then, we estimate from \eqref{symdiff} that
\begin{equation*}
\tilde\bE\int_{0}^{T} \int_{A^{K}} (a\hat\rho^{\gamma}_{\delta} + \delta \hat\rho^{\beta}_{\delta}) \le \frac{\theta}{2} + \left(\tilde\bE\left|\int_{0}^{T} \int_{A^{K} \Delta A^{K}_{\delta}} 1 \right|\right)^{\frac{\Theta}{\beta + \Theta}} \left(\tilde\bE\int_{0}^{T} \int_{A^{K} \Delta A^{K}_{\delta}} (a\rho^{\gamma + \frac{\gamma}{\beta}\Theta}_{\delta} + \delta \rho^{\beta + \Theta}_{\delta})\right)^{\frac{\beta}{\beta + \Theta}}.
\end{equation*}
Since $\eta_{\delta}^* \to \eta^*$ in $C(0, T; C(\Gamma))$ almost surely, given $\beta > 0$, we can choose $\nu$ sufficiently small so that:
\begin{equation*}
\mathbb{P}\left(\|\eta_{\delta}^* - \eta^*\|_{C(0, T; C(\Gamma))} < \alpha_{0} \text{ for all } \delta, \kappa < \nu \right) > 1 - \beta,
\end{equation*}
where $\alpha_{0}$ is chosen so that $\displaystyle \alpha_{0} < \frac{1}{2K}$ and $\displaystyle (2\alpha_{0})^{\frac{\Theta}{\beta + \Theta}} C_{0}^{\frac{\beta}{\beta + \Theta}} \le \frac{\theta}{2}$. Then, defining the event
\begin{equation*}
E_{\theta, \beta} := \{\|\eta_{\delta}^* - \eta^*\|_{C(0, T; C(\Gamma))} < \alpha_{0} \text{ for all } \delta, \kappa < \nu\},
\end{equation*}
we have that on $E_{\theta, \beta}$,
\begin{equation*}
\int_{0}^{T} \int_{A^{K} \Delta A^{K}_{\delta}} \le 2\alpha_{0}
\end{equation*}
and $A^{K} \Delta A^{K}_{\delta} \subset \mathcal{O}_{\alpha} \cap A^{2K}_{\delta}(t, \omega)^{c}$ since $\alpha_{0}$ was chosen so that $\alpha_{0} < 1/(2K)$. Hence, we conclude that
\begin{equation*}
\tilde\bE\left(1_{E_{\theta, \beta}} \int_{0}^{T} \int_{A^{K}} (a\hat\rho^{\gamma}_{\delta} + \delta \hat\rho^{\beta}_{\delta})\right) \le \frac{\theta}{2} + (2\alpha_{0})^{\frac{\Theta}{\beta + \Theta}} C_{0}^{\frac{\beta}{\beta + \Theta}} = \theta,
\end{equation*}
which establishes the desired result.
\fi 
\end{proof}
The two propositions above give us equi-integrability of the pressure in the moving domain.
Next we will bound the pressure in the small tubular neighborhood $T^\delta_{\hat\eta^*_\delta}$ of $\hat\eta^*_\delta$.
{ \begin{proposition}\label{tube_pressure}
    For some $\nu_1>0$ there exists a constant $C$ independent of $\delta$, such that
\begin{equation}
\tilde\bE\int_{0}^{T} \int_{T^{\delta}_{\hat\eta^*_\delta}} (a\hat\rho^{\gamma}_{\delta} + \delta \hat\rho^{\beta}_{\delta}) \le C\delta^{\nu_1}.
\end{equation}
Thus,
\begin{equation}\label{l1conv_tube}
\mathbbm{1}_{T^{\delta}_{\hat\eta^*_\delta}} (a\hat\rho^{\gamma}_{\delta} + \delta \hat\rho^{\beta}_{\delta}) \to 0\qquad\text{ in } L^1(\tilde\Omega\times(0,T)\times\sO_\alpha)).
\end{equation}
\end{proposition}} 
\begin{proof}
We define $$\phi_\delta=\left[\min\left(\frac1\delta(z-(1 + \hat\eta^{*}_{\delta}(t, x, y))), 1\right) \right]^+ \bd{e}_{z}.$$
     We will now use $(\bd{q}, \psi) = (\phi_{\delta}, 0)$ as test functions in the weak formulation \eqref{newdelta} which is justified thanks to Lemma 5.1 in \cite{BO13}. Note that $\nabla \cdot \phi_{\delta} = \frac1\delta$ on the random set $S_\delta:=\{(x,y,z):0\leq z-(1+\hat\eta^*_\delta(x,y)) \leq \delta\}$ (which contains $T^\delta_{\hat\eta^*_\delta}$) and 0 otherwise, almost surely. We hence have that 
\begin{multline*}
\frac1\delta\int_{0}^{T} \int_{T^\delta_{\hat\eta^*_\delta}}(a\hat\rho^{\gamma}_{\delta} + \delta \hat\rho^{\beta}_{\delta}) \leq \frac1\delta\int_{0}^{T} \int_{S_\delta}(a\hat\rho^{\gamma}_{\delta} + \delta \hat\rho^{\beta}_{\delta}) 
= \int_{0}^{T} \int_{\sO_\alpha} (a\hat\rho_{\delta}^{\gamma} + \delta \hat\rho^{\beta}_{\delta})(\nabla \cdot \phi_{\delta})= - \int_{\mathcal{O}_{\alpha}}  \bp_{0,\delta} \cdot \phi_{\delta}(0) \\
 + \int_{\mathcal{O}_{\alpha}} \hat\rho_{\delta}(t) \hat\bu_{\delta}(t) \cdot \phi_{\delta}(t)  +\int_{0}^{T} \int_{\mathcal{O}_{\alpha}} \hat\rho_{\delta} \hat\bu_{\delta} \partial_{t} \phi_{\delta} 
- \int_{0}^{T} \int_{\mathcal{O}_{\alpha}} (\hat\rho_{\delta} \hat\bu_{\delta} \otimes \hat\bu_{\delta}) : \nabla \phi_{\delta} + \int_{0}^{T} \int_{\mathcal{O}_{\alpha}} \hat\mu_{\delta}^{\hat\eta_{\delta}^*} \nabla \hat\bu_{\delta} : \nabla \phi_{\delta}  \\
+ \int_{0}^{T} \int_{\mathcal{O}_{\alpha}} \hat\lambda^{\hat\eta_{\delta}^*}_{\delta} \text{div}(\hat\bu_{\delta}) \text{div}(\phi_{\delta})+ { \frac1\delta \int_{0}^{T} \int_{{T^\delta_{\hat\eta^*_\delta}}} (\hat\bu_{\delta} - \hat v_{\delta} \bd{e}_{z}) \cdot \phi_\delta }+ \int_{0}^{T} \int_{\mathcal{O}_{\alpha}} \mathbbm{1}_{\sO_{\hat\eta^*_\delta}}\bd{F}(\hat\rho_{\delta} , \hat\rho_{\delta} \hat\bu_{\delta}) \cdot \phi_{\delta} d\hat W^1_\delta(t).
\end{multline*}
All the terms here are treated as in Proposition \ref{bdry_pressure}. The only difference is the appearance of the penalty term for which we observe that,
$$ \frac1\delta \tilde\bE\int_{0}^{T} \int_{{T^\delta_{\hat\eta^*_\delta}}} (\hat\bu_{\delta} - \hat v_{\delta} \bd{e}_{z}) \cdot \phi_\delta  \leq \frac1{\delta}\tilde\bE\|\hat\bu_{\delta} - \hat v_{\delta} \bd{e}_{z}\|_{L^2(0,T;L^2(T^\delta_{\hat\eta^*_\delta}))}\|\phi_\delta\|_{L^{2}(0, T; L^{2}(\mathcal{O}_{\alpha}))} \leq  \frac{C}{\sqrt\delta}.$$
\end{proof}

Now, we can combine the results of Propositions \ref{interior_pressure}, \ref{bdry_pressure}, and \ref{tube_pressure} to obtain the main result of this section.

\begin{proposition}\label{L1weakpressure}
We conclude that for a limiting pressure function $\overline{p}$:
\begin{equation*}
a\hat{\rho}^{\delta}_{\delta} + \delta \hat{\rho}^{\beta}_{\delta} \rightharpoonup \overline{p} \text{ weakly in } L^{1}(\tilde\Omega; L^{1}(0, T; L^{1}(\mathcal{O}_{\alpha}))).
\end{equation*}
\end{proposition}
\begin{proof}
Observe that Propositions \ref{interior_pressure} and \ref{bdry_pressure} gives us the desired equi-integrability of $\mathbbm{1}_{\sO_{\hat\eta^*_\delta}}(\hat\rho^\gamma_\delta+\delta\hat\rho_\delta^\beta)$ in $L^1(\tilde\Omega\times(0,T)\times\sO_\alpha)$. Indeed, for any $\epsilon>0$ there exists {{a $\Theta > 0$ and an $l>0$}} such that the pressure estimates, near the moving boundary, found in Proposition \ref{bdry_pressure} hold. Then we take any measurable set $E\subset \tilde\Omega\times[0,T]\times\sO_\alpha$ such that $|E|<(\frac1{C_l}\epsilon)^{r_\Theta}$ where $C_l$ is the deterministic constant appearing in Proposition \ref{interior_pressure} and $1<r_\Theta<\infty$ is an appropriately chosen power that appears in the following calculations due to the H\"older inequality. We then obtain,
 \begin{align}
\notag \int_{E}\mathbbm{1}_{\sO_{\hat\eta^*_\delta}}|a\hat\rho^{\gamma}_\delta + \delta\hat\rho^\beta_\delta | & =  \tilde\bE\int_0^T\int_{ \sO_{\hat\eta^*_\delta}(t)\cap (A^l_{\hat{\eta}^*_\delta}(t))^c}\mathbbm{1}_E|a\hat\rho^{\gamma}_\delta+\delta\hat\rho^\beta_\delta | +	\tilde\bE\int_0^T\int_{ A^l_{\hat\eta^*_\delta}(t)}\mathbbm{1}_E|a\hat\rho^{\gamma}_\delta+\delta\hat\rho^\beta_\delta | \\
\notag &\leq \epsilon+ \left( \tilde\bE\int_0^T\int_{ A^l_{\hat\eta^*_\delta}(t)}\mathbbm{1}_E|a\hat\rho^{\gamma+\Theta}_\delta+\delta\hat\rho^{\beta+\Theta}_\delta |\right) ^{1-\frac1{r_\Theta}}|E|^{\frac1{r_\Theta}}\\
 & \leq \epsilon+ C_l |E|^{\frac1{r_\Theta}} \leq 2\epsilon. \label{equiint}
 \end{align}
 
 By excluding the tubular neighborhood $T^\delta_{\hat\eta^*_\delta}$ (thus the penalty term) and using the same interior and exterior pressure arguments as in Propositions \ref{interior_pressure} and \ref{bdry_pressure} we can extend this result to ${\sO_\alpha\setminus(\sO_{\hat\eta^*_\delta}\cup T^\delta_{\hat\eta^*_\delta})}$.
  Moreover, due to Proposition \ref{tube_pressure} we have that $\mathbbm{1}_{T^\delta_{\hat\eta^*_\delta}}(a\rho_\delta^\gamma+\delta\rho_\delta^\beta) \to 0$ in $L^1(\tilde\Omega\times(0,T)\times\sO_\alpha)$ which implies equi-integrability of $\mathbbm{1}_{T^\delta_{\hat\eta^*_\delta}}(a\rho_\delta^\gamma+\delta\rho_\delta^\beta)$.
  Hence, we conclude that,
\begin{align}\label{conv_weakpressure}
a\hat\rho^\gamma_\delta + \delta\hat\rho^\beta_\delta \rightharpoonup \bar p \text{ weakly in } L^1(\tilde\Omega;L^1(0,T;L^1(\sO_{\alpha}))).
\end{align}
\end{proof}

\subsection{Passage to the limit as $\delta \to 0$ in the weak formulation.} Our next aim is to pass $\delta \to 0$ in \eqref{delta}. Before explicitly passing to the limit in the terms in the weak formulation \eqref{delta}, we first show that the kinematic coupling condition holds for the limiting fluid/structure functions $(\bd{u}, v)$ in the limit $\delta\to0$. {{This is nontrivial, because the kinematic coupling condition does not necessarily hold for the approximate solutions $(\hat{\bd{u}}_{\delta}, \hat{v}_{\delta})$ due to the presence of the penalty term in the weak formulation. However, we will recover the kinematic coupling condition for the limiting functions as a consequence of the uniform bound given in \eqref{penaltybound}, which controls the difference between $\hat{\bd{u}_{\delta}}$ and $\hat v_{\delta} \bd{e}_{z}$ in the exterior neighborhood $T^{\delta}_{\hat\eta^{*}_{\delta}}$ which formally ``shrinks" to $\Gamma_{\eta^{*}}$ in the limit as $\delta \to 0$. This is the content of the following proposition:}}
\begin{proposition}
The kinematic coupling condition $\bd{u}|_{\Gamma_{\eta^*}} = v\bd{e}_{z}$ is satisfied $\tilde\bP$-almost surely.
\end{proposition}
\begin{proof}
First, we compute the difference between the fluid velocity and its trace along the moving interface over the exterior tubular neighborhood defined in \eqref{tube}:
\begin{align*}
\tilde\bE\int_{0}^{T} \int_{T^{\delta}_{\hat\eta^{*}_{\delta}}} \left|\hat{\bu}_{\delta} - \hat{\bu}_{\delta}|_{\Gamma_{\hat\eta^{*}_{\delta}}}\right|^{2} &= \tilde\bE\int_{0}^{T} \int_{\Gamma} \int_{\hat\eta^{*}_{\delta}}^{\hat\eta^{*}_{\delta} + \delta^{\left(\frac{1}{2} - \frac{1}{\beta}\right)}} \Big|\hat{\bu}_{\delta}(t, x, y, z) - \hat{\bu}_{\delta}(t, x, y, \hat\eta^{*}_{\delta}(t, x, y))\Big|^{2} dz dy dx dt \\
&=\tilde\bE \int_{0}^{T} \int_{\Gamma} \int_{\hat\eta^{*}_{\delta}}^{\hat\eta^{*}_{\delta} + \delta^{\left(\frac{1}{2} - \frac{1}{\beta}\right)}} \left|\int_{\hat\eta^{*}_{\delta}}^{z} \partial_{z} \hat{\bu}_{\delta}(t, x, y, w) dw \right|^{2} dz dy dx dt \\
&\le \tilde\bE\left(\delta^{2 \left(\frac{1}{2} - \frac{1}{\beta}\right)} \int_{0}^{T} \int_{\Gamma} \int_{\hat\eta^{*}_{\delta}}^{\hat\eta^{*}_{\delta} + \delta^{\left(\frac{1}{2} - \frac{1}{\beta}\right)}} \int_{\hat\eta^{*}_{\delta}}^{\hat\eta^{*}_{\delta} + \delta^{\left(\frac{1}{2} - \frac{1}{\beta}\right)}} |\partial_{z}\hat{\bu}_{\delta}(t, x, y, w)|^{2} dw dz dy dx dt\right) \\
&\le \delta^{3\left(\frac{1}{2} - \frac{1}{\beta}\right)} \tilde\bE\|\partial_{z}\hat{\bu}_{\delta}\|^{2}_{L^{2}(0, T; L^{2}(T^{\delta}_{\hat\eta^{*}_{\delta}}))}\leq C \delta^{3\left(\frac{1}{2} - \frac{1}{\beta}\right)}.
\end{align*}
We then combine this estimate with the uniform bound \eqref{penaltybound} on the difference between the fluid and structure velocity in the exterior tubular neighborhood to deduce:
\begin{multline*}
\tilde\bE\int_{0}^{T} \int_{\Gamma} \left|\hat{\bu}_{\delta}|_{\Gamma_{\hat\eta^{*}_{\delta}}} - \hat{v}_{\delta} \bd{e}_{z}\right|^{2} \le \frac{C_{r}}{\delta^{\left(\frac{1}{2} - \frac{1}{\beta}\right)}} \left(\tilde\bE\int_{0}^{T} \int_{T^{\delta}_{\hat\eta^{*}_{\delta}}} \left|\hat{\bu}_{\delta} - \hat{\bu}_{\delta}|_{\Gamma_{\hat\eta^{*}_{\delta}}}\right|^{2} + \tilde\bE\int_{0}^{T} \int_{T^{\delta}_{\hat\eta^{*}_{\delta}}} |\hat{\bu}_{\delta} - \hat{v}_{\delta} \bd{e}_{z}|^{2}\right) \\
\le C\left[\delta^{2\left(\frac{1}{2} - \frac{1}{\beta}\right)} + \delta^{-\left(\frac{1}{2} - \frac{1}{\beta}\right)} \left(\tilde\bE\int_{0}^{T} \int_{T^{\delta}_{\hat\eta^{*}_{\delta}}} |\hat{\bu}_{\delta} - \hat{v}_{\delta} \bd{e}_{z}|^{2}\right)\right] \le C\left(\delta^{2\left(\frac{1}{2} - \frac{1}{\beta}\right)} + \delta^{\left(\frac{1}{2} + \frac{1}{\beta}\right)}\right),
\end{multline*}
so we conclude that
\begin{equation}\label{qnoslip}
\tilde\bE\int_{0}^{T} \int_{\Gamma} \left|\hat{\bu}_{\delta}|_{\Gamma_{\hat\eta^{*}_{\delta}}} - \hat{v}_{\delta} \bd{e}_{z}\right|^{2} \to 0, \quad \text{ as } \delta \to 0 .
\end{equation}

The proof would be complete if we can show an appropriate convergence of the left-hand side of \eqref{qnoslip} to a quantity involving the limiting $\bd{u}|_{\Gamma_{\eta^{*}}}$ and $v \bd{e}_{z}$ as $\delta \to 0$. For the structure velocities, we recall that $\hat{v}_{\delta} \to \hat{v}$ strongly $\tilde\bP$-almost surely in $L^{2}(0, T; L^{2}(\Gamma))$. It remains to consider the convergence of the traces of the fluid velocities.

We claim that we have a similar weak convergence of the traces along the moving interface $\hat{\bu}_{\delta}|_{\Gamma_{\hat{\eta}^{*}_{\delta}}}$ to $\hat{\bu}|_{\Gamma_{\eta}}$ for the fluid velocities. To see this, we note that the following function $w$ is well-defined for each $\delta$ on $[0, T] \times \Gamma \times [0, \alpha]$, by the properties of the stopping of the process depending on the parameter $\alpha$:
\begin{equation*}
\bd{w}_{\delta}(t, x, y, z) = \hat{\bd{u}}_\delta(t, x, y, 1 + \hat\eta^{*}_{\delta}(t, x, y) - z).
\end{equation*}
Note that because $\nabla \bd{u}$ is uniformly bounded in $L^{p}(\Omega; L^{2}([0, T] \times \mathcal{O}_{\hat\eta^{*}_{\delta}}))$ for any $1 \le p < \infty$,  the traces $\hat{\bu}_{\delta}|_{\Gamma_{\hat\eta^{*}_{\delta}}} := \bd{w}_{\delta}|_{\Gamma \times \{0\}}$ are well-defined  and, for any $q<2$, 
\begin{align}\label{tracecontinuity}
        \|\bd{w}_\delta(t)|_{z=0}\|_{W^{1-\frac1q,q}(\Gamma)} \leq C_{tr}\|\bd{w}_\delta(t)\|_{W^{1,q}(\Gamma\times[0,\alpha])} \leq C(\alpha)\|\hat\bu_\delta(t)\|_{H^1(\sO_{\hat\eta^*_\delta})},
    \end{align}
    where $C_{tr}$ is the constant appearing in the usual trace theorem that depends only on $\alpha, \Gamma$.

Moreover, note, due to the strong uniform convergence of $\hat\eta^*_\delta$ and weak convergence of $\mathbbm{1}_{\sO_{\hat\eta^*_\delta}}\hat\bu_\delta$ in $L^{2}(0, T; L^{6}(\mathcal{O}_{\alpha}))$ due to Theorem \ref{skorohod}, that
\begin{align*}
    \int_0^T\int_{\Gamma\times[0,\alpha]}\bd{w}_\delta\bq = \int_0^T\int_{\sO_\alpha}\mathbbm{1}_{S^\alpha_{\hat\eta^*_\delta}}\hat\bu_{\delta}\bq \to \int_0^T\int_{\sO_\alpha}\mathbbm{1}_{S^\alpha_{\eta^*}}\bu\bq = \int_0^T\int_{\Gamma\times[0,\alpha]}\bd{w}\bq.
\end{align*}
where $S^\alpha_{\eta}:=\sO_{\eta}\setminus\sO_{\eta-\alpha}$ and $\bd{w}(t,x,y,z)=\bu(t,x,y,1+\eta^*(t,x,y)-z)$. Additionally, the $\delta$-independent almost sure bound \eqref{tracecontinuity}, implies that
$\bd{w}_\delta \rightharpoonup \bd{w}$ weakly in $L^2(0,T;W^{1,q}(\Gamma\times [0,\alpha]))$ almost surely. Hence $\hat{\bu}_\delta|_{\Gamma_{\hat\eta^{*}_{\delta}}}$ converges weakly $\tilde\bP$-almost surely in $L^{q}([0, T] \times \Gamma)$ for $1 \le q < 2$, to $\bd{w}|_{z=0}=\bu|_{\Gamma_{\eta^*}}$.

So combining this weak convergence and the $\tilde\bP$-almost sure strong convergence of $\hat{v}_{\delta} \to v$ in $L^{2}(0, T; L^{2}(\Gamma))$ with \eqref{qnoslip}, we obtain using weak lower semicontinuity that for $1 \le q < 2$:
\begin{equation*}
\tilde\bE\int_{0}^{T} \int_{\Gamma} \Big|\bu|_{\Gamma_{\eta^*}} - v \bd{e}_{z}\Big|^{q} = 0,
\end{equation*}
so that $\bu|_{\Gamma_{\eta^*}} = v\bd{e}_{z}$, $\tilde\bP$-almost surely.
\end{proof}

Finally, we comment on how to pass to the limit in the terms in the weak formulation \eqref{newdelta} for the new random variables on the probability space. The only involved terms in the weak formulation \eqref{newdelta} for the limit passage $\delta \to 0$ are the advection term $\displaystyle \int_{0}^{t} \int_{\mathcal{O}_{\alpha}} (\hat{\rho}_{\delta} \hat{\bu}_{\delta} \otimes \hat\bu_{\delta}) : \nabla \bd{q}$ and the stochastic integrals. We begin with showing convergence of the advection term. A primary difficulty that is involved in both convergences of the advection term and the stochastic integral is that we must estimate the contribution of $\hat{\rho}_{\delta}$ outside of the approximate domain $\mathcal{O}_{\hat{\eta}^{*}_{\delta}}$, since the spatial integrals are over the entire maximal domain $\mathcal{O}_{\alpha}$. 

\medskip

\noindent \textbf{Limit passage for the advection term.} For the limit of the advection term, we will need to use the $(\hat\sF_t)_{t \geq 0}$-stopping time $\tau^{\eta}$ and the $(\hat\sF^\delta_t)_{t \geq 0}$-stopping time $\tau^{\eta}_{\delta}$ because the estimate Proposition \ref{vacuum2} only holds up until the stopping time $\tau^{\eta}_{\delta}$ defined in \eqref{taudelta}. We want to pass to the limit in the weak formulation almost surely, for almost every $t \in [0, \tau^{\eta}]$, where $\tau^{\eta}$ is the stopping time corresponding to the \textit{limiting structure displacement}. So consider any $t \in [0, \tau^{\eta}]$. By properties of stopping times, we have that
\begin{equation*}
\tau^{\eta} \le \liminf_{\delta \to 0} \tau^{\eta}_{\delta}, \quad \tilde{\bP}\text{-almost surely.}
\end{equation*}

So given any (deterministic) time $t \in [0, \tau^{\eta}(\tilde\omega)]$ and a specific outcome $\tilde\omega \in \tilde\Omega$, we can find a subsequence $\delta_{m} \to 0$ (which depends on the outcome $\tilde\omega \in \tilde\Omega$) such that
\begin{equation}\label{liminf}
t \le \tau^{\eta}_{\delta_{m}}(\tilde{\omega}), \quad \text{ for all } m.
\end{equation}
We then want to show that given a sequence of $(\hat{\sF}^\delta_{t})_{t\in[0,T]}$-adapted, essentially bounded, smooth processes $\bd{q}_{\delta}$ that converge $\tilde\bP$-almost surely and strongly to $\bd{q}$ in $C(0, T; C^{k}(\mathcal{O}_{\alpha}))$ for any positive integer $k$, we have that:
\begin{equation}\label{advectionlimit}
\int_{0}^{t} \int_{\mathcal{O}_{\alpha}} (\hat{\rho}_{\delta_m} \hat{\bu}_{\delta_m} \otimes \hat{\bu}_{\delta_m}) : \nabla \bd{q}_{\delta_m} \to \int_{0}^{t} \int_{\mathcal{O}_{\eta}} \rho \bu \otimes \bu : \nabla \bd{q},
\end{equation}
for the specific (deterministic) choice of { $t \in [0, \tau^{\eta}(\tilde\omega)]$}, the specific outcome $\tilde\omega$, and the corresponding (random) subsequence $\delta_m \to 0$. 

To do this, we separate the integral into two contributions, an integral over the physical domain with the exterior tubular neighborhood included and an integral over the remainder of the maximal domain:
\begin{multline}\label{advectiondecomposition}
\int_{0}^{t} \int_{\mathcal{O}_{\alpha}} (\hat{\rho}_{\delta_m} \hat{\bu}_{\delta_m} \otimes \hat{\bu}_{\delta_m}) : \nabla \bd{q}_{\delta_m} \\
= \int_{0}^{t} \int_{\mathcal{O}_{\hat{\eta}_{\delta_m}^{*}} \cup T^{\delta_m}_{\hat{\eta}_{\delta_m}^{*}}} \hat{\rho}_{\delta_m} \hat{\bu}_{\delta} \otimes \hat{\bu}_{\delta_m} : \nabla \bd{q}_{\delta_m} + \int_{0}^{t} \int_{(\mathcal{O}_{\hat{\eta}_{\delta_m}^{*}} \cup T^{\delta_m}_{\hat{\eta}_{\delta_m}^{*}})^{c}} \hat{\rho}_{\delta_m} \hat{\bu}_{\delta_m} \otimes \hat{\bu}_{\delta_m} : \nabla \bd{q}_{\delta_m}.
\end{multline}

\noindent \underline{\textit{The exterior integral.}} We first consider the contribution on the exterior $(\mathcal{O}_{\hat{\eta}^{*}_{\delta}} \cup T^{\delta}_{\hat{\eta}^{*}_{\delta}})^{c}$ and omit the $m$ index on the subsequence to simplify notation, where we emphasize that the (random) subsequence of $\delta_{m}$ although not explicit in the notation, is still an important component of the proof due to \eqref{liminf}. The goal is to show that for this fixed outcome $\tilde\omega \in \tilde\Omega$ and fixed $t \in [0, \tau^{\eta}(\tilde\omega)]$, we have that
\begin{equation}\label{exteriorlimitzero}
\int_{0}^{t} \int_{(\mathcal{O}_{\hat{\eta}^{*}_{\delta}} \cup T^{\delta}_{\hat{\eta}^{*}_{\delta}})^{c}} \hat{\rho}_{\delta} \hat{\bu}_{\delta} \otimes \hat{\bu}_{\delta} : \nabla \bd{q}_{\delta} \to 0.
\end{equation}
Because the random test functions we construct (in the next Section \ref{sec:testfunction}) will satisfy the convergence $\bd{q}_{\delta} \to \bd{q}$ in $C(0, T; C^{k}(\mathcal{O}_{\alpha}))$ for any positive integer $k$, strongly and $\tilde{\bP}$-almost surely, we have that $\|\nabla \bd{q}_{\delta}(\tilde{\omega})\|_{L^{\infty}(\mathcal{O}_{\alpha})}$ is uniformly bounded in $\delta$. For this fixed but arbitrary $\tilde\omega \in \tilde\Omega$ and $t \in [0, \tau^{\eta}(\tilde\omega)]$, we estimate:
\begin{multline*}
\left|\int_{0}^{t} \int_{(\mathcal{O}_{\hat{\eta}^{*}_{\delta}} \cup T^{\delta}_{\hat{\eta}^{*}_{\delta}})^{c}} \hat{\rho}_{\delta} \hat{\bu}_{\delta} \otimes \hat{\bu}_{\delta} : \nabla \bd{q}_{\delta}\right| \\
\le C(\tilde\omega) \|\sqrt{\hat{\rho}_{\delta}}\|_{L^{\infty}(0, T; L^{r}((\mathcal{O}_{\hat{\eta}^{*}_{\delta}} \cup T^{\delta}_{\hat{\eta}^{*}_{\delta}})^{c})} \|\sqrt{\hat{\rho}_{\delta}} \hat{\bu}_{\delta}\|_{L^{\infty}(0, T; L^{2}(\mathcal{O}_{\alpha}))} \|\hat{\bu}_{\delta}\|_{L^{2}(0, T; L^{q}(\sO_\alpha))},
\end{multline*}
for some $q<6$ and $3<r<2\gamma$ so that $\frac1r+\frac1q=\frac12$.

Hence, \eqref{exteriorlimitzero} will follow once we make the following observations:
\begin{itemize}
    \item First, $\|\sqrt{\hat{\rho}_{\delta}} \hat{\bu}_{\delta}(\tilde\omega)\|_{L^{\infty}(0, T; L^{2}(\mathcal{O}_{\alpha}))}$ is uniformly bounded in $\delta$. 
     \item We use the ideas employed in Proposition \ref{vacuum2} and interpolation to conclude that for some $0 < \theta < 1$, depending only on $r$ and $\gamma$, and  any $t\in[0,\tau^\eta_\delta(\tilde\omega)]$:
    \begin{equation}
    \begin{split}\label{nu0}
    & \|\sqrt{\hat{\rho}}_{\delta}\|_{L^{r}((\mathcal{O}_{\hat{\eta}^{*}_{\delta}} \cup T^{\delta}_{\hat{\eta}^{*}_{\delta}})^{c})} =  \|\hat{\rho}_{\delta}\|^\frac12_{L^{\frac{r}2}((\mathcal{O}_{\hat{\eta}^{*}_{\delta}} \cup T^{\delta}_{\hat{\eta}^{*}_{\delta}})^{c})} \\&\le \|\hat{\rho}_{\delta}\|^{\frac{\theta}2}_{L^{1}((\mathcal{O}_{\hat{\eta}^{*}_{\delta}} \cup T^{\delta}_{\hat{\eta}^{*}_{\delta}})^{c}))} \|\hat{\rho}_{\delta}\|_{L^{\gamma}((\mathcal{O}_{\hat{\eta}^{*}_{\delta}} \cup T^{\delta}_{\hat{\eta}^{*}_{\delta}})^{c})}^{\frac{1 - \theta}2}\le C(\tilde\omega)\delta^{\frac{\theta\nu_*}2}=: C(\tilde\omega)\delta^{\nu_0}.
    \end{split}
    \end{equation}
    Here we define $\nu_0=\frac{\theta\nu_*}2$ where $
    \nu_*$ is any constant that satisfies the condition in Proposition \ref{vacuum2} and $\theta$ is determined by the interpolation inequality \eqref{nu0}.
    Here, because of the condition \eqref{liminf} on the subsequence $\delta_m$ we have chosen, we can apply Proposition \ref{vacuum2}, since this lemma only holds for times up to the stopping time $\tau^{\eta}_{\delta}$. Note that here, it is essential that we are applying the result about the $L^{1}$ integrability of the pressure in the exterior domain in Proposition \ref{vacuum2} pathwise, as in \eqref{vacuumpathwise}. 
    \item 
    Thanks to \eqref{mulowerbound} and Theorem \ref{skorohod} we have
    \begin{align}
    \delta^{\nu_0/2} \|\hat{\bu}_{\delta}\|_{L^2(0,T;L^{q}(\mathcal{O}_\alpha))}  &\leq   \delta^{\nu_0/2} \|\nabla\hat{\bu}_{\delta}\|_{L^2(0,T;L^{2}(\mathcal{O}_\alpha))}\leq \|\sqrt{\mu^{\hat\eta^*_\delta}_\delta}\nabla\hat{\bu}_{\delta}\|_{L^2(0,T;L^{2}(\sO_\alpha))}\le C(\tilde\omega).
    \end{align}
\end{itemize}
Then, we get the desired convergence \eqref{exteriorlimitzero}.

\medskip

\noindent \underline{\textit{The interior integral.}} Next, we show that for the (random) subsequence $\delta_{m}$ satisfying \eqref{liminf} for our fixed but arbitrary $t \in [0, \tau^{\eta}(\tilde\omega)]$ and $\tilde\omega \in \tilde\Omega$, and for a sequence of (random) test functions for which $\bd{q}_{\delta}(\tilde\omega) \to \bd{q}(\tilde\omega)$ in $C(0, T; C^{k}(\mathcal{O}_{\alpha}))$ for all nonnegative integers $k$, we have that:
\begin{equation}\label{interioradvection}
\int_{0}^{t} \int_{{{\mathcal{O}_{\hat{\eta}^{*}_{\delta}} \cup T^{\delta}_{\hat{\eta}_{\delta}^{*}}}}} \hat{\rho}_{\delta} \hat{\bu}_{\delta} \otimes \hat{\bu}_{\delta} : \nabla \bd{q}_{\delta} \to \int_{0}^{t} \int_{\mathcal{O}_{\eta}} \rho \bu \otimes \bu : \nabla \bd{q},
\end{equation}
where we have omitted the explicit labeling of the (random) subsequence $\delta_{m} \to 0$ satisfying \eqref{liminf} with the subscript $m$, for simplicity of notation. By the condition \eqref{liminf}, we have that $\displaystyle \int_{0}^{t \wedge \tau^{\eta}} \int_{\mathcal{O}_{\eta}} \rho \bu \otimes \bu : \nabla \bd{q} = \int_{0}^{t\wedge \tau^{\eta}} \int_{\mathcal{O}_{\eta^{*}}} \rho \bu \otimes \bu : \nabla \bd{q}$, and hence, this convergence will follow from Lemma \ref{rhousquared} (see below) and \eqref{convindicator}.

\medskip

\noindent \underline{\textit{Conclusion of advection term limit.}} Thus, combining both the exterior estimate \eqref{exteriorlimitzero} and the interior estimate \eqref{interioradvection} gives the desired convergence \eqref{advectionlimit} for an arbitrary time $t \in [0, \tau^{\eta}(\tilde\omega)]$ and for an arbitrary outcome $\tilde \omega$ in some measurable set of probability one in $\tilde\Omega$. 

\medskip

\noindent \textbf{Convergence of the stochastic integrals.} Next we will show how to treat the stochastic integral. We first begin by showing two essential convergences that will appear several times in the proof of the convergence of the stochastic integral for the fluid, which involves terms of the type $\rho_{\delta}|\hat{\bu}_{\delta}|^{2}$.

\begin{lemma}\label{rhousquared}
For some $r > 1$, we have the following $\tilde\bP$-almost sure convergences: 
\begin{align*}
\mathbbm{1}_{\mathcal{O}_{\hat{\eta}^{*}_{\delta}}} \hat{\rho}_{\delta} |\hat{\bu}_{\delta}|^{2} & \rightharpoonup \mathbbm{1}_{\mathcal{O}_{\eta^{*}}} \rho|\bu|^{2}, \quad \text{ weakly in } L^{r}((0,T) \times \mathcal{O}_{\alpha}),\\
\mathbbm{1}_{\mathcal{O}_{\hat{\eta}^{*}_{\delta}}} \hat{\rho}_{\delta} \hat{\bu}_{\delta} \otimes \hat{\bu}_{\delta} &\rightharpoonup \mathbbm{1}_{\mathcal{O}_{\eta^{*}}} \rho \bu \otimes \bu, \quad \text{ weakly in } L^{r}((0,T) \times \mathcal{O}_{\alpha}).
\end{align*}
\end{lemma}

\begin{proof}
We consider just the first convergence of $\hat{\rho}_{\delta}|\hat{\bu}_{\delta}|^{2}$ since the second convergence is analogous. Note that for some $q \in (3/2, \gamma)$, which is possible since $\gamma > 3/2$, we have that:
\begin{equation*}
\rho|\bu|^{2} \le |\rho|^{q} + |\bu|^{\frac{2q}{q - 1}},
\end{equation*}
where $\frac{2q}{q - 1} < 6$. So by considering the Carath\'{e}odory function $H(\rho, \bu) = \rho|\bu|^{2}$, we obtain from \eqref{caratheodoryconv} the following $\tilde\bP$-almost sure convergence by considering the weak limit of $H(\hat{\rho}_{\delta}, \mathbbm{1}_{\mathcal{O}_{\hat{\eta}^{*}_{\delta}}}\hat{\bu}_{\delta})$:
\begin{equation}\label{cararhou2}
\mathbbm{1}_{\mathcal{O}_{\hat{\eta}^{*}_{\delta}}} \hat{\rho}_{\delta} |\hat{\bu}_{\delta}|^{2} \rightharpoonup \overline{\rho|\bu|^{2}}, \text{ weakly in $L^{r}((0,T) \times \mathcal{O}_{\alpha})$ for some $r > 1$.}
\end{equation}
However, we have not yet identified this weak limit $\overline{\rho|\bu|^{2}}$ from \eqref{caratheodoryconv} explicitly as $\mathbbm{1}_{\mathcal{O}_{\eta^{*}}} \rho|\bu|^{2}$, and if we are able to do this, the proof would be complete.

So we will now identify the weak limit $\overline{\rho|\bu|^{2}} = \mathbbm{1}_{\mathcal{O}_{\eta^{*}}} \rho|\bu|^{2}$. By the application of the Skorohod representation theorem in Theorem \ref{skorohod}, we have the following $\tilde\bP$-almost sure convergences:
\begin{align*}
&\hat{\eta}^{*}_{\delta} \rightharpoonup \eta^{*}, \quad \text{ weakly-star in } L^{\infty}(0, T; H^{2}(\Gamma)),\\
&\partial_{t} \hat{\eta}^{*}_{\delta} \rightharpoonup \partial_{t} {\eta}^{*}, \quad \text{ weakly-star in } L^{\infty}(0, T; L^{2}(\Gamma)).
\end{align*}
In addition, we have from Theorem \ref{skorohod} and the strong convergence of 
\begin{align}\label{convindicator}
{{\mathbbm{1}_{\mathcal{O}_{\hat{\eta}^{*}_{\delta}}}\text{ and }}} \mathbbm{1}_{\left(\mathcal{O}_{\hat{\eta}^{*}_{\delta}} \cup T^{\delta}_{\hat{\eta}^{*}_{\delta}}\right)} \to \mathbbm{1}_{\mathcal{O}_{\eta^{*}}}, \quad \tilde\bP\text{-almost surely in}\quad L^{\infty}(0, T; L^{q}(\Gamma))\quad \forall 1 \le q < \infty,
\end{align}
that
\begin{align*}
&\mathbbm{1}_{\left(\mathcal{O}_{\hat{\eta}^{*}_{\delta}} \cup T^{\delta}_{\hat{\eta}^{*}_{\delta}}\right)} \hat{\bu}_{\delta} \rightharpoonup \mathbbm{1}_{\mathcal{O}_{\eta^{*}}} \bu, \quad \text{ weakly in } L^{2}(0, T; L^{r}(\mathcal{O}_{\alpha})),\\
&\mathbbm{1}_{\left(\mathcal{O}_{\hat{\eta}^{*}_{\delta}} \cup T^{\delta}_{\hat{\eta}^{*}_{\delta}}\right)} \nabla \hat{\bu}_{\delta} \rightharpoonup \mathbbm{1}_{\mathcal{O}_{\eta^{*}}} \nabla \bu, \quad \text{ weakly in } L^{2}(0, T; L^{r}(\mathcal{O}_{\alpha})),
\end{align*}
for any $1 \le r < 2$. In addition,
\begin{equation*}
\hat{\rho}_{\delta} \hat{\bu}_{\delta} \rightharpoonup \rho \bu, \quad \text{ weakly-star in } L^{\infty}(0, T; L^{\frac{2\gamma}{\gamma + 1}}(\mathcal{O}_{\alpha})),
\end{equation*}
where $\frac{2\gamma}{\gamma + 1}>6/5$ since $\gamma > 3/2$. Applying a deterministic compactness result for establishing weak convergence of products of functions defined on moving domains, given by Lemma 2.8 in \cite{BreitFSI}, to the current situation outcome by outcome for each $\tilde\omega \in \tilde\Omega$, we obtain that:
\begin{equation}\label{aubinlionsomega}
\mathbbm{1}_{\mathcal{O}_{\hat{\eta}^{*}_{\delta_m}}} \hat{\rho}_{\delta_m} |\hat{\bu}_{\delta_m}|^{2} (\tilde\omega) \rightharpoonup \mathbbm{1}_{\mathcal{O}_{\eta^{*}}} \rho |\bu|^{2}(\tilde\omega), \quad \text{ weakly in $L^{r}((0, T) \times \mathcal{O}_{\alpha})$},
\end{equation}
for some $r > 1$, along a subsequence $\delta_m \to 0$ that \textit{depends on the outcome $\tilde\omega \in \tilde\Omega$}. Even though this weak convergence in \eqref{aubinlionsomega} is along a subsequence that depends on the outcome $\tilde\omega \in \tilde\Omega$, we get the desired convergence as $\delta \to 0$ (without considering a subsequence) of $\mathbbm{1}_{\sO_{\hat{\eta}^{*}_{\delta}} }\hat{\rho}_{\delta} |\hat{\bu}_{\delta}|^{2} \rightharpoonup \mathbbm{1}_{\mathcal{O}_{\eta^{*}}} \rho |\bu|^{2}$ weakly $\tilde\bP$-almost surely in $L^{r}((0, T) \times \mathcal{O}_{\alpha})$ for some $r > 1$, by using \eqref{cararhou2} and by using uniqueness of the weak limit for each $\tilde\omega \in \tilde\Omega$ to identify $\overline{\rho|\bu|^{2}}$ as $\mathbbm{1}_{\mathcal{O}_{\eta^{*}}} \rho|\bu|^{2}$.
\end{proof}

Using Lemma \ref{rhousquared}, we can establish the following convergence, which will be useful for passing to the limit in the stochastic integral as $\delta \to 0$. 

\begin{lemma}\label{energyconvdelta}
We have the following convergence $\tilde\bP$-almost surely as $\delta \to 0$:
\begin{equation*}
\int_{0}^{T} \int_{\mathcal{O}_{\alpha}} \mathbbm{1}_{\mathcal{O}_{\hat{\eta}_{\delta}^{*}}} \hat{\rho}_{\delta} |\hat{\bu}_{\delta} - \bu|^{2} \to 0.
\end{equation*}
\end{lemma}

\medskip

\if 1 = 0
Before proving this lemma, we observe that $\hat{\bu}_{\delta}$ is not uniformly bounded in $L^{2}(0, T; H^{1}(\mathcal{O}_{\alpha}))$. Rather, we only have that $\mathbbm{1}_{\mathcal{O}_{\hat{\eta}_{\delta}^{*}}} \bu$ and $\mathbbm{1}_{\mathcal{O}_{\hat{\eta}_{\delta}^{*}}} \nabla \bu$ are bounded in $L^{2}(0, T; L^{2}(\mathcal{O}_{\alpha}))$ $\tilde\bP$-almost surely. However, to prove this lemma, it will be helpful to have functions $\hat{\bu}_{\delta, \text{ext}}$ which are in actually bounded in $L^{2}(0, T; H^{s}(\mathcal{O}_{\alpha}))$ $\tilde\bP$ almost surely for $0 \le s < 1$. To do this, we define
\begin{equation*}
\hat{\bu}_{\delta, \text{ext}}(t, x, y, z) = \hat{\bu}_{\delta}(t, x, y, z) + \left[\min\left(\frac{z - \hat{\eta}^{*}_{\delta}(t, x, y)}{\delta^{\frac{1}{2} - \frac{1}{\beta}}}, 1 \right)\right]^{+} \Big(\hat{v}_{\delta}(t, x, y) \bd{e}_{z} - \hat{\bu}_{\delta}(t, x, y, z)\Big)
\end{equation*}
which keeps the same values as $\hat{\bu}_{\delta}$ on $\mathcal{O}_{\hat{\eta}_{\delta}^{*}}$, and then interpolates between $\hat{\bu}_{\delta}$ and $\hat{v}_{\delta} \bd{e}_{z}$ in the $z$ direction on the tubular neighborhood $T^{\delta}_{\hat{\eta}^{*}_{\delta}}$ of width $\delta^{\left(\frac{1}{2} - \frac{1}{\beta}\right)}$, which we recall from \eqref{tube} is defined as
\begin{equation*}
T^{\delta}_{\eta} := \left\{(t, x, y, z) \in [0, T] \times \mathcal{O}_{\alpha} \setminus \mathcal{O}_{\eta} : 0 < (z - 1  - \eta) < \delta^{\left(\frac{1}{2} - \frac{1}{\beta}\right)}\right\},
\end{equation*}
and then takes the value $\hat{v}_{\delta} \bd{e}_{z}$ on $\mathcal{O}_{\alpha} \setminus (\mathcal{O}_{\hat{\eta}^{*}_{\delta}} \cup T^{\delta}_{\hat{\eta}^{*}_{\delta}}$). It is easy to check that $\hat{\bu}_{\delta, \text{ext}}$ is uniformly bounded in $L^{2}(0, T; W^{1, r}(\mathcal{O}_{\alpha}))$ norm for all $1 \le r < 2$, $\tilde\bP$ almost surely (by a constant depending on the outcome), by observing the following facts:
\begin{itemize}
    \item $\|\hat{\bu}_{\delta}\|_{L^{2}(0, T; H^{1}(\mathcal{O}_{\hat{\eta}^{*}_{\delta}}))}$ is uniformly bounded in $\delta$, $\tilde\bP$-almost surely.
    \item $\|\hat{v}_{\delta} \bd{e}_{z}\|_{L^{2}(0, T; H^{1}(\Gamma))}$ is uniformly bounded in $\delta$, $\tilde\bP$-almost surely.
    \item By \eqref{penaltybound} and equivalence of laws, $\tilde{\mathbb{E}} \|\hat{\bu}_{\delta} - \hat{v}_{\delta} \bd{e}_{z}\|^{p}_{L^{2}(0, T; L^{2}(T^{\delta}_{\hat{\eta}^{*}_{\delta}})} \le C\delta^{p/2}$, so we can observe that
    \begin{equation*}
    \frac{1}{\delta^{\left(\frac{1}{2} - \frac{1}{\beta}\right)}} \|\hat{\bu}_{\delta} - \hat{v}_{\delta} \bd{e}_{z})\|_{L^{2}(0, T; L^{2}(T^{\delta}_{\hat{\eta}^{*}_{\delta}}))} \to 0,
    \end{equation*}
    $\tilde\bP$-almost surely. Furthermore, $\nabla \hat{\eta}^{*}_{\delta}(t, x, y)$ is bounded $\tilde\bP$-almost surely in $L^{\infty}(0, T; L^{q}(\Gamma))$ for all $1 \le q < \infty$ (uniformly in $\delta$) by Sobolev embedding. These facts together help establish that $\|\nabla \hat{\bu}_{\delta, \text{ext}}\|_{L^{2}(0, T; L^{r}(\mathcal{O}_{\alpha}))}$ is uniformly bounded for $1 \le r < 2$, by estimating the contribution of the gradient of the interpolation term in the definition of $\hat{\bu}_{\delta, \text{ext}}$.
\end{itemize}
So $\hat{\bu}_{\delta, \text{ext}}$ are functions that are uniformly bounded $\tilde\bP$-almost surely in $L^{2}(0, T; W^{1, r}(\mathcal{O}_{\alpha})$ for $1 \le r < 2$ and hence in $L^{2}(0, T; H^{s}(\mathcal{O}_{\alpha}))$ for any $0 \le s < 1$, and furthermore, $\hat{\bu}_{\delta, \text{ext}}|_{\mathcal{O}_{\hat{\eta}_{\delta}^{*}}} = \hat{\bu}_{\delta}$.

Now we prove Lemma \ref{energyconvdelta}.

\fi
\begin{proof}
We calculate that
\begin{align}\label{energyconvdeltaeq}
\int_{0}^{T} \int_{\mathcal{O}_{\alpha}} \mathbbm{1}_{\sO_{\hat\eta^*_\delta}}\hat{\rho}_{\delta}|\hat{\bd{u}}_{\delta} - \bd{u}|^{2}& = \int_{0}^{T} \int_{\mathcal{O}_{\alpha}}\mathbbm{1}_{\sO_{\hat\eta^*_\delta}} \Big(\hat{\rho}_{\delta} |\hat{\bd{u}}_{\delta}|^{2} - \rho |\bd{u}|^{2}\Big) +  \int_{0}^{T} \int_{\mathcal{O}_{\alpha}}\mathbbm{1}_{\sO_{\eta^*}} (\hat{\rho}_{\delta} - \rho) |\bd{u}|^{2} \nonumber \\
&+ \int_{0}^{T} \int_{\mathcal{O}_{\alpha}} \Big(\mathbbm{1}_{\mathcal{O}_{\hat{\eta}^{*}_{\delta}}} - \mathbbm{1}_{\mathcal{O}_{\eta^{*}}}\Big) (\hat{\rho}_{\delta} - \rho) |\bu|^{2} + 2 \int_{0}^{T} \int_{\mathcal{O}_{\alpha}} \mathbbm{1}_{\sO_{\hat\eta^*_\delta}}(\rho \bd{u} - \hat{\rho}_{\delta} \hat{\bd{u}}_{\delta}) \cdot \bd{u},
\end{align}
and we show that each of the terms on the right-hand side of \eqref{energyconvdeltaeq} converges to $0$, $\tilde\bP$-almost surely.

\medskip

\noindent \textbf{Term 1.} We estimate the first term on the right-hand side of \eqref{energyconvdeltaeq} by noting that 
\begin{align}
\left|\int_0^{T} \int_{\sO_{\alpha}} \mathbbm{1}_{\sO_{\hat\eta^*_\delta}}(\hat\rho_\delta|\hat\bu_\delta|^2-\rho|\bu|^2) \right|&\leq \left|\int_0^{T}\int_{\sO_\alpha} \mathbbm{1}_{\mathcal{O}_{\eta^{*}}} (\mathbbm{1}_{\mathcal{O}_{\hat{\eta}^{*}_{\delta}}} \hat{\rho}_{\delta} |\hat{\bu}_{\delta}|^{2} - \mathbbm{1}_{\mathcal{O}_{\eta^{*}}} \rho |\bu|^{2})\right|\notag \\
& + \left|\int_0^{T} \int_{\mathcal{O}_{\alpha}}\mathbbm{1}_{\mathcal{O}_{\hat{\eta}^{*}_{\delta}}}(\mathbbm{1}_{\mathcal{O}_{\eta^{*}}} - \mathbbm{1}_{\mathcal{O}_{\hat{\eta}^{*}_{\delta}}}) \hat\rho_\delta|\hat\bu_\delta|^{2}\right| + \left|\int_{0}^{T} \int_{\mathcal{O}_{\alpha}} (\mathbbm{1}_{\mathcal{O}_{\eta^{*}}} - \mathbbm{1}_{\mathcal{O}_{\hat{\eta}^{*}_{\delta}}}) \rho|\bu|^{2}\right|.\label{term1deltaestimate}
\end{align}

By Lemma \ref{rhousquared}, $\mathbbm{1}_{\mathcal{O}_{\hat{\eta}_{\delta}^{*}}} \hat{\rho}_{\delta} |\hat{\bu}_{\delta}|^{2} \rightharpoonup \mathbbm{1}_{\mathcal{O}_{\eta^{*}}} \rho |\bu|^{2}$ weakly in $L^{r}((0, T) \times \mathcal{O}_{\alpha})$ for some $r > 1$, $\tilde\bP$-almost surely, so since $\mathbbm{1}_{\mathcal{O}_{\eta^{*}}} \in L^{q}((0, T) \times \mathcal{O}_{\alpha})$ $\tilde\bP$-almost surely, for all $1 \le q < \infty$, we have that $\displaystyle \left|\int_0^{T}\int_{\sO_\alpha} \mathbbm{1}_{\mathcal{O}_{\eta^{*}}} (\mathbbm{1}_{\mathcal{O}_{\hat{\eta}^{*}_{\delta}}} \hat{\rho}_{\delta} |\hat{\bu}_{\delta}|^{2} - \mathbbm{1}_{\mathcal{O}_{\eta^{*}}} \rho |\bu|^{2})\right| \to 0$ as $\delta \to 0$, $\tilde\bP$-almost surely.

Combining \eqref{convindicator} with the $\tilde\bP$-almost sure boundedness of $\hat{\rho}_{\delta}|\hat\bu_{\delta}|^{2}$ and $\rho|\bu|^{2}$ in $L^{\infty}(0, T; L^{\frac{2\gamma}{\gamma + 1}}(\mathcal{O}_{\alpha}))$ independently of $\delta$ (but potentially depending on the outcome $\tilde\omega \in \tilde\Omega$), we obtain that the remaining terms on the right-hand side of \eqref{term1deltaestimate} converge to zero $\tilde\bP$-almost surely as $\delta \to 0$. This shows that Term 1 on the right-hand side of \eqref{energyconvdeltaeq} converges to $0$, $\tilde\bP$-almost surely.

\medskip

\noindent \textbf{Term 2.} Since $\hat{\rho}_{\delta} \to \rho$ weakly-star in $L^{\infty}(0, T; L^{\gamma}(\mathcal{O}_{\alpha}))$ and $\mathbbm{1}_{\mathcal{O}_{\eta^{*}}} |\bu|^{2} \in L^{1}(0, T; L^{3}(\mathcal{O}_{\alpha}))$, $\tilde\bP$-almost surely, we have that
\begin{equation*}
\int_{0}^{T} \int_{\mathcal{O}_{\alpha}} \mathbbm{1}_{\mathcal{O}_{\eta^{*}}} (\hat{\rho}_{\delta} - \rho) |\bu|^{2} \to 0, \quad \tilde\bP\text{-almost surely.}
\end{equation*}

\medskip

\noindent \textbf{Term 3.} Using the fact that $\hat{\rho}_{\delta}$ is uniformly bounded in $L^{\infty}(0, T; L^{\gamma}(\mathcal{O}_{\alpha}))$ $\tilde\bP$-almost surely, $|\bu|^{2} \in L^{1}(0, T; L^{3}(\mathcal{O}_{\alpha}))$ $\tilde\bP$-almost surely, and \eqref{convindicator}, we have that
\begin{equation*}
\int_{0}^{T} \int_{\mathcal{O}_{\alpha}} \Big(\mathbbm{1}_{\mathcal{O}_{\hat{\eta}^{*}_{\delta}}} - \mathbbm{1}_{\mathcal{O}_{\eta^{*}}}\Big) (\hat{\rho}_{\delta} - \rho)|\bu|^{2} \to 0, \quad \text{$\tilde\bP$-almost surely.}
\end{equation*}

\medskip

\noindent \textbf{Term 4.} By using the $\tilde\bP$-almost sure convergence $\hat\rho_{\delta} \hat\bu_{\delta} \rightharpoonup \rho \bd{u}$ weakly-star in $L^\infty(0, T; L^{\frac{2\gamma}{\gamma + 1}}(\mathcal{O}_{\alpha}))$ and the fact that $\mathbbm{1}_{\mathcal{O}_{\eta^{*}}} \bd{u} \in L^{2}(0, T; L^{6}(\mathcal{O}_{\alpha}))$ $\tilde\bP$-almost surely, we have 
\begin{equation*}
\int_{0}^{T} \int_{\mathcal{O}_{\alpha}} \mathbbm{1}_{\mathcal{O}_{\eta^{*}}} (\rho \bu - \hat\rho_{\delta} \hat\bu_{\delta}) \cdot \bd{u} \to 0, \quad \text{$\tilde\bP$-almost surely.}
\end{equation*}
Since uniform bounds combined with \eqref{convindicator} imply that 
\begin{equation*}
\int_{0}^{T} \int_{\mathcal{O}_{\alpha}} \Big(\mathbbm{1}_{\mathcal{O}_{\eta^{*}}} - \mathbbm{1}_{\mathcal{O}_{\hat{\eta}^{*}_{\delta}}}\Big) (\rho \bu - \hat{\rho}_{\delta} \hat{\bu}_{\delta}) \cdot \bu \to 0, \quad \text{$\tilde\bP$-almost surely},
\end{equation*}
we thus obtain the desired convergence that
\begin{equation*}
\int_{0}^{T} \int_{\mathcal{O}_{\alpha}} \mathbbm{1}_{\mathcal{O}_{\hat{\eta}^{*}_{\delta}}} (\rho \bu - \hat{\rho}_{\delta} \hat{\bu}_{\delta}) \cdot \bu \to 0, \quad \text{$\tilde\bP$-almost surely.}
\end{equation*}

\medskip

\noindent \textbf{Conclusion of the proof.} Since each of the four terms on the right-hand side of \eqref{energyconvdeltaeq} converges to zero $\tilde\bP$-almost surely, this completes the proof of the lemma. 
\end{proof}

Before showing convergence of the stochastic integrals, we show two technical lemmas, regarding properties of compositions of the fluid densities with Lipschitz functions $b$.

\begin{lemma}\label{outsidefluid}
There exists $p_{0} > 1$ such that for every Lipschitz function $b: \R \to \R$ with $b(0) = 0$ and $b'(w) = 0$ for $|w|$ sufficiently large and such that for all $1 \le p < p_{0}$:
\begin{equation*}
\mathbbm{1}_{\mathcal{O}_{\hat{\eta}^{*}_{\delta}}^{c}} b(\hat{\rho}_{\delta}) \hat{u}_{\delta} \to 0, \qquad \mathbbm{1}_{\mathcal{O}_{{\hat{\eta}^{*}_{\delta}}}^{c}} (b'(\hat{\rho}_{\delta})\hat{\rho}_{\delta} - b(\hat{\rho}_{\delta})) (\nabla \cdot \hat{\bu}_{\delta}) \to 0,
\end{equation*}
both strongly in $L^{p}([0, T] \times \mathcal{O}_{\alpha})$.
\end{lemma}

\begin{proof}
Consider such a function $b$ and some $1 \le p < p_{0}$ where $p_{0}$ will be appropriately chosen later in the proof. By the boundedness of $b$, $|b(z)| \le Cz^{\frac{3(2 - p)}{4p}}$ and $|b'(z) z| \le Cz^{\frac{3(2-p)}{4p}}$ for all $z \ge 0$. We estimate using \eqref{nu0}:
\begin{align*}
\|\mathbbm{1}_{\mathcal{O}_{\hat{\eta}^{*}_{\delta}}^{c}} b(\hat{\rho}_{\delta}) \hat{u}_{\delta}&\|_{L^{p}([0, T] \times \mathcal{O}_{\alpha})}^{p} + \|\mathbbm{1}_{\mathcal{O}_{{\hat{\eta}^{*}_{\delta}}}^{c}} (b'(\hat{\rho}_{\delta})\hat{\rho}_{\delta} - b(\hat{\rho}_{\delta})) (\nabla \cdot \hat{\bu}_{\delta})\|_{L^{p}([0, T] \times \mathcal{O}_{\alpha})}^{p} \\
&\le C \int_{0}^{t} \int_{(\mathcal{O}_{\hat{\eta}^{*}_{\delta}} \cup T^{\delta}_{\hat\eta^{*}_{\delta}})^{c}} \hat{\rho}_{\delta}^{\frac{3(2-p)}{4}}(|\hat{\bd{u}}_{\delta}|^{p} + |\nabla \cdot \hat{\bd{u}}_{\delta}|^{p}) + C \int_{0}^{t} \int_{T^{\delta}_{\hat{\eta}^{*}_{\delta}}} \hat{\rho}_{\delta}^{\frac{3(2-p)}{4}} (|\hat{\bd{u}}_{\delta}|^{p} + |\nabla \cdot \hat{\bu}_{\delta}|^{p}) \\
&\le C \|\hat{\rho}_{\delta}\|_{L^{\infty}(0, T; L^{3/2}((\sO_{\hat{\eta}^{*}_{\delta}} \cup T^{\delta}_{\hat{\eta}^{*}_{\delta}})^{c})}^{\frac{3(2-p)}{4}} \left(\|\hat{\bu}_\delta\|_{L^{\infty}(0, T; L^{2}(\mathcal{O}_{\alpha}))}^p + \|\nabla \cdot \hat{\bu}_{\delta}\|_{L^{2}(0, T; L^{2}((\sO_{\hat{\eta}^{*}_{\delta}} \cup T^{\delta}_{\hat{\eta}^{*}_{\delta}})^{c}))}^p\right) \\
& \quad + C |T^{\delta}_{\hat{\eta}^{*}_{\delta}}|^{\frac{2\gamma - 3(2-p)}{4\gamma}} \|\hat{\rho}_{\delta}\|_{L^{\infty}(0, T; L^{\gamma}(\mathcal{O}_{\alpha}))}^{\frac{3(2-p)}{4}} \left(\|\hat{\bu}_\delta\|_{L^{\infty}(0, T; L^{2}(\mathcal{O}_{\alpha}))}^p + \|\nabla \cdot \hat{\bu}_\delta\|_{L^{2}(0, T; L^{2}(\sO_{\hat{\eta}^{*}_{\delta}} \cup T^{\delta}_{\hat{\eta}^{*}_{\delta}}))}^p\right) \\
&\le C_{\phi, k}(\tilde\omega) \delta^{\frac{3(2-p)\nu_0}{2}} (1 + \delta^{-\frac{\nu_0 p}{2}}) + C_{\phi, \kappa}(\tilde\omega)\delta^{\frac{2\gamma - 3(2-p)}{4\gamma}},
\end{align*}
where in the last step, we use the fact from Theorem \ref{skorohod} that $\sqrt{\lambda^{{\eta}^{*}}_{\delta}} \nabla \cdot \hat{\bu}_{\delta}$ and $\mathbbm{1}_{\mathcal{O}_{\hat{\eta}^{*}_{\delta}} \cup T^{\delta}_{\hat{\eta}^{*}_{\delta}}}$ are both $\tilde\bP$-almost surely bounded in $L^{2}(0, T;L^{2}(\mathcal{O}_{\alpha}))$ and $\lambda^{\hat\eta^{*}_{\delta}}_{\delta} \ge \delta^{\nu_0}$ on $\mathcal{O}_{\alpha}$ by \eqref{mulowerbound}. We can then choose $p_0 > 1$ sufficiently small so that
\begin{equation*}
\frac{3(2- p)\nu_0}{2} - \frac{\nu_0 p}{2} > 0 \ \ \text{ and } \ \ \frac{2\gamma - 3(2 - p)}{4\gamma} > 0
\end{equation*}
for all $1 \le p < p_0$, recalling that $\gamma > 3/2$.
\end{proof}

We can use the previous technical lemma to prove the following weak continuity result on compositions of Lipschitz functions $b$ with the fluid densities $\hat{\rho}_{\delta}$.

\begin{lemma}\label{brholemma}
    Let $b: \R \to \R$ be a $C^{1}$ function such that $b(0) = 0$ and $b'(w) = 0$ for $|w|$ sufficiently large. Then, the following convergence holds $\tilde\bP$ almost surely:
    \begin{equation*}
    \mathbbm{1}_{\mathcal{O}_{\hat{\eta}^{*}_{\delta}}} b(\hat{\rho}_\delta) \to \mathbbm{1}_{\mathcal{O}_{\eta^{*}}} \overline{b(\rho)} \quad \text{ in } C_{w}(0, T; L^{\gamma}(\mathcal{O}_{\alpha})).
    \end{equation*}
\end{lemma}

\begin{proof}
We use the renormalized weak formulation, which holds $\tilde\bP$-almost surely in the following distributional sense: 
\begin{equation}\label{renorm_delta}
\partial_{t}b(\hat{\rho}_{\delta}) + \text{div}(b(\hat{\rho}_{\delta}) \hat{\bd{u}}_{\delta}) + (b'(\hat{\rho}_{\delta})\hat{\rho}_{\delta} - b(\hat{\rho}_{\delta})) \text{div}(\hat{\bd{u}}_{\delta}) = 0,
\end{equation}
along with the weak limit properties of compositions with Carath\'{e}odory functions. Using the weak limit property stated in \eqref{caratheodoryconv}, we have the following $\tilde\bP$-almost sure convergences:
\begin{equation}
\begin{split}\label{continuitycaratheordory}
b(\hat{\rho}_{\delta}) \rightharpoonup \overline{b(\rho)} \text{ weakly in } L^{\gamma}([0, T] \times \mathcal{O}_{\alpha})&, \quad 
\mathbbm{1}_{\mathcal{O}_{\eta_{\delta}^{*}}} b(\hat{\rho}_{\delta}) \hat{\bd{u}}_{\delta} \rightharpoonup \overline{b(\rho) \bd{u}} \text{ weakly in } L^{\min(\gamma, 6)}([0, T] \times \mathcal{O}_{\alpha}),\\
\mathbbm{1}_{\mathcal{O}_{\hat{\eta}_{\delta}^{*}}} b'(\hat{\rho}_{\delta})\hat{\rho}_{\delta} \text{div}(\hat{\bd{u}}_{\delta}) &\rightharpoonup \overline{b'(\rho)\rho \text{div}(\bd{u})} \text{ weakly in } L^{\min(\gamma, 2)}([0, T] \times \mathcal{O}_{\alpha}),
\\
\mathbbm{1}_{\mathcal{O}_{\hat{\eta}^{*}_{\delta}}} b(\hat{\rho}_{\delta})\text{div}(\hat{\bd{u}}_{\delta}) &\rightharpoonup \overline{b(\rho)\bd{u}} \text{ weakly in } L^{\min(\gamma, 2)}([0, T] \times \mathcal{O}_{\alpha}).
\end{split}
\end{equation}

Fix a choice of an arbitrary deterministic test function $\varphi \in C_{c}^{\infty}(\mathcal{O}_{\alpha})$, and fix an arbitrary outcome $\tilde\omega \in \tilde\Omega$ such that the renormalized weak formulation and the weak convergences in \eqref{continuitycaratheordory} hold. The proof will be complete if we can show that
\begin{equation}\label{brhogoal}
\int_{\mathcal{O}_{\alpha}} \mathbbm{1}_{\mathcal{O}_{\hat{\eta}^{*}_{\delta}}} b(\hat{\rho}_{\delta}) \varphi \to \int_{\mathcal{O}_{\alpha}} \mathbbm{1}_{\mathcal{O}_{\eta^*}} \overline{b(\rho)} \varphi \quad \text{ in } C(0, T; \R). 
\end{equation}
We once again emphasize that for the remainder of the argument, the outcome $\tilde\omega$ is \textit{fixed but arbitrary} in a measurable set of probability one.

\noindent \textbf{Step 1.} \textit{Show that $b(\hat{\rho}_{\delta}) \to \overline{b(\rho)}$ in $C_{w}(0, T; L^{\gamma}(\sO_{\alpha}))$.} By the renormalized weak formulation, for any fixed deterministic test function $\varphi \in C_{c}^{\infty}(\sO_{\alpha})$:
\begin{equation}\label{continuitybrho}
\begin{split}
\int_{\mathcal{O}_{\alpha}} b(\hat{\rho}_{\delta})(t) \varphi &= \int_{\mathcal{O}_{\alpha}}  b(\hat{\rho}_{\delta})(0) \varphi + \int_{0}^{t} \int_{\mathcal{O}_{\alpha}} \mathbbm{1}_{\mathcal{O}_{\hat{\eta}^{*}_{\delta}}} b(\hat{\rho}_{\delta}) \hat\bu_{\delta} \cdot \nabla \varphi + \int_{0}^{t} \int_{\mathcal{O}_{\alpha}} \mathbbm{1}_{\mathcal{O}_{\hat{\eta}^{*}_{\delta}}^{c}} b(\hat{\rho})\hat{\bu}_{\delta} \cdot \nabla \varphi \\
&+ \int_{0}^{t} \int_{\mathcal{O}_{\alpha}} \mathbbm{1}_{\mathcal{O}_{\hat{\eta}^{*}_{\delta}}} (b(\hat{\rho}_{\delta}) - b'(\hat{\rho}_{\delta})\hat{\rho}_{\delta}) \text{div}(\hat{\bd{u}}_{\delta}) \varphi + \int_{0}^{t} \int_{\mathcal{O}_{\alpha}} \mathbbm{1}_{\mathcal{O}_{\hat{\eta}^{*}_{\delta}}^{c}} (b(\hat{\rho}_{\delta}) - b'(\hat{\rho}_{\delta})\hat{\rho}_{\delta})\text{div}(\hat{\bu}_{\delta})\varphi,
\end{split}
\end{equation}
where we divide the integrals into the parts inside and outside $\mathcal{O}^{\hat{\eta}^{*}_{\delta}}$. Note that by Lemma \ref{outsidefluid}, the two integrals in \eqref{continuitybrho} involving $\mathbbm{1}_{\sO_{\hat{\eta}^{*}_{\delta}}^{c}}$ converge to $0$, $\tilde\bP$-almost surely as $\delta \to 0$. Combined with the weak convergences in \eqref{continuitycaratheordory}, we conclude from \eqref{continuitybrho} that $\displaystyle \int_{\mathcal{O}_{\alpha}} b(\hat{\rho}_{\delta})(t) \varphi$ converges as $\delta \to 0$ to
\begin{equation*}
\int_{\mathcal{O}_{\alpha}} b(\rho_{0}) \varphi + \int_{0}^{t} \int_{\mathcal{O}_{\alpha}} \overline{b(\rho) \bd{u}} \cdot \nabla \varphi + \int_{0}^{t} \int_{\mathcal{O}_{\alpha}} (\overline{b(\rho)\text{div}(\bd{u})} - \overline{b'(\rho)\rho\text{div}(\bd{u})}) \varphi
\end{equation*}
Since $\mathbbm{1}_{\mathcal{O}_{\hat{\eta}^{*}_{\delta}}} b(\hat{\rho}_{\delta})\hat{\bd{u}}_{\delta}$ and $\mathbbm{1}_{\mathcal{O}_{\hat{\eta}^{*}_{\delta}}} (b(\hat{\rho}_{\delta}) - b'(\hat{\rho}_{\delta})\hat{\rho}_{\delta})\text{div}(\hat{\bd{u}}_{\delta})$ are bounded in $L^{2}((0,T) \times \mathcal{O}_{\alpha})$ independently of $\delta$ by some constant $C(\tilde\omega)$ depending on $\tilde{\omega} \in \tilde{\Omega}$ by \eqref{caratheodoryconv} and since we have higher bounds for some $p > 1$ in Lemma \ref{outsidefluid} for the terms involving $\mathbbm{1}_{\sO_{\hat{\eta}^{*}_{\delta}}^{c}}$, we conclude from \eqref{continuitybrho} that the functions $\displaystyle t \to \int_{\mathcal{O}_{\alpha}} b(\hat{\rho}_{\delta})(t) \varphi$ $\tilde\bP$-almost surely are uniformly equicontinuous (with an equicontinuity parameter depending only on $\tilde{\omega} \in \tilde{\Omega}$) independently of $\delta$, for our fixed but arbitrary $\tilde\omega \in \tilde\Omega$. Hence, for our fixed $\tilde\omega$, the continuous functions $\displaystyle t \to \int_{\mathcal{O}_{\alpha}} b(\hat{\rho}_{\delta})(t) \varphi$ converge in $C(0, T; \R)$ as $\delta \to 0$ for our fixed $\tilde\omega \in \tilde\Omega$. Since by \eqref{continuitycaratheordory}, $b(\hat{\rho}_{\delta})$ converges weakly to $ \overline{b(\rho)}$ in $L^{r}((0,T) \times \mathcal{O}_{\alpha})$ for $1 < r < \gamma$, we identify the $\tilde\bP$-almost sure limit of $\displaystyle \int_{\mathcal{O}_{\alpha}} b(\hat{\rho}_{\delta})\varphi$ in $C(0, T; \R)$ as $\displaystyle \int_{\mathcal{O}_{\alpha}} \overline{b(\rho)} \varphi$, which concludes Step 1.   

\medskip

\noindent \textbf{Step 2.} \textit{Show that $\displaystyle \int_{\mathcal{O}_{\alpha}} \mathbbm{1}_{\sO_{\hat{\eta}^{*}_{\delta}}} b(\hat{\rho}_{\delta}) \varphi \to \int_{\mathcal{O}_{\alpha}} \mathbbm{1}_{\eta^{*}} \overline{b(\rho)} \varphi$ in $L^{\infty}(0, T; \R)$ for any deterministic $\varphi \in C_{c}^{\infty}(\sO_{\alpha})$.} From Step 1, we already have that $\tilde\bP$-almost surely:
\begin{equation*}
\int_{\sO_{\alpha}} b(\hat{\rho})\varphi \to \int_{\sO_{\alpha}} \overline{b(\rho)} \varphi \quad \text{ in } C(0, T; \R).
\end{equation*}
Thus, Step 2 is immediate once we observe that $\displaystyle \int_{\mathcal{O}_{\alpha}} \mathbbm{1}_{\mathcal{O}_{\hat{\eta}^{*}_{\delta}}^{c}} b(\hat{\rho}_{\delta}) \varphi \to 0$ in $L^{\infty}(0, T; \R)$, $\tilde\bP$-almost surely, by an easier variant of the same argument in Lemma \ref{outsidefluid}, and once we observe that $\overline{b(\rho)} = \mathbbm{1}_{\sO_{\eta^{*}}}\overline{b(\rho)}$. At this stage, the only remaining item left to show is that this convergence is actually in $C(0, T; \R)$ rather than just $L^{\infty}(0, T; \R)$. This will be accomplished in the final step, Step 3.

\medskip

\noindent \textbf{Step 3.} \textit{Show that $\displaystyle t \to  \int_{\sO} \mathbbm{1}_{\mathcal{O}_{\hat{\eta}^{*}_{\delta}}} b(\hat{\rho}_{\delta}(t)) \varphi$ is continuous from $[0, T] \to \R$, for each deterministic $\varphi \in C_{c}^{\infty}(\mathcal{O}_{\alpha})$.} From Step 1, we know that $\displaystyle t \to \int_{\mathcal{O}_{\alpha}} b(\hat{\rho}_{\delta})(t)\varphi$ is a real-valued continuous function on $[0, T]$, $\bP$-almost surely for any deterministic $\varphi \in C^{\infty}_{c}(\mathcal{O}_{\alpha})$. For any fixed open connected set $B \subset \mathcal{O}_{\alpha}$, we can show by approximating $\mathbbm{1}_{B}\varphi$ in the $L^{\frac{\gamma}{\gamma - 1}}(\sO_{\alpha})$ norm by smooth compactly supported functions, that
\begin{equation}\label{fixedsetcontinuity}
t \to \int_{\sO_{\alpha}} \mathbbm{1}_{B} b(\hat{\rho}_{\delta}) \varphi \quad \text{ is continuous for any fixed, open and connected $B \subset \sO_{\alpha}$.}
\end{equation}
Since $\hat{\eta}^{*}_{\delta} \in C(0, T; C(\Gamma))$, $\tilde\bP$-almost surely, we can show the desired $\tilde\bP$-almost sure continuity of $\displaystyle t \to \int_{0}^{t} \int_{\mathcal{O}_{\alpha}} \mathbbm{1}_{\sO_{\hat{\eta}^{*}_{\delta}}} b(\hat{\rho}_{\delta}) \varphi$ by using the boundedness of $b$ and writing for any fixed $t_0 \in [0, T]$:
\begin{align*}
\int_{\mathcal{O}_{\alpha}} \mathbbm{1}_{\sO_{\hat{\eta}^{*}_{\delta}(t)}} b(\hat{\rho}_{\delta})(t) \varphi &- \int_{\mathcal{O}_{\alpha}} \mathbbm{1}_{\sO_{\hat{\eta}^{*}_{\delta}(t_0)}} b(\hat{\rho}_{\delta})(t_0) \varphi \\
&= \int_{\mathcal{O}_{\alpha}} \Big(\mathbbm{1}_{\sO_{\hat{\eta}^{*}_{\delta}(t)}} - \mathbbm{1}_{\sO_{\hat{\eta}^{*}_{\delta}(t_0)}}\Big) b(\hat{\rho}_{\delta})(t) + \int_{\mathcal{O}_{\alpha}} \mathbbm{1}_{\sO_{\hat{\eta}^{*}_{\delta}(t_0)}} \Big(b(\hat{\rho}_{\delta})(t) - b(\hat{\rho}_{\delta})(t_0)\Big),
\end{align*}
where we can use the continuity property for fixed sets in \eqref{fixedsetcontinuity} to estimate the second quantity. This completes the proof of Lemma \ref{brholemma}.
\end{proof}

Now that we have the necessary convergence results from Lemma \ref{energyconvdelta} and Lemma \ref{brholemma}, we can now establish the desired convergence of the stochastic integrals for the limit passage as $\delta \to 0$. 

\begin{lemma}[Convergence of stochastic integrals]\label{stochint}
For any collection of $(\hat{\sF}_{t}^{\delta})_{t\geq 0}$-adapted, essentially bounded, smooth processes $\{\psi_\delta\}_{\delta\geq 0}\in C(0,T;C^\infty(\Gamma))$ and $\{\bd{q}_{\delta}\}_{\delta \ge 0}\in C(0,T;C^\infty(\sO_\alpha))$ such that $\bd{q}_{\delta}$ converges almost surely to some $(\hat{\sF}_{t}^{})_{t\geq 0}$-adapted process $ \bd{q}$ in $C(0, T; H^{l}(\mathcal{O}_{\alpha}))$ and $\psi_\delta $ converges almost surely to some $(\hat{\sF}_{t}^{})_{t\geq 0}$-adapted process $ \psi$ in $C(0, T; H^{l}(\Gamma))$ 
for $l > \frac{5}{2}$, $\tilde\bP$-a.s. such that 
\begin{equation}\label{testbound}
\|\psi_\delta\|_{L^\infty(\tilde\Omega \times [0,T]\times\Gamma)}+\|\bd{q}_{\delta}\|_{L^{\infty}(\tilde\Omega \times [0, T]; H^l_0(\mathcal{O}_{\alpha}))} \le C,
\end{equation}
for a constant $C$ that is independent of $\delta$, we have the following $\tilde\bP$-almost sure convergence:
\begin{align*}
\int_{0}^{ T} \int_{\mathcal{O}_{\alpha}} \mathbbm{1}_{\sO_{\hat\eta^*_\delta}}\bd{F}(\hat{\rho}_{\delta}, \hat{\rho}_{\delta} \hat{\bu}_{\delta}) \cdot \bd{q}_{\delta} d\hat W^1_\delta(t) \to \int_{0}^{ T} \int_{\mathcal{O}_{\eta^{*}}} \overline{\bd{F}(\rho, \rho \bu)} \cdot \bd{q} d\hat W^1(t),\\
\int_{0}^{T} \int_{\Gamma} G(\hat{\eta}_{\delta}, \hat{v}_{\delta} ) \cdot \psi_\delta d\hat W_\delta^2(t) \to \int_{0}^{T} \int_{\Gamma} G(\eta, v) \cdot \psi d\hat W^2(t).
\end{align*}
Here the limiting $\overline{\bd{F}(\rho, \rho \bu)}$ is defined by its action on the orthonormal basis of $\mathcal{U}_0$, via $\overline{\bd{F}(\rho, \rho\bu)}\bd{e}_{k} = \overline{f_{k}(\rho, \rho \bu)}$ where the functions $\overline{{f}_{k}(\rho, \rho \bu)}$ are the weak limits of $f_{k}(\hat\rho_{\delta}, \mathbbm{1}_{\mathcal{O}_{\hat\eta^{*}_{\delta}}}\hat{\bu}_\delta, \mathbbm{1}_{\mathcal{O}_{\hat\eta^{*}_{\delta}}} \nabla \hat{\bu}_\delta)$ in $L^{p}((0, T) \times \mathcal{O}_{\alpha})$ for some $p > 1$, given by the stochastic noise assumption \eqref{fassumption} and the result on weak limits of Carath\'{e}odory functions composed with the approximate solutions, see \eqref{caratheodoryconv}.
\end{lemma}
\begin{proof}
By classical methods in \cite{Ben} (see Lemma 2.1 in \cite{DGHT}, Lemma 2.6.6. in \cite{BFH18}), it suffices to show:
\begin{equation*}
\int_{\mathcal{O}_{\alpha}} \mathbbm{1}_{\mathcal{O}_{\hat{\eta}_{\delta}^{*}}} \bd{F}(\hat{\rho}_{\delta}, \hat{\rho}_{\delta} \hat{\bu}_{\delta}) \cdot \bd{q}_{\delta} \to \int_{\mathcal{O}_{\eta^{*}}} \overline{\bd{F}(\rho, \rho \bu)} \cdot \bd{q} \quad\text{ in $L^{2}(0, T; L_{2}(\mathcal{U}_0; \mathbb{R}))$ in probability}.
\end{equation*}
We observe that this will be established if we show that $\mathbbm{1}_{\sO_{\hat\eta^*_\delta}}{f}_{k}(\rho_{\delta}, \rho_{\delta} \bu_{\delta}) \to \mathbbm{1}_{\sO_{\eta^*}}\overline{{f}_{k}(\rho, \rho \bu)}$ almost surely in $L^{2}(0, T; H^{-l}(\mathcal{O}_{\alpha}))$ for $l > 5/2$, by the strong $\tilde\bP$-almost sure convergence of $\bd{q}_{\delta} \to \bd{q}$ in $C(0, T; C^{m}(\mathcal{O}_{\alpha}))$ for all $m\in\mathbb{N}$. To prove this, it suffices to show the following two $\tilde\bP$-almost sure strong convergences:
\begin{align}\label{negativeSobconv1}
\mathbbm{1}_{\sO_{\hat\eta^*_\delta}}f_k(\hat{\rho}_{\delta}, \hat{\rho}_{\delta} \hat{\bu}_{\delta}) -{\mathbbm{1}_{\sO_{\hat{\eta}_{\delta}^*}}}f_k(\hat{\rho}_{\delta}, \hat{\rho}_{\delta} \bu) \to 0 \quad \text{ in } L^{2}(0, T; H^{-l}(\mathcal{O}_{\alpha})), \\ 
{\mathbbm{1}_{\sO_{\hat{\eta}_{\delta}^*}}}f_k(\hat{\rho}_{\delta}, \hat{\rho}_{\delta} \bu) \to {\mathbbm{1}_{\sO_{\eta^*}}}\overline{f_k(\rho, \rho \bu)} \quad \text{ in } L^{2}(0, T; H^{-l}(\mathcal{O}_{\alpha})).\label{negativeSobconv2}
\end{align}
and the uniform estimate
\begin{equation}\label{deltatail}
\lim_{m \to \infty} \left(\sup_{\delta} \tilde{\mathbb{E}}\int_{0}^{T} \sum_{k = m}^{\infty}\left(\int_{\mathcal{O}_{\alpha}} \mathbbm{1}_{\mathcal{O}_{\hat{\eta}^{*}_{\delta}}} |f_{k}(\hat{\rho}_\delta, \hat{\rho}_\delta \hat{\bu}_\delta) \cdot \bd{q}_{\delta}| + \mathbbm{1}_{\mathcal{O}_{\eta^{*}}} |\overline{f_{k}(\rho, \rho \bd{u})} \cdot \bd{q}|\right)^{2}\right) = 0.
\end{equation}

To prove these statements, we use techniques from the proof of Proposition 4.4.12 in \cite{BFH18} for the analogous convergence of stochastic integrals in the context of stochastic Navier-Stokes equations on a fixed domain, where we must adapt many of the calculations to account for the fact that we only have uniform bounds of $\hat{\bd{u}}_{\delta}$ on (moving) domains $\mathcal{O}_{\hat\eta^*_{\delta}}$, rather than on a fixed domain.

\medskip

\noindent \textbf{Proof of \eqref{deltatail}.} Since $\frac{2\gamma}{\gamma + 1} > \frac{6}{5}$, thanks to \eqref{fassumption}, we note that $f_{k}(\hat{\rho}_{\delta}, \hat{\rho}_{\delta} \hat{\bu}_{\delta})$ is uniformly bounded in $L^{p}(\tilde\Omega; L^{\infty}(0, T; L^{6/5}(\mathcal{O}_{\alpha})))$, and in fact it is also uniformly bounded in $L^{\infty}(0, T; L^{6/5}(\mathcal{O}_{\alpha}))$ $\tilde\bP$-almost surely, thanks to Theorem \ref{skorohod}. So along a subsequence depending on the outcome $\tilde\omega$, $f_{k}(\hat{\rho}_{\delta}, \hat{\rho}_{\delta} \hat{\bu}_{\delta})$ converges weakly-star in $L^{\infty}(0, T; L^{6/5}(\mathcal{O}_{\alpha}))$ to some limit, but this limit must $\tilde\bP$-almost surely be $\overline{f_{k}(\rho, \rho \bu)}$, since this has already been identified via \eqref{caratheodoryconv} as the $\tilde\bP$-almost sure limit of $f_{k}(\hat{\rho}_{\delta}, \hat{\rho}_{\delta} \hat{\bu}_{\delta})$ in $L^{r}((0, T) \times \mathcal{O}_{\alpha})$ for some $r > 1$. So by weak lower semi-continuity and the assumption on the noise \eqref{fassumption}, we conclude that for all $1 \le p < \infty$:
\begin{equation*}
\tilde{\mathbb{E}} \left(\|f_{k}(\hat{\rho}_{\delta}, \hat{\rho}_{\delta} \bu_{\delta})\|_{L^{\infty}(0, T; L^{6/5}(\mathcal{O}_{\alpha}))} + \|\overline{f_{k}(\rho, \bu)}\|_{L^{\infty}(0, T; L^{6/5}(\mathcal{O}_{\alpha})))}\right)^{p} \le C_{p} c_{k}^{p},
\end{equation*}
for some constant $C_{p}$ that is independent of $k$ and $\delta$. Therefore, using the assumption on the test functions in \eqref{testbound}:
\begin{equation*}
\sup_{\delta} \tilde{\mathbb{E}}\int_{0}^{T} \sum_{k = m}^{\infty}\left(\int_{\mathcal{O}_{\alpha}} \mathbbm{1}_{\mathcal{O}_{\hat{\eta}^{*}_{\delta}}} |f_{k}(\hat{\rho}_\delta, \hat{\rho}_\delta \hat{\bu}_\delta) \cdot \bd{q}_{\delta}| + \mathbbm{1}_{\mathcal{O}_{\eta}} |\overline{f_{k}(\rho, \rho \bd{u})} \cdot \bd{q}|\right)^{2} \le C_{p, T}\sum_{k = m}^{\infty} c_{k}^{2},
\end{equation*}
which establishes \eqref{deltatail} since $\displaystyle \sum_{k = 1}^{\infty} c_{k}^{2} < \infty$ by the assumptions on the noise.

\medskip

\noindent \textbf{Proof of \eqref{negativeSobconv1}.} First, by using \eqref{fassumption} we can estimate:
\begin{multline*}
\int_{0}^{T} \|\mathbbm{1}_{\sO_{\hat\eta^*_\delta}}\left(f_k(\hat{\rho}_{\delta},\hat{\rho}_{\delta}\hat{\bu}_{\delta}) - f_k(\hat{\rho}_{\delta}, \hat{\rho}_{\delta} \bu)\right)\|_{H^{-l}(\mathcal{O}_{\alpha})}^{2} \le C\int_{0}^{T} \left(\int_{\mathcal{O}_{\alpha}} \mathbbm{1}_{\sO_{\hat\eta^*_\delta}}|f_k(\hat{\rho}_{\delta}, \hat{\rho}_{\delta} \hat{\bu}_{\delta}) - f_k(\hat{\rho}_{\delta}, \hat{\rho}_{\delta} \bu)|\right)^{2} \\
\le c_k \int_{0}^{T} \left(\int_{\mathcal{O}_{\alpha}}\mathbbm{1}_{\sO_{\hat\eta^*_\delta}} \hat{\rho}_{\delta}|\hat{\bu}_{\delta} - \bu|\right)^{2} \le c_k
{\|\hat{\rho}_{\delta}\|_{L^{\infty}(0, T; L^{1}(\mathcal{O}_{\alpha}))}} \int_{0}^{T} \int_{\mathcal{O}_{\alpha}}\mathbbm{1}_{\sO_{\hat\eta^*_\delta}} \hat{\rho}_{\delta} |\hat{\bu}_{\delta} - \bu|^{2}.
\end{multline*}

Since $\hat\rho_{\delta} \rightharpoonup \rho$ weakly-star in $L^{\infty}(0, T; L^{\gamma}(\mathcal{O}_{\alpha}))$ $\tilde\bP$-almost surely, the norms $\|\hat\rho_{\delta}\|_{L^{\infty}(0, T; L^{\gamma}(\mathcal{O}_{\alpha}))}$ are bounded $\tilde\bP$-almost surely (independently of $\delta$). Combining this with Lemma \ref{energyconvdelta} establishes \eqref{negativeSobconv1}.
\medskip

\noindent \textbf{Proof of \eqref{negativeSobconv2}.} Using the renormalized continuity equation, we established the convergence of certain Carath\'{e}odory functions composed with $\hat{\rho}_{\delta}$ in Lemma \ref{brholemma}, and we will combine the result of this lemma with an approximation and density argument to conclude \eqref{negativeSobconv2}. This is in the spirit of the proof of Proposition 4.4.12 in \cite{BFH18}, adapted in the current setting to handle difficulties arising from the moving domain. Recall from Lemma \ref{brholemma} that we have established convergence of Lipschitz functions of $\hat{\rho}_{\delta}$. Specifically, given an arbitrary $C^{1}$ function $b: \R \to \R$ with $b(0) = 0$ such that $b'(w) = 0$ for $|w|$ sufficiently large, we have that:
\begin{equation}\label{Cwb}
\mathbbm{1}_{\mathcal{O}_{\hat\eta^{*}_{\delta}}} b(\hat{\rho}_{\delta}) \to \mathbbm{1}_{\mathcal{O}_{\eta^{*}}} \overline{b(\rho)} \quad \text{ in } C_{w}(0, T; L^{\gamma}(\mathcal{O}_{\alpha})), \quad \text{$\tilde\bP$-almost surely}.
\end{equation}

From \eqref{Cwb} and the compact embedding $C_{w}(0, T; L^{\gamma}(\mathcal{O}_{\alpha})) \subset \subset L^{2}(0, T; H^{-\sigma}(\mathcal{O}_{\alpha}))$ for $\sigma > \frac{1}{6}$:
\begin{equation}\label{indicatorCwb}
\mathbbm{1}_{\mathcal{O}_{\hat{\eta}_{\delta}^{*}}} b(\hat{\rho}_{\delta}) \to \mathbbm{1}_{\mathcal{O}_{\eta^{*}}} \overline{b(\rho)}, \quad \text{ $\tilde\bP$-almost surely and strongly in $L^{2}(0, T; H^{-\sigma}(\mathcal{O}_{\alpha}))$ for $\sigma > \frac{1}{6}$.}
\end{equation}
Note that, for any Lipschitz function $B: \R^{3} \to \R$, $B(\mathbbm{1}_{\mathcal{O}_{\eta^{*}}}\bu) \in L^{2}(0, T; H^{\sigma}(\mathcal{O}_{\alpha}))$ $\tilde\bP$-almost surely for some $\sigma > \frac{1}{6}$ sufficiently small by Theorem \ref{thm:extension}, since $\bu \in L^{2}(0, T; H^{1}(\mathcal{O}_{\eta^{*}}))$ and $\eta^{*} \in C(0, T; H^{s}(\Gamma))$ are bounded $\tilde\bP$-almost surely, for $3/2 < s < 2$. Therefore, 
we have the following strong convergence for $l > \frac{5}{2}$,
\begin{equation}\label{L2Hnegativel}
\mathbbm{1}_{\mathcal{O}_{\hat{\eta}^{*}_{\delta}}} b(\hat{\rho}_{\delta}) B(\mathbbm{1}_{\mathcal{O}_{\eta^{*}}} \bu)=\mathbbm{1}_{\mathcal{O}_{\hat{\eta}_{\delta}^{*}}} b(\hat\rho_{\delta}) B(\bu) \to \mathbbm{1}_{\mathcal{O}_{\eta^{*}}} \overline{b(\rho)} B(\bu), \quad \tilde\bP\text{-almost surely in $L^{2}(0, T; H^{-l}(\mathcal{O}_{\alpha}))$},
\end{equation}
because, by the weak convergence result in Theorem \ref{skorohod}, we have $\bu=\bu\mathbbm{1}_{\sO_{\eta^*}}$.

Finally, we claim that 
\begin{equation}\label{indicatorstrong}
\mathbbm{1}_{\mathcal{O}_{\hat{\eta}^{*}_{\delta}}} b(\hat{\rho}_{\delta}) B(\bd{u}) \to \mathbbm{1}_{\mathcal{O}_{\eta^{*}}} \overline{b(\rho)B(\bd{u})}, \quad \text{ $\tilde\bP$-almost surely and strongly in $L^{2}(0, T; H^{-l}(\mathcal{O}_{\alpha}))$,}
\end{equation}
where $\overline{b(\rho)B(\bd{u})}$ is the weak limit of $b(\hat{\rho}_{\delta})B(\mathbbm{1}_{\mathcal{O}_{\hat{\eta}^{*}_{\delta}}}\hat{\bd{u}}_{\delta})$ in $L^{\min(\gamma, 6)}([0, T] \times \mathcal{O}_{\alpha})$ obtained from \eqref{caratheodoryconv}. This will follow from \eqref{L2Hnegativel} once we show that $\mathbbm{1}_{\mathcal{O}_{\eta^{*}}} \overline{b(\rho)} B(\bu)$ is the same as $\mathbbm{1}_{\mathcal{O}_{\eta^{*}}} \overline{b(\rho) B(\bu)}$, where we recall from \eqref{caratheodoryconv} that the $\tilde{\mathbb{P}}$-almost sure weak convergences in \eqref{continuitycaratheordory} hold and furthermore, 
\begin{equation*}
b(\hat{\rho}_{\delta})B(\mathbbm{1}_{\mathcal{O}_{\hat{\eta}^{*}_{\delta}}} \hat{\bu}_{\delta}) \rightharpoonup \overline{b(\rho)B(\bu)}, \quad \text{$\tilde\bP$-almost surely, weakly in $L^{\min(\gamma, 6)}((0,T) \times \mathcal{O}_{\alpha})$}. 
\end{equation*}
By the boundedness of $B$ and \eqref{convindicator} we have for $1 \le r < \min(\gamma, 6)$ that:
\begin{align}\label{weakgamma1}
\mathbbm{1}_{\mathcal{O}_{\hat{\eta}_{\delta}^{*}}} b(\hat{\rho}_{\delta}) B(\bu) \rightharpoonup \mathbbm{1}_{\mathcal{O}_{\eta^{*}}} \overline{b(\rho)} B(\bu), \quad \text{$\tilde\bP$-almost surely, weakly in $L^{r}((0, T) \times \mathcal{O}_{\alpha})$},\\
\label{weakgamma2}
\mathbbm{1}_{\mathcal{O}_{\hat{\eta}_{\delta}^{*}}} b(\hat{\rho}_{\delta}) B(\hat{\bu}_{\delta}) \rightharpoonup \mathbbm{1}_{\mathcal{O}_{\eta^{*}}} \overline{b(\rho)B(\bu)}, \quad \text{$\tilde\bP$-almost surely, weakly in $L^{r}((0, T) \times \mathcal{O}_{\alpha})$,}
\end{align}
where in the second convergence, we used that $\mathbbm{1}_{\mathcal{O}_{\hat{\eta}^{*}_{\delta}}} B(\mathbbm{1}_{\mathcal{O}_{\hat{\eta}^{*}_{\delta}}} \hat{\bu}_{\delta}) = \mathbbm{1}_{\mathcal{O}_{\hat{\eta}^{*}_{\delta}}} B(\hat{\bu}_{\delta})$ since $B$ is defined pointwise $B: \R^{3} \to \R$. For any test function $\varphi \in C_{c}^{\infty}([0, T] \times \mathcal{O}_{\alpha})$:
\begin{small}
\begin{align}\label{weaklimitunique}
\left|\int_{0}^{T} \int_{\mathcal{O}_{\alpha}} \mathbbm{1}_{\mathcal{O}_{\hat\eta^{*}_{\delta}}} b(\hat{\rho}_{\delta}) \Big(B(\bu) - B(\hat{\bu}_{\delta})\Big) \varphi \right| \le\text{Lip}(b) \cdot \text{Lip}(B) \cdot \|\varphi\|_{L^{\infty}([0, T] \times \mathcal{O}_{\alpha})} &\int_{0}^{T} \int_{\mathcal{O}_{\alpha}} \mathbbm{1}_{\mathcal{O}_{\hat{\eta}^{*}_{\delta}}} \hat{\rho}_{\delta}|\bu - \hat{\bu}_{\delta}| \nonumber \\
\le \text{Lip}(b) \cdot \text{Lip}(B) \cdot \|\varphi\|_{L^{\infty}([0, T] \times \mathcal{O}_{\alpha})} \cdot  \|\hat{\rho}_{\delta}\|^{1/2}_{L^{\infty}(0, T; L^{1}(\mathcal{O}_{\alpha}))} &\int_{0}^{T} \int_{\mathcal{O}_{\alpha}} \mathbbm{1}_{\mathcal{O}_{\hat{\eta}^{*}_{\delta}}} \hat{\rho}_{\delta} |\bu - \hat{\bu}_{\delta}|^{2} \to 0
\end{align}
\end{small}
$\tilde\bP$-almost surely, by the $\tilde\bP$-almost sure boundedness of $\|\hat{\rho}_{\delta}\|_{L^{\infty}(0, T; L^{1}(\mathcal{O}_{\alpha}))} \le C(\tilde{\omega})$ and Lemma \ref{energyconvdelta}. Thus, we conclude from \eqref{weakgamma1} and \eqref{weakgamma2} that $\mathbbm{1}_{\mathcal{O}_{\eta^{*}}} \overline{b(\rho)} B(\bu) = \mathbbm{1}_{\mathcal{O}_{\eta^{*}}} \overline{b(\rho)B(\bu)}$, from which the desired convergence \eqref{indicatorstrong} follows. 

Now that \eqref{indicatorstrong} is established, the desired $\tilde\bP$-almost sure convergence in \eqref{negativeSobconv2} 
follows by an approximation argument, by approximating $f_{k}(\rho,\rho \bd{u})$ by finite sums of the form $\displaystyle \sum_{i = 1}^{m} b_{m}(\rho) B_{m}(\bd{u})$ almost surely in the $L^{2}(0, T; H^{-l}(\mathcal{O}_{\alpha}))$ norm, where $b_{m}$ and $B_{m}$ are bounded Lipschitz functions with $b_m(0) = 0$ and $b_m'(w) = 0$ for sufficiently large $|w|$. 
\end{proof}

\subsection{Construction of test functions}\label{sec:testfunction}

The aim of this section is to construct smooth test functions for the weak formulation \eqref{newdelta} in order to pass  $\delta\to 0$. Recall that this formulation contains a term penalizing the boundary behavior of the fluid and structure velocities. Hence, we must construct test functions that satisfy the kinematic coupling condition so that this penalty term drops out. We must also ensure that these test functions are appropriately adapted so that the stochastic integrals appearing in the equation \eqref{newdelta} are well-defined.

For that purpose, we begin by considering an $(\hat\sF_t)_{t\geq 0}-$adapted process $\bq$ taking $C^1$-paths in $C^\infty(\sO_\alpha)$ such that $\bq|_{\Gamma_{\eta^*}}=\psi{\bf e}_z$ for some $\psi\in C^\infty(\Gamma)$. Hence $(\bq,\psi)$ is an admissible test pair for the limiting equations. However, $(\bd{q}, \psi)$ might not be an admissible test function for the approximate equation due to the kinematic coupling condition being dependent on the structure displacement. Hence, we construct a sequence of $(\hat\sF^\delta_t)_{t\geq 0}-$adapted test functions that are admissible for the approximate problem with $\delta > 0$, that converge to the target test function $(\bd{q}, \psi)$ for the limiting problem as $\delta \to 0$.

 To do this, we first define test functions $(\bd{q}_{\lambda}, \psi_{\lambda})$, that approximate $(\bq,\psi)$ as $\lambda\to 1$, such that $\bd{q}_{\lambda}$ is smooth on the entire domain $\sO_\alpha$ and is equal to $\psi_{\lambda} \bd{e}_{z}$ on $\Gamma_{\eta^{*}}$.
In particular, 
we define a random function $\tilde{\bd{q}}_{\lambda}$ on $[0, T] \times \mathcal{O}_{\alpha}$ by
\begin{equation*}
\tilde{\bd{q}}_{\lambda}(x, y, z) := \bd{q}(x, y, \lambda z), \quad \text{ if } 0 \le z \le \lambda^{-1}(1 + \eta^*(x, y)),
\end{equation*}
\begin{equation}\label{tildeqlambda}
\tilde{\bd{q}}_{\lambda}(x, y, z) := \psi (x,y)\bd{e}_{z}, \quad \text{ if } z \ge \lambda^{-1}(1 + \eta^*(x, y)).
\end{equation}
{ See Fig. \ref{testpic} for a realization.}
Now, thanks to the fact that
$1 + \eta^{*} \ge \alpha > 0,$ we can find a constant $\sigma_\lambda>0$, depending only on $\alpha,\lambda$, such that (see \cite{TC23})
\begin{equation}
\begin{split}
        &\sigma_\lambda \leq \text{distance between the two curves $1+\eta^*$ and $\lambda^{-1}(1+\eta^*)$},\\
    &\sigma_\lambda \to 0  \text{ as the parameter } \lambda\to 1.
\end{split}
\end{equation}
Next we smooth out $\tilde\bq_\lambda$ by defining,
\begin{equation}\label{qlambda}
\bd{q}_{\lambda} := \tilde{\bd{q}}_{\lambda} \ast  \zeta_{\sigma_\lambda},
\end{equation}
where $\zeta_{\sigma_\lambda}$ is the standard rescaled three-dimensional convolution kernel supported in a ball of radius $\sigma_\lambda$. Then, we note that for all $\lambda \in ( 1,2)$, $\bd{q}_{\lambda}|_{\Gamma_{\eta^{*}}} = \psi_\lambda \bd{e}_{z}$, where
\begin{equation}
\psi_\lambda=\psi\ast \zeta_{\sigma_\lambda}.
\end{equation}
Hence, $(\bd{q}_{\lambda}, \psi_{\lambda})$ is another admissible test function for the limiting weak formulation. 


Next, we will approximate this $(\hat\sF_t)_{t\in[0,T]}$-adapted process by $(\hat\sF^\delta_t)_{t\in[0,T]}$-adapted processes. In fact, due to Proposition 3.3.2 in \cite{CKMT25}, there exist $(\hat\sF^\delta_t)_{t\in[0,T]}$-adapted  processes $\bq_{\lambda,\delta}$ in $C^1(0,T;H^k(\sO_\alpha))$ and $\psi_{\delta,\lambda}$ in $C^1(0,T;H^k(\Gamma))$, for any $k \geq 0$ such that 
\begin{equation}
\begin{split}\label{approxfilq}
&\bq_{\lambda,\delta} \to \bq_\lambda \quad \text{almost surely in } C^1(0,T;C^k(\sO_\alpha))\\
&\psi_{\lambda,\delta}\to\psi_{\lambda} \quad \text{almost surely in } C^1(0,T;C^k(\Gamma)).
\end{split}
\end{equation}
These processes may not satisfy the kinematic coupling condition.
To construct test functions for the approximate weak formulation \eqref{newdelta} which approximate the limiting test function $(\bd{q}_{\lambda}, \psi_{\lambda})$ for $\lambda > 1$, we introduce an arbitrary but fixed parameter $\kappa > 0$. The corresponding fluid test function will be built so that it transitions smoothly, in a layer of width $\kappa$, from $\bq_{\lambda,\delta}$ in the interior of $\sO_{\hat\eta^*_\delta}$ to $\psi_{\lambda,\delta }$ on the fluid-structure interface $\Gamma_{\hat\eta^*_\delta}$ which enforces the kinematic coupling condition along $\Gamma_{\hat{\eta}^{*}_{\delta}}$, $\tilde\bP$-almost surely. 


We begin the construction by recalling from Section \ref{extension}: since for some $s \in (3/2, 2)$, we have $\|\hat \eta^*_{\delta}\|_{L^{\infty}(\Gamma)}<\frac1\alpha$ and thus $\|\eta^*_{}\|_{L^{\infty}(\Gamma)}<\frac1\alpha$ for almost any $\omega \in \tilde\Omega$ and $t\in[0,T]$,
we can construct smooth bounding functions bounding functions $a^{\hat\eta^{*}_{\delta}}_{\kappa}$ and $b^{\hat\eta^{*}_{\delta}}_{\kappa}$ for each $\delta > 0$. Furthermore, we recall that $a^{ \eta^{*}}_{\kappa}$  and $b^{ \eta^{*}}_{\kappa}$ for the limiting structure displacement $\eta^{*}$, 
satisfy the properties \eqref{abbound0}, \eqref{abound}, \eqref{bbound} and Lemma \ref{abconv}.

We further recall from Section \ref{extension} that $\phi_0: \R \to \R$ is chosen to be a smooth function on $[0, \infty)$ such that $\phi_0(w) = 1$ for $w \le 1/4$, $\phi_0(w) = 0$ for $w \ge 3/4$, and $\phi_0$ is decreasing on $[1/4, 3/4]$. We then define a smooth function $\varphi_{\delta, \kappa}$ for each $\hat\eta^{*}_{\delta}$ and each parameter $\kappa$ by: 
\begin{align}
&\varphi_{\delta, \kappa}(t,x, y, z) = \phi_0\left(\frac{z-b^{ \hat\eta^{*}_{\delta}}_{\kappa}(t,x, y) }{C_\alpha\kappa^{1/2}}\right),
\label{phideltakappa}
\end{align}
where the constant $C_\alpha$ is as in \eqref{Calphadef}.

Now, given the smooth test pair $(\bd{q}_{\lambda,\delta}, \psi_{\lambda,\delta})$ on the extended domain, approximating $(\bd{q}_\lambda, \psi_\lambda)$ (and thus $(\bd{q}, \psi)$) for the limiting weak formulation, we can construct a family of test functions $(\bd{q}_{\lambda, \delta, \kappa}, \psi_{\lambda, \delta})$ for $\lambda \in (1, 2)$, $\delta > 0$, and $\kappa > 0$, which satisfy the kinematic coupling condition along the approximate moving boundary 
\begin{equation*}
\bd{q}_{\lambda, \delta, \kappa}|_{\Gamma_{\hat\eta^{*}_{\delta}}} = \psi_{\lambda, \delta} \bd{e}_{z},
\end{equation*}
as follows, using the previously constructed function $\varphi_{\delta,\kappa}$ from \eqref{phideltakappa}:
\begin{equation*}
\bd{q}_{\lambda, \delta, \kappa} = \bd{q}_{\lambda,\delta} + (\psi_{\lambda,\delta} \bd{e}_{z} - \bd{q}_{\lambda,\delta})(1-\varphi_{\delta, \kappa}), \qquad \psi_{\lambda, \delta} = \psi_{\lambda}.
\end{equation*}
Note that this construction ensures that $(\bd{q}_{\lambda, \delta, \kappa}, \psi_{\lambda, \delta})$ is adapted to $(\hat\sF_t^\delta)_{t\in [0,T]}$ and satisfies the kinematic coupling condition so that the penalty in the weak formulation drops out for this pair of test functions.\\
\noindent\textbf{Step 1: $\delta\to 0$.}
We will fix $\kappa$ and $\lambda$ and handle the limit as $\delta \to 0$ first, which requires defining the following limiting function $\varphi_{\kappa}$ analogously to \eqref{phikappa}, corresponding to the limiting structure displacement $\eta^{*}$, where we recall that  $b^{\eta^{*}}_{\kappa}$ bounds from below the limiting structure displacement $\eta^{*}$:
\begin{equation}\label{phikappa}
\varphi_{\kappa}(t,x, y, z) = \phi_0\left(\frac{z-b^{ \eta^{*}}_{\kappa}(t,x, y) }{C_\alpha\kappa^{1/2}}\right). 
\end{equation}
and we similarly define
\begin{equation}\label{qlambdakappa}
\bd{q}_{\lambda, \kappa} = \bd{q}_{\lambda} + (\psi_{\lambda} \bd{e}_{z} - \bd{q}_{\lambda}) (1-\varphi_{\kappa}), \qquad \psi_{\lambda, \kappa} = \psi_{\lambda}.
\end{equation}
We have the following convergence result for the convergence of the approximate test functions $(\bd{q}_{\lambda, \delta, \kappa}, \psi_{\lambda,\delta})$ satisfying $\bd{q}_{\lambda, \delta, \kappa}|_{\Gamma_{\hat\eta_{\delta}^{*}}} = \psi_{\lambda, \delta} \bd{e}_{z}$ almost surely, to a limiting test function $(\bd{q}_{\lambda, \kappa}, \psi_{\lambda})$ in the limit as $\delta \to 0$ which satisfies $\bd{q}_{\lambda, \kappa}|_{\Gamma_{\eta^{*}}} = \psi_{\lambda} \bd{e}_{z}$ almost surely. 

\begin{lemma}\label{qlambdadeltakappa}
For all positive integers $k$ and fixed parameters $\kappa > 0$ and $\lambda \in (1/2, 1)$, $$ (\bd{q}_{\lambda, \delta, \kappa}, \psi_{\lambda, \delta} )\to  (\bd{q}_{\lambda, \kappa},\psi_{\lambda})\quad \text{almost surely in } C([0, T]; C^k(\mathcal{O}_{\alpha})\times C^k(\Gamma)), \quad\text{ as } \delta \to 0.$$
\end{lemma}
\begin{proof}
    Observe that thanks to Lemma \ref{abconv} and the fact that by definition, $\phi_0:\R\to\R$ is smooth, we have that $ \varphi_{\kappa,\delta} \to \varphi_{\kappa}$ almost surely in $C([0, T]; C^k(\mathcal{O}_{\alpha}))$ for any $k\geq 0$ as $\delta \to 0$. This, combined with \eqref{approxfilq}, immediately gives us the desired result.
\end{proof}
\if 1 = 0
\begin{proof}
It suffices to show that $\nabla^{k} \varphi_{\kappa,\delta} \to \nabla^{k} \varphi_{\kappa}$ almost surely in $C([0, T] \times \mathcal{O}_{\alpha})$ as $\delta \to 0$. This, however, follows immediately from the convergence result in Lemma \ref{abconv}.
Indeed, since $a^{ \hat\eta^{*}_{\delta}}_{\kappa} \to a^{ \eta^{*}}_{\kappa}$ in $C(0, T; C^k(\Gamma))$, it suffices by the definitions of $\varphi_{\delta,\kappa}$ and $\varphi_{\kappa}$ to show that 
\begin{equation*}
\nabla^{k} \phi_0\left(\frac{ b^{ \hat\eta_\delta^{*}}_\kappa(\omega, x, y)-z}{\kappa^{1/2}}\right) \to \nabla^{k}\phi_0\left(\frac{ b^{\eta^{*}}_{\kappa}(\omega, x, y)-z}{\kappa^{1/2}}\right) \quad \text{ as } \delta \to 0,
\end{equation*}
but this is an immediate result of the chain rule, the smoothness of $\phi_0: \R \to \R$, and the convergence from Lemma \ref{abconv}. 
\end{proof}

\medskip

This is done by transforming $\bq$ while preserving its adaptiveness and the kinematic coupling condition.
For that purpose, for any $t\in [0,T]$ and given process $\bq$, let the map $\mathcal{C}_{\bq}$ be defined as $$\mathcal{C}_{\bq}(t,\omega,\eta,\eta_\delta)
	=F^\ep_{\eta,\eta_\delta}(\bq(t,\omega)),$$
	where for any $\epsilon<\alpha$,
 $$F^\ep_{\eta,\eta_\delta}({\bq})(x,y,z):={\bq}\left(x,y,\frac{(1+\eta(x,y))}{(1+{\eta_\delta(x,y)}-\ep)}z\right) = {\bq}\circ (\bd{\Phi}^\eta)^{-1}\circ \bd{\Phi}^{\eta_\delta-\epsilon}, $$ 
 is a well-defined map from $C(\bar\sO_{\eta})$ to $C(\bar\sO_{\eta_\delta-\varepsilon})$ 
 for  any $\eta,\eta_\delta\in C(\Gamma)$. Recall that, by assumption, $\bq(t)$ is $\hat\sF_t$-measurable and hence $\hat\sF^\delta_t$-measurable.
 This observation combined with the continuity of the composition operator $F^\ep_{\eta,\eta_\delta}$  gives us, for any $\eta$ and $\eta_\delta$, that the $C^1(\bar\sO_{\eta_\delta})$-valued map $\omega \mapsto \mathcal{C}_{\bg}(\omega,\eta,\eta_\delta)$ is $\hat\sF^\delta_t$-measurable (where $C^1(\bar\sO_{\eta_\delta})$ is endowed with Borel $\sigma$-algebra).
	Note also that for any fixed $\omega$, the maps  $\eta\mapsto \mathcal{C}_{\bg}(\omega,\eta,\eta_\delta)$ and $\eta_\delta\mapsto \mathcal{C}_{\bg}(\omega,\eta,\eta_\delta)$ are continuous.
	Hence we deduce that $\mathcal{C}_\bg$ is a Carath\'eodory function.
	Now, by the construction,  we know that $\eta^*$ and $\hat\eta_\delta^*$ are $(\hat{\sF}^\delta_t)_{t\geq 0}$-adapted. Therefore, we conclude that the $C^1(\bar\sO_{\hat\eta^*_\delta})$-valued process,  
 $$\bq^1_{\delta}(t,\omega):=\mathcal{C}_\bq(\omega,\eta^*(t,\omega),\hat\eta^*_\delta(t,\omega)),$$
	is $(\hat\sF^\delta_t)_{t\geq 0}$-adapted as well. The same conclusions follow for the process $\psi$, using the same argument.
 
  Unfortunately, $\bq^1_{\delta}$ inherits its spatial regularity from the ALE maps $\bd{\Phi}^{\eta^*}$ and $\bd{\Phi}^{\hat\eta^*_\delta}$ which is not
 enough to be used as a test function for the momentum equations. Hence, we will mollify it to attain an appropriate test function.  To that end, we denote by $\bq^2_{\delta}$ the extension of $\bq^1_{\delta}$ by $\psi{\bf e}_z$ constantly in the $z$-direction. Then define $\bq_{\delta,\lambda}$ and $\psi_\lambda$ to be the space regularization of $\bq^2_{\delta}$ and $\psi$ respectively, using the standard 3D mollifiers. We choose the mollification parameter $\lambda=c\epsilon$ where the deterministic constant $c$, depending only on $\alpha$,  is appropriately small (see equation (44) in \cite{TC23}) to ensure that, $$\bq_{\delta,\lambda}|_{\Gamma_{\hat\eta^*_\delta}}=\psi_\lambda{\bf e}_z.$$
Observe that,
\begin{align}\label{conv_q}
    \bq_{\delta,\lambda} \to \bq^1_\lambda \,\,\quad\tilde\bP\text{-almost surely in } \,\, L^\infty(0,T;H^k(\sO_\alpha)), \quad \forall k\geq 1, \quad \text{ as }\delta\to 0,
\end{align}
where $\bq^1_\lambda$ is the space mollification of $F_{\eta,\eta}^\ep(\bq) + \mathbbm{1}_{\sO^c_{\eta^*}}\psi{\bf e}_z.$
We can also see that,
\begin{equation}\label{conv_lambda}
\begin{split}
\bq^1_\lambda &\to \bq  \quad\,\,\tilde\bP-a.s.\text{ in } \,\, L^\infty(0,T;H^k(\sO_{\eta^*})), \quad \text{ as }\lambda\to 0, {k??}\\
\psi_\lambda &\to \psi \quad\,\,\tilde\bP-a.s.\text{ in } \,\, L^\infty(0,T;H^k(\Gamma)), \forall k\geq 1 \quad \text{ as }\lambda\to 0.
\end{split}
\end{equation}
\fi
Observe that, for any $\delta, \kappa, \lambda > 0$, the pair $(\bq_{\lambda,\delta,\kappa},\psi_{\lambda,\delta})$ is an admissible test function for the coupled equation \eqref{newdelta}. We now apply a special version of the It\^o formula, which can be proven by using a regularization argument as outlined in Lemma 5.1 in \cite{BO13}, to obtain that:
\begin{multline}\label{newdeltatest1}
\int_{\mathcal{O}_{\alpha}} \hat\rho_{\delta}(t) \hat{\bu}_{\delta}(t) \cdot \bd{q}_{\lambda,\delta,\kappa}(t) + \int_{\Gamma} \hat{v}_{\delta}(t) \psi_{\lambda,\delta}(t) = \int_{\mathcal{O}_{\alpha}} \bd{p}_{0,\delta} \cdot \bd{q}_{\lambda,\delta,\kappa}(0) + \int_{\Gamma} v_{0} \psi_{\lambda,\delta}(0)\\
+\int_0^t\int_{\mathcal{O}_{\alpha}} \hat{\bp}_{0,\delta} \cdot \partial_t\bd{q}_{\lambda, \delta, \kappa} + \int_0^t\int_{\Gamma} \hat v_{\delta}\partial_t\psi_{\lambda,\delta}
+ \int_{0}^{t} \int_{\mathcal{O}_{\alpha}} (\hat\rho_{\delta} \hat{\bu}_{\delta} \otimes \hat{\bu}_{\delta}) : \nabla \bd{q}_{\lambda, \delta,\kappa} 
+ \int_{0}^{t} \int_{\mathcal{O}_{\alpha}} \Big(a\hat\rho_{\delta}^{\gamma} + \delta \hat\rho_{\delta}^{\beta}\Big) (\nabla \cdot \bd{q}_{\lambda, \delta,\kappa}) \\
- \int_{0}^{t} \int_{\mathcal{O}_{\alpha}} {\mu^{\hat{\eta}^{*}_{\delta}}_{\delta}} \nabla\hat{ \bu}_{\delta} : \nabla \bd{q}_{\lambda,\delta,\kappa}
- \int_{0}^{t} \int_{\mathcal{O}_{\alpha}} {\lambda^{\hat{\eta}^{*}_{\delta}}_{\delta}} \text{div}(\hat{\bu}_{\delta}) \text{div}(\bd{q}_{\lambda,\delta,\kappa}) - \frac{1}{\delta} \int_{0}^{t} \int_{\Gamma} (\hat{\bu}_{\delta}|_{\Gamma_{\hat\eta_{\delta}^{*}}} - \hat v_{\delta} \bd{e}_{z}) \cdot (\bd{q}_{\lambda,\delta,\kappa}|_{\Gamma_{\hat\eta_{\delta}^{*}}} - \psi_{\lambda,\delta} \bd{e}_{z}) \\
- \int_{0}^{t} \int_{\Gamma} \nabla \hat v_{\delta} \cdot \nabla \psi_{\lambda,\delta} - \int_{0}^{t} \int_{\Gamma} \nabla \hat\eta_{\delta} \cdot \nabla \psi_{\lambda,\delta} - \int_{0}^{t} \int_{\Gamma} \Delta \hat\eta_{\delta} \Delta \psi_{\lambda,\delta} \\
+ \int_{0}^{t} \int_{\mathcal{O}_{\alpha}} \bd{F}_\delta(\hat\rho_{\delta} , \hat\rho_{\delta} \hat{\bu}_{\delta} ) \cdot \bd{q}_{\lambda,\delta,\kappa} d\hat W^1_\delta (t) + \int_{0}^{t} \int_{\Gamma} G(\hat\eta_{\delta}, \hat v_{\delta}) \psi_{\lambda,\delta} d\hat W^2_\delta(t),
\end{multline}
holds $\tilde\bP$-almost surely and for almost every $t\in [0,T]$.

We will first pass $\delta \to 0$ using Theorem \ref{skorohod}, the convergence result for test functions in Lemma \ref{qlambdadeltakappa}, and convergence result for the stochastic integrals Lemma \ref{stochint}. This yields:
\begin{equation}\label{newdeltatest}
\begin{split}
&\int_{\mathcal{O}_{\eta^*}} \rho(t) {\bu}(t) \cdot \bd{q}_{\lambda, \kappa}(t) + \int_{\Gamma}  v(t) \psi_\lambda(t) = \int_{\mathcal{O}_{\eta_0}} \bp_{0} \cdot \bd{q}_{\lambda, \kappa}(0) + \int_{\Gamma} v_{0} \psi_\lambda(0)
 \\
&+\int_0^t\int_{\mathcal{O}_{\eta^*}} \rho{\bu} \cdot \partial_t\bd{q}_{\lambda, \kappa}+ \int_{0}^{t} \int_{\mathcal{O}_{\eta^*}} (\rho{\bu} \otimes {\bu}) : \nabla \bd{q}_{\lambda, \kappa} 
+ \int_{0}^{t} \int_{{{\mathcal{O}_{\alpha}}}} \bar p (\nabla \cdot \bd{q}_{\lambda, \kappa}) - \int_{0}^{t} \int_{\mathcal{O}_{\eta^*}} \mu \nabla{ \bu} : \nabla \bd{q}_{\lambda, \kappa}\\
&- \int_{0}^{t} \int_{\mathcal{O}_{\eta^*}} \lambda \text{div}({\bu}) \text{div}(\bd{q}_{\lambda, \kappa})  
+ \int_0^t\int_{\Gamma} v\partial_t\psi_\lambda- \int_{0}^{t} \int_{\Gamma} \nabla  v \cdot \nabla \psi_\lambda - \int_{0}^{t} \int_{\Gamma} \nabla \eta \cdot \nabla \psi_\lambda - \int_{0}^{t} \int_{\Gamma} \Delta \eta \Delta \psi_\lambda \\
&+ \int_{0}^{t} \int_{\mathcal{O}_{\alpha}} \overline{\bd{F}(\rho , \rho {\bu} )} \cdot \bd{q}_{\lambda, \kappa}d\hat W^1 + \int_{0}^{t} \int_{\Gamma} G(\eta,  v) \psi_\lambda d\hat W^2 ,
\end{split}
\end{equation}
holds $\tilde\bP$-almost surely for almost every $t\in (0,\tau^\eta)$ and any specific choice of parameters $\lambda \in (1, 2)$ and $\kappa > 0$. 
 To handle the convergence of the terms involving  the viscosity coefficients $\mu^{\hat{\eta}^*_\delta}_{\delta}$ and $\lambda^{\hat{\eta}^*_\delta}_{\delta}$, we recall that $\mathbbm{1}_{\sO_{\hat\eta^*_\delta}}\nabla\hat\bu_{\delta}$ converges weakly to $\mathbbm{1}_{\mathcal{O}_{\eta^{*}}}\nabla\bu$ in $L^{2}(\sO_\alpha)$ due to Theorem \ref{skorohod}, so since $\mu^{\hat{\eta}^{*}_{\delta}}_{\delta} = \mu$ on $\mathcal{O}_{\hat\eta^{*}_{\delta}}$, we calculate
\begin{equation*}
\mu \int_0^t\int_{\mathcal{O}_{\hat\eta_{\delta}^{*}}}  \nabla \hat{\bu}_{\delta} : \nabla \bd{q}_{\lambda, \delta, \kappa} \to \mu\int_0^t \int_{\mathcal{O}_{\eta^{*}}}\nabla\bu: \nabla \bd{q}_{\lambda, \kappa}, \qquad a.e.\,\, t\in [0,T].
\end{equation*}
The remaining part of the integral outside will vanish in the limit, since 
\begin{equation*}
\int_{0}^{t} \int_{\mathcal{O}_{\alpha} \setminus \mathcal{O}_{\hat{\eta}^{*}_{\delta}}} \mu^{\hat{\eta}^{*}_{\delta}}_{\delta} \nabla \hat{\bu}_{\delta} : \nabla \bd{q}_{\lambda, \delta, \kappa} = \int_{\mathcal{O}_{\alpha}} (\sqrt{\mu^{\hat{\eta}^{*}_{\delta}}_{\delta}} \nabla \hat{\bu}_{\delta}) : \mathbbm{1}_{\mathcal{O}_{\alpha} \setminus \mathcal{O}_{\hat{\eta}^{*}_{\delta}}} (\sqrt{\mu^{\hat{\eta}^{*}_{\delta}}_{\delta}} \nabla \bd{q}_{\lambda, \delta, \kappa}) \to 0 \quad \text{ almost surely}, 
\end{equation*}
by the weak convergence of $\sqrt{\mu_\delta^{\hat{\eta}^{*}_{\delta}}} \nabla \hat{\bu}_{\delta}$ in $L^{2}(\mathcal{O}_{\alpha})$ and the fact that $\mathbbm{1}_{\mathcal{O}_{\alpha} \setminus \mathcal{O}_{\hat{\eta}^{*}_{\delta}}} (\sqrt{\mu^{\hat{\eta}^{*}_{\delta}}_{\delta}} \nabla \bd{q}_{\lambda, \delta, \kappa}) \to 0$ in $L^{q}(\mathcal{O}_{\alpha})$, for any $q\in(1,\infty)$ by the properties of the extension of the viscosity coefficients \eqref{viscosityextension} and the almost sure uniform pointwise convergence of $\nabla \bd{q}_{\lambda, \delta, \kappa}$ in Lemma \ref{qlambdadeltakappa}.\\

\noindent{\bf Step 2: Choosing an appropriate $\kappa$.} Next, we will fix a specific and strategic choice of $\kappa$ in \eqref{newdeltatest} for the test functions $(\bd{q}_{\lambda, \kappa}, \psi_{\lambda})$, which are admissible test functions for the limiting weak formulation satisfying $\bd{q}_{\lambda, \kappa}|_{\Gamma_{{\eta}^{*}}} = \psi_{\lambda} \bd{e}_{z}$.

Given $\lambda \in (1, 2)$, we choose $\kappa$ such that the curve $b^{\eta^*}_{\kappa}+\frac14 C_\alpha \kappa^{\frac12}$ is above the 
set $\bigcup_{w\in \Gamma^\lambda_{\eta^*} }B(w,\sigma_\lambda)$, where $B(w,\sigma_\lambda)$ denotes the 3D ball of radius $\sigma_\lambda$ centered at $w$ and $\Gamma^\lambda_{\eta^*}=\{(x,y,\frac1\lambda(1+\eta^*)): (x,y)\in \Gamma \}$; see Fig. \ref{testpic}. 

\begin{figure}
    \centering
    \includegraphics[width=0.5\linewidth]{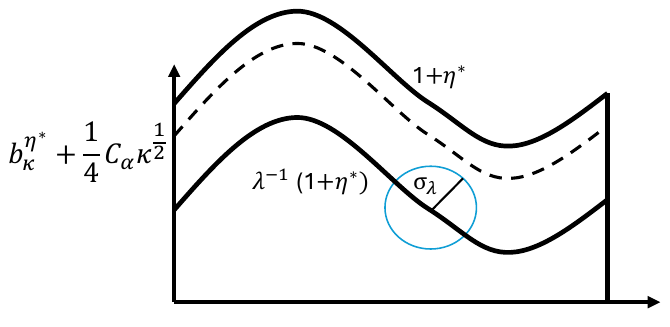}
    \caption{Choice of $\kappa$}
    \label{testpic}
\end{figure}

This will ensure that that for all $(x, y, z) \in \mathcal{O}_{\alpha}$, either $\varphi_{\kappa} = 1$ or $\bd{q}_{\lambda} = \psi_{\lambda} \bd{e}_{z}$. By the definition in \eqref{qlambdakappa}, this means that for this specific choice of $\kappa$, we have that $$\bd{q}_{\lambda, \kappa} = \bd{q}_{\lambda}$$ and hence the weak formulation \eqref{newdeltatest} also holds, almost surely for almost every $t\in (0,\tau^\eta)$, with $\bd{q}_{\lambda}$ in place of $\bd{q}_{\lambda, \kappa}$:
\begin{equation}\label{qlambdatest}
\begin{split}
&\int_{\mathcal{O}_{\eta^*}} \rho(t) {\bu}(t) \cdot \bd{q}_{\lambda}(t) + \int_{\Gamma}  v(t) \psi_{\lambda}(t) = \int_{\mathcal{O}_{\eta_0}} {\bd{p}_{0} }\cdot \bd{q}_{\lambda}(0) + \int_{\Gamma} v_{0} \psi_\lambda(0)
+\int_0^t\int_{\mathcal{O}_{\eta^*}} \rho{\bu} \cdot \partial_t\bd{q}_{\lambda} \\
&+ \int_{0}^{t} \int_{\mathcal{O}_{\eta^*}} (\rho{\bu} \otimes {\bu}) : \nabla \bd{q}_{\lambda} 
+ \int_{0}^{t} \int_{{\mathcal{O}_{\alpha}}} \bar p (\nabla \cdot \bd{q}_{\lambda}) - \int_{0}^{t} \int_{\mathcal{O}_{\eta^*}} \mu \nabla{ \bu} : \nabla \bd{q}_{\lambda}
- \int_{0}^{t} \int_{\mathcal{O}_{\eta^*}} \lambda \text{div}({\bu}) \text{div}(\bd{q}_{\lambda})  \\
&+ \int_0^t\int_{\Gamma} v\partial_t\psi_\lambda- \int_{0}^{t} \int_{\Gamma} \nabla  v \cdot \nabla \psi_\lambda - \int_{0}^{t} \int_{\Gamma} \nabla \eta \cdot \nabla \psi_\lambda - \int_{0}^{t} \int_{\Gamma} \Delta \eta \Delta \psi_\lambda \\
&+ \int_{0}^{t} \int_{\Gamma} G(\eta,  v) \psi_\lambda d\hat W^2 + \int_{0}^{t} \int_{\mathcal{O}_{\alpha}} \overline{\bd{F}(\rho , \rho {\bu} )} \cdot \bd{q}_{\lambda}d\hat W^1 ,
\end{split}
\end{equation}

\noindent\textbf{Step 3: $\lambda\to 1$.} Finally, we pass to the limit as $\lambda \to 1$ in the ``squeezed" test functions $(\bd{q}_{\lambda}, \psi_\lambda)$, which we recall are defined by squeezing by a factor of $\lambda$  and doing a constant extension by $\psi \bd{e}_{z}$ to the maximal domain to get $\tilde{q}_{\lambda}$ defined in \eqref{tildeqlambda}, then using a spatial convolution to get a smooth function $\bd{q}_{\lambda}$ as defined in \eqref{qlambda}. It is clear that for $\lambda \in (1, 2)$, $\|\tilde{\bd{q}}_{\lambda}\|_{W^{1, \infty}(\mathcal{O}_{\alpha})} \le \lambda \|\bd{q}\|_{W^{1, \infty}(\mathcal{O}_{\alpha})}$, so we have that $\|\bd{q}_{\lambda}\|_{W^{1, \infty}(\mathcal{O}_{\alpha})} \le 2\|\bd{q}\|_{W^{1, \infty}(\mathcal{O}_{\alpha})}$ after spatial convolution, for $\lambda \in (1, 2)$. Furthermore, it is clear that $\bd{q}_{\lambda}$ and $\nabla \bd{q}_{\lambda}$ converge pointwise almost everywhere to $\tilde{\bd{q}}$ and $\nabla \tilde{\bd{q}}$, where
\begin{equation*}
\tilde{\bd{q}}|_{\mathcal{O}_{{\eta}^{*}}} = \bd{q}, \qquad \tilde{\bd{q}}|_{\mathcal{O}_{\alpha} \setminus \mathcal{O}_{{\eta}^{*}}} = \psi \bd{e}_{z}.
\end{equation*}

Hence, we can use dominated convergence to pass to the limit as $\lambda \to 1$ in \eqref{qlambdatest} in order to obtain that the limiting weak formulation holds almost surely for $(\tilde{\bd{q}}, \psi)$, which although is not smooth is piecewise smooth. Since $\tilde{\bd{q}}$ agrees with $\bd{q}$ on $\mathcal{O}_{{\eta}^{*}}$, the limiting weak formulation will hold almost surely with the originally given adapted test function $(\bd{q}, \psi)$ satisfying $\bd{q}|_{\Gamma_{{\eta}^{*}}} = \psi \bd{e}_{z}$, once we show that the pressure integral involving $\bar{p}$ is the same as a pressure integral over $\mathcal{O}_{{\eta}^{*}}$, which will require us showing that the pressure vanishes outside $\mathcal{O}_{{\eta}^{*}}$. We will show this in Section \ref{sec:pressure}, by showing that in fact $\overline{p} = \rho^{\gamma}$ and deriving the renormalized continuity equation \eqref{renorm_cont}. In addition, notice that some terms in the limiting weak formulation still contain $\eta^*$ instead of $\eta$, however thanks to \eqref{positivetau} we know that $\eta$ and $\eta^*$ agree at least until the almost surely positive stopping time $\tau^\eta$. 

\subsection{Renormalized solutions}
The aim of this section is to show that the limiting random variables $\bu,\rho$ satisfy continuity equations in the renormalized sense, that is, the equation:
\begin{align}\label{renorm}
	\int_0^t\int_{\sO_\alpha}  b(\rho)(\partial_t\phi+\bu\cdot\nabla \phi)=\int_0^t\int_{\sO_\alpha} (b'(\rho)\rho-b(\rho))(\nabla \cdot\bu)\phi,
\end{align} 
holds $\tilde\bP$-almost surely for almost every $t\in(0,T)$ and for any  $\phi\in C^\infty_c((0,T)\times \sO_{\alpha})$ and $b \in C(\R)$ with $b'(z)=0$ when $z\geq M_b$.
The idea is to choose an appropriate $b$ in the approximate normalized continuity equation satisfied by the new random variables $\hat\bu_\delta,\hat\rho_\delta$ which, at this stage, states that \eqref{renorm_delta} 
holds $\tilde\bP$-almost surely for almost every $t\in(0,T)$ and for any  $\phi$  in $C^\infty_c((0,T)\times \sO_{\alpha})$ and $b \in C(\mathbb{R})$.
To that end, we define
\begin{equation}\label{T}
\begin{split}
    	T(z)= \begin{cases} 
		\begin{array}{cc}
			z & z \in [0,1] \\
			2 &  z \geq 3 \\
			\text{concave otherwise}. 
		\end{array}
	\end{cases},
 \quad \text{and }\quad T_k(z)=kT(\frac{z}k).
\end{split}
\end{equation}
We will thus choose $b=T_k$ in \eqref{renorm_delta}, pass $\delta\to0$ and then pass $k\to\infty$. To analyze these limits we will first study the relevant properties of $T_k$.

First, observe, due to Corollary 6.4 in \cite{BreitHofmanova}, which is a consequence of a fundamental theorem on Young measures, that,
\begin{align}\label{conv_Tk}
	T_k(\hat\rho_\delta) \to \overline{T_k(\rho)} \text{ in } C_w([0,T];L^p(\sO_\alpha)),\quad\tilde\bP\text{-almost surely,} \quad \forall p\in[1,\infty).
\end{align}

\begin{lemma}\label{lem:effectflux}
We have the following convergence result {{for fixed $k$, as $\delta \to 0$:}}
\begin{align}\label{effectflux2}
\tilde\bE\int_0^T\int_{\sO_\alpha}(	\hat\rho_\delta^\gamma + \delta\hat\rho_\delta^\beta - \lambda^{\hat\eta^*_\delta}_\delta \nabla\cdot\hat\bu_\delta)T_k(\hat\rho_\delta) \,dxdt \to \tilde\bE\int_0^T\int_{{\sO_{\alpha}}}(\bar p-\mathbbm{1}_{\sO_{\eta^*}}\lambda\nabla\cdot\bu)\overline{T_k(\rho)}\, dxdt.
\end{align}
\end{lemma}
\begin{proof}
We consider a process $\varphi_\delta \in C^\infty_c(\sO_{\hat\eta^*_\delta})$ such that $\varphi_\delta  \to\varphi \in C^\infty_c(\sO_{\eta^*})$ almost surely.
First, as earlier, we apply Ito's formula with ${\Psi}_k( \rho,\bu)=\int_{\sO_\alpha}\bu\cdot \varphi\nabla\Delta^{-1}\varphi T_k(\rho)$ to the approximate \eqref{delta} and limiting momentum equation \eqref{newdeltatest}. Then using the convergence results obtained in Theorem \ref{skorohod}, \eqref{conv_Tk} and \eqref{conv_weakpressure}
we can see that,
\begin{align}\label{convwithphi}
\tilde\bE\int_0^T\int_{\sO_\alpha}\varphi_\delta^2(	\hat\rho_\delta^\gamma + \delta\hat\rho_\delta^\beta - \lambda^{\hat\eta^*_\delta}_\delta \nabla\cdot\hat\bu_\delta)T_k(\hat\rho_\delta) \, \to \tilde\bE\int_0^T\int_{\sO_\alpha}\varphi^2(\bar p-\lambda\nabla\cdot\bu)\overline{T_k(\rho)}\, .
\end{align}
The proof of \eqref{convwithphi} is standard and so we skip it; see e.g. \cite{BreitHofmanova, BreitSchwarzacherNSF}. 
In \eqref{convwithphi}, we then take $\varphi_\delta \in C^\infty_c(\sO_{\hat\eta^*_\delta})$  such that $\varphi_\delta=1$ on $ A^l_{\hat\eta^*_\delta}$, (see Definition \eqref{Alset}) for some $l>0$ to be chosen later, that converges $\tilde\bP$-almost surely to some $\varphi \in C^\infty_c(\sO_{\eta^*})$ such that $\varphi = 1$ on $ A^{l}_{\eta^*}$. Then we write,
\begin{align*}
&\tilde\bE\int_0^T\int_{\sO_{\hat\eta^*_\delta}}(	\hat\rho_\delta^\gamma + \delta\hat\rho_\delta^\beta - \lambda^{\hat\eta^*_\delta}_\delta \nabla\cdot\hat\bu_\delta)T_k(\hat\rho_\delta) \,-\tilde\bE\int_0^T\int_{\sO_{\alpha}}(\bar p-\lambda\nabla\cdot\bu)\overline{T_k(\rho)}\, \\
	&= \tilde\bE\int_0^T\int_{\sO_{\hat\eta^*_\delta}\cap (A^l_{\hat\eta^*_\delta})^c}(1-\varphi_\delta)(	\hat\rho_\delta^\gamma + \delta\hat\rho_\delta^\beta - \lambda^{\hat\eta^*_\delta}_\delta \nabla\cdot\hat\bu_\delta)T_k(\hat\rho_\delta) \,+\tilde\bE\int_0^T\int_{\sO_{\hat\eta^*_\delta}}\varphi_\delta(	\hat\rho_\delta^\gamma + \delta\hat\rho_\delta^\beta - \lambda^{\hat\eta^*_\delta}_\delta \nabla\cdot\hat\bu_\delta)T_k(\hat\rho_\delta) \,\\
	&-\tilde\bE\int_0^T\int_{\sO_{\eta^*}\cap (A^l_{\eta^*})^c}(1-\varphi)(\bar p-\lambda\nabla\cdot\bu)\overline{T_k(\rho)}\, -\tilde\bE\int_0^T\int_{\sO_{\eta^*}}\varphi(\bar p-\lambda\nabla\cdot\bu)\overline{T_k(\rho)}.\, 
\end{align*}
Observe that for any $\epsilon>0$,  Proposition \ref{bdry_pressure} and an equivalent statement for $\bar p$, give the existence of some $l>0$ for which,
\begin{align*}
\left|	\tilde\bE\int_0^T\int_{\sO_{\hat\eta^*_\delta}\cap (A^l_{\hat\eta^*_\delta})^c}(1-\varphi_\delta)(	\hat\rho_\delta^\gamma + \delta\hat\rho_\delta^\beta )T_k(\hat\rho_\delta)-	\tilde\bE\int_0^T\int_{\sO_{\eta^*}\cap (A^l_{\eta^*})^c}(1-\varphi)\bar p\overline{T_k(\rho)}\right|<\epsilon.
\end{align*}
This observation combined with \eqref{convwithphi} then gives us the desired result \eqref{effectflux2}.
\end{proof}
Next, using the convergence result \eqref{effectflux2} we will next prove that the oscillation defect measure, first used in \cite{FeireislCompressible}, is bounded in expectation.
\begin{lemma}\label{osc}
For the oscillation defect measure we have,
$$osc_{\gamma+1}[\hat\rho_\delta \to \rho](\sO_{\alpha}):=\limsup_{\delta\to 0}\tilde\bE\int_0^T\int_{\sO_\alpha}|T_k(\hat\rho_\delta)-T_k(\rho)|^{\gamma+1} \leq C .$$
\end{lemma}
\begin{proof}
Observe that,
	\begin{align*}
	&\limsup_{\delta\to 0}	\tilde\bE\int_0^T\int_{\sO_\alpha}\left( (\hat\rho_\delta^\gamma+\delta\hat\rho^\beta_\delta)T_k(\hat\rho_\delta) - \bar p \overline{T_k(\rho)}\right) \\
	&= \limsup_{\delta\to 0}\tilde\bE\int_0^T\int_{\sO_\alpha}\left( (\hat\rho_\delta^\gamma+\delta\hat\rho^\beta_\delta-(\rho^\gamma+\delta\rho^\beta))(T_k(\hat\rho_\delta) -  {T_k(\rho)})\right)\\
	&+ \limsup_{\delta\to 0}\tilde\bE\int_0^T\int_{\sO_\alpha}(\hat\rho_\delta^\gamma+\delta\hat\rho^\beta_\delta-\bar p)T_k(\rho)+(\rho^\gamma+\delta\rho^\beta)(T_k(\hat\rho_\delta)-\overline{T_k(\rho)})\\	
		&+\limsup_{\delta\to 0}\tilde\bE\int_0^T\int_{\sO_{\alpha}}(\bar p -(\rho^\gamma+\delta\rho^\beta))(T_k(\rho) -  \overline{T_k(\rho)}) 
\shortintertext{thanks to \eqref{conv_weakpressure}, \eqref{conv_Tk}, and Theorem 11.27 in \cite{FN17}, we obtain}
		&\geq \limsup_{\delta\to 0}\tilde\bE\int_0^T\int_{\sO_\alpha}\left( (\hat\rho_\delta^\gamma+\delta\hat\rho^\beta_\delta-(\rho^\gamma+\delta\rho^\beta))(T_k(\hat\rho_\delta) - {T_k(\rho)})\right)
		\shortintertext{since $z \mapsto z^{\gamma}$ is convex and $T_k$ is concave, $(z^{\gamma}-y^{\gamma})(T_k(z)-T_k(y)) \geq |T_k(z)-T_k(y)|^{\gamma+1}$ for $z,y \geq 0$ we have}
		&\geq \limsup_{\delta\to 0}\tilde\bE \|T_k(\hat\rho_\delta)-T_k(\rho)\|^{\gamma+1}_{L^{\gamma+1}((0,T)\times\sO)}.
	\end{align*}
 On the other hand, since $\mathbbm{1}_{\sO_{\eta^*}}\lambda(\nabla\cdot\bu)$ is the weak limit of $\lambda^{\hat\eta^*_\delta}_\delta(\nabla\cdot\hat\bu_\delta )$ in $L^2(\tilde\Omega,L^2(0,T;L^2(\sO_\alpha)))$, we obtain
\begin{align*}
\limsup_{\delta\to0}&\left|\tilde\bE\int_0^T\int_{\sO_\alpha}\lambda^{\hat\eta^*_\delta}_\delta(\nabla\cdot\hat\bu_\delta )T_k(\hat\rho_\delta)-\mathbbm{1}_{\sO_{\eta^*}}\lambda(\nabla\cdot\bu) \overline{T_k(\rho)}\right| \\
& =\limsup_{\delta\to0}\left|\tilde\bE\int_0^T\int_{\sO_\alpha}\lambda^{\hat\eta^*_\delta}_\delta(\nabla\cdot\hat\bu_\delta )T_k(\hat\rho_\delta)-\lambda^{\hat\eta^*_\delta}_\delta(\nabla\cdot\hat\bu_\delta) \overline{T_k(\rho)}\right| \\
	&\leq \limsup_{\delta\to0} (\tilde\bE\|\lambda^{\hat\eta^*_\delta}_\delta\nabla\cdot\hat\bu_\delta\|^2_{L^2(0,T;L^2(\sO_\alpha))})^\frac12(\tilde\bE\|T_k(\hat\rho_\delta)-\overline{T_k(\rho)}\|^2_{L^2(0,T;L^2(\sO_\alpha))})^\frac12\\
 &\leq \limsup_{\delta\to0} C\|T_k(\hat\rho_\delta)-T_k(\rho)\|_{L^2(\tilde\Omega;L^2(0,T;L^2(\sO_\alpha)))}\\
  &\leq \limsup_{\delta\to0} C\|T_k(\hat\rho_\delta)-T_k(\rho)\|_{L^{\gamma+1}(\tilde\Omega\times (0,T)\times\sO_\alpha)}.
\end{align*}
These bounds and an application of  \eqref{effectflux2} thus give us the desired result as follows,
\begin{align*}
	&\limsup_{\delta\to 0} \|T_k(\hat\rho_\delta)-T_k(\rho)\|^{\gamma+1}_{L^{\gamma+1}(\tilde\Omega\times(0,T)\times\sO_\alpha)}\\ 
	&\leq \limsup_{\delta\to 0}	\tilde\bE\int_0^T\int_{\sO_\alpha}\left( (\hat\rho_\delta^\gamma+\delta\hat\rho^\beta_\delta-\lambda^{\hat\eta^*_\delta}_\delta\nabla\cdot\hat\bu_\delta)T_k(\hat\rho_\delta) - (\bar p-\mathbbm{1}_{\sO_{\eta^*}}\lambda\nabla\cdot\bu) \overline{T_k(\rho)}\right) \\
	&+	\limsup_{\delta\to0} |\tilde\bE\int_0^T\int_{\sO_\alpha}\lambda^{\hat\eta^*_\delta}_\delta(\nabla\cdot\hat\bu_\delta) T_k(\hat\rho_\delta)-\mathbbm{1}_{\sO_{\eta^*}}\lambda(\nabla\cdot\bu )\overline{T_k(\rho)}|\\
	& \leq C\limsup_{\delta\to 0} \|T_k(\hat\rho_\delta)-T_k(\rho)\|_{L^{\gamma+1}(\tilde\Omega\times(0,T)\times\sO_\alpha)}.
\end{align*}
\end{proof}

Lemma \ref{osc} is sufficient to show that the normalized continuity equation \eqref{renorm_delta} holds true in the limit i.e. that \eqref{renorm} holds true.
As mentioned earlier, we will take $b=T_k$ in \eqref{renorm_delta} where $T_k$ is defined in \eqref{T}. This yields, for any $\phi\in C^\infty(0,T;C^\infty_c(\sO_\alpha))$,
\begin{align}\label{tkrenorm}
	\int_0^t\int_{\sO_\alpha}  T_k(\hat\rho_\delta)(\partial_t\phi+\hat\bu_\delta\cdot\nabla \phi)=\int_0^t\int_{\sO_\alpha} (T_k'(\hat\rho_\delta)\hat\rho_\delta-T_k(\hat\rho_\delta))(\nabla \cdot\hat\bu_\delta)\phi .
\end{align} 
To pass $\delta\to 0$ in the equation above, we rewrite it as,
\begin{equation}\label{tkrenormerror}
\int_{0}^{t} \int_{\mathcal{O}_{\alpha}} T_{k}(\hat{\rho}_{\delta})(\partial_{t}\phi + \mathbbm{1}_{\mathcal{O}_{\hat{\eta}^{*}_{\delta}}} \hat{\bd{u}}_{\delta} \cdot \nabla \phi) = \int_{0}^{t} \int_{\mathcal{O}_{\alpha}} (T_{k}'(\hat{\rho}_{\delta}) \hat{\rho}_{\delta} - T_{k}(\hat{\rho}_{\delta})) \mathbbm{1}_{\mathcal{O}_{\hat{\eta}^{*}_{\delta}}} (\nabla \cdot \hat{\bd{u}}_{\delta}) \phi + E_{\delta, \phi},
\end{equation}
where the error term is:
\begin{equation*}
E_{\delta, \phi} := -\int_{0}^{t} \int_{\mathcal{O}_{\alpha} \setminus \mathcal{O}_{\hat{\eta}^{*}_{\delta}}} T_{k}(\hat{\rho}_{\delta}) \hat{\bd{u}}_{\delta} \cdot \nabla \phi + \int_{0}^{t} \int_{\mathcal{O}_{\alpha} \setminus \mathcal{O}_{\hat{\eta}^{*}_{\delta}}} (T_{k}'(\hat{\rho}_{\delta}) \hat{\rho}_{\delta} - T_{k}(\hat{\rho}_{\delta})) (\nabla \cdot \hat{\bd{u}}_{\delta}) \phi.
\end{equation*}
Observe that, due to Corollary 6.4 in \cite{BreitHofmanova}, we have the following weak convergence
$$ \mathbbm{1}_{\sO_{\hat\eta^*_\delta}}(T_k'(\hat\rho_\delta)\hat\rho_\delta-T_k(\hat\rho_\delta))(\nabla \cdot\hat\bu_\delta) \rightharpoonup \tilde T_k \quad \text{ in } L^2(\tilde\Omega\times(0,T)\times\sO_\alpha) .
$$ 


We now pass to the limit in \eqref{tkrenormerror} as $\delta \to 0$. 
Using the fact that $E_{\delta, \phi} \to 0$, $\tilde\bP$-almost surely by Lemma \ref{outsidefluid}, and using the notation \eqref{conv_Tk}, we let $\delta\to 0$ in \eqref{tkrenormerror} and obtain for any $\phi\in C^\infty(0,T;C_c^\infty(\sO_\alpha))$ that
\begin{align}\label{renormTk}
\int_0^t\int_{\sO_\alpha}\overline{T_k(\rho)} (\partial_t\phi +\bu\cdot\nabla\phi)= \int_0^t\int_{\sO_\alpha}\tilde T_k\phi,
\end{align}
holds $\tilde\bP$-almost surely for any $t\in [0,T]$. Note that we used the fact that {{$\mathbbm{1}_{\mathcal{O}_{\hat{\eta}^{*}_{\delta}}}T_{k}(\hat{\rho}_{\delta}) \hat{\bu}_{\delta}$}} converges weakly to $\overline{T_{k}(\rho) \bu}$ $\tilde{\bP}$-almost surely in $L^{\min(\gamma, 2)}([0, T] \times \mathcal{O}_{\alpha})$ by \eqref{caratheodoryconv}, and then we can show that $\overline{T_{k}(\rho) \bu} = \overline{T_{k}(\rho)} \bu$ by showing that {{$\displaystyle \left|\int_0^T\int_{\mathcal{O}_{\alpha}} T_{k}(\hat{\rho}_{\delta}) (\bu - \mathbbm{1}_{\mathcal{O}_{\hat{\eta}^{*}_{\delta}}} \hat{\bu}_{\delta}) \varphi \right| \to 0$}} $\tilde\bP$-almost surely for any test function $\varphi \in C_{c}^{\infty}([0, T] \times \mathcal{O}_{\alpha})$, using an argument similar to \eqref{weaklimitunique}.

Now, by using a standard regularization method (see e.g.~\cite{FeireislCompressible}) we write the renormalized equations for $\overline{T_k(\rho)}$ satisfying \eqref{renormTk}:
\begin{align}\label{renorm-k}
\partial_tb(\overline{T_k(\rho)})+ \text{div}(b(\overline{T_k(\rho)})\bu) + (b'(\overline{T_k(\rho)})\overline{T_k(\rho)}-b(\overline{T_k(\rho)}))(\nabla\cdot\bu) + b'(\overline{T_k(\rho)})\tilde T_k=0,
\end{align}
holds $\tilde\bP$-almost surely in the sense of distributions on $[0,T]\times { \sO_{\alpha}}$. 

Our next aim is to pass $k\to\infty$ in \eqref{renorm-k}.
First, observe that for any $p<\gamma$, thanks to Lemma \ref{osc},  we have that
\begin{align*}
\|\overline{T_k(\rho)} - \rho \|^p_{L^p(\tilde\Omega\times(0,T)\times\sO_\alpha)} &\leq \liminf_{\delta\to 0}\|{T_k(\hat\rho_\delta)}-\hat\rho_\delta\|^p_{L^p(\tilde\Omega\times(0,T)\times\sO_\alpha)}\leq C \tilde\bE\int_0^T\int_{\{|\hat\rho_\delta|\geq k\}} |\hat\rho_\delta|^p \\
&\leq C k^{{p-\gamma}}\tilde\bE\int_0^T\int_{\sO_\alpha} |\hat\rho_\delta|^\gamma \to 0, \qquad \text{ as }k\to \infty.
\end{align*}
That is, for any $p<\gamma$, we have proved that
\begin{align}\label{klimit}
\overline{T_k(\rho)} \to \rho \,\,\text{ in } \,\, L^p(\tilde\Omega\times(0,T)\times\sO_\alpha) \,\,\text{ as }\,\, k\to \infty.
\end{align}
This convergence result allows us to pass $k\to\infty$ in the first three terms on the left-hand side of \eqref{renorm-k}. 
To deal with the last term on the left-hand side of the equation \eqref{renorm-k}, we will next show that
\begin{align}\label{remainder}
b'(\overline{T_k(\rho)})\tilde T_k \to 0 \quad \tilde\bP\text{-almost surely } \quad\text { as } k\to\infty
\end{align}
For that purpose, as proposed in \cite{FeireislCompressible}, we consider the set,
\begin{equation}\label{QkM}
Q_{k,M}:=\{(\omega,t,x)\in \tilde\Omega\times(0,T)\times\sO_\alpha;\,\,\overline{T_k(\rho)} \leq M_b\},
\end{equation}
where $M_b$ is such that
$$b'(z) = 0,\qquad \text{ for } z\geq M_b.$$
Then, we obtain that 
\begin{align*}
 & \tilde\bE\int_0^T\int_{\sO_\alpha}  b'(\overline{T_k(\rho)})\tilde T_k \leq \sup_{z\leq M_b}|b'(z)|\tilde\bE\int_0^T\int_{\sO_\alpha}\mathbbm{1}_{Q_{k,M}}|\tilde T_k|\\
  &\leq \liminf_{\delta\to 0}\tilde\bE\int_0^T\int_{\sO_\alpha}{ \mathbbm{1}_{\sO_{\hat\eta^*_\delta}}}\mathbbm{1}_{Q_{k,M}}|(T_k'(\hat\rho_\delta)\hat\rho_\delta-T_k(\hat\rho_\delta))(\nabla \cdot\hat\bu_\delta)|\\
  &\leq \| \nabla\cdot\hat\bu_\delta\|_{L^2(\tilde\Omega;L^2(0,T;L^2(\sO_{\hat\eta^*_\delta})))}\liminf_{\delta\to 0}\|T_k'(\hat\rho_\delta)\hat\rho_\delta-T_k(\hat\rho_\delta)\|^{\frac{\alpha_0}2}_{L^1(\tilde\Omega\times(0,T)\times\sO_\alpha)}\|T_k'(\hat\rho_\delta)\hat\rho_\delta-T_k(\hat\rho_\delta)\|^{\frac{(1-\alpha_0)(1+\gamma)}2}_{L^{\gamma+1}(Q_{k,M})}
\end{align*}
where $\alpha_0=\frac{\gamma-1}{\gamma}$.
Next, note that since we have $$\sup_{\delta}\|T_k'(\hat\rho_\delta)\hat\rho_\delta-T_k(\hat\rho_\delta)\|_{L^1(\tilde\Omega\times(0,T)\times\sO_\alpha)}\leq k^{1-\gamma}\sup_{\delta}\tilde\bE\int_0^T\int_{\sO_\alpha}\hat\rho_\delta^\gamma \to0\quad \text{as } k\to \infty,$$
we will be done with the proof of \eqref{remainder} if we show that
\begin{equation}\label{Tkprimebound}
    \|T_k'(\hat\rho_\delta)\hat\rho_\delta-T_k(\hat\rho_\delta)\|_{L^{\gamma+1}(Q_{k,M})}\leq C.
\end{equation}
Indeed, observe that due to the fact that $T_k'(z)z \leq T_k(z)$, we have
\begin{align*}
    \|T_k'&(\hat\rho_\delta)\hat\rho_\delta-T_k(\hat\rho_\delta)\|_{L^{\gamma+1}(Q_{k,M})}\leq 2\|T_k(\hat\rho_\delta)\|_{L^{\gamma+1}(Q_{k,M})}\\
    &\leq 2\left(\|T_{k}(\hat{\rho}_{\delta}) - T_{k}(\rho)\|_{L^{\gamma + 1}(\tilde\Omega \times (0, T) \times \mathcal{O}_{\alpha})} + \|T_{k}(\rho)\|_{L^{\gamma + 1}(Q_{k, M})}\right) \\
    &\leq 2\left(\|T_{k}(\hat{\rho}_{\delta}) - T_{k}(\rho)\|_{L^{\gamma + 1}(\tilde\Omega \times (0, T) \times \mathcal{O}_{\alpha})} + \|T_{k}(\rho) - \overline{T_{k}(\rho)}\|_{L^{\gamma + 1}(\tilde\Omega \times (0, T) \times \mathcal{O}_{\alpha})} + \|\overline{T_{k}(\rho)}\|_{L^{\gamma + 1}(Q_{k, M})}\right) \\
    &\leq 2\left(\|T_{k}(\hat{\rho}_{\delta}) - T_{k}(\rho)\|_{L^{\gamma + 1}(\tilde\Omega \times (0, T) \times \mathcal{O}_{\alpha})} + \|T_{k}(\rho) - \overline{T_{k}(\rho)}\|_{L^{\gamma + 1}(\tilde\Omega \times (0, T) \times \mathcal{O}_{\alpha})} + M_b\right),
\end{align*}
where we used the definition of $Q_{k, M}$ in \eqref{QkM} in the last inequality. In addition, by using Lemma \ref{osc}, we can immediately bound $\|T_{k}(\hat{\rho}_{\delta}) - T_{k}(\rho)\|_{L^{\gamma + 1}(\tilde\Omega \times (0, T) \times \mathcal{O}_{\alpha})} \le C$ independently of $\delta$. To show a similar bound that $\|T_{k}(\rho) - \overline{T_{k}}(\rho)\|_{L^{\gamma + 1}(\tilde\Omega \times (0, T) \times \mathcal{O}_{\alpha})} \leq C$ which would establish the result \eqref{Tkprimebound}, we observe that by the convergence of $T_{k}(\hat{\rho}_{\delta}) \to \overline{T_{k}(\rho)}$ in $C_{w}(0, T; L^{p}(\mathcal{O}_{\alpha}))$ $\tilde\bP$-almost surely for $p \in [1, \infty)$, we have that $T_{k}(\hat{\rho}_{\delta}) - T_{k}(\rho) \rightharpoonup \overline{T_{k}(\rho)} - T_{k}(\rho)$ weakly in $L^{\gamma + 1}((0, T) \times \mathcal{O}_{\alpha})$ $\tilde\bP$-almost surely. So by weak lower semicontinuity of norms and Lemma \ref{osc}, we have $\displaystyle \tilde\bE\int_{0}^{T} \int_{\mathcal{O}_{\alpha}} |\overline{T_{k}(\rho)} - T_{k}(\rho)|^{\gamma + 1} dx dt \le C$, which hence establishes \eqref{Tkprimebound}. This proves \eqref{remainder} and thus ultimately proves \eqref{renorm}.
\subsection{Strong convergence of density}\label{sec:pressure}
The aim of this section is to prove that 
\begin{align}\label{pressure}
    \bar p =\rho^\gamma\quad\text{ a.e. on }\tilde\Omega\times(0,T)\times \sO_\alpha,
\end{align}
by appealing to the monotonicity of the pressure. This procedure is standard and follows closely the steps introduced in \cite{FeireislCompressible, BreitHofmanova}. For the sake of completion, we will briefly outline the steps involved in establishing \eqref{pressure} and, for details we will refer the reader to \cite{FeireislCompressible} wherever necessary. We begin by defining,
\begin{align}
L_k(z) = \begin{cases}
    z\ln{z},\quad 0 \leq z<k \\
    z\ln{z} + z\int_k^zT_k(s)/s^2 ds, \quad z \geq k,
\end{cases}
\end{align}
where $T_k$ is defined in \eqref{T}.
Now, we choose $b=L_k$ in \eqref{renorm_delta} and \eqref{renorm}, then we take the difference of the resulting equations and set $\phi\equiv 1$. This yields,
\begin{align*}
\int_{\sO_\alpha} (L_k(\hat\rho_\delta)-L_k(\rho)) (t)\leq \int_{\sO_\alpha}(L_k(\hat\rho_\delta)-L_k(\rho))(0)+\int_0^t\int_{\sO_\alpha}(T_k(\rho)\nabla\cdot\bu-T_k(\hat\rho_\delta)\nabla\cdot\hat\bu_\delta).
\end{align*}
Hence, upon letting $k\to\infty$ in the equation above and using Lemma \ref{lem:effectflux} and monotonicity of pressure, we obtain for any $p<\gamma$ that {{also uses vanishing outside for $\delta > 0$}} 
\begin{align*}
{{\limsup_{\delta \to 0}}} \tilde\bE&\int_{\sO_\alpha} ({L_k(\hat\rho_\delta)}-L_k(\rho)) (t)\le \limsup_{\delta\to 0} \tilde\bE\int_0^t\int_{\sO_\alpha}(T_k(\rho)\nabla\cdot\bu-T_k(\hat\rho_\delta)\nabla\cdot\hat\bu_\delta)\\
&{{=\limsup_{\delta\to 0} \tilde\bE\int_0^t\int_{\sO_\alpha}(T_k(\rho){{\mathbbm{1}_{\sO_{\eta^*}}}}\nabla\cdot\bu-T_k(\hat\rho_\delta)\mathbbm{1}_{\sO_{\hat\eta_\delta^*}}\nabla\cdot\hat\bu_\delta)}}\\
&{{\leq \tilde\bE \int_0^t\int_{\sO_\alpha}(T_k(\rho)-\overline{T_k(\rho)}){{\mathbbm{1}_{\sO_{\eta^*}}}}\nabla\cdot\bu}} \\
&\leq \tilde\bE (\|\nabla\cdot\bu\|_{L^2(0,T;L^2(\sO_{\eta^*}))}\|T_k(\rho)-\overline{T_k(\rho)}\|_{L^2(0,T;L^2(\sO_\alpha))})\\
&\leq \tilde\bE (\|\nabla\cdot\bu\|_{L^2(0,T;L^2(\sO_{\eta^*}))}\|T_k(\rho)-\overline{T_k(\rho)}\|_{L^p((0,T)\times\sO_\alpha)}^{\frac{\gamma-1}{\gamma+1-p}}\|T_k(\rho)-\overline{T_k(\rho)}\|_{L^{\gamma+1}((0,T)\times\sO_\alpha)}^{(\frac{2-p}{\gamma+1-p})}).
\end{align*}


Then the bounds derived in Lemma \ref{osc} and the convergence result \eqref{klimit} implies that,
\begin{align}
    \tilde\bE\int_{\sO_\alpha\times(0,T)} \hat\rho_\delta\ln\hat\rho_\delta \to \tilde\bE\int_{\sO_\alpha\times(0,T)} \rho\ln\rho.
\end{align}
Using the convexity of $\rho \mapsto\rho\ln\rho$ then gives us the desired strong convergence result for the density,
\begin{align}\label{den_strongconv}
    \hat\rho_\delta\to \rho, \quad\text{ in } \quad L^1(\tilde\Omega; L^1(0,T;L^1(\sO_\alpha) )).
\end{align}
This completes the proof of \eqref{pressure}. This strong convergence \eqref{den_strongconv} allows us to explicitly identify the term $\overline{\bd{F}(\rho, \rho \bd{u})}$ in the limiting stochastic integral, see Lemma \ref{stochint}, initially given as the weak limit via \eqref{caratheodoryconv} of $\mathbbm{1}_{\sO_{\hat{\eta}^{*}_{\delta}}} \bd{F}(\hat{\rho}_{\delta}, \hat{\rho}_{\delta} \hat{\bu}_{\delta})$, as $\bd{F}(\rho, \rho \bd{u})$.

Using \eqref{pressure} in the limit of \eqref{qlambdatest} we come to the conclusion that
for every $(\hat\sF_t)_{t \geq 0}-$adapted, essentially bounded smooth process $(\bq,\psi)$ such that 
	$\bq|_{\Gamma_{\eta^*}}=\psi\bd{e}_{z}$, $\bP$-almost surely, the following equation holds for $\bP$-almost surely, for almost every $t \in[0,\tau^\eta]$: 
\begin{equation}
\begin{split}
&{\int_{\sO_{\eta^*}(t)}\rho(t)\bu(t)\bq(t) d\bx +\int_\Gamma\partial_t\eta(t)\psi(t)dz}= \int_{\sO_{\eta_0}}\bp_0\bq(0) d\bx  + \int_\Gamma v_0\psi(0) d\bd{z} \\
&+\int_0^{t }\int_{\sO_{\eta^*}(t)}\rho\bu\cdot \partial_t\bq d\bx dt+\int_0^{t }\int_{\sO_{\eta^*}(t)} \rho\bu\otimes\bu:\nabla\bq d\bx dt
- 2\mu\int_0^{t } \int_{\sO_{\eta^*}(t)}\nabla\bu\cdot \nabla\bq d\bx dt\\
&+\int_0^t\int_{\sO_{\eta^*}(t)} \rho^\gamma (\nabla \cdot\bq )d\bx dt- \lambda \int_0^t\int_{\sO_{\eta^*}(t)} \text{div}(\bu)\text{div}(\bq)d\bx dt \\
&+\int_0^{t }\int_\Gamma\partial_t\eta\partial_t\psi d\bd{z}dt - { \int_0^{t }\int_\Gamma \nabla\partial_{t}\eta \cdot\nabla \psi d\bd{z} dt - \int_{0}^{t} \int_{\Gamma} (\nabla\eta\nabla \psi +\Delta\eta\Delta \psi) d\bd{z}dt}\\
&+\int_0^t\int_{\sO_{\eta^*}(t)} \bd{F}(\rho,\rho\bu)\cdot\bq\,d\hat W^1(t) + \int_0^t\int_\Gamma G(\eta, \partial_t\eta)\psi d\hat W^2(t).
\end{split}\end{equation}
Moreover, $\bu|_{\Gamma_{\eta^*}}=v\bd{e}_z.$
Observe that in this formulation some terms are defined on fluid domains corresponding to the artificial structure displacement $\eta^*$ whereas the others are given in terms of $\eta$. To resolve this discrepancy (cf. \eqref{weaksol}) we appeal to the fact that these two structural displacements are equal (see \eqref{etasequal2}) until the almost surely positive stopping time $\tau^\eta$ (see \eqref{positivetau}) defined in \eqref{tau}. This leads us to our final conclusion stated below.


\medskip

\noindent {\bf Conclusion:} We conclude that the stochastic basis $(\tilde\Omega,\tilde\sF, (\tilde\sF_t)_{t\geq 0},\tilde\bP, \hat W^1,\hat W^2)$ and the random variables $(\bu,\rho,\eta)$ constructed in Theorem \ref{skorohod} and the stopping time $\tau^\eta$ defined in \eqref{tau} determine a martingale solution to the FSI problem in the sense of Definition \ref{def:martingale}.

\section*{Acknowledgements}

J. Kuan was supported by the National Science Foundation under the NSF Mathematical Sciences Postdoctoral Research Fellowship DMS-2303177. K. Tawri was partially supported by the National Science Foundation grant DMS-2407197.

\section*{Appendix A: Equivalence of laws for stopped processes}

Recall that in the proof of the main existence result, we constructed approximate solutions consisting of an approximate structure displacement $\eta_{N}$ and a stopped structure displacement $\eta_{N}^{*}$, which is stopped at the first instance $\tau^\eta_N$, if it exists, at which $\eta_{N}$ leaves desired bounds depending on $\alpha$:
\begin{equation*}
\tau^\eta_N := T \wedge \inf\left\{t > 0 : \inf_{\Gamma}(1 + \eta_{N}(t)) \le \alpha \text{ or } \|\eta_{N}(t)\|_{H^{s}(\Gamma)} \ge \frac{1}{\alpha}\right\}.
\end{equation*}
We then used the Skorohod representation theorem to transfer these approximate solutions to a different probability space, but we want to justify that the new random variables $(\overline{\eta}, \overline{\eta}^{*})$ on the new probability space also have the property that they agree up until the time $\tau^\eta_N$ at which $\tilde{\eta}$ leaves the desired deterministic $\alpha$ bounds:
\begin{equation*}
\tau^\eta_N := T \wedge \inf\left\{t > 0 : \inf_{\Gamma}(1 + \overline{\eta}_{N}(t)) \le \alpha \text{ or } \|\overline{\eta}_{N}(t)\|_{H^{s}(\Gamma)} \le \frac{1}{\alpha}\right\}.
\end{equation*}
The Skorohod representation theorem gives us that there is equivalence of laws, so to use equivalence of laws to establish this, we must show that the set of functions $(\eta_{1}, \eta_{2})$, where $\eta_{2}$ is equal to $\eta_{1}$ stopped at the first time of leaving the desired $\alpha$ bounds, is a measurable set of the phase space $C(0, T; H^{s}(\Gamma))$ for a fixed $3/2 < s < 2$, after which we can use equivalence of laws to conclude. 

To illustrate the main idea behind the measurability of this set of ordered pairs of continuous processes with their stopped processes, we consider an analogue of this situation in the simpler case of continuous real-valued functions in $C(0, T; \R)$. Consider the phase space $C(0, T; \R) \times C(0, T; \R)$ and for a given function $f$ and a given positive number $R > 0$, define
\begin{equation*}
\tau_{R} = \inf\{t \in [0, T] : f(t) \ge R\},
\end{equation*}
and define the stopped process
\begin{equation*}
f^{*}(t) = f(t \wedge \tau_{R}) , \quad \text{ for } t \in [0, T],
\end{equation*}
where in using the star notation, we do not notate the explicit dependence on $R$, even though it is there implicitly. Define the set $B_{R}$ to be the set
\begin{equation}\label{BR}
B_{R} := \{(f, g) \in C(0, T; \R)^{2} : g = f^{*}\},
\end{equation}
or more informally, the set of all $(f, f^{*})$ as $f$ traverses through $C(0, T; \R)$. We claim the following result about $B_{R}$:

\begin{proposition}\label{equalinlaw}
    For each $R > 0$, $B_{R}$ is a measurable subset of $C(0, T; \R)^{2}$.
\end{proposition}

\begin{proof}
The proof will use a time discretization argument, where the main idea will be to associate to each continuous function in $C(0, T; \R)$ and each time discretization parameter a finite sequence of real numbers, where we can more easily impose the ``stopped" process condition that $g = f^{*}$ for $(f, g) \in B_{R}$ for the finite sequences, and where we can impose this condition for infintely many (but countably many) time discretization parameters, to show that $B_{R}$ is a measurable set. 

Let $N$ be the number of subintervals, let $\Delta t = T/N$, and for $n = 0, 1, ..., N$, let $t_n = n\Delta t$. Consider the following continuous map $F_N: C(0, T; \R) \times C(0, T; \R) \to \R^N \times \R^N$:
\begin{small}
\begin{multline}\label{FN}
F_{N}: (f, g) \to \\
\left(\max_{t \in [0, \Delta t]} f(t), \max_{t \in [\Delta t, 2\Delta t]} f(t), ..., \max_{t \in [(N - 1)\Delta t, T]} f(t), \max_{t \in [0, \Delta t]} g(t), \max_{t \in [\Delta t, 2\Delta t]} g(t), ..., \max_{t \in [(N - 1)\Delta t, T]} g(t)\right).
\end{multline}
\end{small}
For each $N$, consider the following measurable subset $E_N$ of $\R^N \times \R^N$, defined as the union of $E_{N, m}$ for $m = 1, 2, ..., N + 1$, where
\begin{equation*}
E_{N, N + 1} = \{(a_{1}, a_{2}, ..., a_{N}, b_{1}, b_{2}, ..., b_{N}) \in \R^{N} \times \R^{N} : a_{i} = b_{i} \text{ and } b_{i} < R, \text{ for all } 1 \le i \le N\},
\end{equation*}
and for $1 \le m \le N$,
\begin{small}
\begin{equation*}
E_{N, m} = \{(a_{1}, a_{2}, ..., a_{N}, b_{1}, b_{2}, ..., b_{N}) \in \R^{N} \times \R^{N} : a_{i} = b_{i} < R \text{ for } 1 \le i < m, \ b_{i} = R \text{ for all } m \le i \le N\}.
\end{equation*}
\end{small}
Then, let $\displaystyle E_{N} = \bigcup_{m = 1}^{N + 1} E_{N, m}$. Note that for $B_{R}$ defined in \eqref{BR} and for the continuous map $F_{N}: C(0, T; \R) \times C(0, T; \R) \to \R^{N} \times \R^{N}$ defined in \eqref{FN}, we have that 
\begin{equation}\label{BRequality}
B_{R} = \bigcap_{N = 1}^{\infty} F_{N}^{-1}(E_{N}),
\end{equation}
where each $F_{N}^{-1}(E_{N})$ is a measurable subset of $C(0, T; \R) \times C(0, T; \R)$ since $E_{N}$ is measurable in $\R^N \times \R^N$ and by the continuity of $F_{N}$. This shows that $B_{R}$ is measurable in $C(0, T; \R) \times C(0, T; \R)$. 

It suffices to show that \eqref{BRequality} holds. By the definition of $B_{R}$, it is easy to show immediately that $\displaystyle B_{R} \subset \bigcap_{N = 1}^{\infty} F_{N}^{-1}(E_{N})$. To show the opposite inclusion, it suffices to show that if $(f, g) \notin B_{R}$, then $(f,g) \notin F_{N}^{-1}(E_{N})$ for some $N$. If $(f, g) \notin B_{R}$, there are two possibilities:
\begin{itemize}
\item $g$ cannot be written as a stopped process $g = h^{*}$ for $h \in C(0, T; \R)$. This can only happen if there exists $s, t \in [0, T]$ with $s < t$ such that $g(s) = R$ and $g(t) \ne R$. Thus, taking $N$ to be sufficiently large, we would see that $(f, g) \notin E_{N, m}$ for any $m$. 
\item $g$ can be written as $g = h^{*}$ for some $h$, but if $t_{0}$ is smallest time for which $g(t_0) = R$ (or $t_{0} = T$ otherwise if such a time does not exist), there exists some $0 \le t < t_0$ such that $f(t) \ne g(t)$, where $g(t) < R$. In this case, using continuity of $f$ and $g$, $(f, g) \notin E_{N, m}$ for any $m$ for some $N$ sufficiently large.
\end{itemize}
This verifies \eqref{BRequality} and concludes the proof of the claim that $B_{R}$ is measurable.

\end{proof}

\section*{Appendix B: Extension by 0}
\begin{theorem}\label{thm:extension}
    Assume that $\eta:\Gamma:=\mathbb{T}^2\to \mathbb{R}$ is an $\alpha$-H\"older continuous function such that $\sup_\Gamma|\eta| \leq L$. Let $\sO_\eta$ be the subgraph defined as $\sO_\eta:=\{(x,y,z):0<z<\eta(x,y), (x,y)\in\Gamma\}$. Now, for any $\bu\in H^s(\sO_\eta)$ define its extension by 0 as:
    \begin{align*}
        \tilde\bu &= \bu \qquad \text{in }\sO_\eta,\\
                \tilde\bu &= 0 \qquad\text{in }\mathbb{R}^3\setminus\sO_\eta.
        \end{align*}
        Then
        \begin{align}\label{extendHs}
            \|\tilde\bu\|_{H^{s\alpha}(\mathbb{R}^3)} \leq C_{s,L,\alpha}\|\bu\|_{H^s(\sO_\eta)},\qquad \forall s<\frac\alpha2.
        \end{align}
        \end{theorem}
        \begin{proof}
 For any open set $\sO\subseteq \mathbb{R}^3$ we let $\mathcal{D}(\sO)$ be the set of all smooth functions with compact support in $\sO$.     
 First we will prove that for $s<\frac12$ and any $\bu\in \mathcal{D}(\sO_\eta)$,
            \begin{align}\label{dist}
                \int_{\sO_\eta}\text{dist}(x,\partial\sO_\eta)^{-2s\alpha}|\bu(x)|^2 \leq C_{s,L,\alpha}\|\bu\|^2_{H^s(\sO_\eta)}.
            \end{align}
            To prove this we consider any $y =(y',y_3)\in \partial\sO_\eta$ and use the triangle inequality to obtain for any $x =(x',x_3)\in \sO_\eta$ that
            \begin{align*}
                |\eta(x')-x_3|& = |\eta(x')-\eta(y') + y_3-x_3| \leq C_\alpha |x'- y'|^\alpha + |x_3-y_3|\\
                &\leq (1+C_\alpha)(|x-y|^\alpha + |x-y|).
            \end{align*}
            This implies, for any $x=(x',x_3)\in\sO_\eta$, that
            \begin{align}
                \text{dist}(x,\partial\sO_\eta)^\alpha \geq \frac{|\eta(x')-x_3|}{(1+C_\alpha)(1+(\text{dist}(x,\partial\sO_\eta))^{1-\alpha})} \geq \frac1{(1+L^{1-\alpha})(1+C_\alpha)}|\eta(x')-x_3| .
            \end{align}
        Applying Lemma 3.31 of \cite{Mc00} we obtain
            \begin{align*}
                \int_{\sO_\eta}\text{dist}(x,\partial\sO_\eta)^{-2s\alpha}|\bu(x)|^2 &\leq \frac1{(1+L^{1-\alpha})(1+C_\alpha)}\int_{\sO_\eta}|\eta(x')-x_3|^{-2s}|\bu(x)|^2\\
                &= \frac{C}{(1+L^{1-\alpha})(1+C_\alpha)}\int_{\mathbb{R}^2}\int_0^\infty t^{-2s}|\bu(x',\eta(x')-t)|^2dtdx'\\
                & \leq \frac{C_s}{(1+L^{1-\alpha})(1+C_\alpha)}\int_{\mathbb{R}^2}\int_{y<\eta(x')} \int_{z<\eta(x')} \frac{|\bu(x',y)-\bu(x',z)|^2}{|x-y|^{1+2s}}dydzdx'\\
                & \leq C_{s,L,\alpha} \|\bu\|_{H^s(\sO_\eta)}
            \end{align*}
       The second to last inequality above follows using identity (1.4.4.9) in \cite{G08} and the calculations that follow this equation on page 30 (see also Theorem 3.31 in \cite{Mc00}).    
This proves \eqref{dist}.

Next, by elementary calculations we have (see e.g. Lemma 1.3.2.6 in \cite{G08} or Theorem 3.33 in \cite{Mc00}):
  \begin{align*}
 \|\tilde\bu\|_{H^s(\mathbb{R}^3)}=\left( \|\bu\|^2_{H^s(\sO_\eta)} + \int_{\sO_\eta}|\bu(x)|^2 w_{s}(x)dx \right)^\frac12,
  \end{align*}
  where the weight $\displaystyle w_s(x)=2\int_{\mathbb{R}^3\setminus \sO_\eta}\frac1{|x-y|^{3+2s}}dy$ satisfies
  $\displaystyle { |w_s(x)|\leq C\text{dist}(x,\partial\sO_\eta)^{-2s}}. $\\

Hence, thanks to \eqref{dist}, for any $\bu\in \mathcal{D}(\sO_\eta)$ we have
\begin{align*}
    \|\tilde\bu\|_{H^{s\alpha}(\mathbb{R}^3)} \leq C_{s,L,\alpha}\|\bu\|_{H^s(\sO_\eta)}.
\end{align*}
Hence, \eqref{extendHs} will follow using the following two density results:
       
\begin{itemize}
\item For any $\sO$ with $C^0$ boundary, Theorem 1.4.2.2 in \cite{G08}, states that
\begin{align*}
\mathcal{D}(\sO) \text{ is dense in } \tilde H^s(\sO):=\{\bu \in H^s(\sO):  \tilde\bu\in H^s(\mathbb{R}^3)\},\quad\forall s>0,
\end{align*}
endowed with the norm $\|\bu\|_{\tilde H^s(\sO)}=\|\tilde\bu\|_{H^s(\mathbb{R}^n)}$.
\item Moreover, if $\sO$ is a $C^{0,\alpha}$ domain, then thanks to Corollary 3.29 (viii) in \cite{CHM17}, we have that 
\begin{align*}
\mathcal{D}(\sO) \text{ is dense in } H^s(\sO)\quad\forall 0<s<\frac\alpha2.
\end{align*}
\end{itemize}

 \end{proof}

\bibliographystyle{plain}
\bibliography{CompressibleFSIBibliography}

\end{document}